\documentclass[reqno,11pt,letterpaper]{amsart}

\usepackage{lipsum}
\usepackage{amsmath}
\usepackage{amssymb}
\usepackage{amsthm}
\usepackage{mathrsfs}
\usepackage{accents}
\usepackage{calc}
\usepackage{arydshln}
\usepackage{upgreek}
\usepackage{slashed}
\usepackage{xifthen}
\usepackage{graphicx}
\usepackage{longtable}
\usepackage[inline]{enumitem}

\usepackage{xcolor}
\definecolor{winered}{rgb}{0.6,0,0}
\definecolor{lessblue}{rgb}{0,0,0.7}

\usepackage[pdftex,colorlinks=true,linkcolor=winered,citecolor=lessblue,urlcolor=lessblue,breaklinks=true,bookmarksopen=true]{hyperref}

\makeatletter
\newcommand{\myitem}[2]{\item[\rm(#2)]\def\@currentlabel{#2}\label{#1}}
\makeatother

\usepackage{titletoc}

\makeatletter

\def\@tocline#1#2#3#4#5#6#7{
\begingroup
  \par
    \parindent\z@ \leftskip#3 \relax \advance\leftskip\@tempdima\relax
                  \rightskip\@pnumwidth plus 4em \parfillskip-\@pnumwidth
    \ifcase #1 
       \vskip 0.6em \hskip 0em 
       \or
       \or \hskip 0em 
       \or \hskip 1em 
    \fi%
    #6
    \nobreak\relax{\leavevmode\leaders\hbox{\,.}\hfill}
    \hbox to\@pnumwidth {\@tocpagenum{#7}}
  \par
\endgroup
}

 \def\l@section{\@tocline{0}{0pt}{0pc}{}{}}

\renewcommand{\tocsection}[3]{%
  \indentlabel{\@ifnotempty{#2}{ 
    \ignorespaces\bfseries{#2. #3}}}
  \indentlabel{\@ifempty{#2}{\ignorespaces\bfseries{#3}}{}} 
    \vspace{1.5pt}}

\renewcommand{\tocsubsection}[3]{%
  \indentlabel{\@ifnotempty{#2}{
    \ignorespaces#2. #3}}
  \indentlabel{\@ifempty{#2}{\ignorespaces #3}{}}
    \vspace{1.5pt}}

\renewcommand{\tocsubsubsection}[3]{%
  \indentlabel{\@ifnotempty{#2}{
    \ignorespaces#2. #3}}
  \indentlabel{\@ifempty{#2}{\ignorespaces #3}{}}
    \vspace{1.5pt}}

\makeatother

\makeatletter
\def\@nomenstarted{0}
\newlength{\@nomenoldtabcolsep}

\newcommand{\nomenstart}
  {%
    \def\@nomenstarted{1}%
    \setlength{\@nomenoldtabcolsep}{\tabcolsep}%
    \setlength{\tabcolsep}{3.5pt}%
    \begin{longtable}{p{0.11\textwidth} p{0.86\textwidth}}
  }

\newcommand{\nomenitem}[2]{%
    \ifcase\@nomenstarted%
      \or 
      \or \\ 
    \fi%
    #1\,{\leavevmode\leaders\hbox{\,.}\hfill} & #2%
    \def\@nomenstarted{2}%
  }%
\newcommand{\nomenend}
  {\\%
      \end{longtable}%
      \setlength{\tabcolsep}{\@nomenoldtabcolsep}%
      \def\@nomenstarted{0}%
  }
\makeatother

\makeatletter
\newcommand{\BIG}{\bBigg@{3.5}}
\newcommand{\vast}{\bBigg@{4}}
\newcommand{\Vast}{\bBigg@{5}}
\newcommand{\VAST}[1]{\bBigg@{#1}}
\makeatother

\allowdisplaybreaks

\numberwithin{equation}{section}
\numberwithin{figure}{section}
\newtheorem{thm}{Theorem}[section]

\newtheorem{prop}[thm]{Proposition}
\newtheorem{lemma}[thm]{Lemma}
\newtheorem{cor}[thm]{Corollary}

\newtheorem*{thm*}{Theorem}
\newtheorem*{prop*}{Proposition}
\newtheorem*{cor*}{Corollary}
\newtheorem*{conj*}{Conjecture}

\theoremstyle{definition}
\newtheorem{definition}[thm]{Definition}

\theoremstyle{remark}
\newtheorem{rmk}[thm]{Remark}
\newtheorem{example}[thm]{Example}

\makeatletter
\newcommand{\fakephantomsection}{%
  \Hy@MakeCurrentHref{\@currenvir.\the\Hy@linkcounter}
  \Hy@raisedlink{\hyper@anchorstart{\@currentHref}\hyper@anchorend}%
  \Hy@GlobalStepCount\Hy@linkcounter%
}
\makeatother

\newcommand{\mc}{\mathcal}
\newcommand{\cA}{\mc A}
\newcommand{\cB}{\mc B}
\newcommand{\cC}{\mc C}
\newcommand{\cD}{\mc D}

\newcommand{\cF}{\mc F}

\newcommand{\cL}{\mc L}

\newcommand{\cO}{\mc O}

\newcommand{\cT}{\mc T}
\newcommand{\cU}{\mc U}
\newcommand{\cV}{\mc V}
\newcommand{\cW}{\mc W}

\newcommand{\ms}{\mathscr}
\newcommand{\sA}{\ms A}
\newcommand{\sB}{\ms B}
\newcommand{\sC}{\ms C}
\newcommand{\sD}{\ms D}

\newcommand{\sH}{\ms H}

\newcommand{\sS}{\ms S}

\newcommand{\sW}{\ms W}

\newcommand{\C}{\mathbb{C}}
\newcommand{\N}{\mathbb{N}}
\newcommand{\R}{\mathbb{R}}
\newcommand{\Z}{\mathbb{Z}}

\newcommand{\Sph}{\mathbb{S}}

\newcommand{\sfl}{\mathsf{l}}
\newcommand{\sfm}{\mathsf{m}}

\newcommand{\sfs}{\mathsf{s}}

\newcommand{\fB}{\mathfrak{B}}

\newcommand{\fF}{\mathfrak{F}}

\newcommand{\fM}{\mathfrak{M}}

\newcommand{\fu}{\mathfrak{u}}

\newcommand{\fw}{\mathfrak{w}}

\newcommand{\Hom}{\operatorname{Hom}}

\newcommand{\Id}{\operatorname{Id}}

\newcommand{\supp}{\operatorname{supp}}

\newcommand{\diag}{\operatorname{diag}}

\DeclareMathOperator{\ad}{ad}

\newcommand{\eps}{\epsilon}

\newcommand{\hra}{\hookrightarrow}
\newcommand{\la}{\langle}

\newcommand{\ol}{\overline}
\newcommand{\pa}{\partial}
\newcommand{\dd}{{\mathrm d}}
\newcommand{\ra}{\rangle}

\newcommand{\wh}{\widehat}
\newcommand{\wt}{\widetilde}
\newcommand{\xra}{\xrightarrow}

\newcommand{\ubar}[1]{\underaccent{\bar}#1}

\newcommand{\bop}{{\mathrm{b}}}

\newcommand{\scop}{{\mathrm{sc}}}

\newcommand{\cl}{{\mathrm{cl}}}

\newcommand{\eop}{{\mathrm{e}}}

\newcommand{\scbtop}{{\mathrm{sc}\text{-}\mathrm{b}}}

\newcommand{\semi}{\hbar}

\newcommand{\cp}{{\mathrm{c}}}

\newcommand{\Diff}{\mathrm{Diff}}

\DeclareMathOperator{\Op}{Op}

\newcommand{\Vb}{\cV_\bop}
\newcommand{\Ve}{\cV_\eop}

\newcommand{\Diffb}{\Diff_\bop}

\newcommand{\Psiscbt}{\Psi_\scbtop}

\newcommand{\Vsc}{\cV_\scop}

\newcommand{\WF}{\mathrm{WF}}
\newcommand{\Ell}{\mathrm{Ell}}
\newcommand{\Char}{\mathrm{Char}}
\newcommand{\singsupp}{\mathrm{\rm sing}\supp}
\newcommand{\esssupp}{\mathrm{ess}\supp}

\newcommand{\Tb}{{}^{\bop}T}

\newcommand{\Tsc}{{}^{\scop}T}

\newcommand{\CI}{\cC^\infty}

\newcommand{\CIc}{\cC^\infty_\cp}

\newcommand{\Hb}{H_{\bop}}

\newcommand{\openbigpmatrix}[1]
  {%
    \def\@bigpmatrixsize{#1}%
    \addtolength{\arraycolsep}{-#1}%
    \begin{pmatrix}%
  }
\newcommand{\closebigpmatrix}
  {%
    \end{pmatrix}%
    \addtolength{\arraycolsep}{\@bigpmatrixsize}%
  }

\newlength{\enummargin}

\newcommand{\usref}[1]{{\upshape\ref{#1}}}

\DeclareGraphicsExtensions{.mps}

\begin{document}

\title[Scaled bounded geometry]{Pseudodifferential operators on manifolds with scaled bounded geometry}

\date{\today}

\begin{abstract}
  We prove a general black box result which produces algebras of pseudodifferential operators (ps.d.o.s) on noncompact manifolds, together with a precise principal symbol calculus. Our construction (which also applies in parameter-dependent settings, with phase space weights and variable differential and decay orders) recovers most of the ps.d.o.\ algebras which have been introduced in recent years as tools for the microlocal analysis of non-elliptic partial differential equations. This includes those used for proving resolvent bounds (b- and scattering algebras and resolved or semiclassical versions thereof), studying waves on asymptotically flat spacetimes (\mbox{3b-}, edge-b-, and desc-algebras), inverting geodesic X-ray transforms (semiclassical foliation and 1-cusp algebras), and many others.

  Our main result rests on the novel notion of manifolds with scaled bounded geometry. A scaling encodes, in each distinguished chart of a manifold with bounded geometry, the amounts in $(0,1]$ by which the components of a uniformly bounded vector field are scaled. This decouples the regularity of the coefficients of elements of the resulting Lie algebra $\cV$ of vector fields from the pointwise size of their coefficients. When the scaling tends to $0$ at infinity, the approximate constancy of coefficients of elements of $\cV$ on increasingly large cubes, as measured using $\cV$, gives rise to a principal symbol which captures $\cV$-operators modulo operators of lower differential order and more decay.
\end{abstract}


\author{Peter Hintz}
\address{Department of Mathematics, ETH Z\"urich, R\"amistrasse 101, 8092 Z\"urich, Switzerland}
\email{peter.hintz@math.ethz.ch}

\maketitle

\section{Introduction}
\label{SI}

We prove a black box result which produces algebras of pseudodifferential operators (ps.d.o.s) on noncompact manifolds, together with a precise (principal) symbol calculus, in a large variety of settings including the semiclassical calculus, b- and scattering calculi, and many others. This result obviates the need to construct, by hand, ps.d.o.\ algebras associated with (structured) singularities or ends, as long as it is only the \emph{local} and \emph{symbolic} behavior of ps.d.o.s that one cares about. Global effects, as often arise in the inversion of elliptic partial differential operators, are not addressed here; neither are normal operators.

The noncompact manifolds of interest are equipped with two pieces of data:
\begin{enumerate}
\item a bounded geometry (b.g.) structure---roughly speaking, a cover by unit cube coordinate charts with uniformly bounded transition functions. (Metrics are not involved, cf.\ Remark~\ref{RmkIBddMet}.)
\item Scalings---roughly speaking, a subdivision of each unit cube into cuboids aligned with the local coordinate axes, in a manner compatible across unit cubes.
\end{enumerate}
The differential operators of interest are then built from vector fields which have bounded size in each cuboid, but whose coefficients vary only on the scale of the unit cubes. In particular, when the cuboids become very small, such operators have roughly constant coefficients on increasingly large collections of cuboids, which gives rise to a further principal symbol which, in a certain sense, captures the limiting constant coefficient models. The reader may jump directly to Definitions~\ref{DefIBdd} and \ref{DefISB}, Figure~\ref{FigISB}, and Theorem~\ref{ThmISBRough} for details.

\bigskip

We first present some broader context for the present paper. In the analysis of partial differential equations (PDE), pseudodifferential operators typically serve one of two purposes.
\begin{enumerate}
\item\label{ItIPx} Precise approximate inverses (parametrices), or true inverses, of \emph{elliptic} (pseu\-do)dif\-fer\-en\-tial operators can often be constructed within a calculus of pseudodifferential operators tailored to the partial differential operator at hand.
\item\label{ItITool} One can use ps.d.o.s as \emph{tools} to microlocalize the study of the PDE in phase space. This is particularly convenient for applications to \emph{non-elliptic} PDE (whose parametrices are more complicated objects, such as paired Lagrangian distributions \cite{MelroseUhlmannLagrangian}).
\end{enumerate}
For PDE on compact manifolds without boundary, the standard ps.d.o.\ calculus \cite{HormanderAnalysisPDE3} serves both purposes. For PDE on noncompact (and also on singular) spaces $M$ however, ps.d.o.\ calculi appropriate for the first purpose are in general significantly more delicate than those for the second: Schwartz kernels of inverses, or of parametrices precise enough to imply Fredholm properties, often have nontrivial behavior far from the diagonal $\diag_M\subset M\times M$, cf.\ the Schwartz kernel $\R^3\times\R^3\ni(x,x')\mapsto(4\pi|x-x'|)^{-1}$ of $\Delta_{\R^3}^{-1}$. A well-established method to encode the (off-diagonal) behavior of Schwartz kernels, originating in Melrose's work \cite{MelroseTransformation,MelroseMendozaB,MelroseAPS}, is to describe them as distributions on a appropriate resolutions (iterated blow-ups) of the double space $\bar M\times\bar M$ where $\bar M$ is a compactification of $M$ to a manifold with corners.

By contrast, when using ps.d.o.s as tools, only the behavior of their Schwartz kernels near the diagonal $\diag_M\subset M\times M$, including in a suitable uniform sense near infinity in $\diag_M$, matters, and one can attempt to use a \emph{bounded geometry} approach to encode uniformity near infinity without introducing compactifications and boundaries or corners.

\begin{definition}[Bounded geometry structure]
\label{DefIBdd}
  (Cf.\ \cite{ShubinBounded}.) Let $M$ be a smooth $n$-dimensional manifold (without boundary). Then a \emph{bounded geometry structure} (b.g.\ structure) on $M$ is a set $\{(U_\alpha,\phi_\alpha)\colon\alpha\in\sA\}$ where $\sA$ is a countable index set and for each $\alpha\in\sA$, $U_\alpha$ is an open subset of $M$ and $\phi_\alpha\colon U_\alpha\to (-2,2)^n$ is a diffeomorphism; and
  \begin{enumerate}
  \item\label{ItIBddFinite} there exists $A\in\N$ so that for all pairwise distinct $\alpha_1,\ldots,\alpha_{A+1}\in\sA$, we have $\bigcap_{a=1}^{A+1} U_{\alpha_a}=\emptyset$;
  \item $M=\bigcup_{\alpha\in\sA} \phi_\alpha^{-1}((-1,1)^n)$;
  \item\label{ItIBddTrans} the transition functions $\tau_{\beta\alpha}:=\phi_\beta\circ\phi_\alpha^{-1}\colon\phi_\alpha(U_\alpha\cap U_\beta)\to\phi_\beta(U_\alpha\cap U_\beta)$ are uniformly bounded in $\CI$, i.e.\ for all $k\in\N_0$ there exists $C_k<\infty$ so that $\|\tau_{\beta\alpha}^j\|_{\cC^k(\phi_\alpha(U_\alpha\cap U_\beta))}\leq C_k$ for all $\alpha,\beta$, where $\tau_{\beta\alpha}^j$ denotes the $j$-th component ($j=1,\ldots,n$) of $\tau_{\beta\alpha}$, and $\|\cdot\|_{\cC^k}$ is the maximum of the supremum norms of a function and its up to $k$-fold coordinate derivatives.
  \end{enumerate}
\end{definition}

We shall frequently refer to the sets $U_\alpha$ as \emph{unit cells} or \emph{distinguished charts}.

\begin{rmk}[Metrics]
\label{RmkIBddMet}
  The more common definition of a \emph{manifold with bounded geometry} is that of a Riemannian manifold $(M,g)$ with positive injectivity radius for which the Riemann curvature tensor and all of its covariant derivatives are uniformly bounded \cite{CheegerGromovTaylorComplete}, \cite[Appendix~1]{ShubinBounded}. That this gives rise to a b.g.\ structure is shown in \cite[pp.\ 63--66]{ShubinBounded} (using balls instead of cubes in Definition~\ref{DefIBdd}, which is easily seen to lead to an equivalent notion); see also \cite{ElderingUniform}. Conversely, given a b.g.\ structure, a metric of bounded geometry can be constructed as $g:=\sum_\alpha \phi_\alpha^*(\chi\,\dd x^2)$ where $\chi\in\CIc((-2,2)^n)$ equals $1$ on $(-1,1)^n$, and $\dd x^2$ denotes the Euclidean metric. For our purposes however, Riemannian structures are irrelevant.
\end{rmk}

We define the spaces
\[
  \CI_{{\rm uni},\fB}(M)\quad\text{resp.}\quad \CI_{{\rm uni},\fB}(M;T M)
\]
of functions and vector fields which, resp.\ whose coefficients, have uniformly bounded $\CI$ seminorms in the charts $U_\alpha$. Uniform (or bounded geometry) differential operators are then finite sums of up to $m$-fold compositions of elements of $\CI_{{\rm uni},\fB}(M;T M)$, with a $0$-fold composition defined to be multiplication with an element of $\CI_{{\rm uni},\fB}(M)$. Pseudodifferential operators (ps.d.o.s) on $M$, which act boundedly on $\CI_{{\rm uni},\fB}(M)$, can be defined by patching together standard quantizations on $\R^n$ (see~\eqref{EqIScOp} below) in the charts $U_\alpha$.

\emph{However}, the bounded geometry perspective is too imprecise and inflexible to capture crucial properties of many standard algebras of (pseudo)differential operators. For example, it cannot capture the typical regularity of coefficients of differential operators on $\R^n$ (cf.\ \eqref{EqIScReg}), nor can (a parameterized version of) it capture the principal symbol in the semiclassical calculus (see~\S\ref{SssIh}). We explain these shortcomings in some detail in~\S\ref{SsIBad}.

In this paper, we rectify these shortcomings via the introduction of \emph{(parameterized) scaled bounded geometry structures}; see~\S\ref{SsISB}. We recover all those pseudodifferential algebra\footnote{At least all those known to the author!} on manifolds with corners (possibly allowing for parameter dependencies) which are characterized via Schwartz kernels on resolved double spaces, including the associated sharp principal symbol maps. More precisely, we recover the \emph{very small} algebras consisting of operators whose Schwartz kernels are supported in a collar neighborhood of the lifted diagonal, and which have conormal coefficients down to the boundary. In particular, by using our general construction, it becomes obsolete to define double and triple spaces \cite{MelroseICM,MelroseDiffOnMwc} for the purpose of constructing novel ps.d.o.\ algebras---provided one only needs this algebra in the sense of~\eqref{ItITool} above. The examples from~\S\ref{SsIBad} are revisited from our new perspective in~\S\S\ref{SssISc2}--\ref{SssIh2}. Further examples are discussed in~\S\S\ref{SssIAniso}--\ref{SssIF}.

\subsection{Shortcomings of the bounded geometry perspective}
\label{SsIBad}

We give two examples to motivate our main result.

\subsubsection{Operators on Euclidean space}
\label{SssISc}

A basic example is $M=\R^n$, with bounded geometry structure
\begin{equation}
\label{EqIScBdd}
  \sA=\Z^n,\quad
  U_\alpha=\alpha+(-4,4)^n,\quad
  \phi_\alpha\colon U_\alpha\ni x\mapsto (x-\alpha)/2\in(-2,2)^n.
\end{equation}
One can associate a scale of Sobolev spaces and algebras of (pseudo)differential operators with a b.g.\ structure. In the case of $\R^n$, these are the standard Sobolev spaces $H^s(\R^n)$ and (uniform) ps.d.o.s $\Psi^m(\R^n)$, defined as quantizations
\begin{equation}
\label{EqIScOp}
  \Op(a)(x,x') := (2\pi)^{-n}\int e^{i(x-x')\xi} a(x,\xi)\,\dd\xi
\end{equation}
of symbols $a\in S^m(\R^n;\R^n)$ satisfying the uniform bounds
\[
  |\pa_x^\beta\pa_\xi^\gamma a(x,\xi)|\leq C_{\beta\gamma}\la\xi\ra^{m-|\gamma|}\qquad\forall\,\beta,\gamma\in\N_0^n,\ x\in\R^n,\ \xi\in\R^n,\quad \la\cdot\ra:=(1+|\cdot|^2)^{1/2}.
\]
(One may in addition enforce that $|x-x'|$ is uniformly bounded on $\supp\Op(a)$ by inserting a cutoff $\chi(x-x')$, $\chi\in\CIc(\R^n)$, in~\eqref{EqIScOp}.)

Consider now the shifted Laplacian $L:=\Delta_{\R^n}+1=\sum_{j=1}^n D_{x^j}^2+1$ (where $D=\frac{1}{i}\pa$), which is a uniformly elliptic operator; or more generally a variable coefficient version $L=\sum_{i,j=1}^n g^{i j}(x)D_{x^i}D_{x^j}+1$ where $(g^{i j})$ is everywhere positive definite and, for some $\eta>0$,
\begin{equation}
\label{EqIScReg}
  |\pa_x^\beta(g^{i j}(x)-\delta^{i j})|\leq C_\eta\la x\ra^{-\eta-|\beta|}\qquad\forall\,\beta\in\N_0^n.
\end{equation}
Using the spaces arising from the b.g.\ structure on $\R^n$, one can show elliptic estimates
\[
  \|u\|_{H^s} \leq C_{s,N}\bigl( \|L u\|_{H^{s-2}} + \|u\|_{H^{-N}}\bigr).
\]
However, ellipticity considerations \emph{within the class of uniform ps.d.o.s} are not precise enough to establish the Fredholm property of $L$. What is missing is a record of the fact that the coefficients of $L$, near any sequence of points $p_i\to\infty$, are essentially constant on \emph{increasingly large cubes} of side length $\frac12|p_i|$ (which are thus unions of an increasingly large number of unit cells);\footnote{This allows for an approximate inversion of $L$ near infinity via the inversion of its constant coefficient models---here the single model $\Delta_{\R^n}+1$---by means of the (inverse) Fourier transform.} put differently, the regularity of the coefficients of $L$ is \emph{stronger} than mere uniform boundedness, with all derivatives, in the charts $(U_\alpha,\phi_\alpha)$. This cannot be encoded only using the b.g.\ structure~\eqref{EqIScBdd}; we shall, however, be able to capture this using our notion of \emph{scaled bounded geometry structures}.\footnote{The constant coefficient models then appear in the guise of a more precise principal symbol which, in the present setting, recovers the boundary principal symbol of the scattering calculus.} Only this more precise principal symbol is precise enough to distinguish $L=\Delta+1$ from the operators $\Delta$ or $\Delta-1$ whose mapping properties are dramatically different. (The operator $L$ considered here can of course be easily analyzed using integration by parts; our aim here being the development of a general-purpose microlocal tool, we pay no attention to this special feature.)

In the concrete setting of $\R^n$, one can encode this higher regularity directly in the symbol class; the resulting symbols are then called \emph{scattering symbols} \cite{VasyMinicourse}, and their quantizations are \emph{scattering pseudodifferential operators}, first introduced under the strongest possible notion of regularity near infinity (smoothness in $|x|^{-1},\frac{x}{|x|}$ for $|x|>1$) by Melrose \cite{MelroseEuclideanSpectralTheory} and generalized in \cite{MazzeoMelroseFibred}. We give detailed comparison with the geometric microlocal perspective in~\S\ref{SsIRel}.

\subsubsection{Semiclassical operators}
\label{SssIh}

We now consider semiclassical ps.d.o.s
\begin{equation}
\label{EqIhOp}
  \Op_h(a)(x,x') := (2\pi h)^{-n}\int e^{i(x-x')\xi/h} a(h,x,\xi)\,\dd\xi
\end{equation}
on $\R^n$, where $h\in(0,1)$ is the semiclassical parameter; these are used to study semiclassical operators, such as $L_h=h^2\Delta_{\R^n}+V-E$ where $V=V(x)$ is a smooth potential and $E\in\R$, in a uniform fashion as $h\to 0$ \cite{ZworskiSemiclassical}.

We first attempt to recover this class of operators using a bounded geometry perspective with parameter $h$: for $h\in(0,1)$, we define the set $U_{h,\alpha}=h(\alpha+(-4,4)^n)$ of size $\sim h$, together with the chart $\phi_{h,\alpha}(x)=(h^{-1}x-\alpha)/2$. The corresponding class of uniform ps.d.o.s then contains operators defined as quantizations
\begin{equation}
\label{EqIhOpBdd}
  (2\pi h)^{-n}\int e^{i\frac{x-x'}{h}\xi} a\Bigl(h,x,\xi\Bigr)\chi\Bigl(\frac{x-x'}{h}\Bigr)\,\dd\xi
\end{equation}
where $\chi\in\CIc((-4,4)^n)$ equals $1$ on $(-2,2)^n$. The cutoff $\chi$ localizes $x,x'$ to essentially the same unit cell $U_{h,\alpha}$. Consider the simplest case of multiplication operators, so $a=a(h,x)$: then the function $a$ on $(0,1]_h\times\R^n_x$ is uniformly smooth with respect to this \emph{parameterized} bounded geometry structure if and only if is uniformly bounded upon applying any number of vector fields $(\phi_{h,\alpha})^*\pa_x\sim h\pa_x$; thus, $a$ is allowed to oscillate on spatial scales $\sim h$.

Typical semiclassical operators (such as $L_h$ above) however have coefficients which are smooth in $x$ uniformly as $h\to 0$. Put differently, the coefficients are essentially constant on \emph{increasingly large cubes}, i.e.\ unions of an increasingly large number $\sim h^{-1}\times\cdots\times h^{-1}=h^{-n}$ of unit cells $U_{h,\alpha}$. This stronger regularity is responsible for the crucial fact that semiclassical ps.d.o.s defined by~\eqref{EqIhOp} have a principal symbol which is well-defined modulo symbols \emph{with extra decay} as $h\to 0$ \cite[\S{9.3.3}]{ZworskiSemiclassical}.

\subsection{Scaled bounded geometry structures; main result}
\label{SsISB}

The discussion of the regularity of coefficients in~\S\S\ref{SssISc}--\ref{SssIh} might suggest the introduction of a structure on top of a b.g.\ structure which describes the increasingly large collections of unit cells on which coefficients of vector fields and operators should be roughly constant. However, in order to avoid having to describe how increasingly large numbers of unit cells fit together, it is more convenient to set this up the other way around. That is, we elevate the b.g.\ structure $\fB$ which gives rise to the desired notion of smoothness of coefficients to the central object. The class of vector fields and operators of interest, which thus shall have these more regular coefficients, then arises from vector fields whose coefficients, in the unit cells of $\fB$, are rescaled by amounts in $(0,1]$ in the local coordinate directions.

Moreover, while in the examples of~\S\ref{SsIBad} the coefficients of naturally arising operators are essentially constant on cubes whose size increases by the same amount in all directions (as $|x|\to\infty$ or $h\to 0$), it is important to allow for more general anisotropic scalings (see~\S\ref{SssIAniso} for an example). We thus arrive at the following notion.

\begin{definition}[Scaled bounded geometry structure]
\label{DefISB}
  Let $M$ be a smooth $n$-dimensional manifold. Then a \emph{scaled bounded geometry structure} (\emph{scaled b.g.\ structure} for short) on $M$ is a set $\fB_\times=\{(U_\alpha,\phi_\alpha,\rho_\alpha)\colon\alpha\in\sA\}$ so that $\fB=\{(U_\alpha,\phi_\alpha)\colon \alpha\in\sA\}$ is a b.g.\ structure on $M$ (called the \emph{underlying b.g.\ structure}), and $\rho_\alpha=(\rho_{\alpha,i})_{i=1,\ldots,n}\in(0,1]^n$ for each $\alpha\in\sA$, with the property that there exist constants $C_\gamma<\infty$ for $\gamma\in\N_0^n$ so that for all $\alpha,\beta$,
  \begin{equation}
  \label{EqISB}
    | \pa^\gamma(\rho_{\alpha,i}\pa_i\tau_{\beta\alpha}^j(x))| \leq C_\gamma\rho_{\beta,j}\qquad\forall\,x\in (-2,2)^n,\quad i,j=1,\ldots,n,\ \ \gamma\in\N_0^n.
  \end{equation}
  (We call $\{\rho_\alpha\}$ a \emph{scaling}.) Given $\fB_\times$, we further define:\footnote{That $\cW,\cV,\cV'$ are indeed Lie algebras, as their name indicates, is shown in~\S\ref{SsSBComp}.}
  \begin{enumerate}
  \item $\bar\rho_\alpha:=\max_{1\leq i\leq n}\rho_{\alpha,i}$;
  \item the \emph{coefficient Lie algebra} $\cW:=\CI_{\rm uni,\fB}(M;T M)$;
  \item the \emph{operator Lie algebra} $\cV\subset\cW$ consisting of all $V\in\cW$ so that, upon writing
    \begin{equation}
    \label{EqISBV}
      (\phi_\alpha)_*V = \sum_{i=1}^n V_\alpha^i\rho_{\alpha,i}\pa_i,\qquad V_\alpha^1,\ldots,V_\alpha^n\in\CI((-2,2)^n),
    \end{equation}
    the functions $V_\alpha^1,\ldots,V_\alpha^n$ are uniformly bounded with all derivatives;
  \item the \emph{large operator Lie algebra} $\cV'$ consisting of all smooth vector fields $V$ on $M$ so that the components $V_\alpha^i$ in each coordinate chart are uniformly (in $\alpha$) bounded in absolute value, and so are the derivatives $(\rho_{\alpha,1}\pa_1)^{\gamma_1}\cdots(\rho_{\alpha,n}\pa_n)^{\gamma_n}((\phi_\alpha)_*V)^j$ for any fixed $\gamma\in\N_0^n$.
  \end{enumerate}
\end{definition}

Bounded geometry structures arise as the special case when all $\rho_{\alpha,i}$ are equal to $1$ (or any other constant $>0$): then $\cV'=\cV=\cW$. The crucial condition~\eqref{EqISB} means that the vector fields $\rho_{\alpha,i}\pa_i$ push forward under $\tau_{\beta\alpha}$ to linear combinations of $\rho_{\beta,j}\pa_j$ with uniformly bounded (with all derivatives) coefficients. Put differently (for $\gamma=0$), the coordinate change maps $\tau_{\beta\alpha}$ map $\rho_\alpha$-cuboids (i.e.\ cuboids with side lengths $\rho_{\alpha,i}$ in the $x^i$-direction) into sets which contain a $C^{-1}\rho_\beta$-cuboid and are contained in a $C\rho_\beta$-cuboid where $C$ is independent of $\alpha,\beta$. See Figure~\ref{FigISB}.

\begin{figure}[!ht]
\centering
\includegraphics{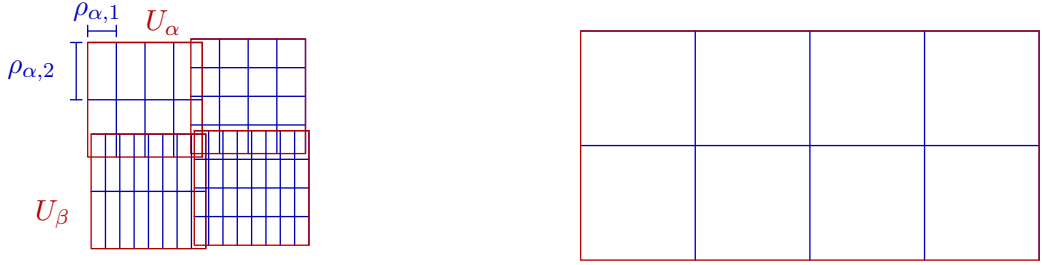}
\caption{\textit{On the left:} Four unit cells for $\cW$ are drawn in red. The smaller cuboids of side lengths $\rho_{\alpha,i}$ are drawn in blue. \textit{On the right:} the top left unit cell from the left, rescaled so that the little cuboids (i.e.\ the unit cells for $\cV'$) appear as unit squares.}
\label{FigISB}
\end{figure}

We denote by
\[
  \Diff_\cV^m(M)
\]
the space of finite sums of up to $m$-fold compositions of elements of $\cV$; $0$-fold compositions are defined to be multiplication operators by elements of $\CI_{{\rm uni},\fB}(M)$. One may regard $\Diff_\cV(\R^n)$ as making rigorous the notion of a space of $\cV'$-differential operators with $\cW$-regular coefficients; see for example the discussion following~\eqref{EqISc2Vp} below. More generally, we define spaces $\rho^l\Diff_\cV^m(M)$ of weighted operators; here $P\in\rho^l\Diff_\cV^m(M)$ if and only if $(\phi_\alpha)_*P=\sum_{|\gamma|\leq m} p_{\alpha\gamma}(x)(\rho_{\alpha,1}\pa_{x^1})^{\gamma_1}\cdots(\rho_{\alpha,n}\pa_{x^n})^{\gamma_n}$ where the weighted coefficients $\bar\rho_\alpha^{-l}p_{\alpha\gamma}$ are uniformly bounded in $\CI$.

To describe principal symbols of $\cV$-operators, we define $\rho^l S^m({}^\cV T^*M)$ for $l,m\in\R$ to consist of all smooth functions $p\in\CI(T^*M)$ with the following property: defining the local coordinate representations
\[
  p_\alpha\colon (-2,2)^n\times\R^n\ni (x,\xi) \mapsto p\biggl(x,\sum_{i=1}^n \xi_i\frac{\dd x^i}{\rho_{\alpha,i}}\biggr),
\]
the symbol seminorms $\bar\rho_\alpha^{-l}\sup \la\xi\ra^{m-|\gamma|}|\pa_x^\beta\pa_\xi^\gamma p_\alpha|$ are, for each $\beta,\gamma\in\N_0^n$, uniformly bounded in $\alpha$. We write $\rho^l P^m({}^\cV T^*M)\subset\rho^l S^m({}^\cV T^*M)$ for the subspace of symbols which on each fiber of $T^*M$ are polynomials of degree $m$.

\begin{thm}[Main result, rough and abridged version]
\label{ThmISBRough}
  There exists a well-defined surjective principal symbol map
  \[
    \upsigma_\cV^m \colon \Diff_\cV^m(M) \to P^m({}^\cV T^*M) / \rho P^{m-1}({}^\cV T^*M),
  \]
  with kernel $\rho\Diff_\cV^{m-1}(M)$, similarly for weighted operators. This map is multiplicative in the sense that
  \[
    \upsigma_\cV^{m_1+m_2}(A_1\circ A_2) = \upsigma_\cV^{m_1}(A_1) \cdot \upsigma_\cV^{m_2}(A_2),\qquad A_j\in\Diff_\cV^{m_j}(M),\ j=1,2.
  \]
  Moreover, there exists a space
  \[
    \Psi_\cV^m(M) \supset \Diff_\cV^m(M)
  \]
  of pseudodifferential operators `microlocalizing' $\Diff_\cV^m(M)$, together with a surjective principal symbol map
  \begin{equation}
  \label{EqISBRoughSymb}
    \upsigma_\cV^m \colon \Psi_\cV^m(M) \to S^m({}^\cV T^*M) / \rho S^{m-1}({}^\cV T^*M)
  \end{equation}
  with kernel $\rho\Psi_\cV^{m-1}(M)$, and a quantization map
  \[
    \Op_\cV\colon S^m({}^\cV T^*M)\to\Psi_\cV^m(M)
  \]
  which maps $1$ to the identity operator and which is surjective modulo an appropriate space $\rho^\infty\Psi_\cV^{-\infty}(M)$ of residual operators; similarly for weighted operators and symbols. The space $\Psi_\cV(M)=\bigcup_{m\in\R}\Psi_\cV^m(M)$ is a graded algebra under composition, and the principal symbol map is multiplicative; this generalizes also to $\bigcup_{m,l\in\R}\rho^l\Psi_\cV^m(M)$. The principal symbol of the commutator of two operators is given by the Poisson bracket of their principal symbols.
\end{thm}

See Theorem~\ref{ThmSBDSymb} for the case of differential operators, and Definition~\ref{DefSBPsdoV} (based on Definitions~\ref{DefSBPsdo} and \ref{DefSBPsdoRes}) for the definition of $\Psi_\cV^m(M)$ (and weighted versions $w\Psi_\cV^m(M)$ thereof) and the principal symbol map. The algebra and symbolic properties of (weighted versions of) $\Psi_\cV(M)$ is the content of Theorem~\ref{ThmSBPsdoComp}. We define ${}^\cV T^*M$ as a vector bundle \emph{with distinguished trivializations} in Definition~\ref{DefSBDCotgt}; it is isomorphic to $T^*M$ as a smooth vector bundle (but, for general scalings, not in the category of vector bundles of bounded geometry). Furthermore,
\begin{enumerate}
\item we define associated scales of \emph{weighted $\cV$-Sobolev spaces} (Definition~\ref{DefSBH}) and obtain the boundedness of $\cV$-ps.d.o.s\ between them (Theorem~\ref{ThmSBPsdoH});
\item we prove the Fredholm property of elliptic operators when $\cV$ is \emph{fully symbolic}, which means that $\bar\rho_\alpha$ tends to $0$ as $\alpha$ leaves every finite subset of $\sA$, i.e.\ as $U_\alpha$ leaves every compact subset of $M$ (see Definition~\ref{DefSBTsSymbolic}). This implies that the principal symbol controls the mapping properties of an elliptic $\cV$-ps.d.o.\ modulo compact errors (see also Theorem~\ref{ThmBHRellich} for the relevant Rellich compactness theorem);
\item we introduce the \emph{$\cV$-wave front set} (Definition~\ref{DefSBWF});
\item operators acting on sections of \emph{uniform vector bundles} are discussed in Remarks~\ref{RmkSBDBundles} and~\ref{RmkSBPsdoVB};
\item we define a locally convex topology on $\Psi_\cV^m(M)$ and its weighted versions so that $\Op_\cV$ is continuous, composition of $\cV$-ps.d.o.s is a continuous bilinear operation, and operator norms on weighted Sobolev spaces are bounded by $\Psi_\cV^m(M)$-seminorms; see Remark~\ref{RmkSBPsdoTop}.
\end{enumerate}

\begin{rmk}[Lie algebra properties of $\cV$]
\label{RmkISBLie}
  We have $[\cV,\cV]\subseteq\rho\cV$ (rather than just $[\cV,\cV]\subseteq\cV$), as follows from a local coordinate computation using~\eqref{EqISBV}. Thus, Theorem~\ref{ThmISBRough} justifies, in considerable generality, the heuristic that improved commutation properties of the Lie algebra $\cV$ imply the existence of an improved principal symbol~\eqref{EqISBRoughSymb} for the associated pseudodifferential calculus.
\end{rmk}

Roughly speaking, elements of $\Psi_\cV^m(M)$ are defined as sums over $\alpha\in\sA$ of quantizations in local charts which are of the form $\Op_\alpha(a_\alpha)=a_\alpha(x^1,\ldots,x^n,\rho_{\alpha,1}D_1,\ldots,\rho_{\alpha,n}D_n)$ for symbols $a_\alpha=a_\alpha(x,\xi)$ which are uniformly bounded in standard symbol classes on $\R^n$; that is,
\[
  (\Op_\alpha(a_\alpha)u)(x) = (2\pi)^{-n}\int_{\R^n\times\R^n} \exp\biggl(i\sum_{j=1}^n (x^j-x^{\prime j})\frac{\xi_j}{\rho_{\alpha,j}}\biggr) a_\alpha(x,\xi)u(x')\,\dd\xi\,\frac{\dd x^{\prime 1}}{\rho_{\alpha,1}}\cdots\frac{\dd x^{\prime n}}{\rho_{\alpha,n}}
\]
(omitting cutoffs in $x-x'$). Suitable \emph{residual operators} must be added to capture the off-diagonal behavior of $\cV$-ps.d.o.s; we consider only operators whose Schwartz kernels are supported a finite distance (as measured by the number of unit cells $U_\alpha$ one must traverse to get from $p$ to $q$ when $(p,q)\in M\times M$ lies in the support of the Schwartz kernel) away from the diagonal in $M\times M$. The definition of a workable class of residual operators turns out to be somewhat subtle: we characterize them through their mapping properties ($\cV$-smoothing, weight-improving, but $\cW$-regularity-preserving) rather than their Schwartz kernels.

\begin{rmk}[Bounded geometry ps.d.o.s]
\label{RmkISBBdd}
  The simplest example of a scaled b.g.\ structure is given by a b.g.\ structure together with the scalings $\rho_{\alpha,i}=1$ for all $\alpha,i$. In this case, Theorem~\ref{ThmISBRough} recovers the usual properties of bounded geometry (pseudo)differential operators \cite{ShubinBounded}.
\end{rmk}

\begin{rmk}[Mixed $\cV$-$\cW$-Sobolev spaces]
\label{RmkIMixed}
  It may happen that solutions of a PDE $P u=f$, where $P$ is a $\cV$-differential operator, have stronger regularity than what is immediately accessible using $\cV$-tools, namely regularity with respect to applications of elements of $\cW\supsetneq\cV$. Examples include the regularity at infinity of $(\Delta-\lambda^2)^{-1}f$ on asymptotically Euclidean spaces \cite[\S{12}]{MelroseEuclideanSpectralTheory} or on asymptotically hyperbolic spaces \cite{GrahamZworskiScattering,ZworskiRevisitVasy} (for compactly supported $f$, say), and the regularity of waves on asymptotically de~Sitter \cite{VasyWaveOndS}, anti-de~Sitter \cite{VasyWaveOnAdS}, and asymptotically flat and asymptotically stationary spacetimes \cite{HintzNonstat}. In such instances, it can be of interest to be able to work $\cV$-microlocally but on function spaces which encode additional integer order $\cW$-regularity. Such mixed function spaces $H_{\cV;\cW}^{(s;k)}$ are discussed in the present general setting in Definition~\ref{DefSBHVW}. Theorem~\ref{ThmSBPsdoH} shows that $\cV$-ps.d.o.s define bounded maps between such mixed function spaces (with the $\cW$-regularity order $k$ being the same for the input and output spaces). In the literature, proofs of this statement in special cases required the introduction of an algebra of $\cW$-differential $\cV$-pseudodifferential operators. This is not necessary in our approach. Instead, we check this directly for quantizations in local coordinates; see the proof of Proposition~\ref{PropSBPsdoH}.
\end{rmk}

\begin{rmk}[Microlocalization locus]
\label{RmkISBMicro}
  Defining a suitable notion of elliptic sets and (operator) wave front sets, and more generally the locus $\fM$ of microlocalization which captures where the $\cV$-algebra is commutative to leading order, is rather delicate; for example, if in~\eqref{EqISBRoughSymb} the weight $\rho$ has infimum equal to $0$, then the $\cV$-algebra is commutative to leading order at `$\rho=0$'. While one can easily make sense of this if a compactification of $M$ (or more precisely ${}^\cV T^*M$) to a reasonable space, e.g.\ a manifold with corners, is specified (in which case $\{\rho=0\}$ may be a boundary hypersurface), there typically exist many such compactifications. See~\S\ref{SsSBCpt}. In order to avoid requiring arbitrary choices, we take an abstract approach and work with a universal compactification of the uniform space $M$ called the \emph{uniform} or \emph{Samuel compactification} \cite{SamuelCompactification}, defined as the spectrum of the *-algebra of uniformly continuous and bounded functions on $M$; this maps onto every other compactification by Lemma~\ref{LemmabgUniv}. Then the \emph{compact} set $\fM$ consists of points in the boundary of the uniform compactification of ${}^\cV T^*M$ which lie at fiber infinity or where $\rho=0$; see Definition~\ref{DefSBTsMicro}. The compactness of $\fM$ is conceptually crucial, as local control of a distribution near each point in $\fM$ automatically gives uniform control modulo distributions with higher regularity and more decay; see e.g.\ Proposition~\ref{PropSBWF}\eqref{ItSBWFEmpty}. --- The price to pay for working with a universal compactification is that $\fM$ does not have an explicit description and has very large cardinality. (See also the discussion preceding Proposition~\ref{PropbgAlt}.)
\end{rmk}

In order to capture parameter-dependent settings, we introduce:

\begin{definition}[Parameterized scaled bounded geometry structure]
\label{DefIPSB}
  Let $M$ be a smooth $n$-dimensional manifold, and let $P$ be a set. Then a \emph{parameterized scaled bounded geometry structure} on $(M,P)$ is a collection $\{\fB_{p,\times}\}$ of scaled bounded geometry structures $\fB_{p,\times}=\{(U_{\alpha_p},\phi_{\alpha_p},\rho_{\alpha_p})\colon\alpha_p\in\sA_p\}$ on $M$ for each $p\in P$ so that the quantities $A$ and the $\CI$ bounds on transition functions in Definition~\usref{DefIBdd} as well as the constants $C_\gamma$ in~\eqref{EqISB} are uniform also in $p\in P$. We define \emph{parameterized bounded geometry structures} analogously.
\end{definition}

This includes scaled b.g.\ structures as the special case when $P$ is a singleton set. Furthermore, it includes semiclassical bounded geometry settings when $P=(0,1]_h$ if all $\fB_{h,\times}$ feature the same b.g.\ structure $\{(U_\alpha,\phi_\alpha)\}$, and the scalings are $\rho_{\alpha_h,i}=h$ for $h\in(0,1]$.

The generalization of Theorem~\ref{ThmISBRough} for parameterized scaled b.g.\ structures, discussed in~\S\ref{SsSBPar}, asserts the existence of an algebra $\Psi_\cV(M)$ of $\cV$-pseudodifferential operators with the usual properties (principal symbol calculus, boundedness on weighted Sobolev spaces). The key point is that such operators are \emph{families}
\[
  (A_p)_{p\in P} \in \Psi_\cV(M)
\]
of operators $A_p\in\Psi_{\cV_p}(M)$ which are quantizations of symbols on ${}^{\cV_p}T^*M$ obeying uniform (in the parameter $p\in P$) bounds; here $\cV_p$ is the operator Lie algebra associated with $\fB_{p,\times}$. Similarly, the relevant Sobolev spaces are $p$-dependent function spaces (corresponding to the scaled b.g.\ structure $\fB_{p,\times}$), and the boundedness of ps.d.o.s is understood as \emph{uniform} (in $p$) boundedness. In the semiclassical bounded geometry setting, this recovers to a large extent the definition of Bahuaud--Guenther--Isenberg--Mazzeo \cite[\S{4}]{BahuaudGuentherIsenbergMazzeoBddGeoWP} of semiclassical operators on manifolds with bounded geometry; the main difference is that our operators do not feature any regularity in the parameter $h$.

\medskip

The following further topics are addressed in~\S\ref{SF}:
\begin{enumerate}
\item \S\ref{SsFWg}: operators and spaces corresponding to certain phase space weights, which for example allows us to recover Vasy's second microlocal b-scattering algebra \cite{VasyLAPLag};
\item \S\ref{SsFVar}: variable regularity (and decay, when $\inf_\alpha\bar\rho_\alpha=0$) orders, with important applications in scattering theory on asymptotically conic spaces \cite{MelroseEuclideanSpectralTheory,VasyMinicourse};
\item \S\ref{SsFFT}: the interaction of certain translation-invariant scaled b.g.\ structures and the Fourier transform, leading for instance to generalizations of the Mellin transform on b-Sobolev spaces \cite{MelroseAPS}, \cite[\S{3.1}]{VasyMicroKerrdS} or the Fourier transform on 3b-Sobolev spaces \cite[\S{4.4}]{Hintz3b};
\item \S\ref{SsFComm}: when there exists a subspace of $\cW$ which spans $\cW$ over $\CI_{{\rm uni},\fB}(M)$ and which contains $\cV$ as an ideal, one can sharpen the class of $\cV$-ps.d.o.s slightly. This improvement is important in some applications (see e.g.\ \cite[\S{2.5.4}]{HintzGlueLocII}).
\end{enumerate}

\subsubsection{Operators on Euclidean space revisited}
\label{SssISc2}

Returning to the setting of~\S\ref{SssISc}, we cover $\R^n$ by the set $U_0:=(-4,4)^n$ with $\phi_0(x)=\frac12 x$ and, for $j=1,\ldots,n$ and $k\in\N_0$,
\begin{subequations}
\begin{equation}
\label{EqISc2Sets}
\begin{split}
  U_{j,k,\pm 1} &:= \Bigl\{ x\in\R^n \colon {\pm}x^j\in(2^k,2^{k+2}),\ \Bigl|\frac{x^l}{x^j}\Bigr| \in (-2,2)\ \forall\,l\neq j \Bigr\}, \\
  \phi_{j,k,\pm 1}(x) &:= \Bigl( \pm 2^{-k}x^j,\frac{x^1}{x^j}, \ldots, \wh{\frac{x^j}{x^j}}, \ldots, \frac{x^n}{x^j} \Bigr) \in (1,4)_{X^1} \times (-2,2)^{n-1}_{(X^2,\ldots,X^n)};
\end{split}
\end{equation}
here the hat denotes the omission of a term. If we further map $(1,4)\mapsto(-2,2)$ via $z\mapsto\frac{4 z-10}{3}$ (which we shall not do for notational simplicity), we obtain a b.g.\ structure on $\R^n$, and with the weights
\begin{equation}
\label{EqISc2Scale}
  \rho_{0,i}=1,\quad
  \rho_{(j,k,\pm 1),i}=2^{-k}\qquad (\forall\,i)
\end{equation}
\end{subequations}
a scaled b.g.\ structure. (Note that when the scales $\rho_{\alpha,i}$ are independent of $i$, then condition~\eqref{EqISB} follows directly from Definition~\ref{DefIBdd}\eqref{ItIBddTrans}.) See Figure~\ref{FigISc2}. Note that $\la x\ra\sim 2^k$ on $U_{j,k,\pm 1}$, and thus
\begin{equation}
\label{EqISc2Charts}
  \phi_{j,k,\pm 1}^*(\pa_{X^1}) = \pm 2^k\Bigl( \pa_{x^j} + \sum_{l=2}^n X^l\pa_{x^l}\Bigr),\qquad
  \phi_{j,k,\pm 1}^*(\pa_{X^l}) = x^j\pa_{x^l}\ \ (l\geq 2)
\end{equation}
are equivalent to $\la x\ra\pa_{x^i}$, $1\leq i\leq n$, in that either set of vector fields can be expressed as a linear combination of the other with uniformly bounded smooth coefficients. Thus the coefficient Lie algebra is
\begin{equation}
\label{EqISc2W}
  \cW = \left\{ \sum_{i=1}^n a^i(x)\la x\ra\pa_{x^i} \colon \forall\,\beta,i\ \exists\,C_{\beta,i}\ \text{s.t.}\ |\pa_x^\beta a^i|\leq C_{\beta,i} \la x\ra^{-|\beta|}\ \forall x\in\R^n \right\}.
\end{equation}
Rescaling the vector fields in~\eqref{EqISc2Charts} by $2^{-k}$ gives the coordinate vector fields on $\R^n$ (up to equivalence as above), and thus the operator Lie algebra is
\[
  \cV = \left\{ \sum_{i=1}^n a^i(x)\pa_{x^i} \colon \forall\,\beta,i\ \exists\,C_{\beta,i}\ \text{s.t.}\ |\pa_x^\beta a^i|\leq C_{\beta,i}\la x\ra^{-|\beta|}\ \forall\,x\in\R^n \right\}.
\]
According to Theorem~\ref{ThmISBRough}, the principal symbol of an element of $\Psi_\cV^m(\R^n)$ is an element of $S^m({}^\cV T^*\R^n)/\la x\ra^{-1}S^{m-1}({}^\cV T^*\R^n)$. For the shifted Laplacian $L=\Delta+1$, it is given by $|\xi|^2+1$ (where we write covectors as $\sum_{j=1}^n \xi_j\,\dd x^j$), which is elliptic in this quotient space (unlike that of $\Delta$ or $\Delta-1$).

\begin{figure}[!ht]
\centering
\includegraphics{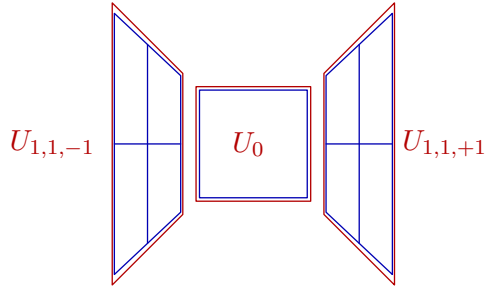}
\caption{Illustration of some unit cells for the b.g.\ structure~\eqref{EqISc2Sets} are drawn in red. The cuboids of side lengths $\rho_{\alpha,i}$ contained in them are drawn in blue. All blue cuboids are comparable in size to a standard unit square, corresponding to the operator Lie algebra being spanned by standard coordinate vector fields.}
\label{FigISc2}
\end{figure}

Finally, the large operator Lie algebra is
\begin{equation}
\label{EqISc2Vp}
  \cV' = \left\{ \sum_{i=1}^n a^i(x)\pa_{x^i} \colon \forall\,\beta,i\ \exists\,C_{\beta,i}\ \text{s.t.}\ |\pa_x^\beta a^i|\leq C_{\beta,i}\ \forall\,x\in\R^n \right\}.
\end{equation}
This is thus the Lie algebra of uniformly bounded vector fields on $\R^n$ corresponding to the b.g.\ structure~\eqref{EqIScBdd}. This explains the sense in which $\Diff_\cV(\R^n)$ comprises all $\cV'$-differential operators with $\cW$-regular coefficients.

\begin{rmk}
\label{RmkISc2Cpt}
  In the language of \cite{MelroseAPS,MelroseEuclideanSpectralTheory}, $\cW$, resp.\ $\cV$ in this example is the space of b-, resp.\ scattering vector fields on the radial compactification $\ol{\R^n}$ with conormal coefficients, i.e.\ $\Vb(\ol{\R^n})\otimes_{\CI(\ol{\R^n})}\cA(\ol{\R^n})$, resp.\ $\Vsc(\ol{\R^n})\otimes_{\CI(\ol{\R^n})}\cA(\ol{\R^n})$; and $\Psi_\cV^m(\R^n)$ is the space of scattering ps.d.o.s of order $m$ with conormal coefficients. (This coincides with the definition of $\Psi^{m,0}(\R^n)$ in \cite{VasyMinicourse}.)
\end{rmk}

\subsubsection{Semiclassical operators revisited}
\label{SssIh2}

Consider again the setting in~\S\ref{SssIh}. For simplicity, we consider $M=\R^n$. We set $P=(0,1]$. For $h\in P$, we define the scaled bounded geometry structure $\fB_{h,\times}$ using the $h$-independent charts $U_j=(-2,2)^n+j$, $j\in\Z^n$, $\phi_j(x)=x-j$, and $j$-independent scalings $\rho_{h,1}=\cdots=\rho_{h,n}=h$. See Figure~\ref{FigIExScl}. Then
\begin{align*}
  \cW &= \left\{ \sum_{i=1}^n a^i(h,x) \pa_{x^i} \colon \forall\,\beta,i\ \exists\,C_{\beta,i}\ \text{s.t.}\ |\pa_x^\beta a^i(h,x)|\leq C_{\beta,i}\ \forall\,h\in(0,1],\ x\in\R^n \right\}, \\
  \cV &= \left\{ \sum_{i=1}^n a^i(h,x) h\pa_{x^i} \colon \forall\,\beta,i\ \exists\,C_{\beta,i}\ \text{s.t.}\ |\pa_x^\beta a^i(h,x)|\leq C_{\beta,i}\ \forall\,h\in(0,1],\ x\in\R^n \right\}.
\end{align*}

\begin{figure}[!ht]
\centering
\includegraphics{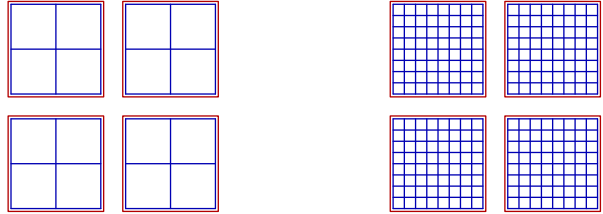}
\caption{Some unit cells for the b.g.\ structure for semiclassical operators with smooth coefficients in red, and cubes of side lengths $h$ in blue. \textit{On the left:} $h=\frac12$. \textit{On the right:} $h=\frac18$.}
\label{FigIExScl}
\end{figure}

\subsubsection{An anisotropic example}
\label{SssIAniso}

We consider asymptotically hyperbolic spaces and wish to encode the conormality of the output of the resolvent on `nice' inputs (cf.\ Remark~\ref{RmkIMixed}). For simplicity, let us work on the upper half space $M=(0,\infty)_x\times\R_y$ (with focus on the region $x<1$) with the b.g.\ structure
\[
  U_{j,k} := (2^{-j-2},2^{-j+2})_x \times (k-2,k+2)_y,\qquad U'_{j,k}:=(j,j+4)_x\times(k-2,k+2)_y
\]
for $j\in\N_0$, $k\in\Z$; we define the scaling
\[
  \rho_{(j,k),1}=1,\ \ \rho_{(j,k),2}=2^{-j},\qquad \rho'_{(j,k),i}=1\ \ (i=1,2).
\]
The coefficient Lie algebra $\cW$ (restricted to $x<1$) thus consists of b-vector fields (linear combinations of $x\pa_x$, $\pa_y$) with conormal coefficients (regularity with respect to b-vector fields), while $\cV$ consists of 0-vector fields \cite{MazzeoMelroseHyp} (linear combinations of $x\pa_x$, $x\pa_y$) with conormal coefficients.

\subsubsection{Further examples}
\label{SssIF}

We demonstrate the versatility of (parameterized) scaled b.g.\ structures by giving a list of examples complementing those discussed in~\S\ref{SsISB}. In these examples, we shall be somewhat imprecise in that we shall only write down the distinguished charts and scalings in the regions of interest in the interiors $M=\bar M^\circ$ of various manifolds with boundaries or corners $\bar M$, rather than providing full b.g.\ structures on the entirety of $M$. For notational clarity, we shall occasionally write $\rho_{\alpha,x}=\rho_{\alpha,i}$ when $x$ is the symbol for the $i$-th local coordinate function.

\begin{enumerate}
\item\label{ItIb} \emph{b-algebra on manifolds with boundary}, introduced in \cite{MelroseTransformation,MelroseMendozaB,MelroseAPS}. Here $\bar M$ is a manifold with boundary; we only consider the local model near a boundary point, so $\bar M=[0,1)_x\times\R^{n-1}_y$. The \emph{b-vector fields} are $x\pa_x$, $\pa_{y^\ell}$, with coefficients which we require to be conormal. Thus, we set
  \begin{equation}
  \label{EqIFb}
    U_{j,k} := (2^{-j-2},2^{-j+2})_x \times (k+(-4,4)^{n-1})_y,\qquad \N\ni j\geq 2,\ k\in\Z^n,
  \end{equation}
  and the scalings are all $1$. (This bounded geometry perspective on the b-algebra is of course not new; see e.g.\ \cite{MazzeoEdge,AmmannLauterNistorLie}.) This admits a straightforward generalization to manifolds with corners, so locally $\bar M=[0,1)^q\times\R^{n-q}$ with unit cells
  \begin{equation}
  \label{EqIFbCorners}
    U_{j_1,\ldots,j_q,k}=(2^{-j_1-2},2^{-j_1+2})_{x^1}\times\cdots\times(2^{-j_q-2},2^{-j_q+2})_{x^q}\times(k+(-4,4)^{n-q})_y,
  \end{equation}
  where $j_1,\ldots,j_q\in\Z$ and $k\in\Z^{n-q}$.
\item\label{ItIsc} \emph{Scattering algebra on manifolds with boundary}, introduced in \cite{MelroseEuclideanSpectralTheory} for geometric scattering theory on generalizations of asymptotically conic manifolds. Here $\bar M$ is a manifold with boundary; we only consider the local model near a boundary point, so $\bar M=[0,1)_x\times\R^{n-1}_y$. The vector fields are $x^2\pa_x$, $x\pa_{y^\ell}$, with coefficients which we require to be conormal. Thus, we use the unit cells~\eqref{EqIFb} with scalings $\rho_{(j,k),i}=2^{-j}$.
\item\label{ItIsch} \emph{Semiclassical scattering algebra} on $\R^n$, used for high energy resolvent estimates in asymptotically Euclidean scattering \cite{VasyZworskiScl}. On $\R^n$ with the b.g.\ structure~\eqref{EqISc2Sets}, we only modify the scaling~\eqref{EqISc2Scale} to depend on the parameter $h\in(0,1]$ via $\rho_{h,0,i}=h$ and $\rho_{h,(j,k,\pm 1),i}=h 2^{-k}$. Principal symbols of semiclassical scattering ps.d.o.s (with coefficients which are uniformly bounded together with all derivatives along $\la x\ra\pa_x$) in $\Psi_\cV^m(\R^n)$ are thus equivalence classes in $S^m/h \la x\ra^{-1}S^{m-1}({}^\cV T^*\R^n)$, similarly for weighted operators. The semiclassical version of the algebra in~\eqref{ItIsc} can be defined similarly.
\item\label{ItIe} \emph{Edge algebra} \cite{MazzeoEdge}, already briefly discussed from the bounded geometry perspective in the proof of \cite[Corollary~(3.23)]{MazzeoEdge}. The local model is $\bar M=[0,1)_x\times\R^{n_Y}_y\times\R^{n_Z}_z$ with the basic \emph{edge vector fields} being $x\pa_x$, $x\pa_{y^\ell}$, $\pa_{z^m}$. If one wishes to work with edge-regular coefficients, one takes as unit cells
  \[
    (2^{-j-2},2^{-j+2})_x\times(2^{-j}(k+(-4,4)^{n_Y}))_y\times(k'+(-4,4)^{n_Z})_z
  \]
  where $k\in\Z^{n_Y}$, $k'\in\Z^{n_Z}$. If one requires edge operators with \emph{b-regular} coefficients, one instead takes $(2^{-j-2},2^{-j+2})_x\times(k''+(-4,4)^{n_Y+n_Z})$, $k''\in\Z^{n_Y+n_Z}$, and the scalings $\rho_{(j,k''),x}=\rho_{(j,k''),z^m}=1$, $\rho_{(j,k''),y^\ell}=2^{-j}$.
\item\label{ItI1cu} \emph{1-cusp algebra}, introduced for inverse problems for geodesic X-ray transforms on asymptotically conic spaces \cite{VasyZachosXray}. On $\bar M=[0,1)_x\times\R^{n-1}_y$, 1-cusp vector fields are spanned by $x^3\pa_x$, $x\pa_{y^\ell}$, with conormal coefficients. This corresponds to using the sets $U_{j,k}$ from~\eqref{EqIFb} and the scalings $\rho_{(j,k),x}=(2^{-j})^2$ and $\rho_{(j,k),y^\ell}=2^{-j}$. Principal symbols of weighted 1-cusp ps.d.o.s in $x^{-\ell}\Psi_\cV^m$ are thus elements $[a]\in x^{-l}S^m/x^{-l+1}S^{m-1}$ (cf.\ \cite[Proposition~1.5]{VasyZachosXray}), where $a\in S^m$ means that $a=a(x,y,\xi_{1\rm c},\eta_{1\rm c})$, with covectors written as $\xi_{1\rm c}\frac{\dd x}{x^3}+\eta_{1\rm c}\frac{\dd y}{x}$, is a symbol of order $m$ in $(\xi_{1\rm c},\eta_{1\rm c})$ with conormal regularity in $(x,y)$. --- The \emph{semiclassical foliation 1-cusp algebra} is a parameterized version of this, with the scaling for the value $h\in(0,1]$ given by $\rho_{h,(j,k),x}=h(2^{-j})^2$, $\rho_{h,(j,k),y^\ell}=h^{\frac12}2^{-j}$; the principal symbol now captures leading order behavior also modulo $\cO(h^{\frac12})$ as $h\searrow 0$.
\item\label{ItIscbt} \emph{Scattering-b-transition algebra}, which first appeared under the name $\Psi_k$ in \cite{GuillarmouHassellResI} with a full calculus, and which under the name $\Psiscbt$ was used as a tool in \cite{HintzKdSMS}, in both cases for the purpose of uniform analysis of resolvents on asymptotically Euclidean or conic spaces. Working near spatial infinity, the local model is $[0,1)_\rho\times\R^{n-1}_\omega$ (where $\rho$ is an inverse radius), and the (spectral) parameter is $\sigma\in[0,1]$; the b.g.\ structure is, independently of $\sigma$, the b-structure $U_{j,k}=(2^{-j-2},2^{-j+2})_\rho\times(k+(-4,4)^{n-1})_\omega$ where $\N\ni j\geq 2$, $k\in\Z^{n-1}$, and the scalings which give rise to the desired scattering-b-transition vector fields $\frac{\rho}{\sigma+\rho}\rho\pa_\rho$, $\frac{\rho}{\sigma+\rho}\pa_{\omega^m}$ are $\rho_{\sigma,(j,k),i}=\frac{2^{-j}}{\sigma+2^{-j}}\sim\frac{\rho}{\sigma+\rho}$ (or equivalently $2^{-j}\sim\rho$ when $2^{-j}\sim\rho\leq\sigma$, and $1$ otherwise). The principal symbol thus captures operators also to leading order as $\frac{\rho}{\sigma+\rho}\searrow 0$, cf.\ \cite[\S{2.4}]{Hintz3b} and \cite[Proposition~2.13]{GuillarmouHassellResI}.
\item\label{ItIhF} \emph{Semiclassical foliation algebra}, utilized for the inversion of ray transforms \cite{VasyXraySemiclassical}. The local model is $M=\R^n=\R_x\times\R^{n-1}_y$ with parameter space $P=(0,1]_h$, unit cells (corresponding to uniform smoothness) $U_{h,j}:=j+(-4,4)^n$, $j\in\Z^n$, and scalings $\rho_{h,j,x}=h$, $\rho_{h,j,y^\ell}=h^{\frac12}$. The operator Lie algebra is thus spanned over the space of uniformly smooth functions by $h\pa_x$, $h^{\frac12}\pa_{y^\ell}$. The principal symbol now also captures the leading order behavior modulo $\cO(h^{\frac12})$ as $h\searrow 0$, recovering the results in \cite[\S2]{VasyXraySemiclassical}. --- The \emph{semiclassical scattering foliation algebra} \cite{VasyXraySemiclassical} is an amalgamation of this with the scattering algebra, so the unit cells are~\eqref{EqIFb} for each $h$, but the scalings are $\rho_{h,(j,k),x}=h 2^{-j}$ and $\rho_{h,(j,k),y^\ell}=h^{\frac12}2^{-j}$.
\item\label{ItI00} \emph{00-algebra}, for sharply localized gluing \cite{HintzGlueID,Mazzeo00} following \cite{ChruscielDelayMapping,DelayCompact}, with 0-regular coefficients. On $\bar M=[0,1)_x\times\R^{n-1}_y$, we take $U_{j,k}=(2^{-j-2},2^{-j+2})_x\times \{|y-2^{-j}k|<2^{-j+2}\}$ for $\N\ni j\geq 2$ and $k\in\Z^{n-1}$ (which is the b.g.\ structure for 0-vector fields $x\pa_x$, $x\pa_{y^\ell}$) and the scalings $\rho_{(j,k),i}=2^{-j}$ for all $i$. The operator Lie algebra is thus spanned by $x^2\pa_x$, $x^2\pa_{y^\ell}$ over the space of functions $a(x,y)$ which are uniformly bounded together with all derivatives along $x\pa_x$, $x\pa_{y^\ell}$. (The reference only considers 00-operators with smooth coefficients on $\bar M$ explicitly; however, the proof of \cite[Proposition~2.7]{HintzGlueID} does use localizers to 0-unit cells.) In addition to the standard principal symbol, there is now also a commutative \emph{boundary principal symbol} capturing leading order behavior as $x\searrow 0$. Both are captured simultaneously by~\eqref{EqISBRoughSymb}.
\item\label{ItIdescb} \emph{Double-edge scattering (desc) algebra}, introduced for the detailed microlocal and asymptotic analysis of massive waves on asymptotically Minkowski spacetimes \cite{SussmanKG}, with \cite{LauterMoroianuDoubleEdge} as a partial precursor. The local model near the intersection of spacelike and null infinity is $\bar M=[0,1)_{\rho_0}\times[0,1)_{x_I}\times\R^{n-2}_\omega$ where $n$ is the spacetime dimension, and desc-vector fields are spanned by $\rho_0^2 x_I\pa_{\rho_0}$, $\rho_0 x_I^2\pa_{x_I}$, $\rho_0 x_I^2\pa_{\omega^\ell}$. Choosing to work with b-regular coefficients, the unit cells are
  \[
    U_{j,k,l} = (2^{-j-2},2^{-j+2})_{\rho_0} \times (2^{-k-2},2^{-k+2})_{x_I} \times (l+(-4,4)^{n-2})
  \]
  where $\N\ni j,k\geq 2$ and $l\in\Z^{n-2}$, with scalings $\rho_{(j,k,l),\rho_0}=\rho_{(j,k,l),x_I}=2^{-j}2^{-k}$ and $\rho_{(j,k,l),\omega^m}=2^{-j}2^{-2 k}$. This gives a fully symbolic algebra, with the principal symbol capturing leading order behavior also at $\rho_0=0$ and $x_I=0$. (Edge-b-regular coefficients would suffice for this purpose.) --- The \emph{edge-b-algebra}, utilized for the analysis of \emph{massless} waves in \cite{HintzVasyScrieb}, can be recovered using the same unit cells but with scalings $\rho_{(j,k,l),\rho_0}=\rho_{(j,k,l),x_I}=1$, $\rho_{(j,k,l),\omega^m}=2^{-k}$. In that paper, regularity with respect to b-vector fields (which in this setup comprise the coefficient Lie algebra $\cW$) is captured relative to microlocal edge-b-regularity (cf.\ the discussion of mixed function spaces in~\S\ref{SssIAniso}).
\item\label{ItI3b} \textit{3b-algebra}, defined in \cite{Hintz3b} and used in \cite{HintzNonstat} in the study of waves on asymptotically stationary and asymptotically flat spacetimes. We only focus on a neighborhood $1\leq r\lesssim t$ of the corner of $\cT\cap\cD$ (in the notation of \cite[Definition~3.1]{Hintz3b}); the vector fields are $r\pa_t$, $r\pa_r$, $\pa_{\omega^m}$ (vector fields on $\Sph^{n-2}$), and we require b-regularity (on the 3b-single space) for the coefficients, which means regularity under $t\pa_t$, $r\pa_r$, $\pa_{\omega^m}$. As unit cells, we thus take
  \begin{equation}
  \label{EqIF3b}
    U_{j,k,l} = (2^{j-2},2^{j+2})_t \times (2^{k-2},2^{k+2})_r \times (l+(-4,4)^{n-2})_\omega,
  \end{equation}
  where $j,k\in\N$ with $1\leq k\leq j$, and $l\in\Z^{n-2}$; and the scalings are $\rho_{(j,k,l),t}=2^{-j+k}$, $\rho_{(j,k,l),r}=\rho_{(j,k,l),\omega^m}=1$. Mixed spaces capturing b-regularity relative to microlocal 3b-regularity are used in \cite[\S\S{2.2.2} and 5]{HintzNonstat}.
\item\label{ItIch} \textit{Semiclassical cone algebra}, defined in \cite{HintzConicPowers} (including a large calculus) and used as a tool in \cite{HintzConicProp} to study the semiclassical propagation of singularities through conic singularities of geometric or analytic (e.g.\ inverse square singularities) character. The local model is $[0,1)_r\times\R^{n-1}_\omega$, with parameter $h\in(0,1]$, and semiclassical cone vector fields $\frac{h}{h+r}r\pa_r$, $\frac{h}{h+r}\pa_{\omega^m}$. We take the coefficients to be uniformly conormal (i.e.\ regularity with respect to $r\pa_r$, $\pa_{\omega^m}$), and correspondingly the unit cells are $U_{j,k}=(2^{-j-2},2^{-j+2})_r\times(k+(-4,4)^{n-1})_\omega$ for $\N\ni j\geq 2$, $k\in\Z^{n-1}$ as in~\eqref{EqIFb}, independently of $h$; the scalings are $\rho_{h,(j,k),i}=\frac{h}{h+2^{-j}}\sim\frac{h}{h+r}$ (or equivalently $1$ when $2^{-j}\sim r\leq h$, and $\frac{h}{2^{-j}}\sim\frac{h}{r}$ otherwise). The principal symbol in this algebra captures operators to leading order also as $\frac{h}{h+r}\searrow 0$; this recovers the corresponding result in \cite[\S{3.2}]{HintzConicProp} (when generalized to variable order settings, cf.\ \S\ref{SsFVar}).
\end{enumerate}

This list is not exhaustive: other algebras which Theorem~\ref{ThmISBRough} covers are, for instance, the \emph{q-} and \emph{Q-algebras} introduced in \cite{HintzXieSdS,HintzGlueID,HintzKdSMS} for certain singular gluing problems, and the closely related \emph{surgery algebra} \cite{McDonaldThesis,MazzeoMelroseSurgery}; the 3-body scattering (and, more generally, $N$-body scattering) algebras introduced by Vasy \cite{VasyThreeBody,VasyManyBody}; and the leC-algebra defined by Sussman \cite{SussmanCoulomb} to study scattering by attractive Coulomb-type potentials at small nonnegative energies.

\begin{rmk}[Relationships of norms]
\label{RmkIFNorms}
  In settings which, in the above descriptions, arise from Lie algebras on manifolds with corners, it often happens that the associated Sobolev norms localize to neighborhoods of suitable boundary hypersurfaces to be equivalent to `simpler' Sobolev norms. For example, the scattering-b-transition Sobolev norm (see~\eqref{ItIscbt}) of a function with support in $\rho\gtrsim\sigma$ is uniformly (as $\sigma\searrow 0$) equivalent to its b-Sobolev norm. This is the content of \cite[Proposition~2.21(2)]{Hintz3b}; and it also follows directly from the bounded geometry perspective since in such regions, the scattering-b-transition unit cells are the same as the b-unit cells. In the same fashion, we recover \cite[Proposition~2.21(1)]{Hintz3b}, \cite[Corollary~3.7]{HintzConicProp} (in the setting~\eqref{ItIch}), and further similar results.
\end{rmk}

There are several important ps.d.o.\ algebras which \emph{cannot} be defined using scaled b.g.\ structures. These include certain second microlocal algebras whose symbols do not lie in the $S_{\rho,\delta}$ symbol classes of H\"ormander \cite[Definition~1.1.1]{HormanderFIO1} with $\rho=1-\delta$ and $\delta\in[0,\frac12)$, for example the semiclassical second microlocal calculus defined in \cite{SjostrandZworskiConvexObstacle} for the study of resonances for scattering by convex obstacles, and the refinement of the cusp calculus defined by Jia \cite{JiaTrapping} to reduce regularity losses at normally hyperbolic trapping. We also do not recover the parabolic calculus used in \cite{GellRedmanGomesHassellSchroedinger} for the Fredholm analysis of the time-dependent Schr\"odinger equation (see also \cite{LascarPropagation}), or calculi on certain Lie groups \cite{BahouriFermanianKammererGallagherHeisenberg}. Except in this final example, one may however entertain the idea of developing modified notions of scaling which guarantee that coordinate transformations across unit cells preserve the relevant local structures (e.g.\ the parabolic scaling in the fibers of the cotangent bundle).

\subsection{Relation with Lie algebras on manifolds with corners}
\label{SsIRel}

Almost all examples given in~\S\ref{SssIF} were introduced from a geometric singular analysis point of view: the starting point is a Lie algebra $\cV$ of vector fields on a manifold with corners $\bar M$ which one wishes to microlocalize.\footnote{Thus, while a central aim of the present paper is to provide general-purpose tools for analysis on singular spaces in the context of Melrose's program \cite{MelroseICM} on geometric microlocal analysis, our implementation is decidedly non-geometric, as it sidesteps manifolds with corners and geometric compactifications altogether.} Following Melrose, spaces of $\cV$-pseudodifferential operators were defined via a description of their Schwartz kernels as distributions on a suitable resolution (iterated blow-up \cite{MelroseDiffOnMwc}) $\bar M^2_\cV$ of $\bar M^2=\bar M\times\bar M$ (the \emph{$\cV$-double space}) which are conormal to the lift $\diag_\cV$ of the diagonal $\diag_{\bar M}\subset\bar M^2$. To show that the space of such operators defines an algebra, one constructs a \emph{$\cV$-triple space} $\bar M^3_\cV$ with the property that the three different projection maps down to $\bar M^2_\cV$ are \emph{b-fibrations}; the composition formula then follows from pullback and pushforward theorems for conormal distributions \cite{MelrosePushfwd}.

Thus far, there is no general purpose method which, given a Lie algebra $\cV$, produces the correct double (let alone triple) spaces; see however \cite{KottkeRochonProducts,MelroseGeneralizedProducts}. In those cases where a double space has been written down, the ps.d.o.\ algebra $\Psi_\cV(M)$ given by Theorem~\ref{ThmISBRough}, \emph{in the case that all scalings are equal to $1$} (or uniformly bounded from below by a positive constant), can be characterized to consist of those operators whose Schwartz kernels have support in a collar neighborhood of $\diag_\cV$; this is sometimes called the \emph{very small $\cV$-calculus/algebra}.

When $\cV$ is a \emph{Lie structure at infinity} \cite{AmmannLauterNistorLieGeometry} on a compact manifold with corners $\bar M$, i.e.\ a Lie subalgebra $\cV\subset\Vb(\bar M)$ whose elements span every tangent space over the interior $M=\bar M^\circ$, then by a result of Bui \cite{BuiLieStructure}, the injectivity radius of associated $\cV$-metrics is positive, and thus $\cV$ induces a bounded geometry structure on $M$. In turn, this structure gives rise to associated algebras of ps.d.o.s \cite{ShubinBounded}. Furthermore, Ammann--Lauter--Nistor \cite{AmmannLauterNistorLie} construct a very small $\cV$-calculus in the absence of a $\cV$-double space by quantizing symbols which are \emph{smooth} down to $\pa\bar M$. (In some sense, the groupoid constructions in \cite{DebordHolonomyGroupoid,CrainicFernandesLieBrackets} used in \cite{BuiLieStructure} do construct a neighborhood of the $\cV$-diagonal of a putative $\cV$-double space.)

Returning again to settings where $\cV$-double spaces are available, consider now the case that the scaling is \emph{nontrivial}, i.e.\ $\inf_\alpha\min_{i=1,\ldots,n}\rho_{\alpha,i}=0$. Then the operators in our algebra $\Psi_\cV(M)$ no longer have support close to $\diag_\cV$, and this is the crucial flexibility which allows for the existence of the refined principal symbol \eqref{EqISBRoughSymb}. For example, in the scattering setting~\eqref{EqISc2Sets}--\eqref{EqISc2Scale}, Schwartz kernels of elements of $\Psi_\cV^m(\R^n)$ are of the form
\begin{equation}
\label{EqIScQuant}
  K_A(x,x') = (2\pi)^{-n}\int e^{i(x-x')\cdot\xi} \chi\Bigl(\frac{|x-x'|}{|x|+|x'|}\Bigr)a(x,\xi)\,\dd\xi\,|\dd x'|
\end{equation}
where $\chi\in\CIc(\R)$ equals $1$ near $0$; thus, they are localized to an asymptotically conic neighborhood of $\{x=x'\}$ rather than to a neighborhood $|x-x'|\leq C$. From the perspective of the scattering double space \cite[\S{21}]{MelroseEuclideanSpectralTheory} of $\ol{\R^n}$, the cutoff $\chi$ is equal to $1$ not only near the scattering diagonal but also in a full neighborhood of the scattering front face. The Fourier transform of Schwartz kernels along the (vector space) fibers of the scattering front face gives rise, geometrically, to the boundary principal symbol of the scattering calculus; thus, it is necessary to localize only weakly, as in~\eqref{EqIScQuant}, in order to capture this boundary principal symbol. --- The semiclassical setting on $\R^n$ is similar: semiclassical ps.d.o.s with uniformly smooth coefficients have Schwartz kernels
\[
  (2\pi)^{-n}\int e^{i(x-x')\cdot\xi/h} \chi(|x-x'|)a(h,x,\xi)\,\dd\xi\,\frac{|\dd x'|}{h^n}
\]
and the Fourier transform of their restriction to the semiclassical front face---which is a bundle of vector spaces, with coordinate $\frac{x-x'}{h}$, over the diagonal $\diag_{\R^n}$ (see \cite[\S{4}]{BahuaudGuentherIsenbergMazzeoBddGeoWP})---gives the semiclassical principal symbol (i.e.\ $a(0,x,\xi)$). (By contrast, bounded geometry ps.d.o.s would feature a localization by $\chi(\frac{|x-x'|}{h})$ which destroys this principal symbol.)

\medskip

We end this section by pointing out the main disadvantages of the (scaled) bounded geometry perspective espoused in the present paper:
\begin{enumerate}
\item generalized inverses of elliptic $\cV$-operators typically do not exist in $\Psi_\cV$ (except in fully symbolic settings);
\item we do not have access to normal operators (i.e.\ restriction maps from $\Psi_\cV$ to noncommutative algebras of operators on boundary hypersurfaces) when microlocalizing Lie algebras $\cV$ on manifolds with corners $\bar M$, since we are at best able to keep track of conormal (but not smooth) regularity of coefficients/symbols down to the boundary hypersurfaces of $\bar M$.
\end{enumerate}
It appears that the full power of the geometric microlocal approach is required when one needs either of these, as the definition of a \emph{large} $\cV$-calculus (whose Schwartz kernels may have nontrivial behavior also on those boundary hypersurfaces of the $\cV$-double space which do not intersect the $\cV$-diagonal) is unavoidable. Recent examples include \cite{GuillarmouHassellResI,GuillarmouHassellResII,GuillarmouHassellSikoraResIII,GuillarmouSherConicLow,AlbinGellRedmanDirac,GrieserTalebiVertmanPhiLowEnergy,FritzschGrieserSchroheCalderon,NuetziDeRham}. On the other hand, in the applications of the calculi in~\S\ref{SssIF} to non-elliptic problems, not only the first but also the second disadvantage is irrelevant: differential operators of interest (which one studies using microlocal \emph{tools}) which have sufficiently regular coefficients on suitable manifolds with corners \emph{do} have well-defined normal operators---which in turn can be studied using scaled b.g.\ ps.d.o.s on the respective boundary hypersurfaces.

\subsection*{Acknowledgments}

I would like to thank Jeff Galkowski for helpful comments which in particular motivated parts of~\S\ref{SsbgCpt}. (For reasons of space, I unfortunately could not follow his suggestion to treat H\"ormander--Weyl calculi from a similar perspective in this paper.) Thanks are also due to Pierre Albin, Daniel Grieser, Ethan Sussman, and Andr\'as Vasy for useful conversations, comments, and suggestions.

\section{Aspects of bounded geometry structures}
\label{Sbg}

We consider a bounded geometry structure $\fB=\{(U_\alpha,\phi_\alpha)\colon\alpha\in\sA\}$ on the manifold $M$.
\begin{enumerate}
\item In~\S\ref{SsbgComp}, we explain in what sense spaces of uniformly smooth functions or vector fields determine $\fB$ (see Proposition~\ref{PropbgCompUniq});
\item in~\S\ref{SsbgDist}, we discuss metrics on $(M,\fB)$;
\item in~\S\ref{SsbgP}, we prove results regarding partitions and refinements of $(M,\fB)$;
\item in~\S\ref{SsbgCpt} finally, we discuss analytic notions related to the \emph{uniform compactification} $\fu M$ of $M$, as motivated in Remark~\ref{RmkISBMicro}.
\end{enumerate}

\subsection{Compatibility and uniqueness}
\label{SsbgComp}

The space $\CI_{{\rm uni},\fB}(M;T M)$ of uniformly bounded vector fields is a module over $\CI_{{\rm uni},\fB}(M)$. It is also a Lie algebra since for two vector fields on $(-2,2)^n$, the $\cC^k$-norms of the coefficients of their commutator are bounded by the product of their $\cC^{k+1}$-norms. We shall explain in which sense $\CI_{{\rm uni},\fB}(M)$ and $\CI_{{\rm uni},\fB}(M;T M)$ (individually) determine $\fB$.

\begin{definition}[Compatibility of b.g.\ structures]
\label{DefbgComp}
  Consider two b.g.\ structures $\fB=\{(U_\alpha,\phi_\alpha)\colon\alpha\in\sA\}$ and $\tilde\fB=\{(\tilde U_{\tilde\alpha},\tilde\phi_{\tilde\alpha})\colon\tilde\alpha\in\tilde\sA\}$. Set $U'_\alpha:=\phi_\alpha^{-1}((-\frac32,\frac32)^n)$ and $\tilde U'_{\tilde\alpha}:=\tilde\phi_{\tilde\alpha}^{-1}((-\frac32,\frac32)^n)$.
  \begin{enumerate}
  \item\label{ItbgCompCoarser} $\fB$ is \emph{coarser} than $\tilde\fB$ (denoted $\fB\geq\tilde\fB$) if the maps
    \begin{equation}
    \label{EqbgComp1}
      \phi_\alpha\circ\tilde\phi_{\tilde\alpha}^{-1} \colon \tilde\phi_{\tilde\alpha}(U'_\alpha\cap\tilde U'_{\tilde\alpha}) \to \phi_\alpha(U'_\alpha\cap\tilde U'_{\tilde\alpha})
    \end{equation}
    are uniformly (in $\alpha\in\sA$, $\tilde\alpha\in\tilde\sA$) bounded with all derivatives.
  \item $\fB$ and $\tilde\fB$ are \emph{compatible} if $\fB\geq\tilde\fB$ and $\tilde\fB\geq\fB$.
  \item $\fB$ is \emph{strongly coarser} than $\tilde\fB$ if the maps~\eqref{EqbgComp1} with $U_\alpha,\tilde U_{\tilde\alpha}$ in place of $U'_\alpha$, $\tilde U'_{\tilde\alpha}$ are uniformly bounded with all derivatives. Similarly, $\fB$ is \emph{strongly compatible} with $\tilde\fB$ if $\fB$ is strongly coarser than $\tilde\fB$ and vice versa.
  \end{enumerate}
\end{definition}

Thus, $\fB\geq\tilde\fB$ implies $\CI_{{\rm uni},\fB}(M)\subset\CI_{{\rm uni},\tilde\fB}(M)$. Roughly speaking, $\fB\geq\tilde\fB$ means that the unit cells of $\fB$ are larger than those of $\tilde\fB$.

The shrinkage to $(-\tfrac32,\tfrac32)^n$ is inconsequential in the following sense. Fix a diffeomorphism $\Psi\colon(-\tfrac32,\tfrac32)^n\to(-2,2)^n$ which is the identity on $[-1,1]^n$. Then:
\begin{enumerate}
\item Define $\phi'_\alpha:=\Psi\circ\phi_\alpha|_{U'_\alpha}$; then $\fB':=\{(U'_\alpha,\phi'_\alpha)\}$ and $\fB$ are strongly compatible (and thus give rise to the same notions of uniformly bounded smooth functions etc.).
\item Suppose $\fB,\tilde\fB$ are compatible. Define $\tilde\phi'_{\tilde\alpha}$ similarly to $\phi'_\alpha$ and set $\tilde\fB':=\{(\tilde U'_{\tilde\alpha},\tilde\phi'_{\tilde\alpha})\}$. Then the b.g.\ structures $\fB'$ and $\tilde\fB'$ are strongly compatible.
\end{enumerate}

\begin{prop}[Uniqueness of b.g.\ structures]
\label{PropbgCompUniq}
  Let $\fB$ and $\tilde\fB$ be two b.g.\ structures on $M$. Then $\fB$ and $\tilde\fB$ are compatible if and only if either of the following conditions holds:
  \begin{enumerate}
  \item\label{ItbgCompUniqFn} $\CI_{{\rm uni},\fB}(M)=\CI_{{\rm uni},\tilde\fB}(M)$;
  \item\label{ItbgCompUniqVF} $\CI_{{\rm uni},\fB}(M;T M)=\CI_{{\rm uni},\tilde\fB}(M;T M)$.
  \end{enumerate}
\end{prop}
\begin{proof}
  Fix $\chi\in\CI((-\tfrac32,\tfrac32)^n)$ to be equal to $1$ on $[-1,1]^n$. Since a smooth function $f$, resp.\ vector field $W$ on $M$ lies in $\CI_{{\rm uni},\fB}(M)$, resp.\ $\CI_{{\rm uni},\fB}(M;T M)$ if and only if $\chi\cdot(\phi_\alpha)_*f$, resp.\ $\chi\cdot(\phi_\alpha)_*W$ is uniformly bounded, resp.\ has uniformly bounded coefficients in $\CI((-2,2)^n)$, the compatibility of $\fB$ and $\tilde\fB$ implies both~\eqref{ItbgCompUniqFn} and \eqref{ItbgCompUniqVF}.

  We first prove that condition~\eqref{ItbgCompUniqFn} implies the compatibility of $\fB$ and $\tilde\fB$. We need to prove uniform $\CI$-bounds on $(-\tfrac32,\tfrac32)^n\ni p\mapsto(\phi_\alpha\circ\tilde\phi_{\tilde\alpha}^{-1})^i(p)$ for all $i,\alpha,\tilde\alpha$. Now, if $\chi\in\CI((-2,2)^n)$ equals $1$ on $[-\tfrac32,\tfrac32]^n$, then the family of functions $f_{\alpha,i}:=\phi_\alpha^*(\chi x^i)$, where $\alpha\in\sA$ and $i=1,\ldots,n$, is a bounded subset of $\CI_{{\rm uni},\fB}(M)$. Since $\CI_{{\rm uni},\fB}(M)$ and $\CI_{{\rm uni},\tilde\fB}(M)$ are Fr\'echet spaces, the identity map $\CI_{{\rm uni},\fB}(M)\to\CI_{{\rm uni},\tilde\fB}(M)$ is continuous by the closed graph theorem; therefore, $\{f_{\alpha,i}\}$ is uniformly bounded in $\CI_{{\rm uni},\tilde\fB}(M)$. This now implies the uniform boundedness in $\CI$ of $(\tilde\phi_{\tilde\alpha}^{-1})^*f_{\alpha,i}=(\phi_\alpha\circ\tilde\phi_{\tilde\alpha}^{-1})^*(\chi x^i)$ and thus of the function $(\phi_\alpha\circ\tilde\phi_{\tilde\alpha}^{-1})^i$ restricted to $\tilde\phi_{\tilde\alpha}(\phi_\alpha^{-1}((-\frac32,\frac32)^n)\cap\tilde U_{\tilde\alpha})$ since $\chi=1$ on the image of this set under $\phi_\alpha\circ\tilde\phi_{\tilde\alpha}^{-1}$. This implies the uniform boundedness of the maps~\eqref{EqbgComp1}, so $\fB\geq\tilde\fB$. The other direction $\tilde\fB\geq\fB$ is proved by reversing the roles of $\fB$ and $\tilde\fB$.

  That condition~\eqref{ItbgCompUniqVF} implies the compatibility of $\fB$ and $\tilde\fB$ can be proved in a similar manner, now considering the uniformly bounded family of vector fields $\phi_\alpha^*(\chi\pa_i)$. Alternatively, one can first show that~\eqref{ItbgCompUniqVF} implies \eqref{ItbgCompUniqFn} and thus reduce to a case already treated. To this end, we claim that
  \begin{equation}
  \label{EqbgCompUniqCI}
    \{ u\in\CI(M) \colon u V\in\CI_{{\rm uni},\fB}(M;T M)\ \forall\,V\in\CI_{{\rm uni},\fB}(M;T M) \} = \CI_{{\rm uni},\fB}(M).
  \end{equation}
  (This allows one to recover $\CI_{{\rm uni},\fB}(M)$ from $\CI_{{\rm uni},\fB}(M;T M)$ and the smooth manifold structure on $M$.) The inclusion `$\supseteq$' in~\eqref{EqbgCompUniqCI} is clear. For the converse, suppose we are given a smooth function $u\notin\CI_{{\rm uni},\fB}(M)$. Then there exists a sequence $\alpha_i$ in $\sA$ so that $(\phi_{\alpha_i}^{-1})^*u$ is unbounded in $\CI((-\frac32,\frac32)^n)$ and thus in $\cC^k$ for some $k$. If one element of $\sA$ appears infinitely often, pass to the corresponding constant subsequence of $\{\alpha_i\}$; otherwise, by Definition~\ref{DefIBdd}\eqref{ItIBddFinite}, we may pass to a subsequence so that the sets $U_{\alpha_i}$ are pairwise disjoint. Upon passing to a further subsequence, there exist a multiindex $\beta\in\N_0^n$, $|\beta|\leq k$, and a sequence of points $x_i\in(-\frac32,\frac32)^n$ converging to some limit $\bar x\in[-\frac32,\frac32]^n$ so that $|(\pa^\beta(\phi_{\alpha_i}^{-1})^*u)(x_i)|\to\infty$. Fix now $\chi\in\CIc((-2,2)^n)$ to be equal to $1$ near $\bar x$. If $\alpha_i=\alpha$ is constant, set $V=\phi_\alpha^*(\chi\pa_1)$, and otherwise set $V=\sum_i\phi_{\alpha_i}^*(\chi\pa_1)$; in both cases, $V\in\CI_{{\rm uni},\fB}(M;T M)$. Then in the charts $\phi_{\alpha_i}$, the first component of the vector field $u V$ has $\pa^\beta$-derivative at $x_i$ tending to infinity; so $u V\notin\CI_{{\rm uni},\fB}(M;T M)$, finishing the proof.
\end{proof}

\subsection{Compatible metrics}
\label{SsbgDist}

While one can define distance functions on $M$ compatible with the bounded geometry structure $\fB$ by taking the Riemannian distance function for a Riemannian metric as in Remark~\ref{RmkIBddMet}, we follow here a more pedestrian route (in line with our insistence that Riemannian structures should not be regarded as central objects for our present purely analytic purposes) to prove the following result:

\begin{prop}[Compatible metrics]
\label{PropbgDist}
  There exists a metric $d\colon M\times M\to[0,\infty)$ which is \emph{compatible} with $\fB$ in the following sense: there exist constants $\delta_0\in(0,1)$ and $C\geq 1$ so that
  \begin{enumerate}
  \item\label{ItbgDistNear} for all $\alpha\in\sA$ and $p,q\in U'_\alpha:=\phi_\alpha^{-1}((-\frac32,\frac32)^n)$, we have
    \begin{equation}
    \label{EqbgDistNear}
      C^{-1} d(p,q) \leq |\phi_\alpha(p)-\phi_\alpha(q)| \leq C d(p,q);
    \end{equation}
  \item\label{ItbgDistFar} for $p,q\in M$, let $D(p,q)$ denote the smallest number $D\in\N_0$ for which there exist $\alpha_0,\alpha_1,\ldots,\alpha_D\in\sA$ with $p\in U_{\alpha_0}$, further $U_{\alpha_i}\cap U_{\alpha_{i+1}}\neq\emptyset$ for $i=0,\ldots,D-1$, and finally $q\in U_{\alpha_D}$; then
    \begin{equation}
    \label{EqbgDistFar}
      \delta_0 D(p,q) \leq d(p,q) \leq \delta_0^{-1}(D(p,q)+1).
    \end{equation}
    The same remains true, for a different constant $\delta_0>0$, if we use the smaller sets $U'_\alpha$ in the definition of $D(p,q)$.
  \end{enumerate}
  The same metric $d$ is also compatible with any other b.g.\ structure $\fB'$ that is strongly compatible with $\fB$.
\end{prop}
\begin{proof}
  For $p,q\in M$, define the set $\sC_{p,q}$ of all triples $(\gamma,T,I)$ of piecewise smooth curves $\gamma\colon[0,1]\to M$ with $\gamma(0)=p$, $\gamma(1)=q$, and tuples $T=(t_0,\ldots,t_k)$, $I=(\alpha_0,\ldots,\alpha_{k-1})$ where $0=t_0<t_1<\ldots<t_k=1$, and $\gamma|_{[t_i,t_{i+1}]}$ is smooth for $i=0,\ldots,k-1$ and contained in $U_{\alpha_i}$. We then set
  \begin{equation}
  \label{EqbgDistLength}
    L(\gamma,T,I) := \sum_{i=0}^{k-1} L(\phi_{\alpha_i}\circ\gamma|_{[t_i,t_{i+1}]}),\qquad
    d(p,q) := \inf_{(\gamma,T,I)\in\sC_{p,q}} L(\gamma,T,I);
  \end{equation}
  here, for a curve $\beta\colon[a,b]\to(-2,2)^n$, we write $L(\beta)=\int_a^b |\beta'(t)|\,\dd t$ for its Euclidean length. The set $\sC_{p,q}$ is non-empty since the sets $U_\alpha$ cover $M$.

  In the subsequent arguments, $c,c'$ and $C$ denote constants which may vary from line to line but which are uniform for all points $p,q$ and all charts. The lower bound in~\eqref{EqbgDistNear} follows (for $C=1$) by taking $\phi_\alpha\circ\gamma$ to be the affine linear curve from $\phi_\alpha(p)$ to $\phi_\alpha(q)$. For the upper bound, consider any $(\gamma,T,I)\in\sC_{p,q}$. Suppose that $\gamma([0,1])\subset U_\alpha$. In view of the uniform bounds on the $\cC^1$-norms of the transition functions between $U_\alpha$ and $U_\beta$ with $U_\alpha\cap U_\beta\neq\emptyset$, we have $L(\gamma,T,I)\geq c|\phi_\alpha(p)-\phi_\alpha(q)|$. On the other hand, if $\gamma([0,1])$ is not contained in $U_\alpha$, there exists $t_0\in(0,1)$ with $\gamma(t_0)\in\phi_\alpha^{-1}(\pa([-\frac74,\frac74]^n))$ and $\gamma([0,t_0))\subset\phi_\alpha^{-1}((-\frac74,\frac74)^n)$. Since $\phi_\alpha(p)\in(-\frac32,\frac32)^n$, the length of $\phi_\alpha\circ\gamma|_{[0,t_0]}$ exceeds some universal constant $c>0$, which is larger than a constant times $|\phi_\alpha(p)-\phi_\alpha(q)|$ (since $|\phi_\alpha(p)-\phi_\alpha(q)|\leq|\phi_\alpha(p)|+|\phi_\alpha(q)|<4\sqrt{n}$). This establishes~\eqref{EqbgDistNear}.

  Turning to~\eqref{EqbgDistFar}, we observe that the upper bound follows by concatenating $D(p,q)+1$ many curves, the $i$-th of which is linear in $U_{\alpha_{i-1}}$. To prove the lower bound, select $(\gamma,T,I)\in\sC_{p,q}$ with $L(\gamma,T,I)\leq 2 d(p,q)$. Since $M$ is covered by the sets $\phi_\alpha^{-1}((-1,1)^n)$, we may modify $T$ and $I$ so that $\gamma(t_i)\in\phi_{\alpha_i}^{-1}([-1,1])^n)$ and $\gamma|_{[t_i,t_{i+1}]}\in\phi_{\alpha_i}^{-1}([-\frac32,\frac32]^n)$ for all $i$; in view of the uniform $\cC^1$ bounds on the $\tau_{\beta\alpha}$, we have
  \begin{equation}
  \label{EqbgDistL}
    L(\gamma,T,I) = \sum_{i=0}^{k-1} L_i \leq C d(p,q),\qquad L_i := L(\phi_{\alpha_i}\circ\gamma|_{[t_i,t_{i+1}]}).
  \end{equation}
  after this modification. For a small constant $c>0$, to be determined below, we furthermore subdivide each segment $\gamma|_{[t_i,t_{i+1}]}$ into a finite number of segments each of which has length $\leq c$; this modification does not change the value of $L(\gamma,T,I)$. We now group the terms in the sum in~\eqref{EqbgDistL}: we set $i_0=0$ and define $i_l$ for $l\geq 1$ inductively as follows, as long as $i_{l-1}\leq k$: we let $i_l\geq i_{l-1}+1$ to be the largest integer $\leq k$ so that $\sum_{i=i_{l-1}}^{i_l-1} L_i<c$. (The existence of such an index $i_l$ uses that $L_i<c$ for all $i$.) Since $\gamma(t_{i_{l-1}})\in\phi_{\alpha_{i_{l-1}}}^{-1}([-1,1]^n)$, we conclude, for sufficiently small $c$, that
  \begin{equation}
  \label{EqbgDistLCover}
    \gamma|_{[t_{i_{l-1}},t_{i_l}]}\subset\phi_{\alpha_{i_{l-1}}}^{-1}((-2,2)^n).
  \end{equation}
  Let $\bar\ell\in\N_0$ be the index with $i_{\bar\ell+1}=k$. In the case $\bar\ell=0$, the curve $\gamma$ is entirely contained in $U_{\alpha_{i_0}}$, so $D(p,q)=0$ and~\eqref{EqbgDistFar} is trivially satisfied. Otherwise, note that for $l\leq\bar\ell-1$ we have $\sum_{i=i_l}^{i_{l+2}-1} L_i\geq c$ by maximality of $i_l$, and therefore
  \[
    \frac{\bar\ell c}{2} \leq \frac12 \sum_{l=0}^{\bar\ell-1} \sum_{i=i_l}^{i_{l+2}-1} L_i \leq \sum_{i=0}^{k-1} L_i \leq C d(p,q).
  \]
  This gives $\bar\ell\leq\frac{2 C}{c}d(p,q)$. But by~\eqref{EqbgDistLCover}, we have $D(p,q)\leq\bar\ell$, and we have therefore proved~\eqref{EqbgDistFar} for $\delta_0=\frac{c}{2 C}$. The argument for the sets $U'_\alpha$ in place of $U_\alpha$ is analogous, except one replaces $[-\frac32,\frac32]^n$ above by the smaller set $[-\frac54,\frac54]^n$.

  To prove the final claim, consider the union $\tilde\fB$ of the b.g.\ structures $\fB$ and $\fB'$, which is itself a b.g.\ structure; defining a metric $\tilde d$ for it as above, one then finds that $C^{-1}\tilde d\leq d\leq C\tilde d$ for some constant $C$. The same holds for the metric $d'$, and thus $C^{-2}d'\leq d\leq C^2 d'$, from which one obtains~\eqref{EqbgDistNear} and \eqref{EqbgDistFar} for $d'$ in place of $d$, with different constants.
\end{proof}

\begin{cor}[Short distances]
\label{CorbgDistShort}
  Let $d$ be as in Proposition~\usref{PropbgDist}. There exists $\delta_1>0$ so that for $p\in\phi_\alpha^{-1}([-1,1]^n)$ and $q\in M$ with $d(p,q)<\delta_1$, we have $q\in U_\alpha=\phi_\alpha^{-1}((-2,2)^n)$ (and thus $|\phi_\alpha(p)-\phi_\alpha(q)|<C d(p,q)$).
\end{cor}
\begin{proof}
  If we take $\delta_1<\delta_0$, then by Proposition~\ref{PropbgDist}\eqref{ItbgDistFar}, there exists $U_\beta$ containing both $p$ and $q$, and $|\phi_\beta(p)-\phi_\beta(q)|<C\delta_1$. Consider the curve $\gamma(s)=\phi_\beta^{-1}(\phi_\beta(p)+s[\phi_\beta(q)-\phi_\beta(p)])$, $s\in[0,1]$, in $U_\beta$; once we show that $\gamma$ remains also in $U_\alpha$, then $|\phi_\alpha(p)-\phi_\alpha(q)|<C C'\delta_1<1$ (for small $\delta_1>0$) where $C'$ bounds the $\cC^1$-norm of the transition function $\tau_{\alpha\beta}$. Now $\gamma(0)=p\in U_\alpha$, and thus $\phi_\alpha(\gamma(s))=\tau_{\alpha\beta}(\phi_\beta(p)+s[\phi_\beta(q)-\phi_\beta(p)])$ is well-defined for small $s$ and attains a value in $[-1,1]^n$ for $s=0$. Let $I\subset[0,1]$ denote the set of all $s_0\in[0,1]$ so that $\gamma(s)\in\phi_\alpha^{-1}([-\frac32,\frac32]^n)$ for $s\in[0,s_0]$. Then $I$ is closed; but for $s_0\in I$, we have $d(\phi_\alpha(\gamma(s_0)),[-1,1]^n)\leq C C'\delta_1$ using the $\cC^1$-bounds on $\tau_{\alpha\beta}$, and thus $\gamma(s_0)\in\phi_\alpha^{-1}([-\frac54,\frac54]^n)$ if we fix $\delta_1>0$ sufficiently small. Therefore, a small neighborhood of $s_0$ lies in $I$. This implies that $I$ is open. Since $I\ni 0$ is non-empty, we are done.
\end{proof}

\subsection{Partitions and refinements}
\label{SsbgP}

We record two technical results related to covering properties of the unit cells. We begin with a partition of $M$ into finitely many unions of pairwise disjoint unit cells.

\begin{lemma}[Partition of distinguished charts]
\label{LemmabgP}
  Let $\ell<2$ and set $U_\alpha(\ell)=\phi_\alpha^{-1}((-\ell,\ell)^n)$. There exist $J<\infty$ and a partition $\sA=\bigsqcup_{j=1}^J\sA_j$ of $\sA$ so that for each $j=1,\ldots,J$, any two sets $U_\alpha(\ell),U_\beta(\ell)$ with $\alpha,\beta\in\sA_j$ are disjoint.
\end{lemma}
\begin{proof}
  Say that a subset $\sB\subset\sA$ satisfies the \emph{disjointness property} if for all $\alpha,\beta\in\sB$ one has $U_\alpha(\ell)\cap U_\beta(\ell)=\emptyset$. Pick any $\alpha_1\in\sA$, and take $\sA_1\subset\sA$ to be a maximal subset satisfying the disjointness property. If $\sA_1\neq\sA$, pick $\alpha_2\in\sA$ and take $\sA_2\subset\sA\setminus\sA_1$ to be a maximal subset satisfying the disjointness property. Define subsets $\sA_j$ for $j=1,2,\ldots$ inductively as long as $\sA_1\cup\cdots\cup\sA_j\neq\sA$. Let $J\in\N\cup\{\infty\}$ be the supremum of $j$; we need to show that $J$ is finite.

  Fix $\eps>0$ with the following property: for all $\alpha,\beta,\gamma\in\sA$, $p\in U_\alpha(\ell)\cap U_\beta(\ell)$, and $q\in U_\alpha(\ell)\cap U_\gamma(\ell)$ with $|\phi_\alpha(p)-\phi_\alpha(q)|<\eps$, one has $q\in U_\beta$ (and symmetrically $p\in U_\gamma$); such an $\eps$ can be chosen to depend only on $\ell$ and on the $\cC^1$-norms of the transition functions $\tau_{\beta\alpha}$. Fix then $J_0$, only depending on $\eps$, $n$, and the constant $A\in\N$ from Definition~\ref{DefIBdd}\eqref{ItIBddFinite}, so that for all $J_0$-tuples of points in $(-\ell,\ell)^n$ there exists a subset of $A+1$ many points with pairwise distances less than $\eps$. We claim that $J\leq J_0$. Suppose this were false, and let $\alpha\in\sA\setminus\bigcup_{j=1}^{J_0}\sA_j$. By the maximality of the sets $\sA_j$, there exists, for each $j=1,\ldots,J_0$, a point $p_j\in U_\alpha(\ell)\cap U_{\alpha'_j}(\ell)$ for some $\alpha'_j\in\sA_j$. By relabeling the points, we may assume that $|\phi_\alpha(p_1)-\phi_\alpha(p_k)|<\eps$ for all $k=1,\ldots,A+1$. But this implies that $p_1\in\bigcap_{k=1}^{A+1}U_{\alpha'_k}$, in contradiction to Definition~\ref{DefIBdd}\eqref{ItIBddFinite}.
\end{proof}

The following result provides a convenient modification of a given b.g.\ structure which will be useful for studying compositions of ps.d.o.s\ which are given via sums of quantizations in local charts. (See the proof of Theorem~\ref{ThmSBPsdoComp}.)

\begin{lemma}[Refinement]
\label{LemmabgRefine}
  Let $1<\ell_1<\ell_2<2$. Then there exists a b.g.\ structure $\fB'$ on $M$ which is strongly compatible with $\fB$ so that if $\phi_\alpha^{-1}((-\ell_1,\ell_1)^n)\cap\phi_\beta^{-1}((-\ell_1,\ell_1)^n)\neq\emptyset$, then $\phi_\beta^{-1}((-\ell_1,\ell_1)^n)\subset\phi_\alpha^{-1}((-\ell_2,\ell_2)^n)$.
\end{lemma}
\begin{proof}
  Fix an integer $S\geq 3$. We subdivide
  \begin{align*}
    &\phi_\alpha^{-1}\Bigl(\Bigl(-\frac32,\frac32\Bigr)^n\Bigr) = \bigcup_{\gamma\in\{\frac18 2^S,\ldots,\frac78 2^S-4\}^n} V_{\alpha,\gamma}, \\
    &\qquad V_{\alpha,\gamma}:=\phi_\alpha^{-1}(Q_\gamma),\quad Q_\gamma = \prod_{j=1}^n \bigl(-2+4\cdot 2^{-S}\gamma_j,-2+4\cdot(2^{-S}(\gamma_j+4))\bigr).
  \end{align*}
  Suppose $p\in V_{\alpha,\gamma}\cap V_{\beta,\delta}\neq\emptyset$, and suppose that $q\in V_{\beta,\delta}$ is such that $|\phi_\beta(q)-\phi_\beta(p)|_\infty<16\cdot 2^{-S}$ (the side length of $Q_\delta$). In view of the uniform $\cC^1$-bounds on $\tau_{\alpha\beta}$, we then also have $q\in U_\alpha$ provided we choose $S$ sufficiently large, and indeed
  \[
    |\phi_\alpha(p)-\phi_\alpha(q)|_\infty < 16 C\cdot 2^{-S} < \frac12
  \]
  for some uniform constant $C>1$.

  The idea is to make $\ell_1$-cubes very small compared to $\ell_2$-cubes. To this end, fix a monotone diffeomorphism $\eta\colon\R\to\R$ with $\eta(-x)=-\eta(x)$ so that $\eta(0)=0$, $\eta(12\cdot 2^{-S}/2)=1$, $\eta(16\cdot 2^{-S}/2)=\ell_1$, $\eta(16 C\cdot 2^{-S}/2)=\ell_2$, and $\eta(\frac14)=2$. Take $U_{\alpha,\gamma}$ to be the preimage under $\phi_\alpha$ of the cube with the same center as $Q_\gamma$ but with side length $2\cdot\frac14$, and let $\phi_{\alpha,\gamma}\colon U_{\alpha,\gamma}\xra{\cong}(-2,2)^n$ to be the composition of $\phi_\alpha$ with the translation $x\mapsto(x^j+2-4\cdot 2^{-S}(\gamma_j+2))_{j=1,\ldots,n}$ (which maps $Q_\gamma$ to $(-\frac14,\frac14)^n$) followed by $\eta\times\cdots\times\eta$. It then follows that $\{(U_{\alpha,\gamma},\phi_{\alpha,\gamma})\colon\alpha\in\sA,\ \gamma\in\{\frac18 2^{-S},\ldots,\frac78 2^S-4\}^n\}$ is a b.g.\ structure with the desired properties.
\end{proof}

\subsection{Compactification and supports}
\label{SsbgCpt}

With $\fB$ fixed throughout, we drop it from the notation and thus write $\CI_{\rm uni}(M)=\CI_{{\rm uni},\fB}(M)$ etc. We moreover fix a \emph{bounded} metric
\[
  d \colon M\times M \to [0,1]
\]
on $M$ which is compatible with the bounded geometry structure in the sense that there exist constants $\delta_0<1$ and $C>1$ so that $p\in\phi_\alpha^{-1}([-1,1]^n)$ and $d(p,q)<\delta\leq\delta_0$ implies $q\in\phi_\alpha^{-1}((-2,2)^n)$ and $|\phi_\alpha(p)-\phi_\alpha(q)|<C\delta$, and conversely for $p,q\in\phi_\alpha^{-1}((-2,2)^n)$ with $|\phi_\alpha(p)-\phi_\alpha(q)|<\delta$ one has $d(p,q)<C\delta$. (In particular, points $p,q$ lying in disjoint sets $\phi_\alpha^{-1}((-1,1)^n)$, $\phi_\beta^{-1}((-1,1)^n)$ have distance $d(p,q)\geq\delta_0$.) Such a function $d$ can be defined by $d(p,q)=\min(\tilde d(p,q),1)$ where $\tilde d$ is a metric as in Proposition~\ref{PropbgDist}.

We denote by $\cC^0_{\rm uni}(M)$ the subspace of the Banach algebra $\cC^0(M)$ of continuous functions $u$ which are uniformly continuous. This means that for all $\eps>0$ there exists $\delta>0$ so that for all $\alpha\in\sA$ and $x,y\in(-2,2)^n$ with $|x-y|<\delta$ we have $|u(\phi_\alpha^{-1}(x))-u(\phi_\alpha^{-1}(y))|<\eps$; or equivalently (for a possibly different $\delta$), $d(p,q)<\delta$ implies $|u(p)-u(q)|<\eps$; or equivalently (for yet another value of $\delta$), the family $\{(\phi_\alpha)_*u\colon\alpha\in\sA\}\subset\cC^0((-2,2)^n)$ is equicontinuous. Equipped with the supremum norm, the space $\cC^0_{\rm uni}(M)$ is again a Banach algebra.

The following compactification was introduced (albeit in a different fashion) in \cite{SamuelCompactification}. We follow, to some extent, the paper \cite{WoodsSamuelCompactification}, but operate entirely in the concrete setting of interest in the present paper.

\begin{definition}[Uniform compactification of $M$]
\label{DefbgCpt}
  The \emph{uniform compactification} $\fu M$ of $(M,\fB)$ is the Gelfand spectrum $\sigma(\cC^0_{\rm uni}(M))$ of $\cC^0_{\rm uni}(M)$ as a unital *-algebra, i.e.\ the set of all continuous\footnote{We recall that continuity is automatic: $\ker\phi$ is a maximal ideal and thus necessarily closed since otherwise its closure would contain the identity element; but the set of invertible elements in a Banach algebra is open.} *-algebra homomorphisms $\phi\colon\cC^0_{\rm uni}(M)\to\C$ equipped with the Gelfand topology, which is the coarsest topology with respect to which the maps $\phi\mapsto\phi(u)$ are continuous for all $u\in\cC^0_{\rm uni}(M)$. If we wish to specify the b.g.\ structure, we write $\fu(M,\fB)$ for $\fu M$.
\end{definition}

Some authors use the terminology \emph{Samuel compactification}.

Since $\fu M$ is a closed subset of the unit ball of $(\cC_{\rm uni}^0(M))^*$ in the weak-* topology, $\fu M$ is a compact Hausdorff space.\footnote{The more well-known Stone--\v{C}ech compactification $\beta M$ is the spectrum of the algebra of bounded continuous functions. Much like $\beta M$, the space $\fu M$ is very large, see \cite[Theorem~4.9(b)]{WoodsSamuelCompactification}; for example, for $M=\R$, restricting an element of $\cC^0_{\rm uni}(\R)$ to the integer points gives a surjective *-homomorphism to $\ell^\infty(\Z)$ (bounded continuous functions on $\Z$), and thus an embedding $\beta\Z\hra\fu\R$.} Of course, in practice one does not work with $\fu M$, but rather with smaller compactifications; this is discussed in and after Lemma~\ref{LemmabgUniv} below.

\begin{prop}[Towards an alternative construction of $\fu M$]
\label{PropbgAlt}
  Let $C=\cC^0_{\rm uni}(M;[0,1])$ and consider $[0,1]^C$ equipped with the product topology.\footnote{This is thus a compact Hausdorff space by Tychonoff's theorem.} Write $\iota\colon M\to[0,1]^C$ for the map $\iota(p)_u:=u(p)$. Then the map
  \[
    j_0 \colon \cC_{\rm uni}^0(M;[0,1]) \to \cC^0\bigl(\ol{\iota(M)};[0,1]\bigr) \qquad j_0(u) \colon \bigl( [0,1]^C \ni (c_v)_{v\in C} \mapsto c_u \in [0,1] \bigr)
  \]
  extends by linearity to an (isometric) *-algebra isomorphism $j\colon\cC_{\rm uni}^0(M)\to\cC^0(\ol{\iota(M)})$.
\end{prop}

Since all $u\in C$ are continuous, the map $\iota$ is a continuous injection. Via $j$, we can identify $\fu M$ with $\sigma(\cC^0(\ol{\iota(M)}))$, which in turn is isomorphic to $\ol{\iota(M)}$; the latter isomorphism is given by evaluations ${\rm ev}_p$ at points $p\in\ol{\iota(M)}$. Thus,
\begin{equation}
\label{EqbgId}
  \ol{\iota(M)} \cong \fu M,\qquad p \mapsto \bigl( \cC^0_{\rm uni}(M) \ni u \mapsto j(u)(p) \bigr).
\end{equation}

\begin{proof}[Proof of Proposition~\usref{PropbgAlt}]
  Note that $j_0(u)(\iota(p))=j_0(u)(( v(p) )_{v\in C})=u(p)$. Therefore, the map $j$, defined by linear extension of $j_0$, satisfies $\|j(u)\|_{\cC^0}=\|u\|_{\cC^0}$ by the density of $\iota(M)\subset\ol{\iota(M)}$; in particular, it is injective.

  We must show that $j_0$ is surjective; that is, given $\tilde u\in\cC^0(\ol{\iota(M)};[0,1])$, we define $u(p):=\tilde u(\iota(p))$, and we need to show that $u\in\cC^0_{\rm uni}(M)$. Let $\eps>0$. For all $\tilde p\in\ol{\iota(M)}$, the continuity of $\tilde u$ at $\tilde p$ (in the product topology on $[0,1]^C$) implies the existence of $\tilde\eps_{\tilde p}>0$, $N_{\tilde p}\in\N$, and $u_{\tilde p,1},\ldots,u_{\tilde p,N_{\tilde p}}\in C\subset\cC^0_{\rm uni}(M)$ with the property that
  \begin{equation}
  \label{EqbgAltPf}
    \max_{1\leq j\leq N_{\tilde p}} | u_{\tilde p,j}(\tilde q) - u_{\tilde p,j}(\tilde p) | < \tilde\eps_{\tilde p} \implies |\tilde u(\tilde q)-\tilde u(\tilde p)| < \frac{\eps}{2}.
  \end{equation}
  (Here we write $u_{\tilde p,j}(\tilde q)=\tilde q_{u_{\tilde p,j}}$ for $\tilde q\in[0,1]^C$.) Let now
  \[
    \cU_{\tilde p} := \Bigl\{ \tilde q\in\ol{\iota(M)} \colon \max_{1\leq j\leq N_{\tilde p}} | u_{\tilde p,j}(\tilde q) - u_{\tilde p,j}(\tilde p) | < \frac{\tilde\eps_{\tilde p}}{2} \Bigr\}.
  \]
  Then $\ol{\iota(M)}=\bigcup_{\tilde p} \cU_{\tilde p}$, so by the compactness of $\ol{\iota(M)}$ there exist finitely many points $\tilde p_1,\ldots,\tilde p_N\in\ol{\iota(M)}$ so that $\ol{\iota(M)}=\bigcup_{i=1}^N \cU_{\tilde p_i}$. Let $\tilde\eps:=\min_{1\leq i\leq N}\tilde\eps_i$. Fix $\delta>0$ so that whenever $p,q\in M$ are such that $d(q,p)<\delta$, then $|u_{\tilde p_i,j}(p)-u_{\tilde p_i,j}(q)|<\frac{\tilde\eps}{2}$ for all $1\leq i\leq N$ and $1\leq j\leq N_{\tilde p_i}$; this is possible since the $u_{\tilde p_i,j}$ are uniformly continuous.

  Let now $p,q\in M$, $d(p,q)<\delta$. Take $i\in\{1,\ldots,N\}$ so that $\iota(p)\in\cU_{\tilde p_i}$, then $|u_{\tilde p_i,j}(\iota(p))-u_{\tilde p_i,j}(\tilde p_i)|<\frac{\tilde\eps}{2}$ for all $j=1,\ldots,N_{\tilde p_i}$ and thus $|\tilde u(\iota(p))-\tilde u(\tilde p_i)|<\frac{\eps}{2}$ by~\eqref{EqbgAltPf}; but also (using $u_{\tilde p_i,j}(p)=u_{\tilde p_i,j}(\iota(p))$ for $p\in M$)
  \[
    | u_{\tilde p_i,j}(\iota(q)) - u_{\tilde p_i,j}(\tilde p_i) | \leq |u_{\tilde p_i,j}(q) - u_{\tilde p_i,j}(p)| + |u_{\tilde p_i,j}(\iota(p)) - u_{\tilde p_i,j}(\tilde p_i)| < \frac{\tilde\eps}{2} + \frac{\tilde\eps}{2} = \tilde\eps,
  \]
  and therefore also $|\tilde u(\iota(q))-\tilde u(\tilde p_i)|<\frac{\eps}{2}$ by~\eqref{EqbgAltPf}; this implies that
  \[
    |u(p)-u(q)| = |\tilde u(\iota(p)) - \tilde u(\iota(q))| < \frac{\eps}{2} + \frac{\eps}{2} = \eps,
  \]
  finishing the proof.
\end{proof}

We shall henceforth typically drop the identification $j$ from the notation, and we also identify $\fu M=\ol{\iota(M)}$.

\begin{lemma}[$M$ and $\fu M$]
\label{LemmabgMsM}
  The map $\iota\colon M\to \fu M$ is open, and thus\footnote{We argued above that $\iota$ is continuous.} a homeomorphism onto its image. In particular, $\iota(M)$ is an open subset of $\fu M$, and $\pa(\fu M)=\fu M\setminus\iota(M)$ is compact.
\end{lemma}
\begin{proof}
  Let $V\subset M$ be open and nonempty. Let $p\in V$. Pick $u\in\cC^0_{\rm uni}(M)$ so that $u(p)=0$ and $u|_{M\setminus V}\geq 1$. Then $\iota(V)$ contains the open subset $\{\tilde q\in[0,1]^C\colon|\tilde q_u|<\frac12\}\cap\iota(M)$ of $\iota(M)$.
\end{proof}

We may thus also drop $\iota$ from the notation and regard $M\subset\fu M$ as the dense interior of the compact space $\fu M$.

\begin{rmk}[Vanishing at the boundary]
\label{RmkbgVanish}
  Let $u\in\cC_{\rm uni}^0(M)$. Then $u|_{\pa(\fu M)}=0$ if and only if $u$ vanishes at infinity, i.e.\ for all $\eps>0$ there exists a compact set $K\subset M$ so that $|u|<\eps$ on $M\setminus K$. The direction `$\Longrightarrow$' follows from the continuity of $u\in\cC^0(\fu M)$, which yields an open neighborhood $U\subset\fu M$ so that $|u|<\eps$ on $U$, and thus $K:=\fu M\setminus U\subset M$ is a compact set with the desired property; the direction `$\Longleftarrow$' follows from $K\cap\pa(\fu M)=\emptyset$ for $K\subset M$, so $|u|_{\pa(\fu M)}|<\eps$ for all $\eps>0$.
\end{rmk}

The following result is stated in \cite[Theorem~1.1]{WoodsSamuelCompactification}; we give a self-contained proof for completeness. We write $\cl_{\fu M}(A)$ for the closure of $A$ in $\fu M$.

\begin{lemma}[Intersections in $\fu M$]
\label{LemmabgCap}
  Let $A,B\subset M$. Then $\cl_{\fu M}(A)\cap\cl_{\fu M}(B)\neq\emptyset$ if and only if $d(A,B)=0$.
\end{lemma}
\begin{proof}
  By Urysohn's lemma, $\cl_{\fu M}(A)\cap\cl_{\fu M}(B)=\emptyset$ is equivalent to the existence of $u\in\cC^0(\fu M)$ with $u|_A=1$, $u|_B=0$. Given such a $u$, the uniform continuity of $u$ implies that there exists $\delta>0$ so that for all $p,q\in M$ with $d(p,q)<\delta$ one has $|u(p)-u(q)|<\frac12$, and thus $d(A,B)\geq\delta$. Conversely, suppose that $d(A,B)\geq\delta>0$, then we can construct a separating function $u$ as follows: in view of $|d(A,p)-d(A,q)|\leq d(p,q)$, we have $d(A,\cdot)\in\cC^0_{\rm uni}(M)$. Pick a continuous function $f\colon\R\to[0,1]$ which equals $1$ on $[0,\frac{\delta}{3}]$ and $0$ on $[\frac{2\delta}{3},\infty)$, then $u:=f\circ d(A,\cdot)\in\cC^0_{\rm uni}(M)$ equals $1$ on $A$ and $0$ on $B$, as required.
\end{proof}

To illustrate how one can work with $\fu M$, we now study the support of uniformly continuous functions. Another example illustrating the utility of the compactification $\fu M$ is given in~\eqref{EqbgSingSupp} below.

\begin{definition}[$\fu M$-support]
\label{DefbgSupp}
  Let $u\in\cC_{\rm uni}^0(M)$. Then its \emph{$\fu M$-support} is $\supp_{\fu M} u:=\supp j(u)\subset \fu M$.
\end{definition}

We have the usual properties
\[
  \supp_{\fu M}(u v) \subset \supp_{\fu M} u \cap \supp_{\fu M} v,\qquad
  \supp_{\fu M}(u+v) \subset \supp_{\fu M} u \cup \supp_{\fu M} v
\]
for $u,v\in\cC^0_{\rm uni}(M)$. Furthermore, one easily shows
\[
  \supp_{\fu M} u \subset M\quad\text{(i.e.\ $\supp_{\fu M} u\cap\pa(\fu M)=\emptyset$)} \iff u\in\cC_{\rm c}^0(M).
\]
Moreover, $\supp_{\fu M} u$ is equal to the closure of its interior $\{u\neq 0\}$; and in particular it does not have isolated points. Finally, note that $\supp u=M\cap\supp_{\fu M} u\subset M$ recovers the standard notion of the support of $u$ on $M$. Conversely, we have
\begin{equation}
\label{EqbgsuppuMCl}
  \supp_{\fu M} u = \cl_{\fu M}(\supp u)
\end{equation}
where $\cl_{\fu M}$ denotes the closure in $\fu M$. Indeed, the inclusion `$\supseteq$' is clear, and to prove `$\subseteq$', we note that if $p\in\supp_{\fu M} u$, then (by the density of $M\subset \fu M$) all open neighborhoods $V\subset \fu M$ of $p$ contain a point $p_V\in V\cap M$ with $u(p_V)\neq 0$ and thus $p_V\in\supp u$; so $p$ is an accumulation point of $\{p_V\}\subset\supp u$, and therefore $p\in\cl_{\fu M}(\supp u)$.

We also note the following testing definition of $\supp_{\fu M} u$: a point $p\in \fu M$ does \emph{not} lie in $\supp_{\fu M} u$ if and only if there exists $\chi\in\cC^0_{\rm uni}(M)$ with $\chi(p)\neq 0$ so that $\chi u=0$. For later purposes, we point out that one can use uniformly \emph{smooth} witnesses $\chi$ for the absence of support:

\begin{lemma}[Smooth witnesses]
\label{LemmabgSmooth}
  Let $u\in\cC^0_{\rm uni}(M)$. Then $p\in \fu M$ does not lie in $\supp_{\fu M} u$ if and only if there exists $\chi\in\CI_{\rm uni}(M)$ with $\chi(p)\neq 0$ and $\chi u=0$.
\end{lemma}
\begin{proof}
  We use a smoothing procedure; the case $p\in M$ is clear, and thus we only consider the case $p\in\pa(\fu M)$, $p\notin\supp_{\fu M} u$. Let $\chi\in\cC^0_{\rm uni}(M)$ with $\chi(p)\neq 0$ and $\chi u=0$. Let $\eta\in\cC^0_{\rm uni}(M)=\cC^0(\fu M)$ be a function which equals $1$ at $p$ and whose $\fu M$-support is disjoint from $\chi^{-1}(0)$; thus $d(\supp\eta,\chi^{-1}(0))>0$ by Lemma~\ref{LemmabgCap}. We can write $\eta=\sum_\alpha (\phi_\alpha)^*\eta_\alpha$ where $\eta_\alpha\in\cC_{\rm c}^0((-2,2)^n)$, with $\supp\eta_\alpha\subset[-1,1]^n$ for all $\alpha$, is an equicontinuous family; and then taking $\psi_\eps(x)=\eps^{-n}\psi_1(x/\eps)$ to be a standard mollifier, with $\psi_1\in\CIc((-1,1)^n)$ and $\int\psi_1(x)\,\dd x=1$, we let $\eta_\eps:=\sum_\alpha (\phi_\alpha)^*(\psi_\eps*\eta_\alpha)$. Then $\eta_\eps\in\CI_{\rm uni}(M)$ converges to $\eta$ in $\cC^0_{\rm uni}(M)$ as $\eps\searrow 0$, and thus $\eta_\eps(p)\neq 0$ for sufficiently small $\eps>0$ (since $j(\eta_\eps)\to j(\eta)$ in $\cC^0(\fu M)$). Since $d(\supp\eta_\eps,\chi^{-1}(0))\geq d(\supp\eta,\chi^{-1}(0))-C\eps$ is positive for all sufficiently small $\eps>0$, this implies that $\eta_\eps u=\frac{\eta_\eps}{\chi}\cdot\chi u=0$ since $\frac{\eta_\eps}{\chi}\in\cC^0_{\rm uni}(M)$.
\end{proof}

\bigskip

We next describe how $\fu M$ relates to more practical (for analytic purposes) compactifications of $M$. We first record the following universal property of $\fu M$:

\begin{lemma}[Smaller compactifications]
\label{LemmabgUniv}
  Suppose $\bar M$ is a compact Hausdorff space so that $M$ is homeomorphically embedded into $\bar M$ as the interior $\bar M^\circ$; and suppose that restriction to $M$ induces a continuous map $r\colon\cC^0(\bar M)\to\cC^0_{\rm uni}(M)$. Then there exists a unique surjective continuous map $\upbeta\colon\fu M\to\bar M$ which is the identity over $M$ and has the property
  \[
    r(\cC^0(\bar M))=\{u\in\cC^0_{\rm uni}(M)\colon u\ \text{is constant on the fibers of}\ \upbeta\}.
  \]
\end{lemma}
\begin{proof}
  The map $r$ induces a continuous map of Gelfand spectra $\upbeta\colon\fu M=\sigma(\cC^0_{\rm uni}(M))\to\sigma(\cC^0(\bar M))\cong\bar M$ by mapping $\phi\colon\cC_{\rm uni}^0(M)\to\C$ to the map $r^*\phi\colon\cC^0(\bar M)\ni u\mapsto\phi(r(u))$; that is, $\upbeta=r^*$. Since $\upbeta(\fu M)\subset\bar M$ is compact, being the continuous image of a compact set, and since $\upbeta(M)=M$ is dense, we must have $\upbeta(\fu M)=\bar M$, i.e.\ $\upbeta$ is surjective. We claim that this map $\upbeta$ satisfies the desired conclusions. The uniqueness of $\upbeta$ follows from the required continuity of $\upbeta$ and the density of $M$ in $\fu M$.

  We claim that $\bar M$ carries the quotient topology of $\fu M/\sim$ where $p\sim q$ if and only if $\upbeta(p)=\upbeta(q)$. Since $\{(p,q)\in\fu M\times\fu M\colon\upbeta(p)=\upbeta(q)\}$ (the graph of $\sim$) is closed, $\fu M/\sim$ is a compact Hausdorff space \cite[\S{10}, no.~4, Proposition 8(a), (d)]{BourbakiTopology1234}. But the continuous map $\fu M\to\bar M$ factors through a continuous map $\fu M/\sim\,\to\bar M$, which is a continuous bijection of compact Hausdorff spaces and thus a homeomorphism, as desired.

  Regarding $\cC^0(\bar M)\subset\cC^0(\fu M)$ via the isometry $r$, a continuous map $u\colon\fu M\to\C$ is constant on the fibers of $\upbeta\colon\fu M\to\bar M$ if and only if it factors through a continuous map $u'\colon\bar M\to\C$, i.e.\ $u=u'\circ\upbeta$; but for $p\in M$, this means $u(p)=u'(p)$, so this is in turn equivalent to $u\in\cC^0(\bar M)$.
\end{proof}

For $u\in\cC^0_{\rm uni}(M)$, we can then consider its support as a subset of $\bar M$,
\begin{equation}
\label{EqbgSuppbarM}
  \supp_{\bar M}u := \upbeta(\supp_{\fu M}u).
\end{equation}
Equivalently, if $p\in\bar M$, then
\begin{equation}
\label{EqbgSuppbarMTest}
  p\notin\supp_{\bar M}(u)\ \iff \exists\,\chi\in\cC^0(\bar M),\ \chi(p)\neq 0,\ \text{such that}\ \chi u=0.
\end{equation}
Indeed, such a function $\chi$ restricts to $M$ as an element of $\cC^0_{\rm uni}(M)$ which does not vanish anywhere on $\upbeta^{-1}(p)$, so $\supp_{\fu M}u\cap\upbeta^{-1}(p)=\emptyset$ or equivalently $p\notin\upbeta(\supp_{\fu M}u)$. Conversely, if $p$ does not lie in the compact set $\upbeta(\supp_{\fu M}u)$, then there exists an open neighborhood $U\subset\bar M$ of $p$ disjoint from $\upbeta(\supp_{\fu M}u)$, and thus for any $\chi\in\cC^0(\bar M)$ with support in $U$ we have $\chi u=0$; choosing $\chi$ to be nonzero at $p$ finishes the proof of~\eqref{EqbgSuppbarMTest}. As a consequence, the definition~\eqref{EqbgSuppbarM} of $\supp_{\bar M}u$ coincides for $u\in\cC^0(\bar M)$ with the standard notion of support.

In the case that $\bar M$ is a smooth manifold with corners---that is, each point $p\in\bar M$ has a neighborhood diffeomorphic to $[0,\infty)^k\times\R^{n-k}$ where $k$ depends on $p$---one can use $\chi\in\CI(\bar M)$ in~\eqref{EqbgSuppbarMTest}. (For manifolds with corners, one often requires the boundary hypersurfaces of $M$ to be embedded \cite{MelroseDiffOnMwc}, but this is not needed here.)

\begin{example}[Supports on compactifications of $\R^n$]
\label{ExbgSuppRn}
  We consider the bounded geometry structure~\eqref{EqIScBdd} on $\R^2$. We write points in $\R^2$ as $(x,y)$. Consider a function $u(x,y)=\phi(x)\chi(y-y_0)$ where $\phi\in\CI(\R)$ equals $1$ on $[0,\infty)$ and $0$ on $(-\infty,-1]$, and $\chi\in\CIc(\R)$ has non-empty support $K:=\supp\chi\neq\emptyset$; we are interested in its support properties `at infinity'. We consider three compactifications of $\R^2$ to smooth manifolds (with or without boundary or corners).
  \begin{enumerate}
  \item In the one-point compactification $\bar M_1=\R^2\sqcup\{\infty\}=\Sph^2$ of $\R^2$, we have $\supp_{\bar M_1}u=\supp_{\R^2}u\cup\{\infty\}$.
  \item In the radial compactification $\bar M_2=(\R^2\sqcup([0,\infty)_\rho\times\Sph^1_\omega))/\sim$, $0\neq x=r\omega\sim(r^{-1},\omega)$ (which is a manifold with boundary $\pa\bar M_2=\Sph^1$), $\supp_{\bar M_2}u=\supp_{\R^2}u\cup\{E\}$ where $E\in\pa\bar M_2$ is the point with $\rho=0$, $\omega=(1,0)\in\Sph^1$.
  \item In the product compactification $\bar M_3=\ol\R\times\ol\R$ where $\ol\R=\R\cup\{-\infty,+\infty\}$, $\supp_{\bar M_3}u=\supp_{\R^2}u\cup(\{+\infty\}\times K)$.
  \end{enumerate}
  The conceptual benefit of working with $\fu\R^2$ is that it includes these cases (and infinitely many others) as special cases in the sense of~\eqref{EqbgSuppbarM}.
\end{example}

\bigskip

We end this section by pointing out that an advantage of compact spaces is that the validity of a local property near every point implies its global uniform validity. For example, let us define for $u\in\cC^0_{\rm uni,\fB}(M)$ its $\cC^k_{{\rm uni},\fB}(M)$-singular support by its complement
\begin{equation}
\label{EqbgSingSupp}
  \fu M\setminus \singsupp_{\cC^k_{{\rm uni},\fB}(M)}u = \bigl\{ p\in\fu M \colon \exists\,\chi\in\CI_{\rm uni}(M),\ \chi(p)\neq 0\ \text{s.t.}\ \chi u\in\cC^k_{{\rm uni},\fB}(M) \bigr\}.
\end{equation}
Then $\singsupp_{\cC^k_{{\rm uni},\fB}(M)}u=\emptyset$ if and only if $u\in\cC^k_{{\rm uni},\fB}(M)$. (By contrast, merely having $\singsupp_{\cC^k}u=\emptyset$, i.e.\ for all $p\in M$ there exists $\chi\in\CI(M)$, $\chi(p)\neq 0$, with $\chi u\in\cC^k$, does \emph{not} imply $u\in\cC^k_{{\rm uni},\fB}(M)$ when $k\geq 1$, due to a lack of uniform control on derivatives of $u$.)

\section{Scaled bounded geometry structures}
\label{SSB}

We now develop the theory of (parameterized) scaled b.g.\ structures (Definitions~\ref{DefISB} and \ref{DefIPSB}) and of the associated classes of (pseudo)differential operators and weighted function spaces in detail. The main results in this section are as follows.
\begin{itemize}
\item \S\ref{SsSBComp}: we show that the operator Lie algebra $\cV$ of a scaled b.g.\ structure $\fB_\times$ essentially determines $\fB_\times$; see Proposition~\ref{PropSBCompUniq}.
\item \S\ref{SsSB2}: we describe a secondary bounded geometry structure (corresponding to the notion of regularity induced by the operator Lie algebra) in Proposition~\ref{PropSB2}.
\item \S\ref{SsSBD}: we study $\cV$-differential operators and their principal symbols; see in particular Theorem~\ref{ThmSBDSymb}.
\item \S\ref{SsSBH}: we introduce weighted $\cV$-Sobolev spaces, and versions thereof with additional $\cW$-regularity. Differential operators act boundedly between these spaces.
\item \S\ref{SsSBPsdo}: the heart of the paper; we define $\cV$-pseudodifferential operators in Definition~\ref{DefSBPsdoV}, prove their mapping properties (Theorem~\ref{ThmSBPsdoH}), and develop their principal symbol calculus (Theorem~\ref{ThmSBPsdoComp}).
\item \S\ref{SsSBWF}: we define Sobolev wave front sets adapted to the scaled b.g.\ structure using the $\cV$-calculus.
\item \S\ref{SsSBPar}: we discuss the generalization of these notions to \emph{parameterized} scaled b.g.\ structures.
\end{itemize}

\subsection{Compatibility and uniqueness}
\label{SsSBComp}

We use the notation $\fB_\times$, $\fB$, $\cW$, $\cV$, $\cV$' from Definition~\ref{DefISB}, and recall $\tau_{\alpha\beta}=\phi_\alpha\circ\phi_\beta^{-1}$. The Lie algebra property of the space $\cV'$ is proved similarly to that of $\cW=\CI_{{\rm uni},\fB}(M;T M)$ (see~\S\ref{SsbgComp}). The properties of $\cV$ are slightly more subtle. We remark that for $\chi\in\CIc((-2,2)^n)$, we have $\phi_\alpha^*(\chi\,\rho_{\alpha,i}\pa_i)\in\cV$ for all $\alpha\in\sA$, $i=1,\ldots,n$; indeed,
\[
  (\phi_\beta)_*\phi_\alpha^*(\chi\,\rho_{\alpha,i}\pa_i)=(\tau_{\alpha\beta}^*\chi)\sum_{j=1}^n\rho_{\alpha,i}(\pa_i\tau_{\beta\alpha}^j)\pa_j = \sum_{j=1}^n V_\beta^j\rho_{\beta,j}\pa_j,\qquad V_\beta^j=(\tau_{\alpha\beta}^*\chi)\frac{\rho_{\alpha,i}\pa_i\tau_{\beta\alpha}^j}{\rho_{\beta,j}},
\]
and the smoothness of $\tau_{\alpha\beta}^*\chi$ together with the condition~\eqref{EqISB} gives the required uniform $\CI$ bounds of $V_\beta^j$. For the proof that $\cV$ is a Lie algebra, let $X,Y\in\cV$ and let $\chi\in\CIc((-2,2)^n)$; then, writing $(\phi_\alpha)_*X=\sum_{i=1}^n X^i\rho_{\alpha,i}\pa_i$ and $(\phi_\alpha)_*Y=\sum_{j=1}^n Y^j\rho_{\alpha,j}\pa_j$, we have
\[
  \chi\cdot(\phi_\alpha)_*[X,Y]=\chi\cdot[(\phi_\alpha)_*X,(\phi_\alpha)_*Y] = \chi\sum_{i,j=1}^n \bigl(X^i\rho_{\alpha,i}(\pa_i Y^j) - Y^i\rho_{\alpha,i}(\pa_i X^j)\bigr)\rho_{\alpha,j}\pa_j;
\]
note then that the expression in parentheses is uniformly bounded in $\CI$.

We shall now explain in what sense $\cV$ determines the underlying scaled b.g.\ structure.

\begin{definition}[Compatibility of scaled b.g.\ structures]
\label{DefSBComp}
  Two scaled b.g.\ structures $\fB_\times=\{(U_\alpha,\phi_\alpha,\rho_\alpha)\}$ and $\tilde\fB_\times=\{(\tilde U_{\tilde\alpha},\tilde\phi_{\tilde\alpha},\tilde\rho_{\tilde\alpha})\}$ are \emph{compatible} if the b.g.\ structures $\{(U_\alpha,\phi_\alpha)\}$ and $\{(\tilde U_{\tilde\alpha},\tilde\phi_{\tilde\alpha})\}$ are compatible and if, moreover, there exist constants $C_\gamma<\infty$ so that the transition functions $\tau_{\alpha\tilde\alpha}:=\phi_\alpha\circ\tilde\phi_{\tilde\alpha}^{-1}$ and $\tau_{\tilde\alpha\alpha}:=\tilde\phi_{\tilde\alpha}\circ\phi_\alpha^{-1}$ satisfy
  \begin{equation}
  \label{EqSBComp}
    |\pa^\gamma\tilde\rho_{\tilde\alpha,i}\pa_i\tau_{\alpha\tilde\alpha}^j| \leq C_\gamma\rho_{\alpha,j},\quad
    |\pa^\gamma\rho_{\alpha,i}\pa_i\tau_{\tilde\alpha\alpha}^j| \leq C_\gamma\tilde\rho_{\tilde\alpha,j}\qquad \forall\, i,j=1,\ldots,n,\ \ \gamma\in\N_0^n,
  \end{equation}
  on $\tilde\phi_{\tilde\alpha}(U'_\alpha\cap\tilde U'_{\tilde\alpha})$ and $\phi_\alpha(U'_\alpha\cap\tilde U'_{\tilde\alpha})$ in the notation of Definition~\usref{DefbgComp} for all $\alpha,\tilde\alpha$.
\end{definition}

\begin{prop}[Uniqueness of scaled b.g.\ structures]
\label{PropSBCompUniq}
  Let $\fB_\times$ and $\tilde\fB_\times$ be two scaled b.g.\ structures on $M$. Then the operator Lie algebras $\cV$ and $\tilde\cV$ of $\fB_\times$ and $\tilde\fB_\times$ are equal if and only if $\fB_\times$ and $\tilde\fB_\times$ are compatible. In particular, $\cV=\tilde\cV$ implies the equality also of the coefficient Lie algebras and large operator Lie algebras of $\fB_\times$ and $\tilde\fB_\times$.
\end{prop}
\begin{proof}
  We only prove one direction, namely that $\cV=\tilde\cV$ implies the compatibility of $\fB_\times$ and $\tilde\fB_\times$. Write $\fB$ and $\tilde\fB$ for the b.g.\ structures underlying $\fB_\times$ and $\tilde\fB_\times$. Following the arguments after~\eqref{EqbgCompUniqCI}, one shows, now using $V=\phi_\alpha^*(\chi\cdot\rho_{\alpha,1}\pa_1)$ or $V=\sum_i\phi_{\alpha_i}^*(\chi\cdot\rho_{\alpha_i,1}\pa_1)$ in the final step of the argument by contradiction, that
  \[
    \CI_{{\rm uni},\fB}(M) = \{ u\in\CI(M) \colon u V\in\cV\ \forall\,V\in\cV \}.
  \]
  Since the space on the right is unchanged when replacing $\cV$ by $\tilde\cV$, in which case it equals $\CI_{{\rm uni},\tilde\fB}(M)$, we deduce $\CI_{{\rm uni},\fB}(M)=\CI_{{\rm uni},\tilde\fB}(M)$ and thus the compatibility of $\fB$ and $\tilde\fB$ by Proposition~\ref{PropbgCompUniq}\eqref{ItbgCompUniqFn}.

  In order to establish~\eqref{EqSBComp}, let $\chi\in\CIc((-2,2)^n)$ be equal to $1$ on $[-\frac32,\frac32]^n$, and consider $V_{\alpha,i}:=\phi_\alpha^*(\chi\cdot\rho_{\alpha,i}\pa_i)$; this is a bounded family of elements of $\cV$, where we equip $\cV$ with the Fr\'echet space structure given by the $\cC^k$-seminorms of the coefficients $V^i$ in the local coordinate description $\chi\cdot(\phi_\alpha)_*V=\sum_{i=1}^n V^i\rho_{\alpha,i}\pa_i$. As in the proof of Proposition~\ref{PropbgCompUniq}, the closed graph theorem implies the boundedness of $V_{\alpha,i}$ in $\tilde\cV$, which means in the notation of Definition~\ref{DefSBComp} that
  \[
    (\tilde\phi_{\tilde\alpha})_*V_{\alpha,i}=(\tau_{\tilde\alpha\alpha})_*(\chi\cdot\rho_{\alpha,i}\pa_i) = \sum_{j=1}^n V_{\tilde\alpha,\alpha,i}^j \tilde\rho_{\tilde\alpha,j}\pa_j
  \]
  where the $V_{\tilde\alpha,\alpha,i}^j$ are uniformly bounded in $\CI$. In particular, this implies that there exists $C<\infty$ so that on $[-\frac32,\frac32]^n$ we have
  \[
    |\rho_{\alpha,i}\pa_i\tau_{\tilde\alpha\alpha}^j|=|V_{\tilde\alpha,\alpha,i}^j\tilde\rho_{\tilde\alpha,j}|\leq C\tilde\rho_{\tilde\alpha,j}
  \]
  for all $\alpha,\tilde\alpha,i,j$, proving the second inequality in~\eqref{EqSBComp} for $\gamma=0$; and the uniform boundedness in $\CI$ of $V_{\tilde\alpha,\alpha,i}^j$ in fact gives~\eqref{EqSBComp} for all $\gamma$. The first inequality follows by exchanging the roles of $\fB_\times$ and $\tilde\fB_\times$.
\end{proof}

\subsection{Secondary bounded geometry structures}
\label{SsSB2}

We briefly address the change of perspective explained prior to Definition~\ref{DefISB}. Let $\fB_\times=\{(U_\alpha,\phi_\alpha,\rho_\alpha)\colon\alpha\in\sA\}$ be a scaled b.g.\ structure. For $D\in\N$ to be determined, and for each $\alpha\in\sA$, define the following objects:
\begin{itemize}
\item $D_{\alpha,i}\in\N$, the largest integer with $\rho_{\alpha,i}D_{\alpha,i}\leq 1$, or $D_{\alpha,i}=4$, whichever is larger;
\item for each $n$-tuple $k=(k_1,\ldots,k_n)$ of integers $0\leq k_i\leq 3 D\cdot D_{\alpha,i}-4$, set $I_{\alpha,k,i}=(-\frac32+\frac{k_i}{D\cdot D_{\alpha,i}},-\frac32+\frac{k_i+4}{D\cdot D_{\alpha,i}})$, denote by $\Psi_{\alpha,k,i}\colon I_{\alpha,k,i}\to(-2,2)$ the monotone affine linear diffeomorphism, and set
  \[
    U_{\alpha,k} := \phi_\alpha^{-1}\Biggl(\prod_{i=1}^n I_{\alpha,k,i}\Biggr),\qquad
    \phi_{\alpha,k} := \Biggl(\prod_{i=1}^n\Psi_{\alpha,k,i}\Biggr)\circ\phi_\alpha \colon U_{\alpha,k}\to(-2,2)^n;
  \]
\item $\rho\fB:=\{(U_{\alpha,k},\phi_{\alpha,k})\colon\alpha\in\sA,\ k\in\prod_{i=1}^n\{0,1,2,\ldots,D_{\alpha,i}-4\}\}$.
\end{itemize}

We refer back to the right hand side of Figure~\ref{FigISB} for an illustration of $\rho\fB$.

\begin{prop}[Secondary b.g.\ structure]
\label{PropSB2}
  In the above notation, $\rho\fB$ is a b.g.\ structure provided we choose $D$ large enough. If $\cV$, resp.\ $\cV'$ denote the operator, resp.\ large operator Lie algebra of $\fB_\times$, then
  \begin{equation}
  \label{EqSB2}
    \cV' = \CI_{{\rm uni},\rho\fB}(M;T M) = \cV \otimes_{\CI_{{\rm uni},\fB}(M)} \CI_{{\rm uni},\rho\fB}(M).
  \end{equation}
\end{prop}
\begin{proof}
  Consider two overlapping charts $U_\alpha,U_\beta$, and let $x_0$ be the center of $\phi_\alpha(U_{\alpha,k})\subset(-\frac32,\frac32)^n$; thus $\phi_\alpha(U_{\alpha,k})$ is a cuboid centered around $x_0$ with side lengths $\frac{4}{D\cdot D_{\alpha,i}}$. Let $k'$ be such that $x'_0:=\tau_{\beta\alpha}(x_0)=\phi_\beta(\phi_\alpha^{-1}(x_0))$ lies in $\phi_\beta(U_{\beta,k'})$. Then for $x\in\phi_\alpha(U_{\alpha,k})$, the condition~\eqref{EqISB} (with $|\gamma|=1$) implies
  \[
    |\tau_{\beta\alpha}^j(x)-(x'_0)^j| = \biggl|\sum_{i=1}^n\Bigl(\int_0^1 \pa_i\tau_{\beta\alpha}^j(x_0+t(x-x_0)) \,\dd t\Bigr)\ (x^i-x_0^i)\biggr| \leq C\rho_{\alpha,j}
  \]
  for a uniform constant $C$, \emph{provided} $\phi_\alpha^{-1}(x_0+t(x-x_0))$ lies in $U_\beta$ for all $t\in[0,1]$; but this condition holds due to $x_0\in(-\frac32,\frac32)^n$, $|x-x_0|_\infty<D^{-1}$, and the uniform $\cC^1$-bounds on $\tau_{\beta\alpha}$ if we choose $D$ sufficiently large. This implies that there is a uniform constant $S<\infty$ so that at most $S$ of the sets $U_{\beta,k'}$ intersect $U_{\alpha,k}$ nontrivially. Denoting by $A$ the covering constant from Definition~\ref{DefIBdd}\eqref{ItIBddFinite}, we conclude that $A\cdot S$ is a covering constant for $\{(U_{\alpha,k})\}$.

  The uniform boundedness in $\CI$ of the transition functions for $\rho\fB$ follows \emph{a fortiori} from the bounds~\eqref{EqISB} for general $\gamma$, as in the rescaled coordinates $\tilde x^i=\frac{x^i}{\rho_{\alpha,i}}$ and $\tilde x'{}^j=\frac{x'{}^j}{\rho_{\beta,j}}$ the bounds~\eqref{EqISB} are equivalent to uniform bounds on the $\tilde x'{}^j$-components of
  \begin{equation}
  \label{EqSB2Coc}
    \pa_{\tilde x^i}(\phi_{\beta}\circ\phi_{\alpha}^{-1})
  \end{equation}
  together with all derivatives along $\rho_{\alpha,i}^{-1}\pa_{\tilde x^i}$.

  Finally, using the same rescaling, one sees that a smooth function $u$, resp.\ vector field $V$ lies in $\CI_{{\rm uni},\rho\fB}(M)$, resp.\ $\CI_{{\rm uni},\rho\fB}(M;T M)$ if and only if its pushforward along $\phi_\alpha$ is of the form $u_\alpha$, resp.\ $\sum_{i=1}^n V_\alpha^i\rho_{\alpha,i}\pa_i$, where $u_\alpha$, resp.\ $V_\alpha^i$ obeys uniform bounds $|(\rho_\alpha\pa_x)^\gamma u_\alpha|\leq C_\gamma$ for all $\alpha\in\sA$ and $\gamma\in\N_0^n$, resp.\ $|(\rho_\alpha\pa_x)^\gamma V_\alpha^i|\leq C_\gamma$ for all $\alpha,i,\gamma$; here we write
  \begin{equation}
  \label{EqSB2Der}
    (\rho_\alpha\pa_x)^\gamma=\prod_{i=1}^n (\rho_{\alpha,i}\pa_{x^i})^{\gamma_i}.
  \end{equation}
  This implies both equalities in~\eqref{EqSB2} and finishes the proof.
\end{proof}

\subsection{Weights, symbols, differential operators}
\label{SsSBD}

From now on, we fix a scaled b.g.\ structure $\fB_\times=\{(U_\alpha,\phi_\alpha,\rho_\alpha)\colon\alpha\in\sA\}$ on $M$, with underlying b.g.\ structure $\fB$, coefficient Lie algebra $\cW=\CI_{{\rm uni},\fB}(M;T M)$, and operator Lie algebra $\cV$. We begin by capturing the class of weights that we can allow in our operators and function spaces.

\begin{definition}[Weights on $(M,\fB)$]
\label{DefSBDWeights}
  Let $U'_\alpha=\phi_\alpha^{-1}((-\frac32,\frac32)^n)$. A \emph{weight} on $(M,\fB)$ is a smooth function $0<w\in\CI(M)$ so that $\frac{W w}{w}\in\CI_{{\rm uni},\fB}(M)$ for all $W\in\cW$. A \emph{weight family} is a collection $\{w_\alpha\colon\alpha\in\sA\}$ of positive real numbers for which there exists $C<\infty$ so that for all $\alpha,\beta\in\sA$ with $U'_\alpha\cap U'_\beta\neq\emptyset$, we have $C^{-1}w_\beta\leq w_\alpha\leq C w_\beta$. Two weights $w,w'$, resp.\ weight families $\{w_\alpha\}$, $\{w'_\alpha\}$ are \emph{equivalent} if $\frac{w}{w'},\frac{w'}{w}\in\CI_{{\rm uni},\fB}(M)$, resp.\ $\frac{w_\alpha}{w'_\alpha},\frac{w'_\alpha}{w_\alpha}\leq C$ (for some uniform constant $C$).
\end{definition}

Given a weight $w$, one can then define the space
\begin{equation}
\label{EqSBDwCI}
  w \CI_{\rm uni,\fB}(M) := \{ w u \colon u \in \CI_{\rm uni,\fB}(M) \}
\end{equation}
of weighted uniformly smooth functions; this space is unchanged when passing from $w$ to an equivalent weight, and the seminorms are unchanged up to equivalence.

\begin{lemma}[Weights and weight families]
\label{LemmaSBDWeights}
  Fix a nonnegative function $\chi\in\CIc((-2,2)^n)$ so that $\chi$ is strictly positive on $[-1,1]^n$.
  \begin{enumerate}
  \item\label{ItISBDWeights1} If $w$ is a weight, then $\bar\fw:=\{\sup_{U'_\alpha} w\colon\alpha\in\sA\}$ and $\ubar\fw=\{\inf_{U'_\alpha} w\colon\alpha\in\sA\}$ are weight families.
  \item\label{ItISBDWeights2} If $\fw=\{w_\alpha\colon\alpha\in\sA\}$ is a weight family, then $W(\fw):=\sum_{\alpha\in\sA} (\phi_\alpha^*\chi)w_\alpha$ is a weight.
  \item\label{ItISBDWeights3} The weights $w$ and $W(\bar\fw)$, $W(\ubar\fw)$ are equivalent; and if $w=W(\fw)$, then the weight families $\fw$ and $\bar\fw$, $\ubar\fw$ are equivalent.
  \end{enumerate}
\end{lemma}
\begin{proof}
  For the proof of part~\eqref{ItISBDWeights1}, let $C<\infty$ denote the supremum over $\alpha\in\sA$, $i=1,\ldots,n$ of the sup norms of $\frac{\pa_i(\phi_\alpha)_*w}{(\phi_\alpha)_*w}=\pa_i((\phi_\alpha)_*\log w)$ on $(-\frac32,\frac32)^n$. If $U'_\alpha\cap U'_\beta\neq\emptyset$, let $x_\alpha,x_\beta\in(-\frac32,\frac32)^n$ be such that $w(\phi_\alpha^{-1}(x_\alpha))<2\inf_{U'_\alpha}w=2\ubar\fw_\alpha$ and $w(\phi_\beta^{-1}(x_\beta))>\frac12\sup_{U'_\beta}w=\frac12\bar\fw_\beta$; and let $p_{\alpha\beta}\in U'_\alpha\cap U'_\beta$. Then
  \begin{align*}
    \biggl|\log\frac{w(\phi_\beta^{-1}(x_\beta))}{w(\phi_\alpha^{-1}(x_\alpha))}\biggr| &\leq \bigl|\log w(\phi_\beta^{-1}(x_\beta)) - \log w(p_{\alpha\beta})\bigr| + \bigl|\log w(p_{\alpha\beta}) - \log w(\phi_\alpha^{-1}(x_\alpha)) \bigr| \\
      &\leq C\bigl( |x_\beta-\phi_\beta(p_{\alpha\beta})| + |\phi_\alpha(p_{\alpha\beta})-x_\alpha| \bigr) \\
      &\leq C\cdot 2\cdot 4\sqrt{n}
  \end{align*}
  This implies $\frac{\bar\fw_\beta}{\ubar\fw_\alpha}\leq C':=4 e^{8 C\sqrt{n}}$ and thus also $\frac{\bar\fw_\beta}{\bar\fw_\alpha}\leq\frac{\bar\fw_\beta}{\ubar\fw_\alpha}\leq C'$ and $\frac{\ubar\fw_\beta}{\ubar\fw_\alpha}\leq\frac{\bar\fw_\beta}{\ubar\fw_\alpha}\leq C'$.

  For part~\eqref{ItISBDWeights2}, note that any $p\in M$ lies in at most $A<\infty$ pairwise distinct sets $U_\alpha$, and the maximum among the corresponding $w_\alpha$ is bounded from above by a uniform constant times the minimum. This implies that $\frac{\sum_\alpha w_\alpha W(\phi_\alpha^*\chi)}{\sum_\alpha w_\alpha\phi_\alpha^*\chi}\in\CI_{{\rm uni},\fB}(M)$ for all $W\in\cW$. The proof of part~\eqref{ItISBDWeights3} follows along similar lines and is left to the reader.
\end{proof}

The scaled b.g.\ structure gives rise to a particularly important weight:

\begin{definition}[Weight family from scaling]
\label{DefSBDScWeight}
  In terms of the scaling $\{\rho_\alpha\}$ of $\fB_\times$, we define
  \[
    \bar\rho := \{\bar\rho_\alpha \colon \alpha\in\sA\},\qquad \bar\rho_\alpha := \max_{1\leq i\leq n}\rho_{i,\alpha}.
  \]
  We furthermore write $\rho\in\CI(M)$ for a weight equivalent to $\bar\rho$ (such as $\rho=W(\bar\rho)$ in the notation of Lemma~\usref{LemmaSBDWeights}), and call $\rho$ a \emph{scaling weight}.
\end{definition}

Since the transition functions $\tau_{\beta\alpha}$ of $\fB$ are uniformly bounded in $\CI$, there exists a uniform constant $C$ so that $\rho_{\beta,j}\leq C\rho_{\alpha,i}$ for all $j,i$ when $U_\alpha\cap U_\beta\neq\emptyset$; by taking the maximum over $j,i=1,\ldots,n$, this shows that $\bar\rho$ is a weight family.

Symbols of $\cV$-(pseudo)differential operators will be functions on the cotangent bundle $T^*M$; it is convenient, however, to encode the scaling directly in the transition functions of a bundle ${}^\cV T^*M\to M$. Note that the transition maps $\tau_{\beta\alpha}$ induce maps between the cotangent bundles of open subsets of $(-2,2)^n$ in the usual fashion; taking the scaling into account, we write
\begin{equation}
\label{EqSBDSymbXi}
  (\tau_{\beta\alpha})_* \colon \sum_{i=1}^n\xi_{\alpha,i}\frac{\dd x_\alpha^i}{\rho_{\alpha,i}} \mapsto \sum_{j=1}^n\xi_{\beta,j}\frac{\dd x_\beta^j}{\rho_{\beta,j}}
\end{equation}
where $x_\alpha^i=\phi_\alpha^i$, $x_\beta^j=\phi_\beta^j$ are the coordinate functions, and $x_\beta=\tau_{\beta\alpha}(x_\alpha)$. Explicitly, since $\pa_{x_\alpha^i}=\sum_j(\pa_{x_\alpha^i}x_\beta^j)\pa_{x_\beta^j}$, we have
\begin{equation}
\label{EqSBDSymbXiCoord}
  \xi_{\alpha,i} = \frac{\rho_{\alpha,i}(\pa_{x_\alpha^i}x_\beta^j)}{\rho_{\beta,j}}\xi_{\beta,j}.
\end{equation}

\begin{definition}[$\cV$-cotangent bundle]
\label{DefSBDCotgt}
  The $\cV$-cotangent bundle ${}^\cV T^*M\to M$ is the bundle with local trivializations $U_\alpha\times\R^n$ and transition functions $(U_\alpha\cap U_\beta)\times\R^n\ni(p,\xi_\alpha)\mapsto(p,\xi_\beta)$ where $\xi_\alpha$, $\xi_\beta$ are related by~\eqref{EqSBDSymbXiCoord}. It is identified with $T^*M$ via $U_\alpha\times\R^n\ni(p,\xi_\alpha)\mapsto\sum\xi_{\alpha,i}\frac{\dd_p x_\alpha^i}{\rho_{\alpha,i}}\in T_{U_\alpha}^*M$.
\end{definition}

The transition functions of ${}^\cV T^*M$ are uniformly bounded in $\CI$ by~\eqref{EqSBDSymbXiCoord} and \eqref{EqbgAltPf}. Note that smooth bundle isomorphism ${}^\cV T^*M\to T^*M$ is not uniformly bounded (in the given trivializations for ${}^\cV T^*M$ and the standard trivializations $(p,\xi)\mapsto\sum_{i=1}^n \xi_i\,\dd x^i$ for $T^*M$) when $\inf_{\alpha,i}\rho_{\alpha,i}=0$. Thus, in this case, ${}^\cV T^*M$ and $T^*M$ are \emph{not} (canonically) isomorphic \emph{as vector bundles of bounded geometry} \cite[\S{A.1.1}]{ShubinBounded} (see also Remark~\ref{RmkSBDBundles} below).

For a function $p_\alpha=p_\alpha(x_\alpha,\xi_\alpha)$ where $x_\alpha\in(-2,2)^n$ and $\xi_\alpha\in\R^n$, we shall write
\[
  ((\tau_{\beta\alpha})_* p_\alpha)(x_\beta,\xi_\beta) = p_\alpha(x_\alpha,\xi_\alpha)
\]
for the same function on ${}^\cV T^*M$ expressed in the chart $U_\beta$. Furthermore, for $p\colon{}^\cV T^*M\to\C$, we write $(\phi_\alpha)_*p\colon(x_\alpha,\xi_\alpha)\mapsto p(\sum_{i=1}^n \xi_{\alpha,i}\frac{\dd x_\alpha^i}{\rho_{\alpha,i}})$ for $p$ expressed in the local trivialization.

\begin{definition}[$\cV$-symbols]
\label{DefSBDSymb}
  Let $m\in\R$. Then $S^m({}^\cV T^*M)$ consists of all functions $a\colon{}^\cV T^*M\to\C$ for which there exist constants $C_{\beta\gamma}$ for all $\beta,\gamma\in\N_0^n$ so that in all local trivializations of ${}^\cV T^*M$ one has
  \[
    | \pa_x^\beta\pa_\xi^\gamma ((\phi_\alpha)_*a)(x,\xi) | \leq C_{\beta\gamma}\la\xi\ra^{m-|\gamma|}\quad\forall\,\alpha\in\sA,\ x\in(-2,2)^n,\ \xi\in\R^n.
  \]
  For a weight $w$ on $(M,\fB)$, we moreover define the space of weighted symbols
  \[
    w S^m({}^\cV T^*M) := \{ w a\colon a\in S^m({}^\cV T^*M) \}.
  \]
  By $P^m({}^\cV T^*M)$ and $w P^m({}^\cV T^*M)$, we denote the subspace of fiberwise polynomials.
\end{definition}

Equipped with the seminorms
\[
  |a|_{S^m;k} := \sup_{\alpha\in\sA}\max_{|\beta|+|\gamma|\leq k} \sup_{(x,\xi)\in(-2,2)^n\times\R^n}\la\xi\ra^{-m+|\gamma|}|\pa_x^\beta\pa_\xi^\gamma((\phi_\alpha)_*a)(x,\xi)|,
\]
the space $S^m({}^\cV T^*M)$ is a Fr\'echet space.

\begin{definition}[$\cV$-differential operators]
\label{DefSBD}
  We write
  \[
    \Diff_\cV^m(M) := \left\{ a+\sum_{k=1}^K \sum_{j=1}^{J_k} V_{k,1}\cdots V_{k,J_k} \colon a\in\CI_{{\rm uni},\fB}(M),\ K,\,J_k\in\N_0,\ V_{k,j}\in\cV \right\}
  \]
  for the space $\cV$-differential operators. Furthermore, for a weight $w$ we set
  \[
    w\Diff_\cV^m(M) = \{ w P_0 \colon P_0\in\Diff_\cV^m(M) \}.
  \]
\end{definition}

It follows from the definition that the space $\Diff_\cV^m(M)$ is unchanged when passing to a compatible scaled b.g.\ structure; in view of Proposition~\ref{PropSBCompUniq}, the notation `$\Diff_\cV$' is therefore justified. We also remark that an $m$-th order differential operator $P\in\Diff^m(M)$ is an element of $\Diff_\cV^m(M)$ if and only if for all $\alpha\in\sA$,
\begin{equation}
\label{EqSBDPushfwdP}
  (\phi_\alpha)_*P = \sum_{|\gamma|\leq m} p_{\alpha\gamma}(x) (\rho_\alpha D)^\gamma,\qquad (\rho_\alpha D)^\gamma=\prod_{j=1}^n (\rho_{\alpha,j}i^{-1}\pa_{x^j})^{\gamma_j},
\end{equation}
where the $p_{\alpha\gamma}$ are uniformly bounded in $\CI((-2,2)^n)$. For $P\in w\Diff_\cV^m(M)$, one requires that $w_\alpha^{-1}p_{\alpha\gamma}$ be uniformly bounded in $\CI((-2,2)^n)$ where $w_\alpha$ is a weight family equivalent to $w$; or equivalently $((\phi_\alpha)_*w^{-1})p_{\alpha\gamma}$ is uniformly bounded.

\begin{thm}[Principal symbols of $\cV$-differential operators]
\label{ThmSBDSymb}
  Let $P\in\Diff_\cV^m(M)$. Define a function $p\in\CI({}^\cV T^*M)$ as follows. For $\alpha\in\sA$, write $(\phi_\alpha)_*P = \sum_{|\gamma|\leq m} p_{\alpha\gamma}(x)(\rho_\alpha D)^\gamma$. Set
  \[
    p_\alpha(x,\xi) = \sum_{|\gamma|\leq m} p_{\alpha\gamma}(x)\xi^\gamma,\qquad x\in U_\alpha,\ \xi\in\R^n,
  \]
  and put $p=\sum_\alpha \chi_\alpha\phi_\alpha^*p_\alpha$ where, for fixed $\chi\in\CIc((-2,2)^n)$ equal to $1$ on $[-1,1]^n$, we set $\chi_\alpha:=\frac{\phi_\alpha^*\chi}{\sum_\beta\phi_\beta^*\chi}$. Then, using the scaling weight $\rho$ from Definition~\usref{DefSBDScWeight}, the equivalence class
  \begin{equation}
  \label{EqSBDSymb}
    \upsigma_\cV^m(P) := [p] \in P^m({}^\cV T^*M) / \rho P^{m-1}({}^\cV T^*M)
  \end{equation}
  (called the \emph{$\cV$-principal symbol of $P$}) is well-defined, in the sense that \begin{enumerate*} \item it does not depend on the choice of $\chi$, and \item it does not change when computing it with respect to a compatible scaled b.g.\ structure.\end{enumerate*} It fits into the short exact sequence
  \begin{equation}
  \label{EqSBDSES}
    0 \to \rho\Diff_\cV^{m-1}(M) \hra \Diff_\cV^m(M) \xra{\upsigma_\cV^m} P^m({}^\cV T^*M) / \rho P^{m-1}({}^\cV T^*M) \to 0.
  \end{equation}
  Furthermore, for weighted operators $P\in w\Diff_\cV^m(M)$, the principal symbol
  \[
    \upsigma_\cV^{m,w}(P) := w \upsigma_\cV^m(w^{-1}P) \in w P^m({}^\cV T^*M)/w \rho P^{m-1}({}^\cV T^*M)
  \]
  is well-defined (in that it does not change when using an equivalent weight instead of $w$), and fits into an analogous short exact sequence. Finally, for $P_j\in w_j\Diff_\cV^{m_j}(M)$, we have $i[P_1,P_2]\in w_1 w_2\rho\Diff_\cV^{m_1+m_2-1}(M)$, with principal symbol given by the Poisson bracket
  \begin{equation}
  \label{EqSBDSymbComm}
    \upsigma_\cV^{m_1+m_2-1,w_1 w_2\rho}(i[P_1,P_2]) = \{ \upsigma_\cV^{m_1}(P_1),\upsigma_\cV^{m_2}(P_2) \}.
  \end{equation}
\end{thm}

In the formula~\eqref{EqSBDSymbComm}, the Poisson bracket of two symbols $p_1,p_2$ is given, in a local chart $U_\alpha$ with coordinates $x\in(-2,2)^n$ and $\xi\in\R^n$ on the fibers of ${}^\cV T^*M$, given by
\begin{equation}
\label{EqSBDPoisson}
  \{ p_1, p_2 \} = \sum_{i=1}^n \pa_{\xi_i}p_1\cdot\rho_{\alpha,i}\pa_{x^i}p_2 - \rho_{\alpha,i}\pa_{x^i}p_1\cdot\pa_{\xi_i}p_2.
\end{equation}
(This is equal to $\pa_\xi p_1 \pa_{\tilde x}p_2-\pa_{\tilde x}p_1 \pa_\xi p_2$ in the rescaled coordinates $\tilde x^i=\frac{x^i}{\rho_{\alpha,i}}$.)

The family $\{\chi_\alpha\}$ is a partition of unity subordinate to $\{U_\alpha\}$ which is uniformly bounded in $\CI$. We will use such partitions frequently in the remainder of the paper.

Only in the case that $\rho$ is equivalent to $1$ (which happens if and only if $\inf_\alpha\rho_\alpha>0$) can one identify $[p]$ with its fiberwise homogeneous degree $m$ part; otherwise, the principal symbol~\eqref{EqSBDSymb} is more precise. For the proof of Theorem~\ref{ThmSBDSymb} and its later pseudodifferential generalization, we now note:

\begin{lemma}[Equivalence class of symbols]
\label{LemmaSBDSymbEquiv}
  Let $\tilde S^m$ be the space of families $\{p_\alpha\colon\alpha\in\sA\}$ where $p_\alpha\colon(-2,2)^n\times\R^n\to\C$ is a symbol of order $m$, and $(\tau_{\beta\alpha})_*p_\alpha=p_\beta+r_{\beta\alpha}$ where $\bar\rho_\beta^{-1}r_{\beta\alpha}$ is uniformly bounded in the space of symbols of order $m-1$. Define an equivalence relation on $\tilde S^m$ as follows: $\{p_\alpha\}\sim\{p'_\alpha\}$ if and only if $\bar\rho_\alpha^{-1}(p_\alpha-p'_\alpha)$ is uniformly bounded in the space of symbols of order $m-1$; write $\tilde S^m/\rho\tilde S^{m-1}$ for the space of equivalence classes. Then the map
  \[
    S^m({}^\cV T^*M)/\rho S^{m-1}({}^\cV T^*M) \ni [p] \mapsto \{ (\phi_\alpha)_*p \colon \alpha\in\sA \} \in \tilde S^m/\rho\tilde S^{m-1}
  \]
  is a linear isomorphism. Its inverse is $[\{p_\alpha\}]\mapsto\sum_\alpha \chi_\alpha\phi_\alpha^*p_\alpha$.
\end{lemma}
\begin{proof}
  We only check the formula for the inverse: we have
  \begin{align*}
    (\phi_\beta)_*\biggl(\sum_\alpha\chi_\alpha\phi_\alpha^*p_\alpha\biggr) &= \sum_\alpha ((\phi_\beta)_*\chi_\alpha) (\tau_{\beta\alpha})_*p_\alpha \\
      &= \sum_\alpha ((\phi_\beta)_*\chi_\alpha) ( p_\beta + r_{\beta\alpha} ) = p_\beta + \sum_\alpha ((\phi_\beta)_*\chi_\alpha)r_{\beta\alpha},
  \end{align*}
  with $\bar\rho_\beta^{-1}$ times the second summand uniformly bounded in symbols of order $m-1$.
\end{proof}

\begin{proof}[Proof of Theorem~\usref{ThmSBDSymb}]
  If we write $x_\beta=\tau_{\beta\alpha}(x_\alpha)$, then
  \[
    \rho_{\alpha,i}\pa_{x_\alpha^i} = \sum_{j=1}^n\psi_i^j\rho_{\beta,j}\pa_{x_\beta^j},\qquad \psi_i^j:=\frac{\rho_{\alpha,i}\pa_{x_\alpha^i}x_\beta^j}{\rho_{\beta,j}};
  \]
  therefore,
  \[
    (\rho_\alpha\pa_{x_\alpha})^\gamma = \prod_{i=1}^n \biggl( \sum_{j=1}^n \psi_i^j\rho_{\beta,j}\pa_{x_\beta^j}\biggr)^{\gamma_i} =\ \prod_{i=1}^n\,\biggl(\,\sum_{|\delta|=\gamma_i} \prod_{j=1}^n (\psi_i^j)^{\delta_j}(\rho_{\beta,j}\pa_{x_\beta^j})^{\delta_j} + R_i\biggr),
  \]
  where the $(|\gamma|-1)$-th order operator $R_i$ is a product of coordinate derivatives of $\rho_{\beta,j}\pa_{x_\beta^j}\psi_i^l$ and compositions of $\rho_{\beta,l}\pa_{x_\beta^l}$ (for varying $l=1,\ldots,n$); in view of the uniform bounds~\eqref{EqISB}, the coefficients of $\bar\rho_\beta^{-1}R_i$ expressed in terms of $\rho_\beta\pa_{x_\beta}$ are thus uniformly bounded when $\alpha,\beta$ vary. Since the symbol of $(\rho_\alpha\pa_{x_\alpha})^\gamma$ pushes forward under $\tau_{\beta\alpha}$ to the symbol of $\prod_i\sum_\delta\prod_j(\psi_i^j)^{\delta_j}(\rho_{\beta,j}\pa_{x_\beta^j})^{\delta_j}$ (cf.\ \eqref{EqSBDSymbXi}), we conclude that $p_\alpha$ pushes forward under $\tau_{\beta\alpha}$ to $p_\beta$ plus a term which, when multiplied by $\bar\rho_\beta^{-1}$, is uniformly bounded in the space of symbols of order $m-1$. This completes the proof of the well-definedness of the principal symbol.

  Regarding the short exact sequence~\eqref{EqSBDSES}, only the surjectivity requires an argument. But given a family $\{p_\alpha\}$ of symbols as in Lemma~\ref{LemmaSBDSymbEquiv}, we may simply set
  \[
    P := \sum_\alpha \chi_\alpha \phi_\alpha^*P_\alpha,\qquad P_\alpha=p_\alpha(x,\rho_\alpha D),
  \]
  where for a function $p_\alpha(x,\xi)$ which is a polynomial in $\xi$, we write $p_\alpha(x,\rho_\alpha D)$ for the differential operator obtained by replacing $\xi_j$ by $\rho_{\alpha,j}D_j$.

  The proof of~\eqref{EqSBDSymbComm} follows from a short calculation in distinguished charts, with the weight $\rho$ coming from the weights $\rho_{\alpha,i}$ in \eqref{EqSBDPoisson}.
\end{proof}

\begin{example}[Laplacian for certain metrics on $\R^n$]
\label{ExSBDbsc}
  We fix the scaled b.g.\ structure on $\R^n$ from \S\ref{SssISc2}; thus $\cV$ is the space of linear combinations of the coordinate vector fields $\pa_{x^i}$ with coefficients which are uniformly bounded together with all derivatives along $\la x\ra\pa_x$. Moreover, $\rho=\la x\ra^{-1}$ is a scaling weight. Let $g=\sum_{i,j=1}^n g_{i j}\dd x^i\otimes\dd x^j$ where $(g_{i j})_{1\leq i,j\leq n}$ is uniformly positive definite and satisfies uniform bounds $|\pa_x^\beta g^{i j}(x)|\leq C_\beta\la x\ra^{-|\beta|}$ for all $\beta\in\N_0^n$. Set $L_z=\Delta_g-z$. Then $\upsigma_\cV^2(L_z)=|\xi|_{g^{-1}}^2-z$ (more precisely, its equivalence class in $P^2_\cV(T^*\R^n)/\la x\ra^{-1}P^1_\cV(T^*\R^n)$). This is elliptic (has an inverse modulo $S^{-1}$ in $S^{-2}({}^\cV T^*\R^n)$) if and only if $z\notin[0,\infty)$.
\end{example}

\begin{rmk}[Operators on vector bundles]
\label{RmkSBDBundles}
  The vector bundles of interest are \emph{vector bundles of bounded geometry} $E\to M$ \cite[\S{A.1.1}]{ShubinBounded}, which means that for each $\alpha\in\sA$ a trivialization $E|_{U_\alpha}\cong U_\alpha\times\C^k$ is chosen so that the transition functions (which are $GL(\C,k)$-valued smooth functions on $U_\alpha\cap U_\beta$) are uniformly bounded in $\CI$. If $E,F\to M$ are vector bundles of bounded geometry, then
  \[
    \Diff_\cV^m(M;E,F)\subset\Diff^m(M;E,F)
  \]
  is characterized as the subspace of operators which, in the charts $U_\alpha$ and in such trivializations of $E|_{U_\alpha}$ and $F|_{U_\alpha}$, are matrices of operators of the form~\eqref{EqSBDPushfwdP} (with coefficients which are uniformly bounded in $\CI$). We then have a principal symbol map
  \[
    \upsigma_\cV^m\colon\Diff_\cV^m(M;E,F) \to P^m({}^\cV T^*M;\pi^*\Hom(E,F)) / \rho P^{m-1}({}^\cV T^*M;\pi^*\Hom(E,F))
  \]
  where $\pi\colon{}^\cV T^*M\to M$ is the base projection. (The space on the right consists of equivalence classes of $\pi^*\Hom(E,F)$-valued symbols.) There is an analogue for weighted operators; and we have a principal symbol short exact sequence analogous to~\eqref{EqSBDSES}.
\end{rmk}

\subsection{Weighted \texorpdfstring{$\cV$}{V}-Sobolev spaces}
\label{SsSBH}

For $s\in\R$, the Sobolev space $H^s(\R^n)$ consists of all tempered distributions $u\in\sS'(\R^n)$ whose Fourier transform $\hat u(\xi)=\int e^{-i x\cdot\xi}u(x)\,\dd x$ defines an element of $\la\xi\ra^{-s}L^2(\R^n)$; and we fix the norm
\[
  \|u\|_{H^s(\R^n)} := \|\la\xi\ra^s\hat u\|_{L^2(\R^n_\xi)}.
\]

\begin{definition}[Weighted $\cV$-Sobolev spaces]
\label{DefSBH}
  Let $s\in\R$, and let $w\in\CI(M)$ be a weight on $(M,\fB)$ equivalent to a weight family $\{w_\alpha\colon\alpha\in\sA\}$. Fix $\chi\in\CIc((-2,2)^n)$ with $\chi=1$ on $[-1,1]^n$; let $\chi_\alpha=\phi_\alpha^*\chi$ and define the map\footnote{Thus $(S_\alpha)_*(\rho_{\alpha,i}\pa_{x_\alpha^i})=\pa_{\tilde x_\alpha^i}$.}
  \begin{equation}
  \label{EqSBHScale}
    S_\alpha \colon (-2,2)^n \ni x_\alpha \mapsto \tilde x_\alpha := \Bigl(\frac{x_\alpha^1}{\rho_{\alpha,1}},\ldots,\frac{x_\alpha^n}{\rho_{\alpha,n}}\Bigr).
  \end{equation}
  Then the space $w H_\cV^s(M)$ consists of all distributions $u\in\sD'(M)$ so that $(\phi_\alpha)_*(\chi_\alpha u)\in H^s(\R^n)$ for all $\alpha\in\sA$ and
  \begin{equation}
  \label{EqSBHNorm}
    \|u\|_{w H_\cV^s(M)}^2 := \sum_\alpha \|w_\alpha^{-1}(S_\alpha)_*(\phi_\alpha)_*(\chi_\alpha u)\|_{H^s(\R^n)}^2.
  \end{equation}
\end{definition}

\begin{rmk}[Uniform Sobolev spaces]
\label{RmkSBHUni}
  In the case that $\rho_{\alpha,i}=1$ for all $\alpha,i$, the maps $S_\alpha$ are the identity maps, and if further $w_\alpha=1$ for all $\alpha$, then the space $H_\cV^s(M)$ thus defined in the usual uniform Sobolev space for the b.g.\ structure $\fB$ (denoted $W_2^s(M)$ in \cite[Appendix~1]{ShubinBounded}).
\end{rmk}

As an illustration for general scalings $\rho_\alpha$, consider the special case $s\in\N$ and set $u_\alpha:=(\phi_\alpha)_*(\chi_\alpha u)$; then, using the notation~\eqref{EqSB2Der}, $\|(S_\alpha)_*u_\alpha\|_{H^s(\R^n)}^2$ is equivalent to (i.e.\ bounded above and below by an $\alpha$-independent constant by)
\begin{equation}
\label{EqSBHNorm2}
  \sum_{|\beta|\leq s} \int |(\rho_\alpha\pa_x)^\beta u_\alpha(x)|^2\,\frac{\dd x^1}{\rho_{\alpha,1}}\cdots\frac{\dd x^n}{\rho_{\alpha,n}} = \sum_{|\beta|\leq s} \int |\pa_{\tilde x}^\beta \tilde u_\alpha(\tilde x)|^2\,\dd\tilde x,
\end{equation}
where we write $\tilde u_\alpha(\tilde x)=u_\alpha(x)$ (i.e.\ $\tilde u_\alpha=(S_\alpha)_*u_\alpha$).

Returning to the case of general $s\in\R$, we note the following classical norm equivalence, whose proof we include for completeness.
\begin{lemma}
\label{LemmaSBHRnNormEq}
  We have $\|u\|_{H^s(\R^n)}^2 \sim \sum_{j\in\Z^n} \| \chi(\cdot)u(\cdot-j) \|_{H^s}^2$.
\end{lemma}

The right hand side here is essentially the same as (and in any case equivalent to) the norm~\eqref{EqSBHNorm} for $w=1$ for the bounded geometry structure~\eqref{EqIScBdd} on $\R^n$.

\begin{proof}[Proof of Lemma~\usref{LemmaSBHRnNormEq}]
  We claim that the operators
  \begin{alignat*}{4}
    \Phi_s &\colon H^s(\R^n) &&\ni u &&\mapsto \bigl(\chi(\cdot)u(\cdot+j)\bigr)_{j\in\Z^n} &&\in \ell^2\bigl(\Z^n;H^s(\R^n)\bigr), \\
    \Psi_s &\colon \ell^2\bigl(\Z^n;H^s(\R^n)\bigr) &&\ni (u_j)_{j\in\Z^n} &&\mapsto \sum_{j\in\Z^n} \chi(\cdot-j)u_j(\cdot-j) &&\in H^s(\R^n)
  \end{alignat*}
  are well-defined and bounded. Note that $\Psi_{-s}=\Phi_s^*$. By complex interpolation, it thus suffices to prove the boundedness of $\Phi_s$ and $\Psi_s$ for $s\in\N_0$. We omit the subscripts `$s$' henceforth. For $s=0$, we have
  \[ 
    \sum_{j\in\Z^n} \|\chi(\cdot)u(\cdot+j)\|_{L^2} = \sum_{j\in\Z^n}\int \chi(x-j)^2|u(x)|^2\,\dd x \leq C\int |u(x)|^2\,\dd x=C\|u\|_{L^2}^2;
  \]
  this proves the boundedness of $\Phi$ and thus of $\Psi=\Phi^*$ for $s=0$. For $s\in\N$, we argue inductively: the boundedness of $\Phi$ follows from
  \begin{align*}
    \sum_{j\in\Z^n} \|\pa_x^{\leq 1}\chi(\cdot)u(\cdot+j)\|_{H^{s-1}}^2 &\lesssim \sum_{j\in\Z^n} \|\chi(\cdot)\pa_x^{\leq 1}u(\cdot+j)\|_{H^{s-1}}^2 + \|[\pa_x,\chi]u(\cdot+j)\|_{H^{s-1}}^2 \\
      &\lesssim \|\pa_x^{\leq 1}u\|_{H^{s-1}}^2 + \|u\|_{H^{s-1}}^2 \\
      &\lesssim \|u\|_{H^s}^2,
  \end{align*}
  where in the second inequality we used that $[\pa_x,\chi]=\tilde\chi[\pa_x,\chi]$ where $\tilde\chi\in\CIc((-2,2)^n)$ equals $1$ on $\supp\chi$ together with the inductive hypothesis (applied with $\tilde\chi$ in place of $\chi$). For $\Psi$, we estimate
  \begin{align*}
    &\|\Psi((u_j)_{j\in\Z^n})\|_{H^s}^2 \\
    &\qquad \sim \|\pa_x^{\leq 1}\Psi((u_j)_{j\in\Z^n})\|_{H^{s-1}}^2 = \Biggl\| \sum_{j\in\Z^n} \pa_x^{\leq 1}(\chi(\cdot-j)u_j(\cdot-j)) \Biggr\|_{H^{s-1}}^2 \\
    &\qquad\lesssim \Biggl\| \sum_{j\in\Z^n} \chi(\cdot-j)\pa_x^{\leq 1}u_j(\cdot-j)\Biggr\|_{H^{s-1}}^2 + \Biggl\|\sum_{j\in\Z^n}\tilde\chi(\cdot-j)[\pa_x,\chi(\cdot-j)]u_j(\cdot-j) \Biggr\|_{H^{s-1}}^2 \\
    &\qquad\lesssim \|(\pa_x^{\leq 1}u_j)_{j\in\Z^n}\|_{\ell^2 H^{s-1}}^2 + \|([\pa_x,\chi]u_j)_{j\in\Z^n}\|_{\ell^2 H^{s-1}}^2 \\
    &\qquad\lesssim \|(u_j)_{j\in\Z^n}\|_{\ell^2 H^s}^2.
  \end{align*}

  We can now complete the proof. Let $X=\sum_{j\in\Z^n}\chi(\cdot-j)^2$; then $1\leq X\leq C:=\sup X<\infty$, and all derivatives of $X$ are uniformly bounded. The boundedness of $\Phi_s$ gives $\|\Phi_s u\|_{\ell^2 H^s}\lesssim\|u\|_{H^s}$. On the other hand, $\frac{1}{X}\Psi_s\Phi_s u=u$ for all $u\in H^s$; since multiplication by $\frac{1}{X}$ is bounded on every Sobolev space, this implies $\|u\|_{H^s}\lesssim\|\Psi_s\Phi_s u\|_{H^s}\lesssim\|\Phi_s u\|_{\ell^2 H^s}$.
\end{proof}

Lemma~\ref{LemmaSBHRnNormEq} implies, in view of the second expression in~\eqref{EqSBHNorm2}, that $H_\cV^s(M)$ is the uniform Sobolev space for the secondary b.g.\ structure $\rho\fB$; thus, recalling the large operator Lie algebra $\cV'$ from Definition~\ref{DefISB}, we have
\[
  H_{\cV'}^s(M) = H_\cV^s(M).
\]
We stick to the notation $H_\cV^s(M)$ however, as it is cleaner notationally; and, more subtly, only through working with $\cV$ and the scaled b.g.\ structure $\fB_\times$ is it possible to define spaces with variable decay order when $\inf_\alpha\bar\rho_\alpha=0$, cf.\ \S\ref{SsFVar}.

It is easy to see (using the invariance of Sobolev spaces under compactly supported coordinate changes) that $H_\cV^s(M)$, and thus also $w H_\cV^s(M)$, as a Hilbert space, is unchanged (up to equivalence of norms) when defined with respect to a b.g.\ structure compatible with $\rho\fB$; correspondingly, these spaces are unchanged when defined with respect to a scaled b.g.\ structure compatible with $\fB_\times$. Similarly, these spaces are unchanged when different cutoffs $\chi$ are used, or indeed when a family $\chi_\alpha\in\CIc(U_\alpha)$, $\alpha\in\sA$, of cutoffs is used subject only to the conditions that $(\phi_\alpha)_*\chi_\alpha\in\CIc((-2,2)^n)$ is uniformly bounded, all $\chi_\alpha$ are nonnegative, and $\inf_M\sum_\alpha\chi_\alpha>0$. One can then easily show that every $V\in\cV$ defines a bounded linear map
\[
  V \colon w H_\cV^s(M) \to w H_\cV^{s-1}(M).
\]
Indeed, using the norm~\eqref{EqSBHNorm}, this follows from the fact that $\tilde V_\alpha:=(S_\alpha)_*(\phi_\alpha)_*V=\sum_{i=1}^n \tilde V_\alpha^i\pa_{\tilde x^i}$, where the $\tilde V_\alpha^i$ have uniformly bounded $\CI$-seminorms on $S_\alpha((-2,2)^n)\subset\R^n$; and moreover the commutators $[V,\chi_\alpha]=V(\chi_\alpha)$ are supported in $U_\alpha$ and have uniformly bounded (in $\CI$) pushforwards along $S_\alpha\circ\phi_\alpha$ (in fact, already their pushforwards along $\phi_\alpha$ are uniformly bounded).

\bigskip

Mixed $\cV$-$\cW$-function spaces, which capture additional integer degrees of $\cW$-regularity, feature frequently in applications where solutions of a $\cV$-equation turn out to have $\cW$-regularity; see Remark~\ref{RmkIMixed}. They are also the right spaces for characterizing the mapping properties of residual operators (Definition~\ref{DefSBPsdoRes}).

\begin{definition}[Weighted $\cV$-Sobolev spaces with additional $\cW$-regularity]
\label{DefSBHVW}
  Let $s\in\R$, $k\in\N_0$, and let $w\in\CI(M)$ be a weight on $(M,\fB)$. Then $w H_{\cV;\cW}^{(s;k)}(M)$ is defined as the space of all $u\in w H_\cV^s(M)$ so that $W_1\cdots W_j u\in w H_\cV^s(M)$ for all $j\leq k$ and $W_1,\ldots,W_j\in\cW$.
\end{definition}

\begin{lemma}[Towards a norm]
\label{LemmaSBHVW}
  There exists a finite subset $\sW\subset\cW$ which spans $\cW$ over $\CI_{{\rm uni},\fB}(M)$. 
\end{lemma}
\begin{proof}
  Fix $\chi\in\CIc((-2,2)^n)$ with $\chi=1$ on $[-1,1]^n$ and consider the vector fields $W_{\alpha,i}:=\phi_\alpha^*(\chi\pa_{x^i})$. These are uniformly bounded in $\cW$. In the notation of Lemma~\ref{LemmabgP}, we now set, for $j=1,\dots,J$ and $i=1,\ldots,n$,
  \[
    W_{j,i} := \sum_{\alpha\in\sA_j} W_{\alpha,i} \in \cW.
  \]
  We claim that $\sW=\{W_{j,i}\}$ spans $\cW$ over $\CI_{{\rm uni},\fB}(M)$. Indeed, given $Z\in\cW$, write $(\phi_\alpha)_*Z=\sum_{i=1}^n Z_\alpha^i \pa_{x^i}$. Set
  \begin{align*}
    &Z_{(1)}:=Z-\sum_{\alpha\in\sA_1}\sum_{i=1}^n \phi_\alpha^*(\chi Z_\alpha^i)W_{\alpha,i} = Z-\sum_{i=1}^n f_{(1)}^i W_{1,i}, \\
    &\qquad f^i_{(1)} := \sum_{\alpha\in\sA_1} \phi_\alpha^*(\chi Z_\alpha^i) \in \CI_{{\rm uni},\fB}(M).
  \end{align*}
  Then for $\alpha\in\sA_1$, the vector field $Z_{(1)}$ pushes forward along $\phi_\alpha$ to $(\phi_\alpha)_*Z-\sum_{i=1}^n \chi Z_\alpha^i \chi\pa_{x^i}=\sum_{i=1}^n (1-\chi^2)Z_\alpha^i\pa_{x^i}$; therefore, $Z_{(1)}$ vanishes on $\bigcup_{\alpha\in\sA_1}\phi_\alpha^{-1}([-1,1]^n)$. The construction of $Z_{(1)}$ moreover ensures that $\supp_M Z_{(1)}\subset\supp_M Z$. We now apply the same procedure to $Z_{(1)}$ and $\sA_2$, thus producing $f_{(2)}^i\in\CI_{{\rm uni},\fB}(M)$ so that $Z_{(2)}:=Z_{(1)}-\sum_{i=1}^n f_{(2)}^i W_{2,i}$ has $M$-support contained in that of $Z_{(1)}$ (thus vanishing on $\bigcup_{\alpha\in\sA_1}\phi_\alpha^{-1}([-1,1]^n)$ still) and vanishes in $\bigcup_{\alpha\in\sA_2}\phi_\alpha^{-1}([-1,1]^n)$. After $J$ steps, we obtain
  \[
    Z_{(J+1)}=Z-\sum_{j=1}^J\sum_{i=1}^n f_{(j)}^i W_{j,i}
  \]
  for suitable $f_{(j)}^i\in\CI_{{\rm uni},\fB}(M)$, with $Z_{(J+1)}$ vanishing on $\bigcup_{\alpha\in\sA_j}\phi_\alpha^{-1}([-1,1]^n)$ for all $j=1,\ldots,J$; thus, $Z_{(J+1)}=0$, and the proof is complete.
\end{proof}

We can now define a norm on $w H_{\cV;\cW}^{(s;k)}(M)$ by setting
\[
  \|u\|_{w H_{\cV;\cW}^{(s;k)}}^2 := \sum_{j=0}^k \sum_{W_1,\ldots,W_j\in\sW} \|W_1\cdots W_j u\|_{w H_\cV^s}^2.
\]
Since multiplication by a function in $\CI_{{\rm uni},\fB}(M)$ defines a bounded linear map on every weighted $\cV$-Sobolev space, this norm is independent (up to equivalence) of the chosen spanning set $\sW$. Furthermore, in the notation of~\eqref{EqSBHNorm}, we have the norm equivalence
\[
  \|u\|_{w H_{\cV;\cW}^{(s;k)}}^2 \sim \sum_\alpha \sum_{|\beta|\leq k} \bigl\|w_\alpha^{-1} (S_\alpha)_*\bigl( \pa_x^\beta(\phi_\alpha)_*(\chi_\alpha u) \bigr) \bigr\|_{H^s(\R^n)}^2
\]

\begin{prop}[Boundedness of $\cV$-differential operators]
\label{PropBHDiff}
  Let $w$, $w'\in\CI(M)$ be weights on $(M,\fB)$; let $s\in\R$, $m\in\N_0$. Then every $P\in w\Diff_\cV^m(M)$ defines a bounded linear map
  \[
    P \colon w' H_\cV^s(M) \to w w' H_\cV^{s-m}(M).
  \]
  More generally, for $k\in\N_0$, we have $P\colon w' H_{\cV;\cW}^{(s;k)}(M)\to w w' H_{\cV;\cW}^{(s-m;k)}(M)$.
\end{prop}

The straightforward proof (which uses $[\cW,\cV]\subset\cW$) is left to the reader. We moreover have a useful density result; we first introduce for $\rho=(\rho_1,\ldots,\rho_n)\in(0,1]^n$ the space $H^{(s;k)}_\rho(\R^n):=H^{s+k}(\R^n)$ with norm
\begin{equation}
\label{EqSBHsk}
  \|u\|_{H_\rho^{(s;k)}(\R^n)}^2 := \sum_{|\beta|\leq k} \| (\rho^{-1}\pa_{\tilde x})^\beta u \|_{H^s(\R^n)}^2,\qquad (\rho^{-1}\pa_{\tilde x})^\beta := \prod_{j=1}^n (\rho_j^{-1}\pa_{\tilde x^j})^{\beta_j}.
\end{equation}

\begin{lemma}[Density]
\label{LemmaSBHDense}
  The space $\CIc(M)$ is dense in $w H_\cV^s(M)$ for all weights $w\in\CI(M)$ on $(M,\fB)$ and $s\in\R$, and also in $w H_{\cV;\cW}^{(s;k)}(M)$ for all $k\in\N_0$.
\end{lemma}
\begin{proof}
  Fixing a nonnegative function $\chi\in\CIc((-2,2)^n)$ with $\chi=1$ on $[-1,1]^n$, set $\chi_\alpha=\frac{\phi_\alpha^*\chi}{\sum_\beta\phi_\beta^*\chi}$. Let $u\in w H_{\cV;\cW}^{(s;k)}(M)$. Then the series $\sum_{\alpha\in\sA}\chi_\alpha u$ converges in $w H_{\cV;\cW}^{(s;k)}(M)$ to $u$, and thus the subspace consisting of function supported in finitely many distinguished open sets is dense. It remains to smooth out the localizations $u_\alpha:=(S_\alpha)_*(\phi_\alpha)_*(\chi_\alpha u)\in H_{\rho_\alpha}^{(s;k)}(\R^n)$. Let $\phi_\eps(\tilde x)=\eps^{-n}\phi(\tilde x/\eps)$, $\phi\in\CIc(\R^n)$, $\int\phi(\tilde x)\,\dd\tilde x=1$. We claim that $\phi_\eps*u_\alpha\to u_\alpha$ in $H_{\rho_\alpha}^{(s;k)}(\R^n)$ as $\eps\searrow 0$. For $k=0$, this follows by taking Fourier transforms in $\tilde x$, i.e.\ $\wh{\phi}(\eps\xi)\wh{u_\alpha}(\xi)\to\wh{u_\alpha}(\xi)$ in $\la\xi\ra^s L^2(\R^n_\xi)$. The case of general $k$ reduces to the case $k=0$ in view of $(\rho_\alpha^{-1}\pa_{\tilde x})^\beta(\phi_\eps*u_\alpha)=\phi_\eps*(\rho_\alpha^{-1}\pa_{\tilde x})^\beta u_\alpha$.
\end{proof}

We end this section by recording two elementary but crucial compactness and duality results.

\begin{thm}[Rellich compactness]
\label{ThmBHRellich}
  Let $w,w'\in\CI(M)$ be weights on $(M,\fB)$ with $w/w'\to 0$ at infinity. Let $s,s'\in\R$. Suppose that $s>s'$. Then the inclusion map
  \begin{equation}
  \label{EqBHRellich}
    w H_\cV^s(M) \hra w' H_\cV^{s'}(M)
  \end{equation}
  is compact. More generally, if $k,k'\in\N_0$, and $s\geq s'$ and $k\geq k'$, with at least one inequality strict, then the inclusion map
  \begin{equation}
  \label{EqBHRellich2}
    w H_{\cV;\cW}^{(s;k)}(M) \hra w' H_{\cV;\cW}^{(s';k')}(M)
  \end{equation}
  is compact.
\end{thm}

This includes the compact inclusion of standard weighted Sobolev spaces $\la x\ra^l H^s(\R^n)\to\la x\ra^{l'}H^{s'}(\R^n)$ (when $s>s'$, $l>l'$) as a special case, and also the inclusion of b-Sobolev spaces $\Hb^{s,l}(\bar M)\to\Hb^{s',l'}(\bar M)$ on compact manifolds $\bar M$ with boundary. Theorem~\ref{ThmBHRellich} has a converse: if $w\leq C w'$ and $s\geq s'$ are such that~\eqref{EqBHRellich} is compact, then $w/w'\to 0$ at infinity and $s>s'$. We leave the proof to the interested reader.

\begin{proof}[Proof of Theorem~\usref{ThmBHRellich}]
  For completeness, we give a proof (which is essentially standard). It suffices to consider the case $w'=1$, so $w\to 0$ at infinity. If $\{w_\alpha\}$ is an equivalent weight family, this implies that $w_\alpha\to 0$ when $\alpha\to\infty$ in $\sA$ (with the discrete topology, i.e.\ for all $\eps>0$ there exists a finite subset $\sA(\eps)\subset\sA$ so that $w_\alpha<\eps$ for $\alpha\in\sA\setminus\sA(\eps)$). Let $\chi\in\CIc((-2,2)^n)$ be nonnegative with $\chi=1$ on $[-1,1]^n$, and set $\chi_\alpha=\frac{\phi_\alpha^*\chi}{\sum_\beta\phi_\beta^*\chi}$.

  We first consider~\eqref{EqBHRellich}. Let $\{u_j\}_{j\in\N}\subset w H_\cV^s(M)$ be a bounded sequence. For each $\alpha\in\sA$, also $\{u_{j,\alpha}\}_{j\in\N}$ with $u_{j,\alpha}=(S_\alpha)_*(\phi_\alpha)_*(\chi_\alpha u)\in H^s(\R^n)$ is a bounded sequence, with norm bounded by $C w_\alpha$ where $C$ can be taken to be a constant times $\|u\|_{w H_\cV^s(M)}$, and with compact support $\supp u_{j,\alpha}\subset S_\alpha((-2,2)^n)$. We may thus pass to a subsequence of $u_{j,\alpha}$ which converges in $H^{s'}(\R^n)$ to some $v_\alpha$ with support in $S_\alpha((-2,2)^n)$. Passing to a diagonal subsequence, we may assume that we have convergence $u_{j,\alpha}\to v_\alpha\in H^{s'}(\R^n)$ for all $\alpha\in\sA$. Set
  \[
    v := \sum_\alpha \phi_\alpha^*S_\alpha^*v_\alpha.
  \]
  Then $v\in H^{s'}_\cV$ since (by Fatou's lemma)
  \[
    \sum_\alpha \|v_\alpha\|_{H^{s'}}^2 \leq \sum_\alpha \liminf_{j\to\infty} \|u_{j,\alpha}\|_{H^{s'}}^2 \leq \liminf_{j\to\infty} \sum_\alpha \|u_{j,\alpha}\|_{H^{s'}}^2 \leq C\liminf_{j\to\infty}\|u_j\|_{H^{s'}}^2 < \infty.
  \]
  Furthermore, for any fixed $\eps>0$, we have $\eps w_\alpha^{-1}\geq 1$ for $\alpha\notin\sA(\eps)$ and thus
  \[
    \|v-u_j\|_{H^{s'}}^2 \leq C\Biggl(\sum_{\alpha\in\sA(\eps)} \|v_\alpha-u_{j,\alpha}\|_{H^{s'}}^2 + \eps \sum_{\alpha\notin\sA(\eps)}w_\alpha^{-1}\|v_\alpha\|_{H^{s'}}^2+\eps\sum_{\alpha\notin\sA(\eps)} w_\alpha^{-1}\|u_{j,\alpha}\|_{H^{s'}}^2\Biggr).
  \]
  The first (finite) sum converges to $0$ as $j\to\infty$. We can estimate the third sum by $C\eps\|u_j\|_{w H_\cV^{s'}(M)}\leq C'\eps$. The second sum is estimated similarly using Fatou's lemma. We conclude that $u_j\to v$ in $H^{s'}_\cV(M)$, as claimed.

  Next, consider~\eqref{EqBHRellich2} for $s>s'$ and $k=k'$. Then all derivatives $(\rho_\alpha^{-1}\pa_{\tilde x})^\gamma u_{j,\alpha}$ for $\gamma\in\N_0^n$, $|\gamma|\leq k$, are uniformly bounded in $H^s$ by $C w_\alpha$; passing to a subsequence, we have $(\rho_\alpha^{-1}\pa_{\tilde x})^\gamma u_{j,\alpha}\to v^{(\gamma)}_\alpha\in H^{s'}(\R^n)$, and in fact $(\rho_\alpha^{-1}\pa_{\tilde x})^\gamma v_\alpha=v_\alpha^{(\gamma)}$ (where $v_\alpha=v_\alpha^{(0)}$) since differentiation commutes with distributional limits; thus $u_{j,\alpha}\to v_\alpha$ in $H_{\rho_\alpha}^{(s;k)}(\R^n)$. The remainder of the proof is then the same, \emph{mutatis mutandis}.

  Finally, when $s=s'$ but $k>k'$ (so $k-1\geq k'$), we write the inclusion~\eqref{EqBHRellich2} as a composition
  \[
    w H_{\cV;\cW}^{(s;k)}(M)\hra w H_{\cV;\cW}^{(s+1;k-1)}(M)\hra w' H_{\cV;\cW}^{(s;k-1)}(M)\hra w' H_{\cV;\cW}^{(s',k')}(M).
  \]
  The first and third arrow are continuous and the middle arrow is compact by what we have shown. This completes the proof.
\end{proof}

\begin{prop}[Duality]
\label{PropSBHDual}
  Fix a uniformly positive $\cV$-density $\mu$ on $M$, by which we mean that $(\phi_\alpha)_*\mu=\mu_\alpha(x)|\frac{\dd x_\alpha^1}{\rho_{\alpha,1}}\cdots\frac{\dd x_\alpha^n}{\rho_{\alpha,n}}|$ where $\mu_\alpha\in\CI((-2,2)^n)$ is uniformly bounded and $\inf_\alpha\inf_{(-2,2)^n}\mu_\alpha>0$. Let $w\in\CI(M)$ be a weight on $(M,\fB)$, and let $s\in\R$. Then the $L^2(M;\mu)$-pairing $(u,v)\mapsto\la u,v\ra:=\int_M u\bar v\,\dd\mu$ extends from $\CIc(M)\times\CIc(M)\to\C$ to a nondegenerate pairing $w H_\cV^s(M)\times w^{-1}H_\cV^{-s}(M)\to\C$, and thus
  \begin{equation}
  \label{EqSBHDual}
    (w H_\cV^s(M))^*=w^{-1}H_\cV^{-s}(M).
  \end{equation}
\end{prop}
\begin{proof}
  Let $\chi\in\CIc((-2,2)^n)$ be nonnegative and equal to $1$ on $[-1,1]^n$, and set $\chi_\alpha=\frac{\phi_\alpha^*\chi}{\sum_\beta\phi_\beta^*\chi}$. Note now that for $u,v\in\CIc(M)$, we can write $\la u,v\ra=\sum_{\alpha,\beta}\la \chi_\alpha u,\chi_\beta v\ra$, where for each $\alpha$ there is a uniform finite upper bound on the number of $\beta$ for which $\chi_\alpha\chi_\beta\neq 0$. Using the well-definedness of the $L^2$-pairing $H^s(\R^n)\times H^{-s}(\R^n)\to\C$ and the coordinate-invariance of $H^s$-spaces, this implies the well-definedness of the pairing~\eqref{EqSBHDual}.

  Next, let $\lambda\in(w H_\cV^s(M))^*$. Since $u\mapsto((S_\alpha)_*(\phi_\alpha)_*\chi_\alpha u)_{\alpha\in\sA}$ defines a continuous map
  \[
    \Phi\colon w H_\cV^s(M)\to\ell^2(\sA;w_\alpha H^s(\R^n))
  \]
  (where $\{w_\alpha\}$ is equivalent to $w$) with $\|\Phi(u)\|_{\ell^2 w_\alpha H^s}=\|u\|_{w H_\cV^s(M)}$, we can extend the functional $\lambda'\colon\Phi(w H_\cV^s(M))\to\C$, $\lambda'(\Phi(u)):=\lambda(u)$, to a continuous linear functional on $\ell^2 w_\alpha H^s$ which is thus realized by $(v_\alpha)_{\alpha\in\sA}\in\ell^2 w_\alpha^{-1}H^{-s}$. That is,
  \[
    \lambda(u)=\sum_\alpha\la (S_\alpha)_*(\phi_\alpha)_*\chi_\alpha u,v_\alpha\ra_{L^2(\R^n;|\dd\tilde x|)}.
  \]
  Let $f_\alpha\in\CI((-2,2)^n)$ be the (uniformly bounded) family of functions for which we have $(S_\alpha)_*(f_\alpha(\phi_\alpha)_*\mu)=|\dd\tilde x|$. We then conclude that $\lambda(u)=\la u,v\ra$ where
  \[
    v=\sum_\alpha \chi_\alpha(\phi_\alpha)^*f_\alpha S_\alpha^*v_\alpha\in w^{-1}H_\cV^{-s}(M),
  \]
  finishing the proof.
\end{proof}

\subsection{Pseudodifferential operators}
\label{SsSBPsdo}

Let $\cB_\times=\{(U_\alpha,\phi_\alpha,\rho_\alpha)\colon\alpha\in\sA\}$ denote a scaled b.g.\ structure on $M$, and let $\cV$ be its operator Lie algebra. Since our aim is to define $\cV$-pseudodifferential operators as tools for microlocal analysis, we begin by defining the locus where microlocal analysis will take place in~\S\ref{SssSBTs} before defining $\cV$-quantization maps and residual operators in~\S\ref{SssSBPsdo}. In~\S\ref{SssSBPsdo2}, we define $\cV$-pseudodifferential operators and study their composition and symbolic properties.

\subsubsection{Phase space b.g.\ structure; microlocalization locus}
\label{SssSBTs}

We shall define a scaled b.g.\ structure on ${}^\cV T^*M$, obtained by combining $\cB_\times$ with the b.g.\ structure on $\R^n$ from \S\ref{SssISc2}. Symbols will then be certain types of weighted uniformly smooth functions.

\begin{definition}[Phase space b.g.\ structure]
\label{DefSBTs}
  We use the notation~\eqref{EqISc2Sets}. For $\alpha\in\sA$, we set
  \begin{equation}
  \label{EqSBTs}
  \begin{alignedat}{2}
    U_{(\alpha,0)} &:= U_\alpha \times (-4,4)^n,&\qquad
    U_{(\alpha,j,k,\pm 1)} &:= U_\alpha \times U_{j,k,\pm 1}, \\
    \phi_{(\alpha,0)} &:= \phi_\alpha \times \phi_0, &\qquad
    \phi_{(\alpha,j,k,\pm 1)} &:= \phi_\alpha \times \phi_{j,k,\pm 1}.
  \end{alignedat}
  \end{equation}
  Set $\sA^*=\{(\alpha,0),\ (\alpha,j,k,\pm 1)\colon \alpha\in\sA,\ j=1,\ldots,n,\ k\in\N_0\}$. Then the \emph{phase space b.g.\ structure} on ${}^\cV T^*M$, with respect to the trivializations of ${}^\cV T^*M$ over each $U_\alpha$ given by Definition~\usref{DefSBDCotgt}, is
  \[
    \fB^* = \{ (U_{\alpha^*},\phi_{\alpha^*}) \colon \alpha^*\in\sA^* \}.
  \]
\end{definition}

\begin{lemma}[Symbols and $\fB^*$]
\label{LemmaSBTsSymbols}
  We use the notation~\eqref{EqSBDwCI}.
  \begin{enumerate}
  \item\label{ItSBTsSymbols0}{\rm (Symbols of order $0$.)} $S^0({}^\cV T^*M)=\CI_{\rm uni,\fB^*}({}^\cV T^*M)$.
  \item\label{ItSBTsSymbols1}{\rm (Elliptic weights.)} Let $\lambda\in S^{-1}({}^\cV T^*M)$ with $\lambda>0$ and $\lambda^{-1}\in S^1({}^\cV T^*M)$. Then $\lambda$ is a weight on $({}^\cV T^*M,\fB^*)$. Moreover, such functions $\lambda$ exist.
  \item\label{ItSBTsSymbols}{\rm (General symbols.)} $S^m({}^\cV T^*M)=\lambda^{-m}\CI_{\rm uni,\fB^*}({}^\cV T^*M)$.
  \end{enumerate}
\end{lemma}
\begin{proof}
  Part~\eqref{ItSBTsSymbols0} is a variant of the well-known characterization of symbols of order $0$ on $\R^n_\xi$ as those smooth functions $a=a(\xi)$ whose pushforwards along $\xi\mapsto 2^{-k}\xi$ define a uniformly bounded (in $k\in\N_0$) family of smooth functions on the annulus $1<|\xi|<4$. For part~\eqref{ItSBTsSymbols1}, we note that in the trivializations of ${}^\cV T^*M$ over $U_\alpha$, the functions $a_\alpha(x,\xi):=\la\xi\ra\lambda(x,\xi)$ are uniformly bounded in $S^0$, and $|a_\alpha|$ has a uniform positive lower bound. Since $\frac{W\la\xi\ra}{\la\xi\ra}\in S^0$ for all vector fields $W\in\{\la\xi\ra\pa_{\xi_i}\}$ (which span the space $\cW$ in~\eqref{EqISc2W}, with $\xi$ in place of $x$), we conclude that $\lambda$ is a weight indeed. The existence of such a weight follows by patching together $\la\xi\ra^{-1}$ in each chart using a uniform non-negative partition of unity subordinate to the cover $\{U_\alpha\}$ of $M$. Part~\eqref{ItSBTsSymbols} is an immediate consequence of the first two parts.
\end{proof}

Spaces of weighted symbols $w S^m({}^\cV T^*M)$ as in Definition~\ref{DefSBDSymb} are well-defined also for weights $w$ on phase space, i.e.\ for $w\in\CI({}^\cV T^*M)$ which are weights on $({}^\cV T^*M,\fB^*)$. A typical example of a weight $w$ on phase space is $w=\lambda^{-m}$, in which case $w S^0=S^m$ by Lemma~\ref{LemmaSBTsSymbols}\eqref{ItSBTsSymbols}. More generally, if $w_0\in\CI(M)$ is a weight on the base, then $w=w_0\lambda^{-m}$ is a phase space weight, and $w S^0=w_0 S^m$.

\begin{definition}[Microlocalization locus]
\label{DefSBTsMicro}
  Denote by $\fu{}^\cV T^*M$ the uniform compactification of $({}^\cV T^*M,\fB^*)$ (see Definition~\usref{DefbgCpt}), and recall from Proposition~\ref{PropbgAlt} the continuous extension map $j\colon\cC^0_{\rm uni,\fB^*}({}^\cV T^*M)\to\cC^0(\fu{}^\cV T^*M)$. Recall the weights $\rho,\lambda$ from Definition~\usref{DefSBDScWeight} and Lemma~\usref{LemmaSBTsSymbols}\eqref{ItSBTsSymbols1}. The \emph{microlocalization locus} of $(M,\fB_\times)$ is then the set
  \[
    \fM := \fu{}^\cV T^*M \cap \{ j(\lambda \rho)=0 \}
  \]
  Let $w\in\CI({}^\cV T^*M)$ be a weight on $({}^\cV T^*M,\fB^*)$. For a symbol $a\in w S^0({}^\cV T^*M)$, we define its \emph{elliptic set} by\footnote{We use that $a b,1\in S^0({}^\cV T^*M)=\CI_{{\rm uni},\fB^*}({}^\cV T^*M)$.}
  \[
    \Ell^w_\fM(a) := \{ \varpi\in\fM \colon \exists\,b\in w^{-1}S^0({}^\cV T^*M),\ \varpi\notin\supp_{\fu{}^\cV T^*M}(a b-1) \} \subset \fM.
  \]
  The \emph{characteristic set} is $\Char_\fM^w(a)=\fM\setminus\Ell_\fM^w(a)$. The \emph{essential support} $\esssupp^w_\fM a\subset\fM$ of $a\in\bigcup_{l,m} w\rho^{-l}S^m({}^\cV T^*M)$ is defined by
  \[
    \fM\ni\varpi\notin\esssupp^w_\fM a \iff \exists\,\chi\in S^0({}^\cV T^*M),\ \chi(\varpi)\neq 0\ \text{s.t.}\ \chi a\in w\rho^\infty S^{-\infty}({}^\cV T^*M),
  \]
  where we write $w\rho^\infty S^{-\infty}=\bigcap_{m,l} w\rho^l S^m$.
\end{definition}

Thus, $\Ell_\fM^w(a)$, resp.\ $\Char_\fM^w(a)$ and $\esssupp^w_\fM(a)$ are open, resp.\ closed subsets of $\fM\subset\fu{}^\cV T^*M$. These are the most general notions of elliptic and characteristic set and essential support which more standard versions are images of, analogously to~\eqref{EqbgSuppbarM}; this is discussed below.

In the special case $w=\lambda^{-m}$, one typically writes $\Ell_\fM^m(a)$, $\esssupp_\fM a$ instead of $\Ell_\fM^w(a)$, $\esssupp_\fM^w(a)$. (Note that for weights $w$ which are products of powers of $\rho$ and $\lambda$, one has $w\rho^\infty S^{-\infty}=\rho^\infty S^{-\infty}$.) We also remark that if $\inf\rho>0$, then $\fM=\fu{}^\cV T^*M\cap\{j(\lambda)=0\}$ simply, and all weights in $\rho$ may be omitted in Definition~\ref{DefSBTsMicro}.

\begin{rmk}[Restricting to $\{\lambda=0\}\cup\{\rho=0\}$]
  The stated condition for membership in $\Ell_\fM^w(a)$ makes sense for any $\varpi\in\fu{}^\cV T^*M$. The reason for restricting attention to $\varpi\in\fM$ is that symbols of $\cV$-(pseudo)differential operators are only well-defined modulo symbols with an additional order $\lambda\rho$ of vanishing; cf.\ Theorem~\ref{ThmSBDSymb}. Similarly, the essential support of ps.d.o.s is a reasonable notion only as a subset of $\fM$; see Definition~\ref{DefSBPsdoV}.
\end{rmk}

\begin{definition}[Fully symbolic operator Lie algebras]
\label{DefSBTsSymbolic}
  We say that $\cV$ is \emph{fully symbolic} if one of the following equivalent conditions holds:
  \begin{enumerate}
  \item $\fM=\pa(\fu{}^\cV T^*M)$.
  \item\label{ItSBTsSymbolic2} For all $\eps>0$ there exists a compact subset $K'\subset{}^\cV T^*M$ so that $\rho\lambda<\eps$ on ${}^\cV T^*M\setminus K$.
  \item\label{ItSBTsSymbolic3} For all $\eps>0$ there exists a compact subset $K\subset M$ so that $\rho<\eps$ on $M\setminus K$.
  \end{enumerate}
\end{definition}

The equivalence of the first two definitions is due to Remark~\usref{RmkbgVanish}. That condition~\eqref{ItSBTsSymbolic2} implies \eqref{ItSBTsSymbolic3} follows by projecting $K'$ to $K$ and noting that on $M\setminus K$, which we identify with the zero section $o$ in ${}^\cV T^*M$ over $M\setminus K$, we have $\rho=\lambda^{-1}\rho\lambda<\lambda^{-1}\eps$, and $\lambda^{-1}|_o$ is uniformly bounded. Conversely, if $\rho<\eps$ on $M\setminus K$, set $K':=\pi^{-1}(K)\cap\{\lambda\geq\eps(\sup_K\rho)^{-1}\}$; then on ${}^\cV T^*M\setminus K'$, we are either over $M\setminus K$ and thus $\rho<\eps$, so $\rho\lambda\leq\|\lambda\|_{L^\infty}\eps$; or we are over $K$, and then $\rho\lambda<\eps$.

For fully symbolic $\cV$, elliptic $\cV$-(pseudo)differential operators are Fredholm; see~\S\ref{SssSBPsdoApp}. The key point is that
\begin{equation}
\label{EqSBTsSymbolicEmbed}
  \text{$\cV$ is fully symbolic} \iff \text{$w\rho^l H_\cV^s(M)\to w\rho^{l'}H_\cV^{s'}(M)$\ is compact for $s>s'$, $l>l'$,}
\end{equation}
likewise for $H_{\cV;\cW}$ spaces. Indeed, since $\rho^l/\rho^{l'}=\rho^{l-l'}\to 0$ at infinity in $M$, this follows at once from Theorem~\ref{ThmBHRellich}.

In analogy to~\eqref{EqbgSingSupp}, we point out:

\begin{lemma}[Trivial essential support]
\label{LemmaSBTsEsssuppGlob}
  Let $a\in\bigcup_{l,m}w\rho^{-l}S^m({}^\cV T^*M)$. Then
  \begin{equation}
  \label{EqSBTsEsssuppGlob}
    \esssupp_\fM^w a=\emptyset \iff a\in w\rho^\infty S^{-\infty}({}^\cV T^*M).
  \end{equation}
\end{lemma}
\begin{proof}
  Only the direction `$\Longrightarrow$' requires an argument. There exists a finite collection $\chi_1,\ldots,\chi_N\in S^0({}^\cV T^*M)$, with $\bigcup_{j=1}^N\{\chi_j\neq 0\}$ covering $\fM$, so that $\chi_j a\in w\rho^\infty S^{-\infty}$ for all $j$. Set $\chi:=\sum_{j=1}^N\chi_j^2$; this has a positive lower bound on $\fM$, and $\chi a\in w\rho^\infty S^{-\infty}$. Pick (using a construction as in the proof of Lemma~\ref{LemmabgSmooth}) a function $\eta\in S^0({}^\cV T^*M)$ with $\supp\eta\cap\fM=\emptyset$ and $\{\eta\neq 0\}\cup\{\chi>0\}=\fu{}^\cV T^*M$. Then $\eta\in\rho^\infty S^{-\infty}({}^\cV T^*M)$ (since $\eta$ remains in $S^0$ upon multiplication by any power of $\lambda,\rho$), and thus $(\chi+\eta^2)a\in w\rho^\infty S^{-\infty}$. Since $\inf(\chi+\eta^2)>0$, this finally gives $a=(\chi+\eta^2)^{-1}\cdot(\chi+\eta^2)a\in w\rho^\infty S^{-\infty}$.
\end{proof}

Similarly, we have
\begin{equation}
\label{EqSBTsEllGlob}
  \Ell_\fM^w(a) = \fM \iff \exists\,b\in w^{-1}S^0({}^\cV T^*M)\ \text{s.t.}\ a b-1 \in \rho S^{-1}({}^\cV T^*M).
\end{equation}
The direction `$\Longleftarrow$' is a consequence of (the proof of) the following lemma.

\begin{lemma}[Elliptic set]
\label{LemmaSBTsEll}
  The set $\Ell_\fM^w(a)$ only depends on the equivalence class $[a]\in w S^0({}^\cV T^*M)/w\rho S^{-1}({}^\cV T^*M)$.
\end{lemma}
\begin{proof}
  Let $b\in w^{-1}S^0$ be such that $\varpi\in\fM\setminus\supp_{\fu{}^\cV T^*M}r$ where $a b=1+r$. If $\tilde a\in w\rho S^{-1}({}^\cV T^*M)$, then $(a+\tilde a)b=1+\tilde a b+r$. But $\tilde a b\in\rho S^{-1}({}^\cV T^*M)=\lambda\rho\CI_{{\rm uni},\fB^*}({}^\cV T^*M)$ vanishes at $\varpi$; therefore, there exists $\chi\in S^0({}^\cV T^*M)$ with $\chi=1$ near $\varpi$ so that $|\chi\tilde a b|<\frac12$. We can then further write $1+\tilde a b+r=1+\chi\tilde a b+r'$ where the support of $r'=(1-\chi)\tilde a b+r$ does not contain $\varpi$; and $(1+\chi\tilde a b)^{-1}\in S^0$. Thus, $\varpi\notin\supp_{\fu{}^\cV T^*M}[(a+\tilde a)\cdot b(1+\chi\tilde a b)^{-1}-1]$, as desired.
\end{proof}

For the direction `$\Longrightarrow$' of~\eqref{EqSBTsEllGlob}, we get finite collections $b_1,\ldots,b_N\in w^{-1}S^0$ and $\chi_1,\ldots,\chi_N\in S^0$ with $(a b_j-1)\chi_j=0$ and $\fM\subset\bigcup_{j=1}^N\{\chi_j\neq 0\}$. For $b_0:=\sum_{j=1}^N b_j\chi_j^2$, this gives $a b_0=\chi:=\sum_{j=1}^N\chi_j^2$. Arguing as in the proof of Lemma~\ref{LemmaSBTsEsssuppGlob}, there exists $\eta\in\rho^\infty S^{-\infty}$ with $\supp\eta\cap\fM=\emptyset$ so that $a b_0+\eta^2=\chi+\eta^2$ has a positive lower bound. For $b:=b_0(\chi+\eta^2)^{-1}$ we conclude that $\fM\cap\supp_{\fu{}^\cV T^*M}(a b-1)=\emptyset$, which implies~\eqref{EqSBTsEllGlob}.

To end this general discussion, we record a technical result on the asymptotic summation of symbols:

\begin{lemma}[Asymptotic summation]
\label{LemmaSBTsAsySum}
  Let $w\in\CI({}^\cV T^*M)$ be a weight on $({}^\cV T^*M,\fB^*)$, and let $m\in\R$. Suppose we are given $a_j\in w\rho^{-l_j}S^{m_j}({}^\cV T^*M)$ for $j\in\N_0$, where $l:=l_0\geq l_1\geq\cdots\to-\infty$ and $m:=m_0\geq m_1\geq\cdots\to-\infty$. Then there exists a symbol $a\in w\rho^{-l} S^m({}^\cV T^*M)$ so that, for all $J\in\N$,
  \begin{equation}
  \label{EqSBTsAsySum}
    a-\sum_{j=0}^{J-1}a_j \in w\rho^{-l_J}S^{m_J}({}^\cV T^*M).
  \end{equation}
  This symbol is unique modulo the space $w\rho^\infty S^{-\infty}({}^\cV T^*M)=\bigcap_{m\in\R} w\rho^m S^{-m}({}^\cV T^*M)$ of residual symbols. Furthermore, if $U\subset\fM$ is an open set with $U\cap\esssupp_\fM^w(a_j)=\emptyset$ for all $j$, then also $U\cap\esssupp_\fM^w(a)=\emptyset$.
\end{lemma}
\begin{proof}
  This is a variant of the standard asymptotic summation argument (similar to the proof of Borel's lemma). By summing finitely many successive $a_j$ and relabeling, we may assume that $l_j=l-j$ and $m_j=m-j$. Let $\lambda\in S^{-1}({}^\cV T^*M)$ be as in Lemma~\ref{LemmaSBTsSymbols}\eqref{ItSBTsSymbols1}. Let $\psi\in\CIc([0,2))$ be equal to $1$ on $[0,1]$. For a sequence $\eps_j\to 0$, we define the locally finite sum
  \begin{equation}
  \label{EqSBTsAsySumPf}
    a := \sum_{j=0}^\infty \psi\Bigl(\frac{\rho\lambda}{\eps_j}\Bigr) a_j.
  \end{equation}
  For $k\in\N_0$, $l,m\in\R$ and $b\in w\rho^{-l}S^m({}^\cV T^*M)$, denote by
  \[
    |b|_{w\rho^{-l}S^m;k} := \sup_{\alpha\in\sA}\max_{|\beta|,|\gamma|\leq k} w_\alpha^{-1}\bar\rho_\alpha^l\la\xi\ra^{-m+|\gamma|}|\pa_x^\beta\pa_\xi^\gamma((\phi_\alpha)_*a)(x,\xi)|
  \]
  the $k$-th uniform symbol seminorm of $b$. For $k\in\N$ and $j\geq k+1$, we then demand that
  \[
    \Bigl|\psi\Bigl(\frac{\rho\lambda}{\eps_j}\Bigr)a_j\Bigr|_{w\rho^{-l+k}S^{m-k};k} \leq 2^{-j};
  \]
  for each index $j\in\N_0$, this gives a finite number of conditions on $\eps_j$ which are satisfied when $\eps_j$ is sufficiently small since, a fortiori, $a_j\in w\rho^{-l+k+1}S^{m-k-1}({}^\cV T^*M)$, and on the support of $\psi(\rho\lambda/\eps_j)$ we have $\rho\lambda<2\eps_j$. Sufficiently late tails of the sum~\eqref{EqSBTsAsySumPf} then converge in any fixed uniform symbol seminorm, and thus $a\in w\rho^{-l}S^m$.

  The property~\eqref{EqSBTsAsySum} is a consequence of the fact that $(1-\psi(\frac{\rho\lambda}{\eps_j}))a_j\in w\rho^\infty S^{-\infty}$. The uniqueness statement follows from the fact that the difference of two asymptotic sums $a,a'$ satisfies $a-a'\in w\rho^{-l_j}S^{m_j}$ for all $j$ and thus is a residual symbol.

  For the final statement, consider any $\chi\in S^0({}^\cV T^*M)$ with $\fM\cap\supp\chi\subset U$. Then $\chi a_j$ is a residual symbol for all $j$, and thus so is every finite truncation of the sum~\eqref{EqSBTsAsySumPf}. But the sum over $j\geq J$ converges in $w\rho^{-l+J}S^{m-J}$; hence $\chi a\in\bigcap_{J\in\N} w\rho^{-l+J}S^{m-J}$ is residual.
\end{proof}

\subsubsection{\texorpdfstring{$\cV$}{V}-pseudodifferential operators I: quantizations and their mapping properties}
\label{SssSBPsdo}

Fix a cutoff function $\psi\in\CIc((-\frac14,\frac14)^n)$ with $\psi=1$ near $0$. For $a\in S^m(T^*\R^n)$ with $\supp a\subset(-\frac54,\frac54)^n\times\R^n$, we define a local quantization map in $\phi_\alpha(U_\alpha)$ in the coordinates $x_\alpha$, where $x_\alpha^j=\phi_\alpha^j$, by
\begin{equation}
\label{EqSBPsdoOpalpha}
\begin{split}
  \Op_\alpha(a)(x_\alpha,x'_\alpha) &:= (2\pi)^{-n} \biggl(\int_{\R^n} \exp\biggl(i\sum_{j=1}^n (x_\alpha^j-x_\alpha^{\prime j})\frac{\xi_j}{\rho_{\alpha,j}}\biggr) \psi(x_\alpha-x'_\alpha) \\
    &\quad\hspace{4.5em}\times a(x_\alpha^1,\ldots,x_\alpha^n,\xi_1,\ldots,\xi_n)\,\dd\xi_1\cdots\dd\xi_n\biggr)\,\frac{|\dd x_\alpha^{\prime 1}\cdots\dd x_\alpha^{\prime n}|}{\rho_{\alpha,1}\cdots\rho_{\alpha,n}}\,.
\end{split}
\end{equation}
We remark that $x\in(-\frac54,\frac54)^n$, $x'\in\R^n$ and $x-x'\in\supp\psi$ implies $x'\in(-\frac32,\frac32)^n$. (The quantization~\eqref{EqSBPsdoOpalpha} acts on a function or distribution $u$ on $\R^n$ via
\[
  (\Op_\alpha(a)u)(x_\alpha)=\int\Op_\alpha(a)(x_\alpha,x'_\alpha)u(x'_\alpha);
\]
the integration density is subsumed into $\Op_\alpha(a)(x_\alpha,x'_\alpha)$.)

Equivalently, we can pass to the perspective of quantizing symbols defined on $\sim\rho_{\alpha,1}^{-1}\times\cdots\times\rho_{\alpha,n}^{-1}$ grids of unit cells of the secondary bounded geometry structure $\rho\fB$ from~\S\ref{SsSB2}; this amounts to introducing
\[
  \tilde x_\alpha^j := \frac{x_\alpha^j}{\rho_{\alpha,j}}\qquad \biggl(\text{thus,}\ \ \xi_{\alpha,j}\frac{\dd x_\alpha^j}{\rho_{\alpha,j}} = \xi_{\alpha,j}\,\dd\tilde x_\alpha^j\biggr),
\]
and similarly $\tilde x'_\alpha=\frac{x_\alpha^{\prime j}}{\rho_{\alpha,j}}$, and noting that
\begin{equation}
\label{EqSBPsdoOpalpha2}
  \Op_\alpha(a)(\tilde x_\alpha,\tilde x'_\alpha) = (2\pi)^{-n} \biggl(\int_{\R^n} e^{i(\tilde x_\alpha-\tilde x'_\alpha)\cdot\xi}\psi(x_\alpha-x'_\alpha) a(\rho_{\alpha,1}\tilde x_\alpha^1,\ldots,\rho_{\alpha,n}\tilde x_\alpha^n,\xi)\,\dd\xi\biggr)\,|\dd\tilde x_\alpha'|
\end{equation}
Crucially, the cutoff $\psi(x'_\alpha-x_\alpha)$ localizes not to unit cells of $\rho\fB$, but to those of $\fB$.

\begin{definition}[$\cV$-quantization]
\label{DefSBPsdo}
  Fix a nonnegative function $\chi\in\CIc((-\frac54,\frac54)^n)$ with $\chi|_{[-1,1]^n}=1$, and let $\chi_\alpha=\frac{\phi_\alpha^*\chi}{\sum_\beta \phi_\beta^*\chi}$. Let $a\in S^m({}^\cV T^*M)$, or more generally $a\in w S^m({}^\cV T^*M)$ where $w\in\CI(M)$ is a weight on $(M,\fB)$ (see Definition~\usref{DefSBDWeights}). Then we define $\Op_\cV(a)=\sum_\alpha\phi_\alpha^*\Op_\alpha((\phi_\alpha)_*(\chi_\alpha a))$, that is, for $u\in\CI(M)$ and $p\in M$,
  \begin{equation}
  \label{EqSBPsdo}
    (\Op_\cV(a)u)(p) := \sum_{\alpha\in\sA} \phi_\alpha^*\Bigl( \bigl[\Op_\alpha\bigl( (\phi_\alpha)_*(\chi_\alpha a)\bigr)((\phi_\alpha)_*u)\bigr](\phi_\alpha(p)) \Bigr).
  \end{equation}
\end{definition}

By property~\eqref{ItIBddFinite} from Definition~\ref{DefIBdd}, and using the standard mapping properties of quantizations of symbols on $\R^n$, such quantizations $\Op_\cV(a)$ define bounded maps on $\CIc(M)$, $\CI(M)$, and indeed $w'\CI_{{\rm uni},\fB}(M)\to w w'\CI_{{\rm uni},\fB}(M)$ for any weight $w'$. (Thus, they can be composed.) More generally:

\begin{prop}[Boundedness of quantizations on mixed Sobolev spaces]
\label{PropSBPsdoH}
  Let $w,w'\in\CI(M)$ be weights on $(M,\fB)$. Let $s,m\in\R$, $k\in\N_0$, and let $a\in w S^m({}^\cV T^*M)$. Then $\Op_\cV(a)$ defines a bounded linear map $w' H_{\cV;\cW}^{(s;k)}(M)\to w w' H_{\cV;\cW}^{(s-m;k)}(M)$.
\end{prop}
\begin{proof}
  Consider first the case $k=0$. Let $\{w_\alpha\}$ and $\{w'_\alpha\}$ be weight families equivalent to $w$ and $w'$, respectively. The symbols $a_\alpha:=w_\alpha^{-1}(\phi_\alpha)_*(\chi_\alpha a)\in S^m(T^*\R^n)$ are then uniformly bounded, and correspondingly their quantizations via $\Op_\alpha$ (using the formula~\eqref{EqSBPsdoOpalpha2}) are uniformly bounded as maps $(S_\alpha)_*\Op_\alpha(a_\alpha) S_\alpha^*\colon H^s(\R_{\tilde x}^n)\to H^{s-m}(\R_{\tilde x}^n)$ (see \cite[Corollary~4.34]{HintzMicro}). But then, for $\ell:=\frac32$ and for $\alpha,\beta\in\sA$ with $U_\alpha(\ell)\cap U_\beta(\ell)\neq\emptyset$ (recalling $U_\alpha=\phi_\alpha^{-1}((-\ell,\ell)^n)$), we can estimate, for $v\in H^s(\R^n)$,
  \begin{align*}
    &\|w^{\prime -1}_\alpha(S_\alpha)_*(\phi_\alpha)_*(\chi_\alpha \phi_\beta^*\Op_\beta(a_\beta)(S_\beta^*v))\|_{H^{s-m}} \\
    &\qquad \leq C \| w^{\prime-1}_\beta(S_\beta)_*(\phi_\beta)_*(\chi_\alpha\phi_\beta^*\Op_\beta(a_\beta)(S_\beta^*v))\|_{H^{s-m}} \\
    &\qquad\leq C\bigl\| w_\beta^{\prime-1}((S_\beta)_*((\phi_\beta)_*\chi_\alpha)) (S_\beta)_*\Op_\beta(a_\beta)(S_\beta^*v) \bigr\|_{H^{s-m}} \\
    &\qquad\leq C\| w^{\prime-1}_\beta v \|_{H^s}
  \end{align*}
  for a uniform (in $\alpha,\beta$) constant $C$ (which may vary from line to line); in the first estimate, we use the uniform boundedness in $\CI$ of the transition maps $(S_\beta\circ\phi_\beta)\circ(S_\alpha\circ\phi_\alpha)^{-1}$, cf.\ the discussion of~\eqref{EqSB2Coc}. Since for fixed $\alpha\in\sA$, the number of $\beta\in\sA$ with $U_\beta(\ell)\cap U_\alpha(\ell)\neq\emptyset$ is bounded by a constant independent of $\alpha$ (see the proof of Lemma~\ref{LemmabgP}), we can add the squares of these inequalities for $v=(S_\beta)_*(\phi_\beta)_* u$ and thus arrive at the desired bound
  \[
    \|\Op_\cV(a)u\|^2_{w w' H_\cV^{s-m}(M)} \leq C\|u\|^2_{w' H_\cV^s(M)}.
  \]

  Consider now $k=1$. It again suffices to consider individual quantizations $\Op_\alpha(a_\alpha)$. Writing $\kappa_\alpha(x_\alpha,x'_\alpha)=(2\pi)^{-n}\int_{\R^n}e^{i\sum_{j=1}^n(x_\alpha^j-x_\alpha^{\prime j})\xi_j/\rho_{\alpha,j}}\psi(x'_\alpha-x_\alpha)a_\alpha(x_\alpha,\xi)\,\dd\xi$, note then that
  \begin{align}
    &\pa_{x_\alpha^i}(\Op_\alpha(a_\alpha)f)(x_\alpha) \nonumber\\
    &\qquad = \int_{\R^n} \pa_{x_\alpha^i}\kappa_\alpha(x_\alpha,x'_\alpha) f(x'_\alpha)\,\frac{\dd x_\alpha^{\prime 1}}{\rho_{\alpha,1}}\cdots\frac{\dd x_\alpha^{\prime n}}{\rho_{\alpha,n}} \nonumber\\
  \label{EqSBPsdoHk}
    &\qquad = \int \Bigl(\bigl((\pa_{x_\alpha^i}+\pa_{x_\alpha^{\prime i}})\kappa_\alpha(x_\alpha,x'_\alpha)\bigr) f(x'_\alpha) + \kappa_\alpha(x_\alpha,x'_\alpha)(\pa_{x_\alpha^{\prime i}}f)(x'_\alpha)\Bigr)\,\frac{\dd x_\alpha^{\prime 1}}{\rho_{\alpha,1}}\cdots\frac{\dd x_\alpha^{\prime n}}{\rho_{\alpha,n}}.
  \end{align}
  But
  \begin{align*}
    &(\pa_{x_\alpha^i}+\pa_{x_\alpha^{\prime i}})\kappa_\alpha(x_\alpha,x'_\alpha) \\
    &\qquad = (2\pi)^{-n}\int_{\R^n}e^{i\sum_{j=1}^n(x_\alpha^j-x_\alpha^{\prime j})\xi_j/\rho_{\alpha,j}} \psi(x'_\alpha-x_\alpha) \pa_{x_\alpha^i}a_\alpha(x_\alpha,\xi)\,\dd\xi.
  \end{align*}
  The key point is that the $x_\alpha$- and $x'_\alpha$-derivatives here only fall on the symbol $a$ \emph{which remains bounded (uniformly) upon such differentiations}. Using the uniform boundedness in $\CI$ of the transition functions $\phi_\beta\circ\phi_\alpha^{-1}$, we now obtain from~\eqref{EqSBPsdoHk} for $f=S_\beta^*v$ the uniform bounds
  \[
    \bigl\| w_\alpha^{\prime-1}(S_\alpha)_*\bigl(\pa_{x_\alpha^i}(\phi_\alpha)_*(\chi_\alpha\phi_\beta^*\Op_\beta(a_\beta)(S_\beta^*v))\bigr) \bigr\|^2_{H^{s-m}} \leq C \sum_{|\gamma|\leq 1}\bigl\| w_\beta^{\prime-1}(S_\beta)_*\bigl( \pa_{x_\beta}^\gamma (S_\beta)^*v \bigr) \bigr\|^2_{H^s},
  \]
  and thus upon summing over $\alpha,\beta$ with $v=(\phi_\beta)_*u$ the bound
  \[
    \|\Op_\cV(a)u\|^2_{w w' H_{\cV;\cW}^{(s-m;1)}(M)} \leq C\|u\|^2_{w' H_{\cV;\cW}^{(s;k)}(M)}.
  \]

  The case of general $k\in\N_0$ is completely analogous.
\end{proof}

\begin{rmk}[Regularity of the symbol]
\label{RmkSBPsdoHReg}
  For the case $k=0$, the proof shows that one only needs to require uniform bounds on the symbols $(\tilde x_\alpha,\xi)\mapsto a_\alpha(\rho_{\alpha,1}\tilde x_\alpha^1,\ldots,\rho_{\alpha,n}\tilde x_\alpha^n,\xi)$ in $S^m$. Put differently, one only needs uniform bounds by $\la\xi\ra^{m-|\gamma|}$ on $|(\rho_\alpha\pa_{x_\alpha})^\beta\pa_\xi^\gamma a_\alpha(x_\alpha,\xi)|$, corresponding to `$\cV$-regular' symbols $a$. (This statement does not use the scaled b.g.\ structure, but only the secondary b.g.\ structure $\rho\fB$. In other words, this is the well-known statement that uniform ps.d.o.s are bounded between uniform Sobolev spaces; cf.\ Remark~\ref{RmkSBHUni}.) More generally, for the boundedness of $\Op_\cV(a)$ between spaces with $k$ orders of $\cW$-regularity, it suffices for the local symbols $a_\alpha$ to be $\cV$-regular in this sense, together with their derivatives in $x_\alpha$ of orders up to $k$.
\end{rmk}

Since $\cV$-quantizations enlarge supports (unless $a$ is polynomial in the fibers of ${}^\cV T^*M)$, one expects, in general, only to be able to write $\Op_\cV(a)\circ\Op_\cV(b)=\Op_\cV(c)+R$ for some `residual' operator $R$. Let us write $d_\fB$ for a metric compatible with $\fB$, as constructed in Proposition~\ref{PropbgDist}. Residual operators should then have Schwartz kernels supported a finite distance (with respect to $d_\fB$) from the diagonal $\diag_M\subset M\times M$ since this is true for operators $\Op_\cV(a)$ and thus for compositions of finitely many such operators. Furthermore, residual operators should be smoothing (on the level of $\cV$-regularity) and improve $\rho$-weights by arbitrary powers $\rho^N$; on the other hand, they typically do not improve, but instead only \emph{preserve}, $\cW$-regularity.

\begin{definition}[Residual operators]
\label{DefSBPsdoRes}
  Let $w\in\CI(M)$ be a weight on $(M,\fB)$; recall the scaling weight $\rho\in\CI(M)$ from Definition~\usref{DefSBDScWeight}. Then the space
  \[
    w\rho^\infty\Psi_\cV^{-\infty}(M)
  \]
  of \emph{residual operators} consists of all linear operators $R\colon\CIc(M)\to\CIc(M)$ with the following properties.
  \begin{enumerate}
  \item For all weights $w'\in\CI(M)$ on $(M,\fB)$ and for all $N\in\R$, $k\in\N_0$, the operator
  \begin{equation}
  \label{EqSBPsdoResMap}
    R \colon w'\rho^{-N}H_{\cV;\cW}^{(-N;k)}(M)\to w w'\rho^N H_{\cV;\cW}^{(N;k)}(M).
  \end{equation}
  is bounded.
  \item\label{ItSBPsdoResAdj} Let $\mu\in\CI_{{\rm uni},\fB}(M;{}^\cV\Omega M)$ be a uniformly positive $\cV$-density. Then also the formal $L^2(M;\mu)$-adjoint $R^*$ of $R$ is bounded as a map~\eqref{EqSBPsdoResMap} for all $w',N,k$.
  \end{enumerate}
  We say that $R$ is \emph{localized near the diagonal} if its Schwartz kernel $K_R$ is supported uniformly closely to the diagonal, i.e.\ there exists a constant $C=C(R)$ so that $d_\fB(p,q)\leq C$ for $(p,q)\in\supp K_R\subset M\times M$. 
\end{definition}

\begin{rmk}[Localization near the diagonal]
\label{RmkSBPsdoResLoc}
  The composition of operators (residual or not) which are localized near the diagonal is again localized near the diagonal. Even though we shall not require this localization in our further development of $\cV$-ps.d.o.s, it can thus always be recovered after the fact if needed (provided of course that all operators involved in one's computation are localized near the diagonal).
\end{rmk}

We note that in the case of standard b.g.\ structures (i.e.\ $\rho=1$) where $\cW=\cV$, we have $H_{\cV;\cW}^{(s;k)}(M)=H_\cV^{s+k}(M)$, i.e.\ residual operators are simply required to be smoothing in the sense of $\cV$-regularity. We further note that condition~\eqref{ItSBPsdoResAdj} only needs to be verified for a single choice of the density $\mu$ since adjoints with respect to another density are conjugations of the original adjoint by the quotient of the densities---but such quotients lie in $\CI_{{\rm uni},\fB}(M)$, and multiplication by elements of $\CI_{{\rm uni},\fB}(M)$ defines bounded operators on every weighted (mixed) $\cV$-Sobolev space.

While typically only the mapping properties encoded in Definition~\ref{DefSBPsdoRes} matter, it is occasionally convenient to have a Schwartz kernel description of residual operators which are localized near the diagonal. Recall the spaces~\eqref{EqSBHsk}.

\begin{lemma}[Schwartz kernels of residual operators]
\label{LemmaSBPsdoResSK}
  Let $\chi\in\CIc((-2,2)^n)$ with $\chi=1$ on $[-1,1]^n$. Recall the notation $S_\alpha\colon x_\alpha\mapsto\tilde x_\alpha=(\frac{x_\alpha^j}{\rho_{\alpha,j}})_{j=1,\ldots,n}$ from Definition~\usref{EqSBHScale}. Suppose that a linear operator $R$ on $\CIc(M)$ is a residual operator with weight $w$ and equivalent weight family $\{w_\alpha\}$. Then for all $N$, the operators
  \[
    R_{\alpha,\alpha'}:=w_\alpha^{-1}\bar\rho_\alpha^{-N}\bar\rho_{\alpha'}^{-N}(S_\alpha)_*\chi (\phi_\alpha)_* R \phi_{\alpha'}^*\chi S_{\alpha'}^*\colon\CIc(S_{\alpha'}((-2,2)^n))\to\CIc(S_\alpha((-2,2)^n))
  \]
  as well as their adjoints (defined with respect to the Lebesgue measure on $\R^n$) are, for all $k\in\N_0$, uniformly bounded as operators
  \begin{equation}
  \label{EqSBPsdoResSKMap}
    R_{\alpha,\alpha'},\ R_{\alpha,\alpha'}^* \colon H^{(-N;k)}_{\rho_{\alpha'}}(\R^n)\to H^{(N;k)}_{\rho_\alpha}(\R^n).
  \end{equation}
  Conversely, if the operators~\eqref{EqSBPsdoResSKMap} are uniformly bounded and $R$ is localized near the diagonal, then $R\in w\rho^\infty\Psi_\cV^{-\infty}(M)$.
\end{lemma}

As a consequence of~\eqref{EqSBPsdoResSKMap}, the Schwartz kernels of $R_{\alpha,\alpha'}$, $R^*_{\alpha,\alpha'}$ (and thus of $R$, $R^*$) are $\CI$.

\begin{proof}[Proof of Lemma~\usref{LemmaSBPsdoResSK}]
  For $u\in H_{\rho_{\alpha'}}^{(-N;k)}(\R^n)$ with norm $\leq 1$, the distribution $\bar\rho_{\alpha'}^{-N}\phi_{\alpha'}^*\chi S_{\alpha'}^*u$ has uniformly (i.e.\ independently of $\alpha'$) bounded $\rho^{-N}H_{\cV;\cW}^{(-N;k)}(M)$-norm, and therefore $R(\bar\rho_{\alpha'}^{-N}\phi_{\alpha'}^*\chi S_{\alpha'}^*u)$ has uniformly bounded $w\rho^N H_{\cV;\cW}^{(N;k)}(M)$-norm. This implies that $R_{\alpha,\alpha'}u$ is uniformly bounded in $H^{(N;k)}_{\rho_\alpha}(\R^n)$; similarly for $R_{\alpha,\alpha'}^*$.

  For the converse, fix a nonnegative function $\eta\in\CIc((-2,2)^n)$ so that $\eta=1$ on $[-1,1]$ and $\supp\eta\subset\chi^{-1}(1)$; set $\eta_\alpha=\frac{\phi_\alpha^*\eta}{\sum_\beta\phi_\beta^*\eta}$. Then
  \begin{equation}
  \label{EqSBPsdoResSK}
    w_\alpha^{-1}\bar\rho_\alpha^{-N}\bar\rho_{\alpha'}^{-N} (S_\alpha)_*(\phi_\alpha)_*\eta_\alpha R\eta_{\alpha'} \phi_{\alpha'}^* S_{\alpha'}^* = ((S_\alpha)_*\tilde\eta_\alpha) R_{\alpha,\alpha'} ((S_{\alpha'})_*\tilde\eta_{\alpha'});
  \end{equation}
  here $\tilde\eta_\alpha=\frac{(\phi_\alpha)_*\eta_\alpha}{\chi}$ is uniformly bounded in $\CI((-2,2)^n)$, and thus $(S_\alpha)_*\tilde\eta_\alpha$ is uniformly bounded as a multiplication operator on $H^{(s;k)}_{\rho_\alpha}(\R^n)$ for all $s,k$; therefore, the operators~\eqref{EqSBPsdoResSK} are uniformly bounded from $H_{\rho_{\alpha'}}^{(-N;k)}(\R^n)\to H_{\rho_\alpha}^{(N;k)}(\R^n)$. Write now $R=\sum_{\beta,\alpha'\in\sA} \eta_\beta R\eta_{\alpha'}$. For $u\in w'\rho^{-N}H_{\cV;\cW}^{(-N;k)}(M)$, and writing $\chi_\alpha=\phi_\alpha^*\chi$, we then have
  \[
    (\phi_\alpha)_*(\chi_\alpha R u) = \sum_{\beta,\alpha'}(\phi_\alpha)_*(\chi_\alpha \eta_\beta R\eta_{\alpha'}u).
  \]
  The support condition on the Schwartz kernel of $R$ implies that there exists a constant $J<\infty$ so that for all $\alpha\in\sA$ there are at most $J$ pairwise distinct pairs $(\beta,\alpha')\in\sA\times\sA$ for which $\chi_\alpha\eta_\beta R\eta_{\alpha'}$ is not the zero operator. Let $\sA(\alpha)$ denote the set of all $\alpha'$ in such pairs; thus $|\sA(\alpha)|\leq J$, and also $d_\fB(p,q)$ has a uniform upper bound for all $p\in U_\alpha$ and $q\in U_{\alpha'}$. We then conclude that
  \begin{align*}
    \|R u\|_{w w'\rho^N H_{\cV;\cW}^{(N;k)}(M)}^2 &= \sum_\alpha \| (w_\alpha w'_\alpha)^{-1}(S_\alpha)_*(\phi_\alpha)_*(\chi_\alpha R u)\|_{H^N_{\rho_\alpha}(\R^n)}^2 \\
      &\leq C\sum_\alpha \sum_{\alpha'\in\sA(\alpha)} \| w^{\prime-1}_\alpha (S_{\alpha'})_*(\phi_{\alpha'})_*(\chi_{\alpha'}u) \|_{H^{-N}_{\rho_{\alpha'}}(\R^n)}^2 \\
      &\leq C\|u\|_{w' H_{\cV;\cW}^{(-N;k)}(M)}^2,
  \end{align*}
  where the constant $C$ may change from line to line; in the second inequality, we use that for every $\alpha'\in\sA$ there are at most $J'$ many $\alpha\in\sA$ with $\alpha'\in\sA(\alpha)$, where $J'$ does not depend on $\alpha'$. This completes the proof.
\end{proof}

\begin{rmk}[Topology on $w\rho^\infty\Psi_\cV^{-\infty}(M)$]
\label{RmkSBPsdoResTop}
  On the space $w\rho^\infty\Psi_\cV^{-\infty}(M)$, we take the operator norms of the maps~\eqref{EqSBPsdoResMap} and of their $L^2$-adjoints as seminorms. Thus, $w\rho^\infty\Psi_\cV^{-\infty}(M)$ is a Fr\'echet space.
\end{rmk}

\begin{cor}[Localizing residual Schwartz kernels]
\label{CorSBPsdoResSKMult}
  Write $\fB\times\fB=\{(U_\alpha\times U_{\alpha'},\phi_\alpha\times\phi_{\alpha'})\colon\alpha,\alpha'\in\sA\}$ for the product b.g.\ structure on $M\times M$, and let
  \[
    \CI_{{\rm uni},\fB\times\fB,\diag}(M\times M)\subset\CI_{{\rm uni},\fB\times\fB}(M)
  \]
  denote the subring of functions $f\in\CI_{{\rm uni},\fB\times\fB}(M\times M)$ for which there exists $C<\infty$ so that $f$ is constant in $\{(p,q)\in M\times M\colon d_\fB(p,q)>C\}$. Then the space $w\rho^\infty\Psi_\cV^{-\infty}(M)$ is a module over $\CI_{{\rm uni},\fB\times\fB,\diag}(M\times M)$ via pointwise multiplication of Schwartz kernels. In fact, it is a module over $\bigcup_{N,N'\in\R}(\pi_1^*\rho)^{-N}(\pi_2^*\rho)^{-N'}\CI_{{\rm uni},\fB\times\fB,\diag}(M\times M)$ where $\pi_1,\pi_2\colon M\times M\to M$ are the two projection maps.
\end{cor}
\begin{proof}
  The final statement is a straightforward consequence of the first. Let thus $f\in\CI_{{\rm uni},\fB\times\fB,\diag}(M\times M)$; it suffices to consider the situation that $f$ is localized near the diagonal. Set $f_{\alpha,\alpha'}=(\phi_\alpha\times\phi_{\alpha'})_*f$, which is thus uniformly bounded in $\CI((-2,2)^n\times(-2,2)^n)$. By Lemma~\ref{LemmaSBPsdoSKRes} with $\chi\in\CI((-\frac32,\frac32)^n)$, it suffices to show that if $R_{\alpha,\alpha'}\colon H_{\rho_{\alpha'}}^{(-N;k)}(\R^n)\to H_{\rho_\alpha}^{(N;k)}(\R^n)$, with uniformly bounded operator norm, has (smooth) Schwartz kernel $K_{R_{\alpha,\alpha'}}$ supported in $S_\alpha((-\frac32,\frac32)^n)\times S_{\alpha'}((-\frac32,\frac32)^n)$, then the operator norm of the operator with Schwartz kernel $((S_\alpha\times S_{\alpha'})_*f_{\alpha,\alpha'})K_{R_{\alpha,\alpha'}}$ as a map $H_{\rho_{\alpha'}}^{(-N;k)}(\R^n)\to H_{\rho_\alpha}^{(N;k)}(\R^n)$ is uniformly bounded as well. (The bounds for $R_{\alpha,\alpha'}^*$ are then proved in exactly the same fashion.) Using the uniform boundedness of the $\CI$ seminorms of $f_{\alpha,\alpha'}$, it suffices to verify this in the case $k=0$.

  We restate the remaining task as follows: if $f\in\CIc((0,1)^n\times(0,1)^n)$ and $R\colon H^{-N}(\R^n)\to H^N(\R^n)$, $(R u)(\tilde x)=\int K_R(\tilde x,\tilde x')u(\tilde x')\,\dd\tilde x'$, has operator norm $\leq 1$ and satisfies $\supp K_R\subset S_\alpha((0,1)^n)\times S_{\alpha'}((0,1)^n)$, then the operator with Schwartz kernel $((S_\alpha\times S_{\alpha'})_*f)K_R$ maps $H^{-N}(\R^n)\to H^N(\R^n)$ with operator norm bounded by some $\CI$-seminorm of $f$. Let $\chi\in\CIc((0,1)^n)$ be such that $\supp f\subset\chi^{-1}(1)\times\chi^{-1}(1)$, and expand $f$ in Fourier series via
  \[
    f(x,x') = \chi(x)\chi(x')f(x,x') = \sum_{k,k'\in\Z^n} f_{k,k'} \chi(x)e^{2\pi i k\cdot x} \chi(x')e^{2\pi i k'\cdot x'},
  \]
  where, for all $N'\in\N$, we have the bound
  \[
    |f_{k,k'}|\leq C_{N'}(1+|k|+|k'|)^{-2 N'}\leq C'_{N'}(1+|k|)^{-N'}(1+|k'|)^{-N'}
  \]
  where $C'_{N'}$ is bounded by a seminorm of $f$. Consider now a term
  \begin{align*}
    &(S_\alpha\times S_{\alpha'})_*\bigl(\chi(x)e^{2\pi i k\cdot x} \chi(x')e^{2\pi i k'\cdot x'}\bigr) K_R(\tilde x,\tilde x') \\
    &\qquad = \chi(S_\alpha^{-1}\tilde x) e^{2\pi i k\cdot S_\alpha^{-1}(\tilde x)} K_R(\tilde x,\tilde x') \chi(S_{\alpha'}^{-1}\tilde x') e^{2\pi i k'\cdot S_{\alpha'}^{-1}(\tilde x')}.
  \end{align*}
  Now, multiplication with $\chi(S_\alpha^{-1}\tilde x)=\chi(\rho_{\alpha,1}\tilde x^1,\ldots,\rho_{\alpha,n}\tilde x^n)$ is uniformly (i.e.\ independently of $\rho_\alpha\in(0,1]^n$) bounded on $H^s(\R^n)$, as follows for $s\in\N_0$ by direct differentiation, thus for $s\geq 0$ by interpolation, and finally for $s\leq 0$ by duality. On the other hand, the operator norm on $H^s(\R^n)$ of multiplication by $e^{2\pi i k'\cdot S_\alpha^{-1}(\tilde x)}=\exp(2\pi i \rho_{\alpha,j}k'_j \tilde x^j)$ is bounded by $\breve C_s(1+|k'|)^s$ for $s\in\N_0$, thus also for $s\geq 0$ by interpolation, and thus by $\breve C_{|s|}(1+|k'|)^{|s|}$ for $s\in\R$ by duality. For any $N\geq 0$, we thus have an operator norm bound
  \begin{align*}
    &\|(S_\alpha\times S_{\alpha'})_*\bigl(f_{k,k'}\chi(x)e^{2\pi i k\cdot x} \chi(x')e^{2\pi i k'\cdot x'}\bigr) K_R(\tilde x,\tilde x')\|_{\cL(H^{-N}(\R^n),H^N(\R^n))} \\
    &\qquad \leq C'_{N'} \breve C_N (1+|k|)^{-N'+N} (1+|k'|)^{-N'+N}.
  \end{align*}
  For fixed $N$, we then take $N'>N+n$ and obtain, by summing over $k,k'\in\Z^n$, the desired operator norm bound on $((S_\alpha\times S_{\alpha'})_*f)K_R$.
\end{proof}

That Definition~\ref{DefSBPsdoRes} is consistent with the properties of quantizations of residual symbols is the content of the following result.

\begin{prop}[Quantizations of residual symbols]
\label{PropSBPsdoOpRes}
  Consider a residual symbol $a\in w\rho^\infty S^{-\infty}({}^\cV T^*M)$. Then $\Op_\cV(a)\in w\rho^\infty\Psi_\cV^{-\infty}(M)$. More generally, if $A$ is an operator which can be written as $A=\Op_\cV(a_m)$ with $a_m\in w\rho^m S^{-m}({}^\cV T^*M)$ for a sequence of orders $m$ tending to $-\infty$, then $A\in w\rho^\infty\Psi_\cV^{-\infty}(M)$.
\end{prop}
\begin{proof}
  Only the mapping properties of the adjoint $\Op_\cV(a)^*$ with respect to a uniformly positive $\cV$-density in Definition~\ref{DefSBPsdoRes}\eqref{ItSBPsdoResAdj} do not follow directly from Proposition~\ref{PropSBPsdoH}. But the adjoints of the local quantizations $\Op_\alpha$ in~\eqref{EqSBPsdo} are again quantizations of symbols with uniformly controlled seminorms by Lemma~\ref{LemmaSBPsdoLeft} below; and these are in turn left quantizations of uniformly controlled symbols by the same Lemma. Thus (the proof of) Proposition~\ref{PropSBPsdoH} again applies and finishes the proof.
\end{proof}

To complete the proof, we need two lemmas (which are essentially standard, cf.\ \cite[Proposition~4.10, Theorem~4.8]{HintzMicro}).

\begin{lemma}[Schwartz kernels as oscillatory integrals]
\label{LemmaSBPsdoSKRes}
  Let $\alpha\in\sA$ and set $x_\alpha^i=\phi_\alpha^i$ and $\tilde x_\alpha^i=\frac{x_\alpha^i}{\rho_{\alpha,i}}$, likewise $\tilde x_\alpha^{\prime i}=\frac{x^{\prime i}_\alpha}{\rho_{\alpha,i}}$. Let $\kappa=\kappa(x_\alpha,x'_\alpha)$. Suppose that for all $\beta,\beta',\gamma,N$,
  \begin{equation}
  \label{EqSBPsdoSKRes}
    |(\rho_\alpha\pa_{x_\alpha})^\beta(\rho_\alpha\pa_{x_\alpha'})^{\beta'}(\pa_{x_\alpha}+\pa_{x_\alpha'})^\gamma\kappa(x_\alpha,x'_\alpha)| \leq C_{\beta\beta'\gamma N}(1+|\tilde x_\alpha-\tilde x'_\alpha|)^{-N}
  \end{equation}
  in the notation~\eqref{EqSB2Der}. Then there exists a unique symbol $a\in S^{-\infty}(T^*\R^n)$ so that
  \[
    \kappa(x_\alpha,x'_\alpha) = (2\pi)^{-n}\int_{\R^n} e^{i\sum_{j=1}^n (x_\alpha^j-x_\alpha^{\prime j})\xi_{\alpha,j}/\rho_{\alpha,j}} a(x_\alpha,\xi_\alpha)\,\dd\xi_\alpha.
  \]
  We have $\supp a\subset\pi_1(\supp\kappa)\times\R^n$ where $\pi_1\colon\R^n\times\R^n\to\R^n$ is the projection to the first factor. Moreover, the symbol seminorms of $a$ are bounded by uniform constants (i.e.\ independent of $\alpha$) times the optimal constants $C_{\beta\beta'\gamma N}$ in~\eqref{EqSBPsdoSKRes}.
\end{lemma}

Since $\rho_{\alpha,j}\pa_{x_\alpha^{\prime j}}=-\rho_{\alpha,j}\pa_{x_\alpha^j}+\rho_{\alpha,j}(\pa_{x_\alpha^j}+\pa_{x_\alpha^{\prime j}})$, requiring~\eqref{EqSBPsdoSKRes} only for $\beta'=0$ implies~\eqref{EqSBPsdoSKRes} for all $\beta'$ (and likewise with $\beta$ in place of $\beta'$).

\begin{proof}[Proof of Lemma~\usref{LemmaSBPsdoSKRes}]
  Since $\kappa$ is the inverse Fourier transform $\check a(x_\alpha,\tilde z)|_{\tilde z^j=(x^j_\alpha-x^{\prime j}_\alpha)/\rho_{j,\alpha}}$ of $a$ in $\xi_\alpha$, we must have (writing $x^{\prime j}_\alpha=x_\alpha^j-\rho_{\alpha,j}\tilde z^j$)
  \[
    a(x_\alpha,\xi_\alpha) = \int_{\R^n} e^{-i\tilde z\cdot\xi_\alpha}\kappa(x_\alpha^1,\ldots,x_\alpha^n,x_\alpha^1-\rho_{\alpha,1}\tilde z^1,\ldots,x_\alpha^n-\rho_{\alpha,n}\tilde z^n)\,\dd\tilde z^1\cdots\dd\tilde z^n.
  \]
  The rapid $\tilde z$-decay of the integrand implies the convergence of the integral and all of its $\xi_\alpha$-derivatives. Furthermore, $|\xi_\alpha|^{2 N} e^{-i\tilde z\cdot\xi_\alpha}=(-\pa_{\tilde z}^2)^N e^{-i\tilde z\cdot\xi_\alpha}$; integrating by parts in $\tilde z$ produces $2 N$-fold derivatives of $\kappa$ along the vector fields $\rho_{\alpha,j}\pa_{x_\alpha^{\prime j}}$. Finally, the $\pa_{x_\alpha}$-regularity of $a$ follows from the $(\pa_{x_\alpha}+\pa_{x'_\alpha})$-regularity of $\kappa$.
\end{proof}

\begin{lemma}[Left and right reduction]
\label{LemmaSBPsdoLeft}
  Let $a=a(x_\alpha,x'_\alpha,\xi_\alpha)\in S^m(\R^n\times\R^n;\R^n)$; that is, for all $k\in\N_0$,
  \[
    |a|_{S^m;k} := \max_{|\beta|+|\beta'|+|\gamma|\leq k} \sup_{(x_\alpha,x'_\alpha,\xi_\alpha)\in\R^n\times\R^n\times\R^n}\la\xi\ra^{-m+|\gamma|}|\pa_{x_\alpha}^\beta\pa_{x'_\alpha}^{\beta'}\pa_{\xi_\alpha}^\gamma a(x_\alpha,x'_\alpha,\xi_\alpha)| < \infty.
  \]
  Then there exists a unique symbol $a_L\in S^m(T^*\R^n)$ so that
  \begin{align}
  \label{EqSBPsdoLeftQuant}
    &\Op'_\alpha(a) := (2\pi)^{-n}\int_{\R^n} e^{i\sum_{j=1}^n(x_\alpha^j-x_\alpha^{\prime j})\xi_{\alpha,j}/\rho_{\alpha,j}} a(x_\alpha,x'_\alpha,\xi_\alpha)\,\dd\xi_\alpha \\
    &\quad =\Op'_\alpha(a_L) = (2\pi)^{-n}\int_{\R^n} e^{i\sum_{j=1}^n(x_\alpha^j-x_\alpha^{\prime j})\xi_{\alpha,j}/\rho_{\alpha,j}}a_L(x_\alpha,\xi_\alpha)\,\dd\xi_\alpha. \nonumber
  \end{align}
  The symbol $a_L$ has the following additional properties:
  \begin{enumerate}
  \item\label{ItSBPsdoLeft1} for $k\in\N_0$, define (using the notation~\eqref{EqSB2Der})
    \[
      q_k(x_\alpha,\xi_\alpha) := a_L(x_\alpha,\xi_\alpha) - \sum_{|\beta|\leq k-1} \frac{1}{\beta!} \bigl(\pa_{\xi_\alpha}^\beta(\rho_\alpha D_{x'_\alpha})^\beta a(x_\alpha,x'_\alpha,\xi_\alpha)\bigr)\big|_{x'_\alpha=x_\alpha}.
    \]
    Then the $S^{m-k}(T^*\R^n)$-seminorms of $\bar\rho_\alpha^{-k}q_k$ are bounded uniformly (i.e.\ with constant independent of $\alpha$) by the $S^m$-seminorms of $a$;
  \item\label{ItSBPsdoLeft2} if $\chi,\tilde\chi\in S^0(T^*\R^n)$ are such that $\supp\chi\subset\tilde\chi^{-1}(1)$ and
  \[
    a_{\tilde\chi}(x_\alpha,x'_\alpha,\xi_\alpha):=\tilde\chi(x_\alpha,\xi_\alpha)a(x_\alpha,x'_\alpha,\xi_\alpha)
  \]
  is an element of $S^{m'}$, then $\chi a_L\in S^{m'}$, with seminorms bounded uniformly by those of $a_{\tilde\chi}$. The same conclusion holds under the assumption that $(x_\alpha,x'_\alpha,\xi_\alpha)\mapsto\tilde\chi(x'_\alpha,\xi_\alpha)a(x_\alpha,x'_\alpha,\xi_\alpha)$ lies in $S^{m'}$.
  \end{enumerate}
  Similarly, $\Op'_\alpha(a)=\Op'_\alpha(a_R)$ for a unique $a_R=a_R(x'_\alpha,\xi_\alpha)\in S^m$ which is an asymptotic sum $a_R\sim\sum_\beta \frac{1}{\beta!}(\pa_{\xi_\alpha}^\beta(-\rho_\alpha D_{x_\alpha})^\beta a)|_{x_\alpha=x'_\alpha}$.
\end{lemma}
\begin{proof}
  We write the Taylor expansion of $a$ around the diagonal $x'_\alpha=x_\alpha$ as
  \begin{align*}
    a(x_\alpha,x'_\alpha,\xi_\alpha) &= \sum_{|\beta|\leq k-1} \frac{1}{\beta!}((\rho_\alpha\pa_{x_\alpha'})^\beta a)(x_\alpha,x_\alpha,\xi_\alpha) \Bigl(\frac{x'_\alpha-x_\alpha}{\rho_\alpha}\Bigr)^\beta + r_k(x_\alpha,x'_\alpha,\xi_\alpha), \\
    r_k(x_\alpha,x'_\alpha,\xi_\alpha) &= \sum_{|\beta|=k} \frac{k}{\beta!} \int_0^1 (1-t)^{k-1} ((\rho_\alpha\pa_{x'_\alpha})^\beta a)(x_\alpha,x_\alpha+t(x'_\alpha-x_\alpha),\xi_\alpha) \Bigl(\frac{x'_\alpha-x_\alpha}{\rho_\alpha}\Bigr)^\beta\,\dd t.
  \end{align*}
  Here, for $z=(z^1,\ldots,z^n)\in\R^n$, we write $(z/\rho_\alpha)^\beta:=\prod_{j=1}^n (z^j/\rho_{\alpha,j})^{\beta_j}$. In the oscillatory integral~\eqref{EqSBPsdoLeftQuant}, write then $(\frac{x'_\alpha-x_\alpha}{\rho_\alpha})^\beta e^{i(x_\alpha-x'_\alpha)\cdot\xi_\alpha/\rho_\alpha}=(-D_{\xi_\alpha})^\beta e^{i(x_\alpha-x'_\alpha)\cdot\xi_\alpha/\rho_\alpha}$ and integrate by parts in $\xi_\alpha$. Let $\tilde a_L$ denote an asymptotic sum of $\frac{1}{\beta!}(\pa_{\xi_\alpha}^\beta(\rho_\alpha D_{x'_\alpha})^\beta a)|_{x'_\alpha=x_\alpha}$ (see Lemma~\ref{LemmaSBTsAsySum}). Then the difference
  \[
    \kappa_R := \Op'_\alpha(a) - \Op'_\alpha(\tilde a_L),
  \]
  a function of $(x_\alpha,x'_\alpha)$, can be written as
  \[
    \kappa_R = \Op'_\alpha(q_k) + \Op'_\alpha(r_k).
  \]
  For any fixed $\beta,\beta',\gamma,N$, the estimate~\eqref{EqSBPsdoSKRes} holds for both $\Op'_\alpha(q_k)$ and $\Op'_\alpha(r_k)$ (by the usual integration by parts arguments) when $k$ is sufficiently large (thus $q_k$ is a symbol of sufficiently negative order). Lemma~\ref{LemmaSBPsdoSKRes} thus implies that $\kappa_R=\Op'_\alpha(r)$ for a symbol $r=r(x_\alpha,\xi_\alpha)$ so that the $S^{m-k}$-seminorms of $\bar\rho_\alpha^{-k}r$ are bounded uniformly by the $S^m$-seminorms of $a$ for all $k$. Setting $a_L:=\tilde a_L+r$ completes the construction; part~\eqref{ItSBPsdoLeft1} is immediate.

  For part~\eqref{ItSBPsdoLeft2}, we only need to note that for any differential operator $D$ in $(x_\alpha,x'_\alpha,\xi_\alpha)$, one has $\chi D a_{\tilde\chi}=\chi\tilde\chi D a+\chi[D,\tilde\chi]a=\chi D a$.
\end{proof}

Finally, we record that localizations of quantizations of general symbols away from the diagonal (as measured using $d_\fB$) are residual; this allows one to localize Schwartz kernels arbitrarily closely to the diagonal.

\begin{lemma}[Off-diagonal behavior of quantizations]
\label{LemmaSBPsdoOffDiag}
  Let $\eps>0$, and let $\eta\in\CI_{{\rm uni},\fB\times\fB}(M\times M)$ be $1$ in the $\eps$-neighborhood $\{(p,q)\in M\times M\colon d_\fB(p,q)<\eps\}$ of the diagonal in $M\times M$. Let $w$ be a weight, and let $a\in w S^m({}^\cV T^*M)$. With $\Op_\cV(a)$ denoting the Schwartz kernel, we then have $(1-\eta)\Op_\cV(a)\in w\rho^\infty\Psi_\cV^{-\infty}(M)$.
\end{lemma}
\begin{proof}
  In view of Corollary~\ref{CorSBPsdoResSKMult}, it suffices to note that if $\eta_0\in\CIc(\R^n)$ is equal to $1$ near $0$ (and is chosen to satisfy $\supp((\phi_\alpha\times\phi_\alpha)^*\eta_0)\subset\{d_\fB(p,q)<\eps\}$ for all $\alpha\in\sA$), then in the integral kernel
  \[
    \int_{\R^n} e^{i\sum_{j=1}^n(x_\alpha^j-x_\alpha^{\prime j})\xi_j/\rho_{\alpha,j}} \psi(x'_\alpha-x_\alpha) \bigl(1-\eta_0(x_\alpha-x'_\alpha)\bigr) a_\alpha(x_\alpha,\xi)\,\dd\xi
  \]
  we can act on the exponential factor with $|x_\alpha-x_\alpha'|^{-2}\sum_{j=1}^n\rho_{\alpha,j}(x_\alpha^j-x_\alpha^{\prime j})i^{-1}\pa_{\xi_j}$, and integrate by parts in $\xi$. If $a_\alpha$ is uniformly bounded in $S^m$, say, then $\rho_{\alpha,j}\pa_{\xi_j}a_\alpha(x_\alpha,\xi)$ is uniformly bounded in $\bar\rho_\alpha S^{m-1}$ for all $j=1,\ldots,n$. By iterating this argument, one can replace $a_\alpha$ by a symbol with uniform bounds in $\bar\rho_\alpha^N S^{m-N}$ for any desired $N$. Thus, Proposition~\ref{PropSBPsdoOpRes} applies and finishes the proof.
\end{proof}

\subsubsection{\texorpdfstring{$\cV$}{V}-pseudodifferential operators II: definition, symbolic properties, composition}
\label{SssSBPsdo2}

We now merge Definitions~\ref{DefSBPsdo} and \ref{DefSBPsdoRes} into the central definition of the present paper.

\begin{definition}[$\cV$-ps.d.o.s]
\label{DefSBPsdoV}
  Let $m\in\R$, and let $w\in\CI(M)$ be a weight on $(M,\fB)$. Then the space of \emph{weighted $\cV$-pseudodifferential operators (of order $m$ with weight $w$)} is defined by
  \[
    w\Psi_\cV^m(M) := \{ \Op_\cV(a) + R \colon a\in w S^m({}^\cV T^*M),\ R\in w\rho^\infty\Psi_\cV^{-\infty}(M) \}.
  \]
  In the notation of Definition~\usref{DefSBTsMicro}, the \emph{principal symbol} of $A=\Op_\cV(a)+R$ is
  \[
    \upsigma_\cV^{m,w}(A) := [a] \in w S^m({}^\cV T^*M)/w\rho S^{m-1}({}^\cV T^*M),
  \]
  and its \emph{operator wave front set} is $\WF^{\prime w}_\cV(A):=\esssupp_\fM^w a$.
\end{definition}

When the weight $w$ is clear from the context, we shall write $\WF'_\cV(A)$ instead of $\WF^{\prime w}_\cV(A)$. Directly from the definition, we have a short exact sequence
\[
  0 \to w\rho\Psi_\cV^{m-1}(M) \hra w\Psi_\cV^m(M) \xra{\upsigma_\cV^m} w S^m({}^\cV T^*M)/w\rho S^{m-1}({}^\cV T^*M) \to 0.
\]

Using Lemma~\ref{LemmaSBTsEll}, we define the \emph{elliptic set} of $A$ as $\Ell_\cV^{m,w}(A) := \Ell^{w\lambda^{-m}}_\fM(\sigma_\cV^{m,w}(A))$; and the \emph{characteristic set} as $\Char_\cV^{m,w}(A)=\fM\setminus\Ell_\cV^{m,w}(A)$.

\begin{thm}[Boundedness of $\cV$-ps.d.o.s]
\label{ThmSBPsdoH}
  Let $w,w'\in\CI(M)$ be weights on $(M,\fB)$; let $s,m\in\R$. Then every $A\in w\Psi_\cV^m(M)$ defines a bounded linear map $w' H_\cV^s(M)\to w w' H_\cV^{s-m}(M)$, and more generally $w' H_{\cV;\cW}^{(s;k)}(M)\to w w' H_{\cV;\cW}^{(s-m;k)}(M)$ for all $k\in\N_0$.
\end{thm}
\begin{proof}
  This is an immediate consequence of Proposition~\ref{PropSBPsdoH} and Definition~\ref{DefSBPsdoRes}.
\end{proof}

\begin{rmk}[Other function spaces]
\label{RmkSBSobOther}
  Since uniform ps.d.o.s on $\R^n$ are bounded also between $L^p$-based Sobolev spaces for $p\in(1,\infty)$ as well as H\"older spaces of non-integer order, Theorem~\ref{ThmSBPsdoH} generalizes to the boundedness of $\cV$-ps.d.o.s also between analogues of such spaces for b.g.\ structures (and also to refinements of such spaces which measure additional $k\in\N_0$ degrees of $\cW$-regularity). This only requires the function spaces to contain, resp.\ be contained in, $\bigcap_{m,l} w\rho^l H_\cV^m(M)$, resp.\ $\bigcap_{m,l} w\rho^l H_\cV^m(M)$ (and the versions with additional $\cW$-regularity).
\end{rmk}

\begin{rmk}[Ps.d.o.s on vector bundles]
\label{RmkSBPsdoVB}
  When $E,F\to M$ are vector bundles of bounded geometry (see Definition~\ref{RmkSBDBundles}), one defines $w\Psi_\cV^m(M;E,F)$ as sums of matrices of quantizations of symbols of class $w S^m$ in local charts and residual operators which are smoothing (in the sense of $\cV$-regularity and $\rho$-decay) between Sobolev spaces of sections of $E,F$. We leave the details to the interested reader.
\end{rmk}

\begin{rmk}[Topology on $w\Psi_\cV^m(M)$]
\label{RmkSBPsdoTop}
  Let $\chi,\tilde\chi\in\CIc((-2,2)^n)$ be equal to $1$ on $[-1,1]^n$, with $\tilde\chi=1$ on $\supp\chi$; set $\chi_\alpha:=\frac{\phi_\alpha^*\chi}{\sum_\beta\phi_\beta^*\chi}$ and $\tilde\chi_\alpha:=\phi_\alpha^*\tilde\chi$. Given $A\in w\Psi_\cV^m(M)$, consider then the decomposition $A=\sum_\alpha \tilde\chi_\alpha A\chi_\alpha+R$ where $R=\sum_\alpha(1-\tilde\chi_\alpha)A\chi_\alpha\in w\rho^\infty\Psi_\cV^{-\infty}(M)$. As seminorms on $w\Psi_\cV^m(M)$, we then take the $w\rho^\infty\Psi_\cV^{-\infty}(M)$-seminorms of $R$ (see Remark~\ref{RmkSBPsdoResTop}) as well as the supremum over $\alpha\in\sA$ of the $S^m$-seminorms of the unique symbols $a_\alpha\in S^m((-2,2)^n;\R^n)$ so that $\tilde\chi_\alpha A\chi_\alpha$, in the chart $\phi_\alpha$, is given by $\Op'_\alpha(a_\alpha)$ where we define $\Op'_\alpha$ as in~\eqref{EqSBPsdoOpalpha} but (to enforce uniqueness) without the cutoff $\psi$.
\end{rmk}

\begin{rmk}[Relationship with $\cV'$-ps.d.o.s]
\label{RmkSBPsdoVprime}
  Recall Proposition~\ref{PropSB2}, and define $\Psi_{\cV'}^m(M)$ using Definition~\ref{DefSBPsdoV} with respect to the b.g.\ structure $\rho\fB$ and the trivial scaling (i.e.\ all scalings are $1$). (Thus, the operator and coefficient Lie algebras coincide and are equal to $\cV'$.) Then
  \[
    \Psi_\cV^m(M)\subset\Psi_{\cV'}^m(M);
  \]
  similarly for spaces of weighted ps.d.o.s. We stress that the validity of this inclusion rests in particular on the fact that we do not require residual $\cV'$-ps.d.o.s to be localized near the diagonal (as measured using a metric $d_{\rho\fB}$ as given in Proposition~\ref{PropbgDist}): after all, when $\inf_{\alpha,i}\rho_{\alpha,i}=0$, Schwartz kernels of $\cV$-quantizations are only $d_\fB$-, but \emph{not} $d_{\rho\fB}$-localized near the diagonal.
\end{rmk}

We proceed to develop further basic properties of $w\Psi_\cV^m(M)$:
\begin{itemize}
\item $w\Psi_\cV^m(M)$ only depends on $\cV$ and $m,w$;
\item $\bigcup_{m,w} w\Psi_\cV^m(M)$ is an algebra, and the principal symbol map is multiplicative.
\end{itemize}

Definition~\ref{DefSBPsdoRes} ensures that the space $w\rho^\infty\Psi_\cV^{-\infty}(M)$ is unchanged when passing from $\fB_\times$ to a compatible scaled b.g.\ structure $\tilde\fB_\times$. We proceed to justify the notation $w\Psi_\cV^m(M)$ (and also of $\upsigma_\cV^{m,w}(A)$ and $\WF'_\cV(A)$) by showing that $w\Psi_\cV^m(M)$ is unchanged when using $\tilde\fB_\times$ instead of $\fB_\times$ in its definition. In view of Proposition~\ref{PropSBCompUniq}, it suffices to study the (uniform) effect of coordinate changes on quantizations $\Op_\alpha(a_\alpha)$ in local charts where $a_\alpha\in S^m(T^*\R^n)$ has support in $(-\frac54,\frac54)^n\times\R^n$. Moreover, by Lemma~\ref{LemmaSBPsdoOffDiag}, we may furthermore localize the Schwartz kernels of such local quantizations to a neighborhood $|x_\alpha-x'_\alpha|<\eps$ for any fixed $\alpha$-independent $\eps>0$. Using a uniform partition of unity subordinate to the b.g.\ structure $\{(\tilde U_{\tilde\alpha},\tilde\phi_{\tilde\alpha})\}$ underlying $\tilde\fB_\times$, with the supports of the cutoff functions contained in $\tilde\phi_{\tilde\alpha}^{-1}((-\frac32,\frac32)^n)$, we may furthermore localize $a_\alpha=a_\alpha(x,\xi)$ in $x$ to have $x$-support in $(-\frac32,\frac32)^n\cap(\phi_\alpha\circ\tilde\phi_{\tilde\alpha}^{-1}((-\frac32,\frac32)^n))$, and thus (using the near-diagonal localization) the Schwartz kernel of $\Op_\alpha(a_\alpha)$ to have support in $[(-\frac74,\frac74)^n\cap(\phi_\alpha\circ\tilde\phi_{\tilde\alpha}^{-1}((-\frac74,\frac74)^n))]^2$. We now describe the symbol $\tilde a_{\tilde\alpha}$ for which $\Op_\alpha(a_\alpha)=\Op_{\tilde\alpha}(\tilde a_{\tilde\alpha})$; for simplicity of notation, the roles of $a_\alpha,\alpha,\tilde\alpha$ will be played by $a,\alpha,\beta$, and we consider the coordinate change $\tau_{\alpha\beta}$ between $\phi_\alpha$ and $\phi_\beta$.

\begin{lemma}[Coordinate change]
\label{LemmaSBPsdoCh}
  Fix $C=\sup_{\alpha,\beta\in\sA}\sum_{j=1}^n\|\tau_{\beta\alpha}^j\|_{\cC^1}$. Fix $0<\eps<(4 C)^{-1}$ and $0<l<2-2\eps$. Let $\alpha,\beta\in\sA$, and suppose $a\in S^m(T^*\R^n)$ has support in $((-l+2\eps,l-2\eps)^n\cap\tau_{\alpha\beta}((-l+2 C\eps,l-2 C\eps)^n))\times\R^n$ For $\delta>0$, let $\eta_\delta\in\CIc((-2\delta,2\delta)^n)$ be equal to $1$ on $[-\delta,\delta]^n$. Define the operator $A_\alpha$ to have Schwartz kernel
  \begin{align*}
    &K_{A_\alpha}(x_\alpha,x'_\alpha) := \kappa_{A_\alpha}(x_\alpha,x'_\alpha)\frac{|\dd x^{\prime 1}_\alpha\cdots\dd x^{\prime n}_\alpha|}{\rho_{\alpha,1}\cdots\rho_{\alpha,n}},  \\
    &\qquad \kappa_{A_\alpha}(x_\alpha,x'_\alpha) := (2\pi)^{-n} \int_{\R^n} e^{i\sum_{j=1}^n (x_\alpha^j-x_\alpha^{\prime j})\xi_{\alpha,j}/\rho_{\alpha,j}} \eta_\eps(x'_\alpha-x_\alpha) a(x_\alpha,\xi_\alpha)\,\dd\xi_\alpha.
  \end{align*}
  Define the operator $A_\beta$ by $A_\beta u:=\tau_{\alpha\beta}^*(A_\alpha(\tau_{\beta\alpha}^*u))$ (whose Schwartz kernel $K_{A_\beta}$ is supported in $(-l+2 C\eps,l-2 C\eps)^n\times(-l,l)^n$). Then we can write its Schwartz kernel as $K_{A_\beta}(x_\beta,x'_\beta)=\kappa_{A_\beta}(x_\beta,x'_\beta)\frac{|\dd x^{\prime 1}_\beta\cdots\dd x^{\prime n}_\alpha|}{\rho_{\beta,1}\cdots\rho_{\beta,n}}$ where
  \[
    \kappa_{A_\beta}(x_\beta,x'_\beta) = (2\pi)\int_{\R^n} e^{i\sum_{j=1}^n(x_\beta^j-x_\beta^{\prime j})\xi_{\beta,j}/\rho_{\beta,j}}\eta_{2 C\eps}(x'_\beta-x_\beta)b(x_\beta,\xi_\beta)\,\dd\xi_\beta
  \]
  for a symbol $b\in S^m(T^*\R^n)$ with support in $(-l+2 C\eps,l-2 C\eps)^n\times\R^n$ which has the following properties:
  \begin{enumerate}
  \item\label{ItSBPsdoChTransf} the $S^{m-1}(T^*\R^n)$-seminorms of $\bar\rho_\alpha^{-1}(b-(\tau_{\beta\alpha})_*a)$ are uniformly bounded by those of $a$;
  \item\label{ItSBPsdoChEsssupp} if $\chi\in S^0(T^*\R^n)$ is such that $\chi a\in S^{m'}$, then $((\tau_{\beta\alpha})_*\chi)b\in S^{m'}$, with seminorms bounded uniformly by those of $\chi a$.
  \end{enumerate}
\end{lemma}

Part~\eqref{ItSBPsdoChTransf} explains why the principal symbol of an element of $\Psi_\cV^m(M)$ is well-defined in $S^m/\rho S^{m-1}({}^\cV T^*M)$; part~\eqref{ItSBPsdoChEsssupp}, applied to $\chi$ which are the pushforward to $U_\alpha$ of an element of $S^0({}^\cV T^*M)$, is the reason for the well-definedness of the essential support of $\cV$-ps.d.o.s.

\begin{proof}[Proof of Lemma~\usref{LemmaSBPsdoCh}]
  The Schwartz kernel of $A_\beta$ is equal to
  \[
    K_{A_\beta}(x_\beta,x'_\beta) = \kappa_{A_\alpha}\bigl(\tau_{\alpha\beta}(x_\beta),\tau_{\alpha\beta}(x'_\beta)\bigr)\frac{|\dd(\tau_{\alpha\beta}^1(x'_\beta))\cdots\dd(\tau_{\alpha\beta}^n(x'_\beta))|}{\rho_{\alpha,1}\cdots\rho_{\alpha,n}}.
  \]
  The density factor is equal to
  \[
    \frac{\rho_{\beta,1}\cdots\rho_{\beta,n}}{\rho_{\alpha,1}\cdots\rho_{\alpha,n}} \bigl|\det \bigl(\pa_i\tau_{\alpha\beta}^j(x'_\beta))_{i,j=1,\ldots,n}\bigr)\bigr| \frac{|\dd x_\beta^{\prime 1}\cdots\dd x_\beta^{\prime n}|}{\rho_{\beta,1}\cdots\rho_{\beta,n}} = \biggl|\det\biggl(\frac{\rho_{\beta,i}\pa_i\tau_{\alpha\beta}^j(x'_\beta)}{\rho_{\alpha,j}}\biggr)\biggr|\frac{|\dd x_\beta^{\prime 1}\cdots\dd x_\beta^{\prime n}|}{\rho_{\beta,1}\cdots\rho_{\beta,n}},
  \]
  with the determinant factor being uniformly bounded in $\CI$ in view of~\eqref{EqISB}. In the oscillatory integral $\kappa_{A_\alpha}$ on the other hand, we write (using the `Kuranishi trick' \cite[\S{2.1}]{HormanderFIO1})
  \begin{align*}
    &\sum_{j=1}^n \bigl(\tau_{\alpha\beta}^j(x_\beta)-\tau_{\alpha\beta}^j(x'_\beta)\bigr)\frac{\xi_{\alpha,j}}{\rho_{\alpha,j}} = \sum_{i,j=1}^n (x_\beta^i-x_\beta^{\prime i})\frac{\Phi_i^j(x_\beta,x'_\beta)\xi_{\alpha,j}}{\rho_{\beta,i}}, \\
    &\hspace{8em} \Phi_i^j(x_\beta,x'_\beta)=\int_0^1\frac{\rho_{\beta,i}\pa_i\tau_{\alpha\beta}^j(x'_\beta+s(x_\beta-x'_\beta))}{\rho_{\alpha,j}}\,\dd s.
  \end{align*}
  In view of~\eqref{EqISB}, there exists $\delta_0>0$ (independent of $\alpha,\beta$) so that the matrix $\Phi(x_\beta,x'_\beta):=(\Phi_i^j(x_\beta,x'_\beta))_{i,j=1,\ldots,n}$ is invertible (and with $|\det\Phi(x_\beta,x'_\beta)|$ bounded above and below by a uniform constant times $|\det\Phi(x_\beta,x_\beta)|$) whenever $|x_\beta-x'_\beta|<4\delta_0$. For the localization $\kappa_{A_\beta,2\delta_0}(x_\beta,x'_\beta):=\eta_{2\delta_0}(x'_\beta-x_\beta)\cdot\kappa_{A_\beta}(x_\beta,x'_\beta)$, we can then perform the change coordinates $\xi_\alpha=\Phi(x_\beta,x'_\beta)^{-1}\theta_\beta$. The left reduction (see Lemma~\ref{LemmaSBPsdoLeft} below) applied to the symbol
  \[
    \eta_\eps(\tau_{\alpha\beta}(x_\beta')-\tau_{\alpha\beta}(x_\beta)) \eta_{2\delta_0}(x'_\beta-x_\beta) a\bigl(\tau_{\alpha\beta}(x_\beta), \Phi(x_\beta,x'_\beta)^{-1}\theta_\beta\bigr) J(x_\beta,x'_\beta),
  \]
  where $J(x_\beta,x'_\beta)=|\det(\frac{\rho_{\beta,i}\pa_i\tau_{\alpha\beta}^j(x'_\beta)}{\rho_{\alpha,j}})|\cdot|\det\Phi(x_\beta,x'_\beta)|^{-1}$ equals $1$ at $x'_\beta=x_\beta$, produces the desired expression for $\kappa_{A_\beta,2\delta_0}(x_\beta,x'_\beta)$ and the statements~\eqref{ItSBPsdoChTransf}--\eqref{ItSBPsdoChEsssupp}. Note also that $\sum_{j=1}^n\Phi_i^j(x_\beta,x_\beta)\xi_{\alpha,j}=\xi_{\beta,i}$ according to~\eqref{EqSBDSymbXiCoord} (with $\alpha,\beta$ exchanged). The operator $(1-\eta_{2\delta_0})\kappa_{A_\beta}$ is residual, and thus also of the stated form by Lemma~\ref{LemmaSBPsdoSKRes} (and it does not affect the validity of~\eqref{ItSBPsdoChTransf}--\eqref{ItSBPsdoChEsssupp}).
\end{proof}

Passing back to the notation from the discussion preceding Lemma~\ref{LemmaSBPsdoCh}, we have now expressed $\Op_\alpha(a_\alpha)=\Op_{\tilde\alpha}(\tilde a_{\tilde\alpha})$ with full control on $\tilde a_{\tilde\alpha}$. The following result completes the proof that $\sum_{\tilde\alpha\in\tilde\sA}\Op_{\tilde\alpha}(\tilde a_{\tilde\alpha})$ is of the form $\widetilde{\Op}_\cV(\tilde a)+\tilde R$ where we write $\widetilde{\Op}_\cV(\tilde a)$ for the quantization map relative to $\tilde\fB_\times$, and where $\tilde R\in w\rho^\infty\Psi_\cV^{-\infty}(M)$. Again, we work with $\fB_\times$ and drop tildes for simplicity of notation.

\begin{prop}[Sums of local quantizations]
\label{PropSBPsdoSum}
  Let $w\in\CI(M)$ be a weight and $\{w_\alpha\}$ an equivalent weight family. For each $\alpha\in\sA$, suppose we are given $a_\alpha\in S^m(T^*\R^n)$ with support in $(-\frac32,\frac32)^n\times\R^n$, and suppose that all seminorms of $w_\alpha^{-1}a_\alpha$ in $S^m(T^*\R^n)$ are uniformly bounded. Then there exist $a\in w S^m({}^\cV T^*M)$ and $R\in w\rho^\infty\Psi_\cV^{-\infty}(M)$ so that
  \begin{equation}
  \label{EqSBPsdoSumOp}
    \sum_{\alpha\in\sA}\phi_\alpha^*\Op_\alpha(a_\alpha)=\Op_\cV(a)+R.
  \end{equation}
  Setting $a_0:=\sum_{\alpha\in\sA} \phi_\alpha^*a_\alpha\in S^m({}^\cV T^*M)$, where $\phi_\alpha\colon{}^\cV T^*_{U_\alpha}M\to(-2,2)^n\times\R^n$ (analogously to the notation used in Definition~\usref{DefSBDSymb}), we moreover have
  \begin{equation}
  \label{EqSBPsdoSumSym}
    a - a_0 \in w\rho S^{m-1}({}^\cV T^*M),\qquad
    \Ell_\fM^w(a)=\Ell_\fM^w(a_0).
  \end{equation}
  Finally, if $\chi\in S^0({}^\cV T^*M)$ is such that $((\phi_\alpha)_*\chi)a_\alpha$ is uniformly bounded in $S^{-\infty}(T^*\R^n)$, then $\esssupp_\fM^w(a)\cap\{\chi\neq 0\}=\emptyset$.
\end{prop}
\begin{proof}
  Fix a uniform partition of unity $\{\chi_\alpha\}$ as in Definition~\ref{DefSBPsdo}. Then
  \begin{align*}
    &\biggl(\sum_\alpha \phi_\alpha^*\Op_\alpha(a_\alpha)\biggr) - \Op_\cV(a_0) \\
    &\qquad = \sum_{\alpha,\beta} \biggl( \phi_\alpha^*\Op_\alpha\Bigl(\bigl((\phi_\alpha)_*\chi_\beta\bigr)a_\alpha\Bigr) - \phi_\beta^*\Op_\beta\Bigl(\bigl((\phi_\beta)_*\chi_\beta\bigr)(\tau_{\beta\alpha})_*a_\alpha\Bigr) \biggr).
  \end{align*}
  Using Lemma~\ref{LemmaSBPsdoCh}, we can write
  \[
    \phi_\alpha^*\Op_\alpha\Bigl(\bigl((\phi_\alpha)_*\chi_\beta\bigr)a_\alpha\Bigr) = \phi_\beta^*\Op_\beta\Bigl((\tau_{\beta\alpha})_*\bigl((\phi_\alpha)_*\chi_\beta\bigr)a_\alpha + r_{\beta\alpha}\Bigr) + Q_{\beta\alpha}
  \]
  where $r_{\beta\alpha}$ is uniformly bounded in $w_\beta\rho_\beta S^{m-1}$, moreover $r_{\beta\alpha}$ and $Q_{\beta\alpha}$ (the latter capturing residual error terms from localizing to a neighborhood of the diagonal) are nonzero only when $U_\beta\cap U_\alpha\neq\emptyset$, and $Q_1:=\sum_{\alpha,\beta} Q_{\beta\alpha}\in w\rho^\infty\Psi_\cV^{-\infty}(M)$. Since $(\tau_{\beta\alpha})_*((\phi_\alpha)_*\chi_\beta)a_\alpha=((\phi_\beta)_*\chi_\beta)(\tau_{\beta\alpha})_*a_\alpha$, we conclude that
  \[
    \biggl(\sum_\alpha \phi_\alpha^*\Op_\alpha(a_\alpha)\biggr) - \Op_\cV(a_0) = \sum_\beta \phi_\beta^*\Op_\beta(a_{1,\beta}) + Q_1
  \]
  where $a_{1,\beta}=\sum_\alpha r_{\beta\alpha}$ is uniformly bounded in $w_\beta\rho_\beta S^{m-1}$. Therefore, the collection $\{a_{1,\beta}\}$ satisfies the same hypotheses as $\{a_\alpha\}$ except with $m$ shifted by $-1$ and $w_\alpha$ replaced by $w_\alpha\bar\rho_\alpha$.

  Proceeding iteratively, we thus obtain symbols $a_j=\sum_\beta\phi_\beta^*a_{j,\beta}\in w\rho^j S^{m-j}({}^\cV T^*M)$ and residual operators $Q_j\in w\rho^\infty\Psi_\cV^{-\infty}(M)$ so that for $k\geq 1$,
  \[
    \sum_\alpha \phi_\alpha^*\Op_\alpha(a_\alpha) - \Op_\cV(a_0+\cdots+a_{k-1}) = \sum_\beta \phi_\beta^*\Op_\beta(a_{k,\beta}) + Q_k,
  \]
  with $a_{k,\beta}$ uniformly bounded in $w_\beta\bar\rho_\beta^k S^{m-k}$. Moreover, the Schwartz kernels of each $Q_{(j)}$ are supported in a fixed neighborhood (with respect to $d_\fB$) of the diagonal in $M\times M$. In view of part~\eqref{ItSBPsdoChEsssupp} of Lemma~\ref{LemmaSBPsdoCh}, the essential support of each $\sum_\beta\phi_\beta^*a_{k,\beta}$ satisfies the same condition as that of $a$ in the statement of the proposition.

  Let now $a\in w S^m({}^\cV T^*M)$ be an asymptotic sum of the $a_i$, $i\in\N_0$. Then
  \begin{equation}
  \label{EqSBPsdoSumR}
  \begin{split}
    R &:= \sum_\alpha \phi_\alpha^*\Op_\alpha(a_\alpha) - \Op_\cV(a) \\
      &\,= \Op_\cV(r_k) + \sum_\beta \phi_\beta^*\Op_\beta(a_{k,\beta}) + Q_{(k)},
  \end{split}
  \end{equation}
  with $r_k:=(a_0+\cdots+a_{k-1})-a\in w\rho^k S^{m-k}({}^\cV T^*M)$. By (the proof of) Proposition~\ref{PropSBPsdoH}, the operator $R$ thus defines a bounded map from $w'H_{\cV;\cW}^{(-N;j)}(M)$ to $w w'\rho^k H_{\cV;\cW}^{(-N-m+k;j)}(M)$ for all $w',N,j,k$, and likewise for $R^*$. Therefore, $R$ is residual, and the proof is complete.
\end{proof}

\bigskip

We next turn to algebra properties of $\Psi_\cV(M)$.

\begin{thm}[Adjoints, compositions, and commutators of $\cV$-ps.d.o.s]
\label{ThmSBPsdoComp}
  Let $w,w_1,w_2\in\CI(M)$ be weights on $(M,\fB)$. Let $m,m_1,m_2,l,l_1,l_2\in\R$. Fix a uniformly positive $\cV$-density on $M$. Then
  \begin{align*}
    A \in w\rho^{-l}\Psi_\cV^m(M) &\implies A^* \in w\rho^{-l}\Psi_\cV^m(M), \\
    A_j \in w_j\rho^{-l_j}\Psi_\cV^{m_j}(M) &\implies A_1\circ A_2 \in (w_1 w_2)\rho^{-(l_1+l_2)}\Psi_\cV^{m_1+m_2}(M).
  \end{align*}
  Furthermore, $\upsigma_\cV^{m,l}(A^*)=\ol{\upsigma_\cV^{m,l}(A)}$ and $\upsigma_\cV^{m_1+m_2,l_1+l_2}(A_1\circ A_2)=\upsigma_\cV^{m_1,l_1}(A_1)\cdot\upsigma_\cV^{m_2,l_2}(A_2)$. The principal symbol of $i[A_1,A_2]\in w\rho^{-(l_1+l_2-1)}\Psi_\cV^{m_1+m_2-1}(M)$ is the Poisson bracket of the principal symbols of $A_1,A_2$. Finally,
  \begin{equation}
  \label{EqSBPsdoCompWF}
    \WF_\cV^{\prime w_1 w_2}(A_1\circ A_2) \subset \WF_\cV^{\prime w_1}(A_1)\cap\WF_\cV^{\prime w_2}(A_1).
  \end{equation}
  The map $A\mapsto A^*$ is continuous, and the map $(A_1,A_2)\mapsto A_1\circ A_2$ is jointly continuous.
\end{thm}
\begin{proof}
  The membership $A^*\in w\rho^{-l}\Psi_\cV^m(M)$ is a simple application of the left reduction result (Lemma~\ref{LemmaSBPsdoLeft}). To study $A_1\circ A_2$, write
  \[
    A_j = \Op_\cV(a_j) + R_j,\qquad a_j \in w_j\rho^{-l_j}S^{m_j}({}^\cV T^*M),\quad R_j\in w_j\rho^\infty\Psi_\cV^{-\infty}(M).
  \]
  The membership $R_1\circ R_2\in(w_1 w_2)\rho^\infty\Psi_\cV^{-\infty}(M)$ is an immediate consequence of Definition~\ref{DefSBPsdoRes}.

  Consider the composition $\Op_\cV(a_1)\circ\Op_\cV(a_2)$. We define $\Op_\cV$ using the b.g.\ structure from Lemma~\ref{LemmabgRefine} with $\ell_1=\frac32$, $\ell_2=\frac74$, and with the scaling on all cells arising from the subdivision of $U_\alpha$ given by $\rho_\alpha$. For a term $\phi_\alpha^*\Op_\alpha(a_{1,\alpha})\circ\phi_\beta^*\Op_\beta(a_{2,\beta})$ with $\phi_\alpha^{-1}((-\frac32,\frac32)^n)\cap\phi_\beta^{-1}((-\frac32,\frac32)^n)\neq\emptyset$, we then write $\phi_\beta^*\Op_\beta(a_{2,\beta})=\phi_\alpha^*(\tau_{\beta\alpha}^*\Op_\beta(a_{2,\beta}))$ as $\phi_\alpha^*\Op_\alpha(a_{2,\beta,\alpha})$ using Lemma~\ref{LemmaSBPsdoCh}. Thus, it suffices to study compositions in a single chart. But using Lemma~\ref{LemmaSBPsdoLeft}, we can write $\Op_\alpha(a)\circ\Op_\alpha(b)=\Op'_\alpha(a_L)\circ\Op'_\alpha(b_R)=\Op'_\alpha(a_L b_R)$ where $a_L=a_L(x_\alpha,\xi_\alpha)$ and $b_R=b_R(x'_\alpha,\xi_\alpha)$; and then further $\Op'_\alpha(a_L b_R)$ equals $\Op_\alpha(c)$ plus a residual operator. The uniformity statements in Lemmas~\ref{LemmaSBPsdoLeft} and \ref{LemmaSBPsdoCh} imply, together with Proposition~\ref{PropSBPsdoSum}, that $\Op_\cV(a_1)\circ\Op_\cV(a_2)=\Op_\cV(a)+R'$ where $a\in(w_1 w_2)\rho^{-(l_1+l_2)}S^{m_1+m_2}({}^\cV T^*M)$ (with the usual full symbol expansion in each chart) and $R\in(w_1 w_2)\rho^\infty\Psi_\cV^{-\infty}(M)$.

  Finally, Proposition~\ref{PropSBPsdoH} implies that the compositions $\Op_\cV(a_1)\circ R_2$ and $R_1\circ\Op_cV(a_2)$ are residual. We leave the proof of the continuity statements to the reader.

  The principal symbol of the commutator arises as usual from the $|\beta|=1$ terms of the local expressions for the left reduced full symbols of $A_1\circ A_2$ and $A_2\circ A_1$.
\end{proof}

\subsubsection{Application: Fredholm theory for elliptic operators in fully symbolic settings}
\label{SssSBPsdoApp}

As a simple applications of the `$\cV$-calculus' $\Psi_\cV(M):=\bigcup_{m,w} w\Psi_\cV^m(M)$ (with $w\in\CI(M)$ denoting weights on $(M,\fB)$), we discuss the Fredholm theory for elliptic operators in the case that $\cV$ is fully symbolic (see Definition~\ref{DefSBTsSymbolic}). (One often says instead that the \emph{$\cV$-calculus is fully symbolic}, as e.g.\ in \cite{SussmanKG}.) The scattering calculus discussed in~\S\ref{SssISc2} is an example; further examples are the 00- and desc-algebras from \S\ref{SssIF}\eqref{ItI00}, \eqref{ItIdescb}.

\begin{thm}[Elliptic $\cV$-ps.d.o.s are Fredholm]
\label{ThmSBPsdoFred}
  Suppose $\cV$ is fully symbolic. Let $P\in w\Psi_\cV^m(M)$ be elliptic, i.e.\ $\upsigma_\cV^{m,w}(P)\in w S^m/w\rho S^{m-1}({}^\cV T^*M)$ is elliptic. Then, for all weights $w'\in\CI(M)$ on $(M,\fB)$ and for all $s\in\R$, the operator
  \[
    P \colon w' H_\cV^s(M) \to w w' H_\cV^{s-m}(M)
  \]
  is Fredholm. The analogous conclusion holds also on $H_{\cV;\cW}^{(s;k)}$ spaces.
\end{thm}
\begin{proof}
  Let $Q=\Op_\cV(q)\in w^{-1}\Psi_\cV^{-m}(M)$ where $q\in w^{-1}S^{-m}({}^\cV T^*M)$ satisfies $q p-1\in \rho S^{-1}({}^\cV T^*M)$. Then $Q P=I+R$ where $R\in\rho\Psi_\cV^{-1}(M)$ is compact on $w' H_\cV^s(M)$ by~\eqref{EqSBTsSymbolicEmbed}.\footnote{As usual, using the principal symbol calculus, one can refine $Q$ to yield a residual remainder $R\in\rho^\infty\Psi_\cV^{-\infty}(M)$.} Therefore, $P$ has finite-dimensional kernel and closed range. Similarly, the fact that $P Q-I$ is compact on $w w' H_\cV^{s-m}(M)$ implies that $P$ has finite-dimensional cokernel.
\end{proof}

\begin{example}[Geometric operators for metrics on $\R^n$]
\label{ExSBPsdoFred}
  Returning to Example~\usref{ExSBDbsc}, we conclude that $L_z=\Delta_g-z$ is Fredholm as a map
  \begin{equation}
  \label{EqSBPsdoFredLz}
    L_z \colon w H_\cV^s(M)\to w H_\cV^{s-2}(M),\qquad z\in\C\setminus[0,\infty),
  \end{equation}
  for all weights $w\in\CI(\R^n)$ on $(\R^n,\fB)$ for the scaled b.g.\ structure from~\S\ref{SssISc2}; in fact, $L_z$ is \emph{invertible}. Indeed, all such weights $w$ can easily be seen to polynomially bounded, and thus (using an elliptic parametrix) $\ker_{w H_\cV^s(M)}L_z\subset\la x\ra^{-\infty}H^\infty(\R^n)=\sS(\R^n)$; but then $L_z u=0$ implies $0=\la L_z u,u\ra_{L^2(\R^n)}$ and thus $u=0$. The same argument applies to $L_z^*$. --- Note that if instead we use the standard b.g.\ structure~\eqref{EqIScBdd} on $\R^n$, i.e.\ with scaling $\rho_{\alpha,i}=1$, then e.g.\ $w=e^{c x^1}$ is a weight on $\R^n$, but the Fredholm property of~\eqref{EqSBPsdoFredLz} typically fails (e.g.\ for $c\geq 1$); and this (trivially scaled) b.g.\ structure is \emph{not} fully symbolic.
\end{example}

\subsection{\texorpdfstring{$\cV$}{V}-microlocalization: wave front sets}
\label{SsSBWF}

In order to facilitate the analysis of non-elliptic problems, we define here notions of wave front sets associated with $\cV$.

\begin{definition}[$\cV$-Sobolev wave front set]
\label{DefSBWF}
  Let $w\in\CI(M)$ be a weight on $(M,\fB)$, recall the scaling weight $\rho\in\CI(M)$ from Definition~\usref{DefSBDScWeight}, and recall the microlocalization locus $\fM$ from Definition~\usref{DefSBTsMicro}. Let $u\in w\rho^{-\infty}H_\cV^{-\infty}(M)=\bigcup_{m,l} w\rho^l H_\cV^m(M)$. Then
  \[
    \WF_\cV^{m,w,l}(u) \subset \fM
  \]
  is the complement of the set of all $\varpi\in\fM$ for which there exists an operator $A\in\Psi_\cV^0(M)$ which is elliptic at $\varpi$ and for which $A u\in w\rho^l H_\cV^m(M)$. For $k\in\N_0$, the set $\WF_{\cV;\cW}^{(m;k),w,l}(u)\subset\fM$ is defined analogously for $u\in w\rho^{-\infty}H_{\cV;\cW}^{(-\infty;k)}(M)$.
\end{definition}

\begin{prop}[Properties of the $\cV$-Sobolev wave front set]
\label{PropSBWF}
  Let $u\in w\rho^{-\infty}H_\cV^{-\infty}(M)$ and $A\in w'\Psi_\cV^{m'}(M)$, $B\in\Psi_\cV^0(M)$.
  \begin{enumerate}
  \item\label{ItSBWFEmpty} $u\in w\rho^l H_\cV^m(M)$ if and only if $\WF_\cV^{m,w,l}(u)=\emptyset$.
  \item{\rm (Microlocality.)} $\WF_\cV^{m-m',w w',l}(A u)\subset\WF_\cV^{m,w,l}(u)\cap\WF_\cV^{\prime w'}(A)$.
  \item{\rm (Microlocal elliptic regularity.)} If $\WF_\cV'(B)\subset\Ell_\cV^{m,w}(A)$, then for any $N$ there exists $C_N$ so that
    \[
      \|B u\|_{w\rho^l H_\cV^m(M)}\leq C_N\Bigl(\|A u\|_{w w'\rho^l H_\cV^{m-m'}(M)}+\|u\|_{w\rho^{-N}H_\cV^{-N}}\Bigr).
    \]
  \end{enumerate}
  
\end{prop}
\begin{proof}
  The first part is analogous to the discussion of~\eqref{EqbgSingSupp}. The non-trivial direction `$\Longleftarrow$' follows, by compactness of $\fM$, from the existence of a finite collection $a_1,\ldots,a_N\in S^0({}^\cV T^*M)$ of symbols with $\Op_\cV(a_j)u\in w\rho^l H_\cV^m(M)$ and $\bigcup_{j=1}^N\Ell^w_\fM(a_j)=\fM$: we then have $A u\in w\rho^l H_\cV^m(M)$ for the elliptic operator $A=\sum_{j=1}^N\Op_\cV(a_j)^2\in\Psi_\cV^0(M)$. With $B\in\Psi_\cV^0(M)$ denoting a parametrix, i.e.\ $B A=I+R$ for $R\in\rho^\infty\Psi_\cV^{-\infty}(M)$, we then have $u=B(A u)-R u\in w\rho^l H_\cV^m(M)$, as desired.

  The proofs of the remaining two parts are standard applications of the principal symbol calculus, see e.g.\ \cite[Proposition~6.3]{HintzMicro}.
\end{proof}

We shall not attempt to study the propagation of singularities in $\fM$ here. We only remark that in applications where a concrete compactification of ${}^\cV T^*M$ is given, one can use the $\cV$-calculus to prove results on the real principal type or radial point propagation for solutions of hyperbolic equations in the usual symbolic manner, as described e.g.\ in \cite{VasyMinicourse,DyatlovZworskiBook,HintzMicro}.

\subsection{Parameterized scaled bounded geometry structures}
\label{SsSBPar}

We now describe the generalization of scaled b.g.\ structures to parameterized scaled b.g.\ structures with parameter set $P$,
\[
  \{\fB_{p,\times}\colon p\in P\},\qquad
  \fB_{p,\times}=\{(U_{p,\alpha},\phi_{p,\alpha},\rho_{p,\alpha})\colon\alpha\in\sA_p\};
\]
see Definition~\ref{DefIPSB}. Denote the b.g.\ structure underlying $\fB_{p,\times}$ by $\fB_p=\{(U_{p,\alpha},\phi_{p,\alpha})\}$, and write $\fB_\times:=(\fB_{p,\times})_{p\in P}$ for the collection of all scaled b.g.\ structures.

On the space $P\times M=\bigsqcup_{p\in P}\{p\}\times M$, with $P$ given the discrete topology, we can again consider the space of uniformly continuous functions $u=u(p,x)$, i.e.\ $u(p,\cdot)\in\cC^0_{{\rm uni},\fB_p}(M)$ which have a $p$-independent modulus of continuity. Correspondingly, the constructions in~\S\ref{SsbgCpt} apply and give the uniform compactification
\[
  \fu(P\times M)
\]
of $P\times M$. When $P$ is finite, we have $\fu(P\times M)=\bigsqcup_{p\in P}\{p\}\times\fu(M,\fB_p)$ (which is $P\times\fu M$ when all $\fB_p$ are compatible). In general, however, the compactification $\fu(P\times M)$ also contains, among other things, accumulation points of sequences $(p_j,x)$ where $\{p_j\}$ is a sequence in $P$.

We denote by $\cV_p$ and $\cW_p$ the operator, resp.\ coefficient Lie algebras of $\fB_{p,\times}$, and by
\[
  {}^{\cV_p}T^*M \to M
\]
the $\cV_p$-cotangent bundle (Definition~\ref{DefSBDCotgt}), which is equipped with its scaled b.g.\ structure $\fB_{p,\times}^*$ from Definition~\ref{DefSBTs}. We may then apply the constructions in~\S\ref{SsbgCpt} to the space
\[
  {}^\cV T^*M:=\bigsqcup_{p\in P}\{p\}\times{}^{\cV_p}T^*M.
\]
Denote by $\rho_p\in\CI(M)$ a scaling weight for $(M,\fB_{p,\times})$ from Definition~\ref{DefSBDScWeight} (with constants which are independent of $p$, cf.\ the discussion following Definition~\ref{DefSBParSymb} below), and by $\lambda_p\in S^{-1}({}^{\cV_p}T^*M)$ a uniformly bounded family of weights as in Lemma~\ref{LemmaSBTsSymbols} for which $\lambda^{-1}\in S^1({}^{\cV_p}T^*M)$ is uniformly bounded as well. Write $\rho,\lambda$ for the functions which on $\{p\}\times M$, resp.\ $\{p\}\times{}^{\cV_p}T^*M$ are given by $\rho_p,\lambda_p$. Analogously to Definition~\ref{DefSBTsMicro}, the microlocalization locus is then the compact subset
\begin{equation}
\label{EqSBParMicro}
  \fM := \fu{}^\cV T^*M) \cap \{ j(\rho\lambda)=0 \}.
\end{equation}

We also have Sobolev spaces $H_{\cV_p}^s(M)$ and weighted versions $w_p H_{\cV_p}(M)$ thereof, where $w_p\in\CI(M)$ is a weight on $(M,P_p)$.

\begin{rmk}[Subsets of $P$]
\label{RmkSBParSubsets}
  In applications, one often only works with operators defined for (typically countable) subsets of parameters $p\in P$ (e.g.\ in the semiclassical setting: sequences of semiclassical parameters $h_j\searrow 0$). Since one can implement this simply by replacing $P$ by such a subset (and then defining the microlocalization locus as in~\eqref{EqSBParMicro} for the new set $P$), we do not explicitly allow for symbols to be defined only on subsets of $P$ here.
\end{rmk}

\begin{rmk}[Compatibility of different scaled b.g.\ structures]
\label{RmkSBPar}
  In applications, two scaled b.g.\ structures $\fB_{p,\times}$ and $\fB_{p',\times}$ for different values of $p,p'\in P$ are often strongly compatible, though not uniformly so in $p$ and $p'\in P$. (This implies, for example, the equality of all Sobolev spaces $H_{\cV_p}^s(M)$ as sets, but still allows for the norms to not be uniformly equivalent as $p$ varies.) Since imposing this condition would not lead to any simplifications, we do not do this here.
\end{rmk}

We now generalize Definitions~\ref{DefSBDSymb}, \ref{DefSBTsMicro}, and \ref{DefSBPsdo}.

\begin{definition}[Symbols]
\label{DefSBParSymb}
  Let $w=(w_p)_{p\in P}$ be a weight on $(P\times M,\fB_\times)$. Let $m\in\R$. Then we say that a family $a=(a_p)_{p\in P}$ lies in $w S^m({}^\cV T^*M)$ if $a_p\in w_p S^m({}^{\cV_p}T^*M)$ for all $p\in P$, with uniformly bounded symbol seminorms. The elliptic set and essential support of $a$, as subsets of $\fM$, are defined in complete analogy with Definition~\usref{DefSBTsMicro}.
\end{definition}

Carefully note that for $w=(w_p)_{p\in P}$ to be a weight on $(P\times M,\fB_\times)$, it is necessary \emph{but not sufficient} that each $w_p$ be a weight on $(M,\fB_{p,\times})$; a sufficient additional condition is that the ratios $w_{p,\alpha}/w_{p,\beta}$ for $\alpha,\beta\in\sA_p$ with $U_{p,\alpha}\cap U_{p,\beta}\neq\emptyset$ are uniformly bounded.

\begin{definition}[$\cV$-ps.d.o.s: parameterized setting]
\label{DefSBParPsdo}
  In the notation of Definition~\usref{DefSBParSymb}, a \emph{$\cV$-pseudodifferential operator} $A\in w\Psi_\cV^m(M)$ is a family $A=(A_p)_{p\in P}$ of operators $A_p=\Op_{\cV_p}(a_p)+R_p$ where the quantization map $\Op_{\cV_p}$ is defined as in Definition~\usref{DefSBPsdo}, $(a_p)_{p\in P}\in w S^m({}^\cV T^*M)$, and $(R_p)_{p\in P}$ is a family of residual operators $R_p$ on $(M,\fB_{p,\times})$ for which the conditions in Definition~\usref{DefSBPsdoRes} hold with uniform constants and operator norm bounds for all weights $w'$ in~\eqref{EqSBPsdoResMap}. The principal symbol of $A$ is the equivalence class of $a$ in $w S^m({}^\cV T^*M)/w\rho S^{m-1}({}^\cV T^*M)$. Finally, the space $w\Diff_\cV^m(M)$ of \emph{$\cV$-differential operators} is the subspace of $w\Psi_\cV^m(M)$ consisting of differential operators.
\end{definition}

An operator $A=(A_p)_{p\in P}\in w\Psi_\cV^m(M)$ defines a \emph{uniformly bounded} family of operators
\[
  A_p\colon w'_p H_{\cV_p}^s(M)\to w_p w'_p H_{\cV_p}^{s-m}(M),
\]
and more generally $A_p\colon w'_p H_{\cV_p;\cW_p}^{(s;k)}(M)\to w_p w'_p H_{\cV_p;\cW_p}^{(s-m;k)}(M)$. (This is in fact a special case of Theorem~\ref{ThmSBPsdoH}, applied to the scaled b.g.\ structure $(P\times M,\fB_\times)$ and the input $u=(u_p)_{p\in P}$ where $u_p=0$ for all $p\neq p_0\in P$ and $u_{p_0}\in w'_{p_0}H_{\cV_{p_0}}^s(M)$ with norm $\leq 1$, since then $A u=(A_p u_p)_{p\in P}$ has norm $\leq C$ for some constant $C$ only depending on $A$; but the norm of $A u$ is bounded from below by the $w_p w'_p H_{\cV_p}^{s-m}(M)$-norm of $A_{p_0}u_{p_0}$.\footnote{Strictly speaking, one needs to apply this argument for all countable subsets $P'\subset P$ of parameters (so that $P'\times M$ is second countable and the total index set $\bigsqcup_{p'\in P'}\sA_{p'}$ is countable, as required in Definition~\ref{DefIBdd}); but the uniform bound on the operators $A_{p'}$ is independent of $P'$.}) The properties of adjoints, compositions, commutators, and principal symbols in Theorem~\ref{ThmSBPsdoComp} carry over verbatim.

The membership $A=(A_p)_{p\in P}\in w\Diff_\cV^m(M)$ is characterized by requiring that each $A_p$ is a differential operator, and $(\phi_{p,\alpha})_*A_p=\sum_{|\beta|\leq m} w_{p,\alpha} a_{p,\alpha,\beta}(x)(\rho_{p,\alpha}\pa_x)^\beta$ where the $\CI((-2,2)^n)$-seminorms of the coefficients $a_{p,\alpha,\beta}$ are uniformly (in $p\in P$, $\alpha\in\sA_p$, $\beta\in\N_0^n$) bounded.

\begin{rmk}[Fully symbolic $\cV$]
\label{RmkSBParSymb}
  Analogously to Definition~\ref{DefSBTsSymbolic}, one may call $\cV$ fully symbolic if, for every $\eps>0$, there exists a compact subset $K\subset P\times M$ (with $P$ given the discrete topology) so that $\rho|_{P\times M\setminus K}<\eps$. In parameterized settings, this is occasionally stronger than necessary. For example, the Lie algebra $\cV=(\cV_{h_j})_{j\in\N}$, $h_j\searrow 0$, of semiclassical vector fields on $\R^n$ with uniformly smooth coefficients (see~\S\ref{SssISc2}) does \emph{not} satisfy this condition due to the lack of decay of the scalings at spatial infinity; nonetheless, if $P=(P_{h_j})_{j\in\N}\in\Psi_\cV^m(\R^n)$ is elliptic, then it is invertible between $H_{\cV_{h_j}}^s(\R^n)=H_{h_j}^s(\R^n)$ and $H_{\cV_{h_j}}^{s-m}(\R^n)=H_{h_j}^{s-m}(\R^n)$ when $h_j$ is sufficiently small. Indeed, $P$ has a parametrix $Q\in\Psi_\cV^{-m}(\R^n)$ with $Q P=I+R$, $R\in h\Psi_\cV^{-1}(\R^n)$ (i.e.\ $R=(R_{h_j})_{j\in\N}$ and $h_j^{-1}R\in\Psi_\cV^{-1}(\R^n)$), and therefore $Q P-I$ has \emph{small} operator norm on any fixed weighted $\cV$-Sobolev space when $h_j$ is sufficiently small, which implies the claim.
\end{rmk}

Similarly, the notion of wave front set (Definition~\ref{DefSBWF}) generalizes to the parameterized setting, with $\WF_\cV^{m,w,l}(u)\subset\fM$ defined for $u=(u_p)_{p\in P}\in w\rho^{-\infty}H_\cV^{-\infty}(M)$, by which we mean that $u_p\in w_p\rho_p^{-N}H_{\cV_p}^{-N}(M)$ for some fixed $N$, with norm bounded uniformly in $p$.

\subsection{Admissible compactifications of phase space}
\label{SsSBCpt}

In practice, one works not with $\fu{}^\cV T^*M$ but with more structured compactifications of ${}^\cV T^*M$ (or equivalently of $T^*M$, which is isomorphic to ${}^\cV T^*M$ as a vector bundle and thus, a fortiori, diffeomorphic as a manifold). The following definition is more demanding than strictly necessary, but is permissive enough to capture all phase space compactifications appearing in the literature.

\begin{definition}[Admissible compactifications]
\label{DefSBCpt}
  A compact manifold with corners $\ol{T^*M}$ (with embedded boundary hypersurfaces), together with a smooth embedding ${}^\cV T^*M\hra\ol{T^*M}$ with open dense range equal to the interior of $\ol{T^*M}$, is an \emph{admissible compactification} of ${}^\cV T^*M$ if the following conditions are satisfied.
  \begin{enumerate}
  \item\label{ItSBCptbg} Denote by $\fB_{\ol{T^*M}}$ the b.g.\ structure on ${}^\cV T^*M$ for the b-algebra on $\ol{T^*M}$, see~\eqref{EqIFbCorners}. Then $\fB_{\ol{T^*M}}\geq\fB^*$ (see Definitions~\usref{DefbgComp}\eqref{ItbgCompCoarser} and~\usref{DefSBTs}).
  \item\label{ItSBCptrho} There exists a weight $\rho$ on $M$ which is equivalent to the weight family $\{\bar\rho_\alpha\}$ and whose preimage under the projection ${}^\cV T^*M\to M$ is the product of powers of boundary defining functions of $\ol{T^*M}$, i.e.
  \begin{equation}
  \label{EqSBCptRho}
    \rho=\rho_0\prod_{H\in\sH_\rho} \rho_H^{\alpha_H}
  \end{equation}
  where $\rho_0\in\CI(\ol{T^*M})$ is a positive function (thus also $\rho_0^{-1}\in\CI(\ol{T^*M})$), and $H$ ranges over a subset $\sH_\rho$ of the collection of boundary hypersurfaces of $\ol{T^*M}$, $\rho_H\in\CI(\ol{T^*M})$ is a defining function of $H$ (that is, $\rho_H^{-1}(0)=H$, $\dd\rho_H\neq 0$ on $H$, and $\rho_H>0$ on $\ol{T^*M}\setminus H$), and finally $\alpha_H>0$ for all $H$.
  \item\label{ItSBCptlambda} There exists a positive elliptic symbol $\lambda\in S^{-1}({}^\cV T^*M)$ with $\lambda^{-1}\in S^1({}^\cV T^*M)$ which is the product of powers of boundary defining functions of $\ol{T^*M}$, i.e.
  \begin{equation}
  \label{EqSBCptLambda}
    \lambda=\lambda_0\prod_{H\in\sH_\lambda} \rho_H^{\beta_H},\qquad \CI(\ol{T^*M})\ni\lambda_0>0,\ \ \beta_H>0.
  \end{equation}
  \end{enumerate}
  Let $\sH:=\sH_\rho\cup\sH_\lambda$. The \emph{microlocalization locus of $\ol{T^*M}$} is then defined as
  \[
    \fM_{\ol{T^*M}} := \bigcup_{H\in\sH} H \subset \pa\ol{T^*M}.
  \]
\end{definition}

Write $\cA(\ol{T^*M})$ for the space of bounded conormal functions $a$ on $T^*M$, i.e.\ $P a\in L^\infty(\ol{T^*M})$ for all $P\in\Diffb(\ol{T^*M})$. Condition~\eqref{ItSBCptbg} implies
\begin{equation}
\label{EqSBCptCIIncl}
  \CI(\ol{T^*M})\subset\cA(\ol{T^*M})=\CI_{{\rm uni},\fB_{\ol{T^*M}}}({}^\cV T^*M)\subset\CI_{{\rm uni},\fB^*}({}^\cV T^*M)=S^0({}^\cV T^*M).
\end{equation}
It also implies the analogous statement for uniformly continuous functions, and thus gives a continuous map $\fu({}^\cV T^*M,\fB^*)\to\ol{T^*M}$ of compact Hausdorff spaces whose image is thus compact and dense (since it contains ${}^\cV T^*M$); this map is therefore surjective. (We already used this argument in the proof of Lemma~\ref{LemmabgUniv}.) This in turn gives rise to a surjective map
\begin{equation}
\label{EqSBCptfM}
  \fM \to \fM_{\ol{T^*M}}
\end{equation}
between the zero sets of $\lambda\rho$ in $\fu{}^\cV T^*M$ and $\ol{T^*M}$, respectively.

In view of~\eqref{EqSBCptCIIncl}, quantizations of functions which are products of powers $\rho_H^{-m}$, $H\in\sH$, and a smooth or conormal function on $\ol{T^*M}$ are then well-defined $\cV$-pseudodifferential operators. One can then define elliptic sets, essential supports, and operator wave front sets in two equivalent ways: as the images in $\fM_{\ol{T^*M}}$ of the corresponding subsets of $\fM$ under the map~\eqref{EqSBCptfM}; or via testing with elements of $\CI(\ol{T^*M})$ which localize near the point in question in $\ol{T^*M}$.

Since our $\cV$-calculus is not designed to keep track of notions of smoothness which are stronger than uniform smoothness in unit cells (i.e.\ $\cW$-regularity), it is cleaner to work directly with the less regular symbol classes
\[
  S^{(m_H)_{H\in\sH}}(\ol{T^*M}) := \Biggl(\prod_{H\in\sH} \rho_H^{-m_H}\Biggr)S^0({}^\cV T^*M),\qquad m_H\in\R.
\]
Sums of quantizations of such symbols and residual operators define the space
\[
  \Psi_\cV^{(m_H)_{H\in\sH}}(M).
\]
The weight $\prod_{H\in\sH}\rho_H^{-m_H}$ is an instance of a phase space weight on $({}^\cV T^*M,\fB^*)$; thus, using the material developed in~\S\ref{SsFWg} below, the principal symbol of an operator
\[
  A = \Op_\cV(a)+R \in \Psi_\cV^{(m_H)_{H\in\sH}}(M)
\]
is a well-defined element
\[
  [a] \in S^{(m_H)_{H\in\sH}} / S^{(m_H-\gamma_H)_{H\in\sH}}(\ol{T^*M}),\qquad \gamma_H:=\alpha_H+\beta_H,
\]
where $\alpha_H,\beta_H$ are taken from \eqref{EqSBCptRho}--\eqref{EqSBCptLambda}, and we set $\alpha_H=0$, resp.\ $\beta_H=0$ when $H\notin\sH_\rho$, resp.\ $H\notin\sH_\lambda$. (This is merely a reformulation of~\eqref{EqFWgSymb}.) One says that the calculus is \emph{symbolic at all $H$ where $\gamma_H>0$}.

\begin{rmk}[Partial compactifications]
\label{RmkSBCptPartial}
  It is often useful to relax the condition that $\ol{T^*M}$ be compact in Definition~\ref{DefSBCpt}. One then needs to replace $\CI(\ol{T^*M})$ by $\CI(\ol{T^*M})\cap\CI_{{\rm uni},\fB_{\ol{T^*M}}}(\ol{T^*M})$.
\end{rmk}

\begin{example}[Fiber-radial compactification]
\label{ExSBCptFiber}
  If $M$ is compact (thus $\Psi_\cV(M)=\Psi(M)$ is the standard algebra of pseudodifferential operators), an admissible compactification of $T^*M$ is the fiber-radial compactification $\ol{T^*}M$ \cite[\S{E.1.3}]{DyatlovZworskiBook}. In this case, one can take $\rho=1$, and for $\lambda$ a boundary defining function of fiber infinity $S^*M\subset\ol{T^*}M$. When $M$ is noncompact and equipped with a scaled b.g.\ structure, an admissible partial compactification (see Remark~\ref{RmkSBCptPartial}) is the fiber-radial compactification $\ol{{}^\cV T^*}M$ of ${}^\cV T^*M$.
\end{example}

\begin{example}[Scattering symbols on compact manifolds with boundary]
\label{ExSBCptSc}
  We revisit example~\eqref{ItIsc} from \S\ref{SssIF}, and consider instead of a compact manifold with boundary $\bar M$ only a chart $\bar M=[0,1)_x\times\R^{n-1}_y$. The radially compactified scattering cotangent bundle over $\bar M$ is then $\ol{\Tsc^*}\bar M=\bar M\times\ol{\R^n}$, with a point $(x,y;\xi,\eta)$ with $x>0$ corresponding to the covector $\xi\frac{\dd x}{x^2}+\eta\frac{\dd y}{x}$. We can take $\rho=x$ and $\lambda=\la(\xi,\eta)\ra^{-1}$. The corresponding microlocalization locus, i.e.\ the zero set of $\rho\lambda$, is the union of the boundary hypersurfaces $\pa\bar M\times\ol{\R^n}$ and $\bar M\times\pa\ol{\R^n}$ of $\ol{\Tsc^*}\bar M$. See also \cite[\S\S{1} and {5}]{MelroseEuclideanSpectralTheory}.
\end{example}

In the parameterized setting with parameter space $P$, compactifications of practical interest are typically compactifications of the total space ${}^\cV T^*M$ to a smooth manifold with corners. Definition~\ref{DefSBCpt} carries over to this setting \emph{mutatis mutandis}. For example, for the semiclassical b.g.\ algebra on the compact manifold $M$, an admissible compactification is given by
\[
  \ol{{}^\semi T^*}M := [0,1]_h \times \ol{T^*}M,
\]
where an interior point $(h,\zeta)$ with $h>0$ and $\zeta\in T^*_z M$ is identified with the covector $h^{-1}\zeta$ in ${}^{\cV_h}T^*M$. Valid choices are then $\rho=h$ and $\lambda=\la\zeta\ra^{-1}$ (with respect to any fixed Riemannian metric on $M$), so the microlocalization locus is the union of $[0,1]\times S^*M$ and $\{0\}\times\ol{T^*}M$.

\section{Further topics}
\label{SF}

Throughout this section, we fix a scaled b.g.\ structure $\fB_\times$ on the smooth manifold $M$, with operator, resp.\ coefficient Lie algebra $\cV$, resp.\ $\cW$.

\subsection{Operators with phase space weights}
\label{SsFWg}

As pointed out after Definition~\ref{DefSBTsMicro}, the symbols whose quantizations we have discussed thus far are of class $w S^0({}^\cV T^*M)$ where $w=w_0 w_1$, with $w_0\in\CI(M)$ a weight on $(M,\fB)$ and $w_1=\lambda^{-m}$ in the notation of Lemma~\ref{LemmaSBTsSymbols}.

\begin{lemma}[Phase space weights]
\label{LemmaFWg}
  Let $w\in\CI({}^\cV T^*M)$ be a weight on $({}^\cV T^*M,\fB_\times)$ in the notation of Definition~\usref{DefSBTs}. Then we can write $w=w_0 w_1$ where $w_0\in\CI(M)$ is a weight on $(M,\fB)$ and $w_1\in\CI({}^\cV T^*M)$ is a weight on $({}^\cV T^*M,\fB_\times)$ for which there exist $m_-<m_+$ and $C>1$ with $C^{-1}\lambda^{m_+}\leq w_1\leq C\lambda^{m_-}$.
\end{lemma}
\begin{proof}
  Set $w_0:=w|_o$ where $o\subset{}^\cV T^*M$ denotes the zero section. The unit cell $U_{(\alpha,j,k,\pm 1)}\subset U_\alpha\times\R^n$ of the phase space b.g.\ structure $\fB^*$ defined in~\eqref{EqSBTs}, where $\lambda^{-1}\sim|\xi|\sim 2^k$, can be reached from the zero section by traversing $k$ many $\fB^*$-unit cells, and thus $|w|$ on this unit cell is bounded above, resp.\ below by $k^C$, resp.\ $k^{-C}$ times its value at the zero section over $U_\alpha$, where $C$ is a constant depending on $w$ (cf.\ Definition~\ref{DefSBDWeights}). Therefore, $w_1:=w_0^{-1}w$ is bounded above, resp.\ below by $\lambda^{-C}$, resp.\ $\lambda^{C}$.
\end{proof}

For phase space weights $w\in\CI({}^\cV T^*M)$, we may then define
\[
  \Psi_\cV^w(M) := \{ \Op_\cV(a) + R \colon a\in w S^0({}^\cV T^*M),\ R\in w_0\rho^\infty\Psi_\cV^{-\infty}(M) \}
\]
where $w_0\in\CI(M)$ is as in Lemma~\ref{LemmaFWg}; the principal symbol of $A=\Op_\cV(a)+R\in\Psi_\cV^w(M)$ is the equivalence class
\begin{equation}
\label{EqFWgSymb}
  [a]\in w S^0/w\rho S^{-1}({}^\cV T^*M).
\end{equation}
Theorem~\ref{ThmSBPsdoComp} generalizes without any modifications to its proof; the reason is that we have a full symbol calculus in local charts $U_\alpha$ modulo symbols which are uniformly bounded in $\bar\rho_\alpha^N S^{-N}$ for all $N$. (See in particular Lemmas~\ref{LemmaSBPsdoLeft}\eqref{ItSBPsdoLeft2} and \ref{LemmaSBPsdoCh}\eqref{ItSBPsdoChEsssupp}.) In the notation of Lemma~\ref{LemmaFWg}, we have
\begin{equation}
\label{EqFWgPsdoIncl}
  w_0\Psi_\cV^{-m_+}(M) = \Psi_\cV^{w_0\lambda^{m_+}}(M) \subset\Psi_\cV^w(M)\subset \Psi_\cV^{w_0\lambda^{m_-}}(M) = w_0\Psi_\cV^{-m_-}(M).
\end{equation}
We can then define the associated precise classes of $\cV$-Sobolev spaces by
\begin{equation}
\label{EqFWgHV}
  H_\cV^w(M) := \{ u\in w_0 H_\cV^{m_-}(M) \colon A u\in L^2(M) \}
\end{equation}
for any fixed elliptic operator $A\in\Psi_\cV^{w^{-1}}(M)$, with norm
\[
  \|u\|_{H_\cV^w(M)}^2 := \|u\|_{w_0 H_\cV^{m_-}(M)}^2 + \|A u\|_{L^2(M)}^2.
\]
(Different choices of $m_-$ and $A$ lead to equivalent norms.) In the special case $w=w_0\lambda^s$, this recovers Definition~\ref{DefSBH}, with the condition $A u\in L^2(M)$ being automatically satisfied by Theorem~\ref{ThmSBPsdoH}.

\begin{prop}[Properties of $H_\cV^w(M)$ spaces]
\label{PropFWgHV}
  Let $w,w'$ be phase space weights.
  \begin{enumerate}
  \item\label{ItFWgHVMap} Every $P\in\Psi_\cV^w(M)$ defines a bounded map $H_\cV^{w'}(M)\to H_\cV^{w w'}(M)$.
  \item\label{ItWgHVDual} In the same sense as in Proposition~\usref{PropSBHDual}, we have
    \begin{equation}
    \label{EqFWgHVDual}
      (H_\cV^w(M))^* = H_\cV^{w^{-1}}(M)
    \end{equation}
  \end{enumerate}
\end{prop}
\begin{proof}
  Let $u\in H_\cV^{w'}$, so (with primes added to the notation of Lemma~\ref{LemmaFWg}) $u\in w_0' H_\cV^{m'_-}$ and $A'u\in L^2$ where $A'\in\Psi_\cV^{w'{}^{-1}}$ is elliptic. By~\eqref{EqFWgPsdoIncl}, $P u\in w_0 w_0' H_\cV^{m'_-+m_-}$. Moreover, if $A''\in\Psi_\cV^{(w w')^{-1}}$, then $A''P\in\Psi_\cV^{w'{}^{-1}}$; for $B'\in\Psi_\cV^{w'}$ with $B'A'=I+R$, $R\in\rho^\infty\Psi_\cV^{-\infty}$, we then have
  \[
    A''P u = A''P(B'A'-R)u = (A''P B')A'u - (A''P R)u.
  \]
  In the first summand, $A''P B'\in\Psi_\cV^0$ is bounded on $L^2$, and in the second summand, $A''P R\in w^{\prime-1}_0\rho^\infty\Psi_\cV^{-\infty}(M)$ maps $u\in w'_0 H_\cV^{m'_-}$ into $L^2$ as well; thus $A''P u\in L^2$. This proves part~\eqref{ItFWgHVMap}.

  For part~\eqref{ItWgHVDual}, let $\lambda\in(H_\cV^w)^*$. In the notation used in~\eqref{EqFWgHV}, define the isometry $\Phi\colon H_\cV^w\to w_0 H_\cV^{m_-}\oplus L^2$, $u\mapsto(u,A u)$. Denote by $\tilde\lambda\in(w_0 H_\cV^{m_-}\oplus L^2)^*$ the extension of $\Phi(H_\cV^w)\xra{\Phi^{-1}} H_\cV^w\xra{\lambda}\C$ with the same norm; then Proposition~\ref{PropSBHDual} produces $v\in w_0^{-1}H_\cV^{-m_-}$ and $v'\in L^2$ so that $\tilde\lambda(u,A u)=\la u,v\ra_{L^2}+\la A u,v'\ra_{L^2}$, and thus $\lambda(u)=\la u,u^*\ra$ where $u^*=v+A^*v'\in H_\cV^{w^{-1}}$ by part~\eqref{ItFWgHVMap}; this proves `$\subseteq$' in~\eqref{EqFWgHVDual}. Conversely, by this formula, every $u^*$ of this form induces an element of $(H_\cV^w)^*$. To prove `$\supseteq$' in~\eqref{EqFWgHVDual}, it thus remains to note that
  \[
    H_\cV^{w^{-1}} = w_0^{-1}H_\cV^{m_-} + A^*L^2;
  \]
  the nontrivial inclusion `$\subseteq$' follows by taking $B\in\Psi_\cV^{w^{-1}}$ to be a parametrix of $A^*$ with $A^*B=I+R$, $R\in\rho^\infty\Psi_\cV^{-\infty}$, so $u\in H_\cV^{w^{-1}}$ can be written as $u=A^*(B u)+R u$, which is of the desired form.
\end{proof}

\begin{example}[Second microlocal b/scattering algebra]
\label{ExFWgbsc}
  Using phase space weights, we can recover the spaces $\Psi_{\scop,\bop}^{s,r,l}(X)$ of second microlocal b/scattering pseudodifferential operators defined in \cite[\S5]{VasyLowEnergy}. For notational simplicity, we consider only the case $X=[0,\infty)_x$ and work near $x=0$. The unit cells for the b.g.\ structure corresponding to b-analysis are $(2^{-j-2},2^{-j+2})$ for $j\in\N$. The phase space unit cells are thus $U_{j,k}=(2^{-j-2},2^{-j+2})_x\times(2^{k-2},2^{k+2})_\xi$ for $k\in\N$, and $U_{j,0}=(2^{-j-2},2^{-j+2})_x\times(-4,4)_\xi$. Then $\Psi_{\scop,\bop}^{s,r,l}(X)$ is essentially equal to (i.e.\ differs only on the level of residual operators from) $\Psi_\cV^w(X)$ for the phase space weight $w=(\frac{\la\xi\ra^{-1}}{x+\la\xi\ra^{-1}})^{-s}(x+\la\xi\ra^{-1})^{-r}(\frac{x}{x+\la\xi\ra^{-1}})^{-l}$, or, on the level of an equivalent weight family, $w_{j,0}=2^{j l}$ and
  \[
    w_{j,k} = \Bigl(\frac{2^{-k}}{2^{-j}+2^{-k}}\Bigr)^{-s} (2^{-j}+2^{-k})^{-r} \Bigl(\frac{2^{-j}}{2^{-j}+2^{-k}}\Bigr)^l.
  \]
  (The verification that $w$ is indeed a weight is conceptually most straightforwardly done in this case by checking its conormality on the radially compactified b-cotangent bundle, or even better on its blow-up studied in \cite[Lemma~5.1]{VasyLowEnergy}.)
\end{example}

Operators with phase space weights can also be defined in the \emph{parameterized} scaled b.g.\ setting. As an application, this allows one to recover the resolved algebra $\bigcup_{m,l,\nu,\delta}\Psi_{{\rm b,res}}^{m,l,\nu,\delta}$ (and its symbol calculus) which was introduced by Vasy \cite{VasyLowEnergyLag} for low energy resolvent analysis. We leave the details to the interested reader.

\subsection{Variable orders and mildly exotic calculi}
\label{SsFVar}

The need for spaces with variable regularity (and decay, when $\inf\rho=0$) orders has frequently arisen in recent applications when lower or upper bounds imposed on the regularity at different points in phase space (typically due to thresholds conditions in radial point estimates) do not permit constant orders. See for instance \cite[Proposition~5.28]{VasyMinicourse} and Example~\ref{ExFVarSc} below, \cite[\S{5}]{BaskinVasyWunschRadMink}, and \cite[\S{5.2}]{HintzNonstat}. (Variable order spaces were originally introduced in \cite{VishikEskinVariable,UnterbergerVariable,DuistermaatCarleman}.) A typical example of a variable order symbol is $\la\xi\ra^{a(x,\xi)}$ where $a\in S^0(\R^n_x;\R^n_\xi)$ (typically, $a\in S^0_{\rm cl}(\R^n;\R^n)$); this is not of class $S^{\sup a}(\R^n;\R^n)$ since derivatives along $\pa_x$ and $\la\xi\ra\pa_\xi$ lose a factor of $\log(1+\la\xi\ra)$. More generally, one can consider mildly exotic symbols which lose a factor of $\la\xi\ra^\delta$ for $\delta\in[0,\frac12)$, corresponding to the H\"ormander class $S_{\rho,\delta}^{\sup a}(\R^n;\R^n)$ where $\rho=1-\delta$. Restricting to uniform symbols, we say that $a=a(x,\xi)\in S_{\rho,\delta}^m(\R^n;\R^n)$ if
\begin{equation}
\label{EqFVarSymbols}
  |(\la\xi\ra^{-\delta}\pa_x)^\beta(\la\xi\ra^\rho\pa_\xi)^\gamma a(x,\xi)| \leq C_{\beta\gamma}\la\xi\ra^m.
\end{equation}
A typical example is $\la\xi\ra^{a(x,\xi)}$ where $a\in S_{1-\delta',\delta'}^0(\R^n;\R^n)$ for $\delta'<\delta$ (so in particular $a\in S^0$ is allowed). (This includes symbols on suitable inhomogeneous blow-ups of submanifolds of fiber infinity in the radially compactified cotangent bundle.)

\begin{lemma}[B.g.\ structure for $(1-\delta,\delta)$ symbols on $\R^n$]
\label{LemmaFVarBg}
  Let $0<\delta<\frac12$ and fix an integer $C_\delta\geq 2(\frac53)^{\frac{1}{\delta}}$. For $\ell\in\Z^n$ and $i=1,\ldots,n$, $\N\ni k\geq 3$, and $j\in\{-C_\delta k,\ldots, C_\delta k\}^{n-1}$, set
  \begin{equation}
  \label{EqFVarBgU}
  \begin{split}
    U_{0,\ell} &:= \Bigl(\frac12\ell + (-2,2)^n\Bigr)_x \times (-4,4)_\xi^n, \\
    U_{i,k,j,\ell,\pm} &:= \bigl( k^{-1}(\ell+(-4,4)^n) \bigr)_x \times \Bigl\{ \xi\in\R^n \colon {\pm}\xi_i \in \bigl( (k-2)^{\frac{1}{\delta}}, (k+2)^{\frac{1}{\delta}}\bigr), \\
      &\quad \hspace{15.5em}\wh{\xi_i}\in k^{\frac{1}{\delta}-1}(j+(-4,4)^{n-1}) \Bigr\},
  \end{split}
  \end{equation}
  where we write $\wh{\xi_i}=(\xi_1,\ldots,\xi_{i-1},\xi_{i+1},\ldots,\xi_n)$. Define the maps $\phi_{0,\ell}(x,\xi)=(x-\frac12\ell,\frac12\xi)$ and $\phi_{i,k,j,\ell,\pm}(x,\xi)=(\frac12(k x-\ell),(\pm\xi_i)^\delta-k,\frac12(k^{-\frac{1}{\delta}+1}\wh{\xi_i}-j))$. Then $\fB_{\delta,\R^n}^*:=\{(U_\alpha,\phi_\alpha)\colon\alpha=(0,\ell),\ (i,k,j,\ell,\pm)\}$ is a b.g.\ structure on $\R^n_x\times\R^n_\xi$, and
  \begin{equation}
  \label{EqFVarBg}
    \CI_{{\rm uni},\fB_{\delta,\R^n}^*}(\R_x^n\times\R_\xi^n) = S^0_{1-\delta,\delta}(\R^n;\R^n).
  \end{equation}
\end{lemma}

Unlike the b.g.\ structure in Definition~\ref{DefSBTs}, $\fB_{\delta,\R^n}^*$ is no longer the product of a b.g.\ structure on $\R^n$ and another one on $\R^n$. If one wishes to have a b.g.\ structure which in the limit $\delta\searrow 0$ recovers the structure from Definition~\ref{DefSBTs}, one can use $\pm\xi_i\in((1+\delta(k-2)\log 2)^{\frac{1}{\delta}},(1+\delta(k+2)\log 2)^{\frac{1}{\delta}})$, with similar modifications for $\wh{\xi_i}$.

\begin{proof}[Proof of Lemma~\usref{LemmaFVarBg}]
  The purpose of the constant $C_\delta$ is that on $U_{i,k,j,\ell,\pm}$, the largest possible values of each component of $\wh{\xi_i}$ exceeds $k^{\frac{1}{\delta}-1}(C_\delta k)=k^{\frac{1}{\delta}}C_\delta\geq 2(k+2)^{\frac{1}{\delta}}$, thus twice the maximum of $|\xi_i|$. This ensures that the sets $\phi_\alpha^{-1}((-1,1)^n)$ indeed cover $\R^n\times\R^n$. We only argue for the equality~\eqref{EqFVarBg}: on $U_{i,k,j,\ell,\pm}$, we have $\la\xi\ra\sim k^{\frac{1}{\delta}}$. Thus, the pullbacks under $\phi_{i,k,j,\ell,\pm}$ of coordinate derivatives are essentially $k^{-1}\pa_x\sim\la\xi\ra^{-\delta}\pa_x$ and (given that $(k+2)^{\frac{1}{\delta}}-(k-2)^{\frac{1}{\delta}}=4(k-2)^{\frac{1}{\delta}-1}\delta^{-1}+\cO(k^{\frac{1}{\delta}-2})$) $k^{\frac{1}{\delta}-1}\pa_\xi\sim\la\xi\ra^{1-\delta}\pa_\xi$.
\end{proof}

\begin{definition}[$(1-\delta,\delta)$-phase space b.g.\ structure]
\label{DefFVarBg}
  Let $\fB=\{(U_\alpha,\phi_\alpha)\colon\alpha\in\sA\}$ be a b.g.\ structure on $M$. With respect to the trivializations of ${}^\cV T^*M$ over each $U_\alpha$ given by Definition~\ref{DefSBDCotgt}, and using the notation~\eqref{EqFVarBgU}, we define
  \begin{alignat*}{2}
    U_{(\alpha,0)} &:= U_\alpha \times (-4,4)^n, &\qquad
    U_{(\alpha,i,k,j,\ell,\pm)} &:= (\phi_\alpha^{-1}\times\Id)U_{i,k,j,\ell,\pm} \subset {}^\cV T^*_{U_\alpha}M, \\
    \phi_{(\alpha,0)} &:= \phi_\alpha \times \frac12\Id, &\qquad
    \phi_{(\alpha,i,k,j,\ell,\pm)} &:= \phi_{i,k,j,\ell,\pm}\circ(\phi_\alpha\times\Id),
  \end{alignat*}
  where $i=1,\ldots,n$, $\N\ni k\geq 3$, $j\in\{-C_\delta k,\ldots,C_\delta k\}$, and $\ell\in\{-2 k+4,\ldots,2 k-4\}^n$. We then define the \emph{$(1-\delta,\delta)$-phase space b.g.\ structure} to be
  \[
    \fB_\delta^* := \bigl\{(U_{(\alpha,\beta)},\phi_{(\alpha,\beta)})\colon\alpha\in\sA,\ \beta\in\{0,(i,k,j,\ell,\pm)\} \bigr\}.
  \]
\end{definition}

The equality~\eqref{EqFVarBg} then generalizes to
\[
  \CI_{{\rm uni},\fB_\delta^*}({}^\cV T^*M) = S^0_{1-\delta,\delta}({}^\cV T^*M),
\]
where the space on the right is defined analogously to Definition~\ref{DefSBDSymb} but now using the bounds~\eqref{EqFVarSymbols} in distinguished charts and trivializations. If we set $\fB_0^*:=\fB^*$, then for $0\leq\delta'\leq\delta$ we have
\begin{equation}
\label{EqFVarC0Incl}
  \cC^0_{{\rm uni},\fB_{\delta'}^*}({}^\cV T^*M)\subset\cC^0_{{\rm uni},\fB_\delta^*}({}^\cV T^*M).
\end{equation}

The appropriate microlocalization locus is
\[
  \fM_\delta := \fu({}^\cV T^*M,\fB_\delta^*) \cap \{ j(\lambda\rho)=0 \},
\]
where $\rho,\lambda$ are as in Definition~\usref{DefSBDScWeight} and Lemma~\usref{LemmaSBTsSymbols}\eqref{ItSBTsSymbols1}. The inclusion~\eqref{EqFVarC0Incl} induces a continuous map $\fu({}^\cV T^*M,\fB_\delta^*)\to\fu({}^\cV T^*M,\fB_0^*)$, which by the same arguments as in the proof of Lemma~\ref{LemmabgUniv} is surjective, and thus a surjective continuous map
\[
  \fM_\delta\to\fM
\]
of compact Hausdorff spaces.\footnote{This map is not injective when $\delta>0$, as otherwise it would be a homeomorphism; but that it is not (reflecting the fact that using $(1-\delta,\delta)$-symbols, one can microlocalize more finely). Indeed, in a chart $U_\alpha$, consider the points $p_{(k)}=(0,\xi_{(k)})$, $\xi_{(k)}=(k,0,0,\ldots,0)$, and $q_{(k)}=(0,\eta_{(k)})$, $\eta_{(k)}=(k,k^{1-\delta},0,\ldots,0)$. Fix compatible metrics $d_\delta$ and $d_0$ for $\fB_\delta^*$ and $\fB_0^*$, respectively; then $\inf_{k,k'} d_\delta(p_{(k)},q_{(k')})>0$, whereas $\inf_k d_0(p_{(k)},q_{(k)})=0$, so the closures of the sets $\{p_{(k)}\}$ and $\{q_{(k')}\}$ are disjoint in the first case, but not so in the second case (cf.\ Lemma~\ref{LemmabgCap}).}

We can now set
\begin{subequations}
\begin{equation}
\label{EqFVarPsi}
  w\Psi_{\cV,1-\delta,\delta}^m(M) := \{ \Op_\cV(a)+R \colon a\in w S^m_{1-\delta,\delta}({}^\cV T^*M),\ R\in w\rho^\infty\Psi_\cV^{-\infty}(M) \},
\end{equation}
the properties of which (algebra properties, boundedness on Sobolev spaces) are the same as those of $w\Psi_\cV(M)$, \emph{except} the principal symbol of $A=\Op_\cV(a)+R$ is now
\begin{equation}
\label{EqFVarPsiSymb}
  \upsigma_{\cV,1-\delta,\delta}^{m,w}(A) = [a] \in w S_{1-\delta,\delta}^m / w\rho S_{1-\delta,\delta}^{m-1+2\delta}({}^\cV T^*M).
\end{equation}
\end{subequations}
Theorem~\ref{ThmSBPsdoComp} thus remains valid, \emph{mutatis mutandis}. (We leave the minor modifications of the proofs to the interested reader.)

We proceed to discuss operators with variable orders and weights; \emph{for notational simplicity, we only consider order functions which are standard symbols} rather than $(1-\delta',\delta')$ symbols.

\begin{definition}[Variable order symbols, ps.d.o.s, Sobolev spaces]
\label{DefFVarSymbPsdo}
  Let $\sfm,\sfl\in S^0({}^\cV T^*M)$, and let $w\in\CI(M)$ be a weight on $(M,\fB)$. In the notation of Definition~\usref{DefSBDScWeight} and Lemma~\usref{LemmaSBTsSymbols}\eqref{ItSBTsSymbols1}, we then write
  \begin{align*}
    \rho^{-\sfl}S^\sfm({}^\cV T^*M) &:= \rho^{-\sfl}\lambda^{-\sfm}\bigcap_{\delta\in(0,\frac12)}S^0_{1-\delta,\delta}({}^\cV T^*M), \\
    w\Psi_\cV^{\sfm,\sfl}(M) &:= \{ \Op_\cV(a)+R \colon a\in w\rho^{-\sfl}S^\sfm({}^\cV T^*M),\ R\in w\rho^\infty\Psi_\cV^{-\infty}(M) \},
  \end{align*}
  with the principal symbol of $A=\Op_\cV(a)+R$ being
  \begin{equation}
  \label{EqFVarSymbPsdo}
    \upsigma_\cV^{\sfm,\sfl}(A) = [a] \in w\rho^{-\sfl}S^\sfm \,\Big/ \bigcap_{\delta\in(0,\frac12)}w\rho^{-\sfl+1-2\delta}S^{\sfm-1+2\delta}({}^\cV T^*M).
  \end{equation}
  Fixing an elliptic operator $A\in w\rho^{-\sfl}\Psi_\cV^\sfm(M)$, we furthermore define
  \[
    w H_\cV^{\sfm,\sfl}(M) := \{ u\in w\rho^{\inf\sfl}H_\cV^{\inf\sfm}(M) \colon A u\in L^2(M) \}.
  \]
\end{definition}

A topology on the space $w\Psi_\cV^{\sfm,\sfl}(M)$ can be defined similarly to Remark~\ref{RmkSBPsdoTop}.

If only $\sfl$ is variable but $\sfm=m$ is not, then the principal symbol is well-defined in the sharper space $w\rho^{-\sfl}S^m/\bigcap_{\delta\in(0,\frac12)} w\rho^{-\sfl+1-2\delta}S^{m-1}({}^\cV T^*M)$; analogously when $\sfm$ is variable but $\sfl$ is not. The space $\rho^{-\sfl}S^\sfm({}^\cV T^*M)$ only depends on $\sfm,\sfl$ modulo $\bigcup_{\eps>0}\rho^\eps S^{-\eps}({}^\cV T^*M)$. Furthermore,
\[
  \rho^{-\sfl}S^\sfm({}^\cV T^*M) \subset \bigcap_{\delta\in(0,\frac12)} \rho^{-\sup\sfl}S_{1-\delta,\delta}^{\sup\sfm}({}^\cV T^*M).
\]
Note that if in a chart $U_\alpha$ we replace $\rho$ by the constant $\bar\rho_\alpha$, then derivatives of $\bar\rho_\alpha^{-\sfl}$ along $\rho\pa_x$ or $\la\xi\ra\pa_\xi$ gain a power of $\bar\rho_\alpha\log\bar\rho_\alpha$, which leads to the (arbitrarily small) loss in the $\rho$-weight in~\eqref{EqFVarSymbPsdo}.

The weights, resp.\ orders of variable order operators are multiplicative, resp.\ additive under composition, and indeed Theorem~\ref{ThmSBPsdoH} remains valid, \emph{mutatis mutandis}, in the variable order case. As the microlocalization locus, one may work with $\fM$. Furthermore, by mimicking the proof of Proposition~\ref{PropFWgHV}, one shows
\[
  w\Psi_\cV^{\sfm,\sfl}(M) \ni A \colon w' H_\cV^{\sfm',\sfl'}(M) \to w w' H_\cV^{\sfm'-\sfs,\sfl'-\sfl}(M),\quad
  (w H_\cV^{\sfm,\sfl}(M))^*=w^{-1}H_\cV^{-\sfm,-\sfl}(M).
\]

\begin{example}[Variable scattering decay]
\label{ExFVarSc}
  We use the scaled b.g.\ structure~\eqref{EqISc2Sets}--\eqref{EqISc2Scale}, with operator Lie algebra $\cV$ given by scattering vector fields on $\ol{\R^n}$ with conormal coefficients. Consider $\sfl_0(x,\xi):=-\frac12-\eps\frac{x\cdot\xi}{|x||\xi|}$ where $\eps>0$ is fixed. Then $\sfl$ can be extended from $|x|>1$, $|\xi|>\frac{c}{2}>0$, to an element of $S^0({}^\cV T^*M)$, and for $\lambda>c$ then, the limiting absorption principle states that $(\Delta-\lambda+i 0)^{-1}\colon H_\cV^{s-2,\sfl+1}(\R^n)\to H_\cV^{s,\sfl}(\R^n)$ (for any $s\in\R$); see \cite[Proposition~5.28]{VasyMinicourse}.
\end{example}

\begin{rmk}[Phase space weights and variable orders]
\label{RmkFVarTs}
  Suppose $w$ is a phase weight on $({}^\cV T^*M,\fB^*)$, so $w=w_0 w_1$ where $w_0\in\CI(M)$ is a weight on $(M,\fB)$ and $C^{-1}\lambda^{m_+}\leq w_1\leq C\lambda^{m_-}$. Scaling $\lambda$ by a constant factor so that $\lambda<\frac12$, we then have $w_1=\lambda^\sfm$ where $\sfm=(\log w_1)/(\log\lambda)$ takes values in a bounded interval. Furthermore, for $V\in\CI_{{\rm uni},\fB^*}({}^\cV T^*M)$ (in local coordinates $\pa_x$ and $\la\xi\ra\pa_\xi$), one computes $V\sfm=\frac{V w_1}{w_1}\frac{1}{\log\lambda}-\sfm\frac{V\lambda}{\lambda}$, which is bounded; higher derivatives are treated similarly. Thus, $\sfm\in S^0({}^\cV T^*M)$. Therefore,
  \[
    w S^0({}^\cV T^*M) \subset w_0 S^{-\sfm}({}^\cV T^*M),
  \]
  similarly for ps.d.o.s. The small benefit of phase space weights (when they are usable) is that the principal symbol is slightly more precise; compare~\eqref{EqFWgSymb} with \eqref{EqFVarSymbPsdo}. Note however that, conversely, variable order operators are typically \emph{not} operators with phase space weights.
\end{rmk}

\begin{rmk}[Parameterized version]
\label{RmkFVarParam}
  The discussion in the present section generalizes to variable order symbols, ps.d.o.s, and Sobolev spaces depending on a parameter; indeed, the discussion in~\S\ref{SsSBPar} applies with only straightforward notational modifications. For example, in a semiclassical bounded geometry setting, one can work with variable semiclassical orders. Applications in the literature include \cite[\S{2.3}]{GalkowskiThinBarriers} and \cite{HintzVasyCauchyHorizon}; see also Example~\ref{ExFFTHV}.
\end{rmk}

\begin{rmk}[Admissible compactifications]
\label{RmkFVarCpt}
  The discussion in~\S\ref{SsSBComp} generalizes to mildly exotic and variable order symbols. The main change is that in Definition~\ref{DefSBCpt}\eqref{ItSBCptbg} (and subsequently) one needs to replace $\fB^*$ by $\fB_\delta^*$, and correspondingly $S^0$ by $S^0_{1-\delta,\delta}$.
\end{rmk}

Consider finally a finite collection
\[
  \rho_1,\ldots,\rho_N \in \CI_{{\rm uni},\fB}(M)
\]
of \emph{bounded} weights on $(M,\fB)$, and assume that $\rho_j$ dominates $\rho$ in that
\[
  \frac{\rho}{\rho_j} \in \CI_{{\rm uni},\fB}(M),\quad j=1,\ldots,N.
\]
Spaces of ps.d.o.s with weights $w=\rho_1^{-l_1}\cdots\rho_N^{-l_N}$ where the orders $l_j$ are \emph{constants} can be defined as $w\Psi_\cV^m(M)$ simply. However, if one needs \emph{variable} orders $l_j$, the setup of Definition~\ref{DefFVarSymbPsdo} does not suffice when $N\geq 2$. Instead, one works with symbols
\begin{subequations}
\begin{equation}
\label{EqFVarMultiWeight1}
  S^{\sfm,(\sfl_1,\ldots,\sfl_N)}({}^\cV T^*M) := \Biggl(\prod_{j=1}^N \rho_j^{-\sfl_j}\Biggr)\lambda^{-\sfm}\bigcap_{\delta\in(0,\frac12)} S^0_{1-\delta,\delta}({}^\cV T^*M)
\end{equation}
and their quantizations
\begin{equation}
\label{EqFVarMultiWeight2}
  w\Psi_\cV^{\sfm,(\sfl_1,\ldots,\sfl_N)}(M) := \Op_\cV\bigl(w S^{\sfm,(\sfl_1,\ldots,\sfl_N)}({}^\cV T^*M)\bigr) + w\rho^\infty\Psi_\cV^{-\infty}(M);
\end{equation}
the principal symbol is now valued in
\begin{equation}
\label{EqFVarMultiWeight3}
  w S^{\sfm,(\sfl_1,\ldots,\sfl_N)} / \bigcap_{\delta\in(0,\frac12)} w \rho S^{\sfm-1+2\delta,(\sfl_1-2\delta,\ldots,\sfl_N-2\delta)}({}^\cV T^*M).
\end{equation}
\end{subequations}

\subsection{Translation symmetries and Fourier transforms}
\label{SsFFT}

In applications where on a manifold $M=\R_t\times X$ one takes a Fourier (or Mellin) transform in $t$ to reduce the PDE under study on $M$ to a family of PDEs on $X$, it is important to have Plancherel type theorems available in order to identify $L^2$-based Sobolev norms on $M$ with suitable $L^2$-type norms on $\R_\sigma\times X$ where $\sigma$ is the Fourier-dual variable to $t$; see e.g.\ \cite[\S{3.1}]{VasyMicroKerrdS}, \cite[Proposition~4.24 and \S{7.3}]{Hintz3b}, \cite[Step (iii) in the proof of Proposition~5.19]{HintzNonstat}. The goal of this section is to prove a general result of this type (including for spaces with variable orders), which we will accomplish via the study of $t$-translation invariant $\cV$-ps.d.o.s for suitable scaled b.g.\ structures on such product manifolds $M$.

\begin{definition}[Scaled b.g.\ structure on $\R\times X$]
\label{DefFFTbg}
  Let $\fB_{X,\times}=\{(U_\alpha,\phi_\alpha,\rho_\alpha)\colon\alpha\in\sA\}$ be a scaled b.g.\ structure on the $(n-1)$-dimensional manifold $X$, with underlying b.g.\ structure denoted $\fB_X$. Write $\hat\cV_0\subset\CI(X;T X)$ for the operator Lie algebra of $(X,\fB_{X,\times})$. Let moreover $\{\rho_{\alpha,0}\colon\alpha\in\sA\}$ and $\{\tau_\alpha\colon\alpha\in\sA\}$ be weight families (Definition~\ref{DefSBDWeights}), with $0<\rho_{\alpha,0}\leq 1$. Then the associated scaled b.g.\ structure on $M=\R\times X$ is
  \begin{alignat*}{2}
    &\fB_\times = \{ (U_{(j,\alpha)},\phi_{(j,\alpha)},\rho_{(j,\alpha)}) \colon \alpha\in\sA,\ j\in\Z \}, \hspace{-20em}&& \\
    &\qquad U_{(j,\alpha)} &&= \Bigl( (j-2)\frac{\tau_\alpha}{\rho_{\alpha,0}}, (j+2)\frac{\tau_\alpha}{\rho_{\alpha,0}}\Bigr)_t \times U_\alpha, \\
    &\qquad \phi_{(j,\alpha)} &&= \phi_{T,(j,\alpha)} \times \phi_\alpha, \qquad \phi_{T,(j,\alpha)}:=\frac{\rho_{\alpha,0}}{\tau_\alpha}(\cdot)-j, \\
    &\qquad \rho_{(j,\alpha),0} &&= \rho_{\alpha,0},\quad \rho_{(j,\alpha),i}=\rho_{\alpha,i}\ \ (i=1,\ldots,n-1).
  \end{alignat*}
  We write $\fB$ for the underlying b.g.\ structure. We moreover write $\rho_X$, $\rho_0$, and $\tau\in\CI(X)$ for weights on $(X,\fB_X)$ which are equivalent to $\{\bar\rho_\alpha\colon\alpha\in\sA\}$, $\{\rho_{\alpha,0}\colon\alpha\in\sA\}$, and $\{\tau_\alpha\colon\alpha\in\sA\}$, respectively. (So $\rho_X$ is a scaling weight on $(X,\fB_{X,\times})$.)
\end{definition}

On $U_{(j,\alpha)}$, with local coordinates $t\in\R$ and $x_\alpha^i=\phi_\alpha^i$, define $T_{(j,\alpha)}$ and $\tilde x_\alpha$ via
\[
  T_{(j,\alpha)} = \frac{\rho_{\alpha,0}}{\tau_\alpha}t-j,\qquad
  \tilde x_\alpha^i = \frac{x_\alpha^i}{\rho_{\alpha,i}}\ \ (i=1,\ldots,n-1).
\]
Then elements of the coefficient Lie algebra $\cW=\CI_{{\rm uni},\fB}(M;T M)$ on $(M,\fB_\times)$ are, in local coordinates, uniformly bounded linear combinations of $\pa_{T_{(j,\alpha)}}=\frac{\tau_\alpha}{\rho_{\alpha,0}}\pa_t$ and $\pa_{x_\alpha^i}=\rho_{\alpha,i}^{-1}\pa_{x_\alpha^i}$, while elements of the operator Lie algebra $\cV$ are uniformly bounded linear combinations of
\begin{equation}
\label{EqFFTcV}
  \rho_{\alpha,0}\pa_{T_{(j,\alpha)}} = \tau_\alpha\pa_t,\qquad
  \rho_{\alpha,i}\pa_{x_\alpha^i} = \pa_{\tilde x_\alpha^i}.
\end{equation}
Thus, $\tau_\alpha$ specifies the scaling of $t$-derivatives, while $\frac{1}{\rho_{\alpha,0}}\geq 1$ is the factor by which unit cells are longer in the $t$-direction than $\tau_\alpha$.

\begin{example}[Operators on Euclidean space]
\fakephantomsection
\label{ExFFTEucl}
  \begin{enumerate}
  \item\label{ItFFTEuclsc} If one defines $\fB_{X,\times}$ on $X=\R^{n-1}$ to be $U_\alpha=\alpha+(-4,4)^{n-1}$, $\alpha\in\Z^{n-1}$, with trivial scalings $\rho_{\alpha,i}=1$, and $\tau_\alpha=\rho_{\alpha,0}=1$, then $\fB$ is equal to the b.g.\ structure~\eqref{EqIScBdd}; the operator Lie algebra consists of linear combinations of $\pa_t$, $\pa_{x^i}$ ($i=1,\ldots,n-1$) with uniformly bounded smooth coefficients on $\R^n$.
  \item\label{ItFFTEuclb} In order to gain access to the boundary principal symbol (which will descend to the spectral, i.e.\ Fourier transform, side), one instead takes $\fB_{X,\times}$ the scaled b.g.\ structure~\eqref{EqISc2Sets}--\eqref{EqISc2Scale} (on $\R^{n-1}$), and then on $\fB_X$-unit cells $U_\alpha=U_{j,k,\pm 1}$, where $|x|\sim 2^k$, one takes $\tau_\alpha=1$ and $\rho_{\alpha,0}=2^{-k}$. The resulting scaled b.g.\ structure $\fB_\times$ is not equivalent to~\eqref{EqISc2Sets}--\eqref{EqISc2Scale}, but rather to that corresponding to 3-body scattering operators on $\bar M:=[\ol{\R\times\R^{n-1}};\pa(\ol\R\times\{0\})]$ with 3b-regular coefficients. (The setting~\eqref{ItFFTEuclsc} gives rise to 3sc-operators with 3sc-regular coefficients---which are indistinguishable from scattering operators on $\R^n$ with scattering-regular coefficients.)
  \item\label{ItFFTEucl3b} With $\fB_X$ as in~\eqref{ItFFTEuclb}, take now $\tau_\alpha=2^k$ and $\rho_{\alpha,0}=1$. Then $\fB_\times$ is the scaled b.g.\ structure (with trivial scalings) with coefficient Lie algebra given by 3b-vector fields on $\bar M$ with 3b-regular coefficients, cf.\ \S\ref{SssIF}\eqref{ItI3b}.
  \end{enumerate}
\end{example}

\begin{example}[b-operators on manifolds with boundary]
\label{ExFFTb}
  Let $\bar M$ be a compact manifold with boundary. The inward pointing normal bundle ${}^+N\bar\pa M$ of the boundary is where the normal operator \cite[\S{4.15}]{MelroseAPS} lives and where the Mellin transform is ultimately used \cite[\S{5.1}]{MelroseAPS}. The same construction as in Example~\ref{ExFFTEucl}\eqref{ItFFTEuclsc} produces the b-algebra on ${}^+N\bar\pa M$ based on the vector fields $\rho\pa_\rho$, $\pa_{x^i}$ where $\rho$ is a fiber-linear defining function of the zero section in ${}^+N\bar\pa M$, under the identification $t=-\log\rho$ (and thus the Mellin transform in $\rho$ is the same as the Fourier transform in $t$).
\end{example}

Fix a uniformly positive $\hat\cV_0$-density $\mu_X$ on $X$ to define $L^2(X)$, and let $\mu:=|\dd t|\mu_X$. (Thus, $\tau^{-1}\mu=|\frac{\dd t}{\tau}|\mu_X$ is a uniformly positive $\cV$-density on $M$.) Define the Fourier transform on $\CIc(M)$, resp.\ its inverse, by
\[
  (\cF u)(\sigma,p) = \int_\R e^{-i\sigma t}u(t,p)\,\dd t,\quad\text{resp.}\quad
  (\cF^{-1}v)(t,p) = (2\pi)^{-1}\int_\R e^{i\sigma t}v(\sigma,p)\,\dd\sigma,
\]
where $\sigma\in\R$, $p\in X$. Plancherel's theorem gives an isomorphism $\cF\colon L^2(\R\times X;\mu)\xra{\cong}L^2(\R;L^2(X;\mu_X))$.

Let now $u\in L^2(\R\times X)$; for simplicity, let us assume that $\supp u\subset\R\times U'_\alpha$ where $U'_\alpha=\phi_\alpha^{-1}([-\frac32,\frac32]^n)$. Then, in the coordinates $t,x_\alpha$, the membership $u\in H_\cV^1(\R\times X)$ is equivalent to $\hat u,\tau_\alpha\sigma\hat u,\rho_{\alpha,i}\pa_{x_\alpha^i}\hat u\in L^2(\R_\sigma;L^2(X))$, or equivalently
\begin{equation}
\label{EqFFTH1}
  \hat u \in L^2\bigl(\R_\sigma; \la\tau_\alpha\sigma\ra H_{\hat\cV_\sigma}^1(U_\alpha)\bigr),\qquad \|v\|_{H_{\hat\cV_\sigma}^1(U_\alpha)}^2 := \|v\|_{L^2}^2 + \sum_{i=1}^{n-1} \| \la\tau_\alpha\sigma\ra^{-1}\rho_{\alpha,i}\pa_{x_\alpha^i}v\|_{L^2}^2,
\end{equation}
where $\|\hat u(\sigma)\|_{\la\tau_\alpha\sigma\ra H_{\hat\cV_\sigma}^1(U_\alpha)}=\|\la\tau_\alpha\sigma\ra^{-1}\hat u(\sigma)\|_{H_{\hat\cV_\sigma}^1(U_\alpha)}$. (For $\sigma=0$, the $H_{\hat\cV_0}^1(U_\alpha)$-norm appearing here, for functions supported in $U'_\alpha$, is equivalent to the Sobolev norm on $X$ with respect to the operator Lie algebra $\hat\cV_0$ from Definition~\ref{DefFFTbg}.) (See Proposition~\ref{PropFFTHV} below for the full result.)

The derivatives appearing in~\eqref{EqFFTH1} motivate the following definition.

\begin{definition}[Spectral scaled b.g.\ structure on $X$]
\label{DefFFTSpec}
  In the notation of Definition~\usref{DefFFTbg}, we define for $\sigma\in\R$
  \[
    \hat\fB_{\sigma,\times} := \{ (U_\alpha,\phi_\alpha,\rho_{\sigma,\alpha}) \colon \alpha\in\sA \},\qquad
    \rho_{\sigma,\alpha,i} := \frac{\rho_{\alpha,i}}{\la\tau_\alpha\sigma\ra}.
  \]
  We write $\hat\fB_\times:=(\hat\fB_{\sigma,\times})_{\sigma\in\R}$ for the parameterized scaled b.g.\ structure on $X$ with parameter space $\R_\sigma$, and $\hat\fB$ for the underlying parameterized b.g.\ structure (so $\hat\fB=(\hat\fB_\sigma)_{\sigma\in\R}$ where $\hat\fB_\sigma=\fB_X$ for all $\sigma\in\R$). The associated operator Lie algebra is denoted $\hat\cV=(\hat\cV_\sigma)_{\sigma\in\R}$.
\end{definition}

An element of $\hat\cV$ is thus a family $(V_\sigma)_{\sigma\in\R}$ of elements of $\hat\cV_\sigma$ (the operator Lie algebra of $(X,\hat\fB_{\sigma,\times})$) whose coefficients, when expressed in local coordinates as linear combinations of $\rho_{\sigma,\alpha,i}\pa_{x_\alpha^i}$, obey uniform $\CI$ bounds. Thus, we have $\hat\cV_\sigma=\la\tau\sigma\ra^{-1}\hat\cV_0$.

Observe that if $w\in\CI(X)$ is a weight on $(X,\fB_X)$, then its pullback along $\R_\sigma\times X\to X$, which we still denote by $w$, is a weight on $(\R\times X,\hat\fB)$; this in particular applies to the scaling weight $\rho_X\in\CI(X)$. As the scaling weight of $(X,\hat\fB_\times)$ we may take
\begin{equation}
\label{EqFFThatrho}
  \hat\rho := \rho_X\la\tau\sigma\ra^{-1}.
\end{equation}
Furthermore, $\la\tau\sigma\ra^{-1}$ is a weight on $(\R\times X,\hat\fB)$ which dominates $\hat\rho$ in that $\frac{\hat\rho}{\la\tau\sigma\ra^{-1}}\in\CI_{{\rm uni},\hat\fB}(\R\times X)$.

\begin{prop}[Spectral family: differential operators]
\label{PropFFTSpec}
  Let $w\in\CI(X)$ be a weight on $(X,\fB_X)$, and let $m\in\N_0$. Denote by
  \[
    w\Diff_{\cV,\rm I}^m(M) \subset w\Diff_\cV^m(M)
  \]
  the subspace of $t$-translation invariant operators $A$, i.e.\ $[\pa_t,A]=0$ (or equivalently: $A$ commutes with translations $(t,p)\mapsto(t+c,p)$ for all $c\in\R$). Define the \emph{spectral family} of $A$ as
  \begin{equation}
  \label{EqFFTSpec}
    \hat A = (\hat A_\sigma)_{\sigma\in\R},\qquad (\hat A_\sigma u)(p) = \bigl(e^{-i\sigma t}A(e^{i\sigma t}u)\bigr)(0,p).
  \end{equation}
  Then we have
  \begin{subequations}
  \begin{equation}
  \label{EqFFTSpecMem}
    \hat A\in w\la\tau\sigma\ra^m\Diff_{\hat\cV}^m(X).
  \end{equation}
  Let $a\in w P^m({}^\cV T^*M)$ be a representative of $\upsigma_\cV^{m,w}(A)\in w P^m/w\rho P^{m-1}$ where $\rho\in\CI(M)$ is a $t$-translation invariant scaling weight on $(M,\fB_\times)$ (thus $\rho\sim\max(\rho_0,\rho_X)$), and define $\hat a_\sigma(\varpi):=a(\sigma\,\dd t+\varpi)$, $\varpi\in T^*X$, and $\hat a:=(\hat a_\sigma)_{\sigma\in\R}$. Then
  \begin{equation}
  \label{EqFFTSpecSymb}
    \upsigma_{\hat\cV}^{m,w\la\tau\sigma\ra^m}(\hat A) = [\hat a] \in w\la\tau\sigma\ra^m P^m / w\la\tau\sigma\ra^{m-1}\rho_X P^{m-1}({}^{\hat\cV}T^*X).
  \end{equation}
  \end{subequations}
\end{prop}
\begin{proof}
  When $m=0$, we have $\hat A_\sigma=A$ (multiplication operator by the function $a\in w\CI_{{\rm uni},\fB_X}(X)$) for all $\sigma$, and the statement is trivial.  By the multiplicativity of the principal symbol map, it then suffices to consider, for $\chi\in\CIc((-2,2)^{n-1})$, the case that $A$ is one of the vector fields $\phi_\alpha^*(\chi\tau D_t)$ (with $D_t=i^{-1}\pa_t$), $\phi_\alpha^*(\chi\rho_{\alpha,i}\pa_{x_\alpha^i})$, cf.\ \eqref{EqFFTcV}. The corresponding spectral families are $(\phi_\alpha^*\chi)\tau\sigma$ and $\phi_\alpha^*(\chi\rho_{\alpha,i}\pa_{x_\alpha^i})$, which we can rewrite as $\la\tau\sigma\ra \frac{\tau\sigma}{\la\tau\sigma\ra}\phi_\alpha^*\chi$ and $\la\tau\sigma\ra\phi_\alpha^*(\chi\rho_{\sigma,\alpha,i}\pa_{x_\alpha^i})$. This gives~\eqref{EqFFTSpecMem}--\eqref{EqFFTSpecSymb}.
\end{proof}

The space in~\eqref{EqFFTSpecMem} is rather large in that a general element of $w\la\tau\sigma\ra^m\Diff_{\hat\cV}^m(X)$ is \emph{not} equal to the spectral family of any element of $w\Diff_{\cV,\rm I}^m(M)$; what the space~\eqref{EqFFTSpecMem} does not capture is the polynomial dependence of $\hat A_\sigma$ on $\sigma$. Note however that ellipticity of $A$ \emph{does} get inherited by $\hat A$.

We now lay the groundwork for the pseudodifferential generalization of Proposition~\ref{PropFFTSpec}.

\begin{lemma}[Construction of translation-invariant ps.d.o.s]
\label{LemmaFFTChar}
  We use the notation of Definition~\usref{DefFFTbg}. Let $m\in\R$, and let $w\in\CI(X)$ be a weight on $(X,\fB_X)$. Denote by
  \[
    w S^m_{\rm I}({}^\cV T^*M)\subset w S^m({}^\cV T^*M)
  \]
  the subspace of symbols which are invariant under the lift to ${}^\cV T^*M$ of translations in $t$. Fix a nonnegative function $\chi\in\CIc((-\frac54,\frac54)^{n-1})$ with $\chi|_{[-1,1]^{n-1}}=1$, and let $\chi_\alpha=\frac{\phi_\alpha^*\chi}{\sum_\beta\phi_\beta^*\chi}$. Let $\psi\in\CIc((-\frac14,\frac14)^{n-1})$ be equal to $1$ near $0$. For $a\in S^m_{\rm I}({}^\cV T^*M)$, define
  \begin{equation}
  \label{EqFFTCharOp}
  \begin{split}
    &\Op_{\cV,\rm I}(a)(t,x_\alpha,t',x'_\alpha) := \sum_\alpha(\Id\times\phi_\alpha)^*\Op_{\alpha,\rm I}((\phi^\tau_\alpha)_*(\chi_\alpha a)), \\
    &\ \Op_{\alpha,\rm I}(a_\alpha)(t,x_\alpha,t',x'_\alpha) := (2\pi)^{-n} \Biggl(\int_{\R^n} \exp\biggl(i\biggl[\,\sum_{j=1}^{n-1} (x_\alpha^j-x_\alpha^{\prime j})\frac{\xi_j}{\rho_{\alpha,j}} + (t-t')\frac{\zeta}{\tau}\biggr]\biggr) \\
    &\ \hspace{4em} \times \psi(x'_\alpha-x_\alpha)\psi\Bigl(\frac{\rho_0}{\tau}(t-t')\Bigr) a_\alpha(0,x_\alpha;\zeta,\xi)\,\dd\zeta\,\dd\xi_1\cdots\dd\xi_{n-1}\Biggr)\,\frac{|\dd t'\,\dd x_\alpha^{\prime 1}\cdots\dd x_\alpha^{\prime n-1}|}{\tau\rho_{\alpha,1}\cdots\rho_{\alpha,n}}\,,
  \end{split}
  \end{equation}
  where $\phi_\alpha^\tau\colon T_{(0,p)}^*M\ni\sigma\,\dd t+\xi \mapsto(\phi_\alpha(p);\tau(p)\sigma,\rho_{\alpha,1}\xi_1,\ldots,\rho_{\alpha,n-1}\xi_{n-1})\in(-2,2)^{n-1}\times\R^n_{\zeta,\xi}$. Then for appropriate choices in the definition of $\Op_\cV$ in Definition~\ref{DefSBPsdo}, we have $\Op_{\cV,\rm I}-\Op_\cV\colon w S^m_{\rm I}({}^\cV T^*M)\to w\rho^\infty\Psi_\cV^{-\infty}(M)$, and
  \begin{equation}
  \label{EqFFTCharOp2}
    \Op_{\cV,\rm I}\bigl(w S^m_{\rm I}({}^\cV T^*M)\bigr) \subset w\Psi_{\cV,\rm I}^m(M) = \bigl\{ A \in w\Psi_\cV^m(M) \colon A\ \text{is $t$-translation invariant} \bigr\}.
  \end{equation}
\end{lemma}

The map $\phi_\alpha^\tau$ expresses a covector on $M$ in the trivialization of ${}^\cV T^*_{U_{0,\alpha}}M$ given by the local frame $\frac{\dd t}{\tau}$, $\frac{\dd x_\alpha^i}{\rho_{\alpha,i}}$. If instead of $\frac{\dd t}{\tau}$ we used $\frac{\dd T_{(j,\alpha)}}{\rho_{\alpha,0}}=\frac{\dd t}{\tau_\alpha}$ (the two choices differing by a factor of $\frac{\tau_\alpha}{\tau}$ which is uniformly bounded in $\CI$ together with its reciprocal), this would match the trivializations in Definition~\ref{DefSBDCotgt}; our present minor modification produces cleaner formulae for spectral families below.

\begin{proof}[Proof of Lemma~\usref{LemmaFFTChar}]
  For notational simplicity, we only consider the unweighted case $w=1$. Let $\chi_T\in\CIc((-\frac54,\frac54))$ be equal to $1$ on $[-1,1]$, and let $\chi_{T,\alpha,j}=\frac{\phi_{T,(j,\alpha)}^*\chi_T}{\sum_k\phi_{T,(k,\alpha)}^*\chi_T}$ in the notation of Definition~\ref{DefFFTbg}; these are functions of $t$ only. Then $\chi_{(j,\alpha)}:=\chi_\alpha\chi_{T,\alpha,j}$ is a uniform partition of unity on $(M,\fB)$ which we may use for defining the quantization of symbols on ${}^\cV T^*M$ in Definition~\ref{DefSBPsdo} (with $\alpha$ and $\sA$ there replaced by $(j,\alpha)$ and $\Z\times\sA$). Quantizing $a\in S^m_{\rm I}({}^\cV T^*M)$ and using that $T_{(j,\alpha)}-T'_{(j,\alpha)}=\frac{\rho_{\alpha,0}}{\tau_\alpha}(t-t')$ almost produces the expression~\eqref{EqFFTCharOp}, except $\phi_\alpha^\tau$ is replaced by $\phi_{(j,\alpha)}$ (for an arbitrary $j$ by translation-invariance, say $j=0$ for concreteness), $\tau$ by $\tau_\alpha$, and (solely for notational clarity) $\zeta$ by $\zeta_\alpha$. Changing variables via $\frac{\zeta_\alpha}{\tau_\alpha}=\frac{\zeta}{\tau}$, i.e.\ $\zeta_\alpha=\frac{\tau_\alpha}{\tau}\zeta$, yields~\eqref{EqFFTCharOp} except with $\psi(x'_\alpha-x_\alpha)\psi(\frac{\rho_{\alpha,0}}{\tau_\alpha}(t-t'))$ in place of $\psi(x'_\alpha-x_\alpha)\psi(\frac{\rho_0}{\tau}(t-t'))$. Since the difference of these two cutoffs vanishes near the diagonal, replacing the former cutoff by the latter produces a $t$-translation invariant error of class $\rho^\infty\Psi_\cV^{-\infty}(M)$ by Lemma~\ref{LemmaSBPsdoOffDiag}.
\end{proof}

\begin{prop}[Spectral family: pseudodifferential operators]
\label{PropFFTSpec2}
  Let $a\in w S^m_{\rm I}({}^\cV T^*M)$, $A=\Op_{\cV,\rm I}(a)\in w\Psi_{\cV,\rm I}^m(M)$. Define the spectral family $\hat A=(\hat A_\sigma)_{\sigma\in\R}$ of $A$ as in~\eqref{EqFFTSpec}. Then there exists a residual symbol $r_\sigma\in w\rho_X^\infty\la\tau\sigma\ra^{-\infty}S^{-\infty}({}^{\hat\cV}T^*X)$ (cf.\ \eqref{EqFFThatrho}) so that
  \begin{equation}
  \label{EqFFTSpec2Op}
    \hat A_\sigma=\Op_{\hat\cV_\sigma}(\hat a_\sigma),\qquad \hat a_\sigma = a|_{\sigma\,\dd t+T^*X} + r_\sigma.
  \end{equation}
  In particular,
  \begin{equation}
  \label{EqFFTSpec2Mem}
    \hat A \in w\la\tau\sigma\ra^m\Psi_{\hat\cV}^m(X).
  \end{equation}
\end{prop}
\begin{proof}[Proof of Proposition~\usref{PropFFTSpec2}]
  To compute the spectral family of $\Op_{\alpha,\rm I}(a_\alpha)$ in~\eqref{EqFFTCharOp} at frequency $\sigma\in\R$, multiply by $e^{i(t'-t)\sigma}$, change variables via $\zeta=\tau\sigma'$ and $\xi=\la\tau\sigma'\ra\hat\xi$, and perform the integrations in $t'$ and $\sigma'$ to arrive at
  \begin{align*}
    &(2\pi)^{-(n-1)}\Biggl(\int_{\R^{n-1}} \exp\biggl(i \sum_{j=1}^{n-1}(x_\alpha^j-x_\alpha^{\prime j})\frac{\hat\xi_j}{\rho_{\alpha,j}\la\tau\sigma\ra^{-1}}\biggr) \psi(x'_\alpha-x_\alpha) b_\alpha(x_\alpha,\xi)\,\dd\xi\Biggr) \\
    &\hspace{22em} \times \frac{|\dd x_\alpha^{\prime 1}\cdots\dd x_\alpha^{\prime n-1}|}{\rho_{\alpha,1}\la\tau\sigma\ra^{-1}\cdots\rho_{\alpha,n-1}\la\tau\sigma\ra^{-1}}, \\
    &\qquad b_\alpha(x_\alpha,\hat\xi) = \frac{1}{2\pi}\int_\R \frac{1}{\rho_0}\hat\psi\Bigl(\frac{\tau\sigma-\zeta}{\rho_0}\Bigr)a_\alpha(0,x_\alpha;\zeta,\la\tau\sigma\ra\hat\xi)\,\dd\zeta.
  \end{align*}

  If we had $\psi\equiv 1$, then $b_\alpha(x_\alpha,\hat\xi)=a_\alpha(0,x_\alpha;\tau\sigma,\la\tau\sigma\ra\hat\xi)$; this is uniformly bounded in $w_\alpha\la\tau\sigma\ra^m S^m((-2,2)^{n-1};\R^{n-1}_{\hat\xi})$. It thus suffices to show (in view of $\rho_0\lesssim\rho$ where $\rho$ is a scaling weight on $(M,\fB_\times)$) that every $w_\alpha\rho_0^N\la\tau\sigma\ra^{-N}S^{-N}((-2,2)^n_{x_\alpha};\R^n_{\hat\xi})$-seminorm of $b_\alpha-\hat a_\sigma$ is uniformly (in $\alpha\in\sA$) bounded by a seminorm of $a_\alpha\in w_\alpha S^m((-2,2)^n_{t,x_\alpha};\R^n_{\zeta,\xi})$. We prove the $L^\infty$ bounds; derivatives are estimated similarly. We shorten the notation by considering $a=a(\zeta,\xi)\in S^m$, $b(\hat\xi)=(2\pi)^{-1}\int_\R\frac{1}{\rho_0}\hat\psi(\frac{\tau\sigma-\zeta}{\rho_0})a(\zeta,\la\tau\sigma\ra\hat\xi)\,\dd\zeta$, and using $\int\hat\psi(\zeta)\,\dd\zeta=1$ to write
  \[
    b(\hat\xi) - a(\tau\sigma,\la\tau\sigma\ra\hat\xi) = \frac{1}{2\pi}\int_\R \frac{1}{\rho_0}\hat\psi\Bigl(\frac{\tau\sigma-\zeta}{\rho_0}\Bigr)\bigl( a(\zeta,\la\tau\sigma\ra\hat\xi) - a(\tau\sigma,\la\tau\sigma\ra\hat\xi) \bigr)\,\dd\zeta.
  \]
  Taylor expanding the second factor in the integrand around $\zeta=\tau\sigma$ to some order $N\in\N$, the integrals of the terms involving $\frac{1}{j!}(\pa_\zeta^j a)(\tau\sigma,\la\tau\sigma\ra\hat\xi) (\zeta-\tau\sigma)^j$, $1\leq j<N$, vanish since $(\pa^j\psi)(0)=0$. The absolute value of the remainder term is $\frac{1}{2\pi(N-1)!}$ times
  \begin{align*}
    &\biggl|\int_0^1 \int (1-\lambda)^{N-1} (\pa_\zeta^N a)(\tau\sigma+\lambda(\zeta-\tau\sigma),\la\tau\sigma\ra\hat\xi) (\zeta-\tau\sigma)^N \frac{1}{\rho_0}\hat\psi\Bigl(\frac{\tau\sigma-\zeta}{\rho_0}\Bigr)\,\dd\zeta\,\dd\lambda\biggr| \\
    &\qquad \leq C_{N'}\int_0^1 \int |(\pa_\zeta^N a)(\tau\sigma+\lambda\rho_0\nu,\la\tau\sigma\ra\hat\xi)| \rho_0^N\nu^N \la\nu\ra^{-N'}\,\dd\nu\,\dd\lambda \\
    &\qquad \leq C C_{N'}\rho_0^N\int_0^1 \int (1+|\tau\sigma+\lambda\rho_0\nu|+\la\tau\sigma\ra|\hat\xi|)^{m-N} \nu^N\la\nu\ra^{-N'}\,\dd\nu\,\dd\lambda,
  \end{align*}
  where we performed the change of variables $\nu=\frac{\zeta-\tau\sigma}{\rho_0}$ and used that $|\hat\psi(-\nu)|\leq C_{N'}\la\nu\ra^{-N'}$ for all $N'$; and $C$ is a seminorm of $a$. For $N>m$, and taking $N'=N+2$, this is bounded by a constant times $\rho_0^N(1+|\hat\xi|)^{\frac{m-N}{2}}(1+|\tau\sigma|)^{\frac{m-N}{2}}$ since $\int_\R\nu^N\la\nu\ra^{-N'}\,\dd\nu<\infty$. Since $N$ is arbitrary, the proof is complete.
\end{proof}

Consider the map
\begin{equation}
\label{EqFFTPsi}
  F\colon\R\times T^*X\ni(\sigma,\varpi)\mapsto\sigma\,\dd t+\varpi\in \sigma\,\dd t+T^*X \subset T^*_{\{0\}\times X}(\R\times X).
\end{equation}
In the local trivializations of ${}^{\hat\cV_\sigma}T^*X$, resp.\ ${}^\cV T^*M$ over $U_\alpha$, resp.\ $\R\times U_\alpha$
\begin{alignat*}{3}
  &(-2,2)^{n-1}\times\R^{n-1}&&\ni(x_\alpha,\hat\xi)&&\mapsto\sum_{i=1}^{n-1} \hat\xi_i\frac{\dd x_\alpha^i}{\rho_{\alpha,i}\la\tau\sigma\ra^{-1}}, \\
  \text{resp.}\ \ &(\R\times(-2,2)^{n-1})\times\R^n&&\ni(t,x_\alpha;\zeta,\xi) &&\mapsto \zeta\,\frac{\dd t}{\tau}+\sum_{i=1}^{n-1} \xi_i\frac{\dd x_\alpha^i}{\rho_{\alpha,i}},
\end{alignat*}
this is given by $F\colon(\sigma,x_\alpha,\hat\xi)\mapsto(0,x_\alpha;\tau\sigma,\la\tau\sigma\ra\hat\xi)$. Therefore, $F$ defines a uniformly bounded smooth map $F\colon{}^{\hat\cV}T^*X\to{}^\cV T^*M$. Thus, pullback along $F$ defines a map $\cC^0_{{\rm uni},\fB^*}({}^\cV T^*M)\to\cC^0_{{\rm uni},\hat\fB^*}({}^{\hat\cV}T^*X)$ for the associated phase space b.g.\ structures, and thus a map $\fF\colon\hat\fM\to\fM$ of the microlocalization loci of corresponding to $\hat\cB_\times$ and $\cB_\times$. Now, given a $t$-translation invariant symbol $a=a(0,x_\alpha;\zeta,\xi)$, formula~\eqref{EqFFTSpec2Op} expresses $\hat a_\sigma$ modulo a residual symbol as
\begin{equation}
\label{EqFFTSpec2Expl}
  \hat a_\sigma(x_\alpha,\hat\xi)\equiv (F^*a)(0,x_\alpha,\hat\xi) = a(0,x_\alpha;\tau\sigma,\la\tau\sigma\ra\hat\xi).
\end{equation}
This implies that $\Ell(\hat A)=\fF^{-1}(\Ell(A))$ and $\WF'(\hat A)=\fF^{-1}(\WF'(A))$.

\begin{prop}[Fourier transform: constant orders]
\label{PropFFTHV}
  Let $s\in\R$, and let $w\in\CI(X)$ be a weight on $(X,\fB_X)$. Then the Fourier transform in $t$ defines an isomorphism
  \begin{equation}
  \label{EqFFTHV}
    \cF \colon w H_\cV^s(\R\times X) \to L^2\bigl(\R_\sigma;w\la\tau\sigma\ra^{-s}H_{\hat\cV_\sigma}^s(X)\bigr).
  \end{equation}
\end{prop}
\begin{proof}
  It suffices to consider the case $w=1$. For $s=1$, \eqref{EqFFTHV} follows from~\eqref{EqFFTH1} upon summing in $\alpha$ using the finite intersection property of b.g.\ structures (Definition~\ref{DefIBdd}\eqref{ItIBddFinite}), and the case of general $s\in\N$ is a simple generalization. We can treat all real $s\geq 0$ via testing with an elliptic $\cV$-ps.d.o.\ $\Lambda=\Op_{\cV,\rm I}(\lambda)\in\Psi_{\cV,\rm I}^s(M)$ since $u\in H_\cV^s(\R\times X)$ if and only if $u,\Lambda u\in L^2(\R\times X)$, which is equivalent to $\sigma\mapsto \hat u(\sigma),\hat\Lambda_\sigma\hat u(\sigma)$ lying in $L^2(\R_\sigma;L^2(X))$. But since by Proposition~\ref{PropFFTSpec2} the spectral family $\hat\Lambda=(\hat\Lambda_\sigma)_{\sigma\in\R}\in\la\tau\sigma\ra^s\Psi_{\hat\cV}^s(X)$ is elliptic, this is in turn equivalent to $\hat u(\cdot)\in L^2(\R_\sigma;\la\tau\sigma\ra^{-s}H_{\hat\cV_\sigma}^s(X))$.

  The case of $s\leq 0$ follows by duality using Proposition~\ref{PropSBHDual}.
\end{proof}

\begin{example}[Applications]
\fakephantomsection
\label{ExFFTHV}
  \begin{enumerate}
  \item In the setting of Example~\ref{ExFFTEucl}\eqref{ItFFTEuclsc}, this recovers the elementary isomorphism $\cF\colon H^s(\R\times\R^{n-1})\to L^2(\R;\la\sigma\ra^{-s}H_{\la\sigma\ra^{-1}}^s(\R^{n-1}))$ where $\|u\|_{H_h^s}=\|\la h D\ra^s\hat u\|_{L^2}$; passing to logarithmic coordinates via $t=-\log\rho$, this is a result about the Mellin transform which is stated in \cite[(3.9)]{VasyMicroKerrdS}.
  \item In the 3b-setting of Example~\ref{ExFFTEucl}\eqref{ItFFTEucl3b}, we may take $\tau=\la r\ra$, and thus $\hat\cV_\sigma=\la\tau\sigma\ra^{-1}\hat\cV_0=(1+|\sigma|\la r\ra)^{-1}\Vb(\ol{\R^{n-1}})$; for $\sigma\in\pm[0,1)$, this gives rise to the scattering-b-transition algebra from~\S\ref{SssIF}\eqref{ItIscbt}, and for $\sigma\in\pm[1,\infty)$ to the semiclassical scattering algebra from~\S\ref{SssIF}\eqref{ItIsch}. Thus, Proposition~\ref{PropFFTHV} recovers \cite[Proposition~4.24]{Hintz3b}. (Also \cite[Proposition~4.26]{Hintz3b} can be recovered via the introduction of a different scaled b.g.\ structure via Definition~\ref{DefFFTbg}; we leave this to the interested reader.)
  \end{enumerate}
\end{example}

The case of $t$-translation invariant variable orders $\sfs,\sfl\in S^0_{\rm I}({}^\cV T^*M)$ can be treated in a completely analogous fashion, namely via testing by $t$-translation invariant $\cV$-ps.d.o.s of variable order. To wit, let $a\in w\rho_X^\sfl S_{\rm I}^\sfs({}^\cV T^*M)$; then formula~\eqref{EqFFTSpec2Op} remains valid, and $\hat a=(\hat a_\sigma)_{\sigma\in\R}$ defines a symbol of variable order
\begin{subequations}
\begin{equation}
\label{EqFFTHV3Symb}
  \hat a \in w\rho_X^{\hat\sfl}\la\tau\sigma\ra^{\hat\sfs}S^{\hat\sfs}({}^{\hat\cV}T^*X),\qquad \hat\sfl_\sigma=\sfl|_{\sigma\,\dd t+T^*X},\ \ \hat\sfs_\sigma=\sfs|_{\sigma\,\dd t+T^*X}.
\end{equation}
(See~\eqref{EqFFTSpec2Expl} for the local coordinate expressions of $\hat\sfl_\sigma$ and $\hat\sfs_\sigma$.) Therefore, the spectral family of $\Op_{\cV,\rm I}(a)$ defines an element $\hat A\in w\rho_X^{\hat\sfl}\la\tau\sigma\ra^{\hat\sfs}\Psi_{\hat\cV}^{\hat\sfs}(X)$. Proposition~\ref{PropFFTHV} then generalizes to the isomorphism
\begin{equation}
\label{EqFFTHV3Sob}
  \cF \colon w\rho_X^\sfl H_\cV^\sfs(\R\times X) \to L^2\bigl(\R_\sigma; w\rho_X^{\hat\sfl_\sigma}\la\tau\sigma\ra^{-\hat\sfs_\sigma}H_{\hat\cV_\sigma}^{\hat\sfs_\sigma}(X)\bigr),
\end{equation}
\end{subequations}
where the space on the right is the space of all elements of $L^2(\R_\sigma;w\rho_X^{l_0}\la\tau\sigma\ra^{-s_0}H_{\hat\cV_0}^{s_0}(X))$ (with $l_0\leq\inf\sfl$ and $s_0\leq\inf\sfs$) whose image under a fixed elliptic element of $w\rho_X^{\hat\sfl}\la\tau\sigma\ra^{\hat\sfs}\Psi_{\hat\cV}^{\hat\sfs}(X)$ lies in $L^2(\R_\sigma;L^2(X))$. As an application, this can be used to recover \cite[Proposition~4.29]{Hintz3b}.

\subsection{Commutator \texorpdfstring{$\cW$}{W}-vector fields}
\label{SsFComm}

We record a sharpening of $\Psi_\cV$ of interest in some applications, e.g.\ \cite[\S{2.5.4}]{HintzGlueLocII}. We only discuss the scaled b.g.\ setting, and leave the purely notational modifications for parameterized settings to the reader.

\begin{definition}[Commutator $\cW$-vector fields]
\label{DefFComm}
  We write $\cW_{[\cV]}\subset\cW$ for the space of all $W\in\cW$ so that $[W,V]\in\cV$ for all $V\in\cV$.
\end{definition}

The key assumption in this section is that
\begin{equation}
\label{EqFComm}
  \text{$\cW_{[\cV]}$ spans $\cW$ over $\CI_{{\rm uni},\fB}(M)$},
\end{equation}
i.e.\ every $W\in\cW$ can be written as a finite linear combination $\sum f_i W_i$ where $f_i\in\CI_{{\rm uni},\fB}(M)$ and $W_i\in\cW_{[\cV]}$. Let us consider the conditions on $W\in\cW$ to lie in $\cW_{[\cV]}$ in distinguished charts: writing $(\phi_\alpha)_*W=\sum_{k=1}^n W^k(x)\pa_k$, we compute
\begin{equation}
\label{EqFCommCalc}
  [V^j\rho_{\alpha,j}\pa_j,(\phi_\alpha)_*W] = \Bigl(\frac{\rho_{\alpha,k}V^k(\pa_k W^j)}{\rho_{\alpha,j}} - W^k(\pa_k V^j)\Bigr)\rho_{\alpha,j}\pa_j.
\end{equation}
This has uniformly (in $\alpha$) bounded coefficients in $\CI([-\frac32,\frac32]^n)$ for all $V\in\cV$ (expressed in local coordinates as $V^j\rho_{\alpha,j}\pa_j$) if and only if $\frac{\rho_{\alpha,k}\pa_k W^j}{\rho_{\alpha,j}}$ is uniformly bounded in $\CI([-\frac32,\frac32]^n)$. If, say, $\rho_{\alpha,j}\to 0$ but $\rho_{\alpha,k}=1$ for some sequence of $\alpha,j,k$, this requires $W^j$ to vary only by an increasingly small amount $\lesssim\rho_{\alpha,k}$ in the $x^k$-direction on $U^\alpha$. This shows that elements of $\cW_{[\cV]}$ cannot be constructed using a partition of unity from purely local constructions in each $U_\alpha$. (In particular, one typically has $\cW_{[\cV]}\subsetneq\cW$; equivalently, $\cW_{[\cV]}$ is \emph{not} a $\CI_{{\rm uni},\fB}(M)$-module.) Rather, they have a global character.

\begin{example}[Commutator vector fields in the edge-setting]
\label{ExFComm}
  Let $\bar M$ be compact with fibered boundary $Z-\pa\bar M\xra{\pi} Y$, and consider on the interior $M=\bar M^\circ$ the scaled b.g.\ structure with coefficient Lie algebra $\cW=\Vb(\bar M)$ and operator Lie algebra $\cV=\Ve(\bar M)$ (with conormal coefficients), cf.\ \S\ref{SssIF}\eqref{ItIe}. Then
  \[
    \cW_{[\cV]}\supset\bigl\{ W\in\CI(\bar M;\Tb\bar M) \colon W|_{\pa\bar M}\ \text{is a lift of a vector field on $Y$} \bigr\}.
  \]
  In local coordinates $x\geq 0$, $y\in\R^{n_Y}$, $z\in\R^{n_Z}$, elements of the space on the right are of the form $a(x,y,z)x\pa_x+(b(y)+x c(x,y,z))\pa_y+d(x,y,z)\pa_z$. Thus, condition~\eqref{EqFComm} is satisfied in this case. This example is related to \cite[\S{5.1}]{HintzVasyScrieb}. (Another example is discussed in \cite[\S{2.1.1}]{HintzNonstat}.)
\end{example}

Under the assumption~\eqref{EqFComm}, we shall write $w\rho^\infty\Psi_{\cV,[\,]}^{-\infty}(M)$ for the space of all residual operators $R\in w\rho^\infty\Psi_\cV^{-\infty}(M)$ in Definition~\ref{DefSBPsdoRes} which satisfy the additional condition
\begin{equation}
\label{EqFCommRes}
  \ad_{W_1}\cdots\ad_{W_k}R,\ \ad_{W_1}\cdots\ad_{W_k}R^* \in w\rho^\infty\Psi_\cV^{-\infty}(M)\qquad\forall\,k\in\N,\ W_1,\ldots,W_k\in\cW_{[\cV]}.
\end{equation}
Here $\ad_{W_1}=[W_1,\cdot]$ is the commutator. (In the case of operators acting on vector bundles, we require the same for all elements $W_1,\ldots,W_k$ of $\Diff_\cW^1$ whose principal symbols are scalar and agree with those of some elements of $\cW_{[\cV]}$.) We remark that under condition~\eqref{EqFCommRes}, it suffices to assume the validity of~\eqref{EqSBPsdoResMap} for $R,R^*$ for $k=0$.

\begin{prop}[Sharper ps.d.o.\ algebra]
\label{PropFCommPsdo}
  Assume that~\eqref{EqFComm} holds. Define
  \begin{equation}
  \label{EqFCommPsdo}
    \rho^{-l}\Psi_{\cV,[\,]}^m(M)\subset\rho^{-l}\Psi_\cV^m(M)
  \end{equation}
  as the space of operators of the form $A=\Op_\cV(a)+R$ where $a\in\rho^{-l}S^m({}^\cV T^*M)$ and $R\in\rho^\infty\Psi_{\cV,[\,]}^{-\infty}(M)$. Then $\ad_{W_1}\cdots\ad_{W_k}A\in\rho^{-l}\Psi_{\cV,[\,]}^m(M)$ for all $A\in\rho^{-l}\Psi_{\cV,[\,]}^m(M)$, $k\in\N$, $W_1,\ldots,W_k\in\cW_{[\cV]}$. Furthermore, $\bigcup_{m,l\in\R}\rho^{-l}\Psi_{\cV,[\,]}^m(M)$ is closed under compositions and adjoints. Analogous statements hold for operators with weights, phase space weights, and variable orders.
\end{prop}

We reiterate that it is in general \emph{not} true that $[\Psi_{\cV,[\,]}^m,W]\subset\Psi_{\cV,[\,]}^m$ for \emph{all} $W\in\cW$; this already fails in the case $m=1$ for the subspace $\cV\subset\Psi_{\cV,[\,]}^1$.

\begin{proof}[Proof of Proposition~\usref{PropFCommPsdo}]
  We first check the analogue of Corollary~\ref{CorSBPsdoResSKMult}. Thus, let $R\in\rho^\infty\Psi_{\cV,[\,]}^{-\infty}(M)$ with Schwartz kernel $K_R$, and let $f\in\CI_{{\rm uni},\fB\times\fB}(M\times M)$. Then $f K_R$ defines an element $R'$ of $w\rho^\infty\Psi_\cV^{-\infty}(M)$ by Corollary~\ref{CorSBPsdoResSKMult}; we need to show that $f K_R\in\rho^\infty\Psi_{\cV,[\,]}^{-\infty}(M)$. If now $W\in\cW_{[\cV]}$, then the Schwartz kernel of $[W,R']$ is equal to $W_1(f K_R)-W_2^*(f K_R)$ where the vector field $W_j$ on $M\times M$ is equal to $W$ acting on the first (for $j=1$), resp.\ second (for $j=2$) factor, and $W_2^*$ acts on right densities. Upon trivializing the right density bundle using a uniform $\cW$-density, we have $W_2^*\equiv-W_2\bmod\CI_{{\rm uni},\fB}(M)$. It then suffices to observe that
  \[
    W_1(f K_R)-W_2^*(f K_R)=f(W_1 K_R-W_2^*K_R)+((W_1+W_2)f)K_R,
  \]
  with the first summand being $f$ times the Schwartz kernel of $[W,R]$, and the second one being $(W_1-W_2)f\in\CI_{{\rm uni},\fB\times\fB}(M\times M)$ times the Schwartz kernel of $R$. Higher order commutators are treated similarly.

  Next, we prove that the conclusion in Proposition~\ref{PropSBPsdoOpRes} can be strengthened to membership in $w\rho^\infty\Psi_{\cV,[\,]}^{-\infty}(M)$ (which then implies an analogous strengthening of Lemma~\ref{LemmaSBPsdoOffDiag}). To this end, we shall prove the following statement for $a\in S^m((-2,2)^n;\R^n)$ and $W\in\cW_{[\cV]}$: write $(\phi_\alpha)_*W=\sum_{k=1}^n W^k\pa_k$; then the operator whose Schwartz kernel is given by the action of $\wt W:=\sum_{j=1}^n W^j(x_\alpha)\pa_{x_\alpha^j}+W^j(x'_\alpha)\pa_{x_\alpha^{\prime j}}$ on the quantization
  \begin{equation}
  \label{EqFCommPsdoQuant}
    K_a(x_\alpha,x'_\alpha) := (2\pi)^{-n}\int_{\R^n} e^{i\sum_{j=1}^n(x_\alpha^j-x_\alpha^{\prime j})\xi_j/\rho_{\alpha,j}} \psi(x'_\alpha-x_\alpha) a(x_\alpha,\xi)\,\dd\xi
  \end{equation}
  is itself such a quantization (upon replacing $\psi\in\CIc((-\frac32,\frac32)^n)$ by a cutoff of the same class which is equal to $1$ on $\supp\psi$) for a new symbol of class $S^m$ which is uniformly (in $\alpha$) bounded by $a$. The only term in $\wt W K_a$ which is not of the same form as $K_a$ itself (upon enlarging the cutoff) arises from differentiating the exponential. The latter produces
  \begin{align*}
    &\biggl(\sum_{j=1}^n \frac{W^j(x_\alpha)-W^j(x'_\alpha)}{\rho_{\alpha,j}}\xi_j\biggr) i e^{i\sum_{j=1}^n (x_\alpha^j-x_\alpha^{\prime j})\xi_j/\rho_{\alpha,j}} \\
    & \qquad = \sum_{j,k=1}^n \int_0^1 \frac{\rho_{\alpha,k}(\pa_k W^j)(x'_\alpha+t(x_\alpha-x'_\alpha))}{\rho_{\alpha,j}}\,\dd t \times (\xi_j\pa_{\xi_k})e^{i\sum_{j=1}^n (x_\alpha^j-x_\alpha^{\prime j})\xi_j/\rho_{\alpha,j}}.
  \end{align*}
  Plugging this into~\eqref{EqFCommPsdoQuant}, we can shift $\xi_j$ onto $a_\alpha$ and integrate by parts in $\xi_k$. The uniform $\CI$ bounds on $\frac{\rho_{\alpha,k}\pa_k W^j}{\rho_{\alpha,j}}$ (cf.\ the discussion following~\eqref{EqFCommCalc}) then yield the desired claim upon left reduction. Higher order commutators are analyzed by iterating this argument.

  The same arguments show that for $A=\Op_\cV(a)$ with $a\in S^m({}^\cV T^*M)$, we have $[W,A]\in\Psi_{\cV,[\,]}^m(M)$ for all $W\in\cW_{[\cV]}$; similarly for higher order commutators. The algebra property of $\bigcup_{m,l\in\R}\rho^{-l}\Psi_{\cV,[\,]}^m(M)$ now follows from that of $\rho^\infty\Psi_{\cV,[\,]}^{-\infty}(M)$. The proof is complete.
\end{proof}

\bibliographystyle{alphaurl}


\end{document}